\documentclass{article}
\usepackage[left=3.6cm,right=3.6cm]{geometry}
\usepackage{amssymb,amsthm,amsmath,amsxtra,dsfont}
\usepackage[nottoc]{tocbibind}   
\usepackage{hyperref} \hypersetup{colorlinks=true, citecolor=blue, linkcolor=black}
\numberwithin{equation}{section} 
\usepackage{mathrsfs}
\usepackage{cancel}
\usepackage{tikz}
\usepackage{lscape}
\usetikzlibrary{decorations}
\usetikzlibrary{decorations.markings,arrows}
\usepackage{enumitem}
\usepackage[colorinlistoftodos]{todonotes}

\usetikzlibrary{decorations}
\usetikzlibrary{decorations.markings,arrows}

\usepackage{graphicx}
\usepackage{pgfplots}
\newcommand{\bigO}{\mathcal{O}}

\pgfplotsset{every axis/.append style={
                    axis x line=middle,    
                    axis y line=middle,    
                    axis line style={->}, 
                    ymajorticks=false,
                    xtick={-1,1},                     
                    }}
\tikzset{>=stealth}

\newcommand{\1}{\mathbf{1}}

\newcommand{\C}{\mathbb{C}}

\newcommand{\E}{\mathds{E}}

\newcommand{\h}{h_N}

\renewcommand{\i}{\mathbf{i}}

\newcommand{\N}{\mathbb{N}}
\renewcommand{\O}{\mathcal{O}}
\renewcommand{\o}{o}

\newcommand{\X}{\mathrm{X}}
\renewcommand{\P}{\mathds{P}}
\newcommand{\q}{\mathbb{Q}}

\newcommand{\R}{\mathbb{R}}

\renewcommand{\S}{[-1,1]}

\newcommand{\supp}{\operatorname{supp}}
\newcommand{\tr}{\operatorname{Tr}}

\newcommand{\U}{\mathscr{U}}

\newtheorem{theorem}{Theorem}[section]
\newtheorem{definition}[theorem]{Definition}
\newtheorem{proposition}[theorem]{Proposition}
\newtheorem{corollary}[theorem]{Corollary}
\newtheorem{lemma}[theorem]{Lemma}
\newtheorem{assumption}[theorem]{Assumptions}

\theoremstyle{definition} \newtheorem{remark}{Remark}[section]

\title{How much can the eigenvalues of a random Hermitian matrix fluctuate?}
\date{}
\author{T. Claeys, B. Fahs, G. Lambert and C. Webb}

\begin{document}

\maketitle

\begin{abstract}
The goal of this article is to study how much the eigenvalues of large Hermitian random matrices deviate from certain deterministic locations -- or in other words, to investigate optimal rigidity estimates for the eigenvalues. We do this in the setting of one-cut regular unitary invariant ensembles of random Hermitian matrices -- the Gaussian Unitary Ensemble being the prime example of such an ensemble.

Our approach to this question combines extreme value theory of log-correlated stochastic processes, and in particular the theory of multiplicative chaos, with asymptotic analysis of large Hankel determinants with Fisher-Hartwig symbols of various types, such as merging jump singularities, size-dependent impurities, and jump singularities approaching the edge of the spectrum. In addition to optimal rigidity estimates, our approach sheds light on the fractal geometry of the eigenvalue counting function.
\end{abstract}

\tableofcontents

\clearpage




\section{Introduction and statement of results}\label{sec:intro}

A fundamental problem in the study of large random matrices is to describe the asymptotic behavior of the eigenvalues of a random matrix as its size tends to infinity. The most classical result of this flavor, due to Wigner, states that the empirical distribution of the eigenvalues of any Wigner matrix\footnote{A Hermitian Wigner matrix is a random (Hermitian) matrix whose entries are independent up to the symmetry constraint, and have zero expectation as well as constant variance.}  converges, after an appropriate normalization, almost surely to the semi-circle law as the size of the matrix tends to infinity  --  see e.g. \cite[Chapter 2]{AGZ}.

Nowadays, it is well known that the empirical distribution can be approximated by the semi-circle law not only on the macroscopic scale, namely the setting of Wigner's theorem, but also on all mesoscopic scales. This is known as the local law and is a far-reaching generalization of Wigner's theorem. Moreover, it is  a key tool for proving universality of local statistics, see \cite{ESY11, BEYY16}. 

Local laws hold for a very large class of random matrix models ranging from generalized Wigner matrices and $\beta$--ensembles, to adjacency matrices of random graphs, see e.g. \cite{AEK17,  BEY14, EKYY13a} for further references.
An important feature of a local law is that it allows to obtain estimates for the maximal fluctuations of eigenvalues -- such results are known as rigidity estimates, see e.g. \cite[Theorem 2.4]{BEY14}. In this article, we strengthen the available estimates for the fluctuations of eigenvalues for a class of random matrices known as unitary invariant one-cut regular random Hermitian matrices. More precisely, we establish optimal upper and lower bounds for these maximal fluctuations.

Before describing our model and results in full detail, let us focus on the Gaussian Unitary Ensemble (GUE), which is the most basic example of a one-cut regular unitary invariant Hermitian ensemble, and also of a Wigner matrix. In our normalization convention, the entries of an $N\times N$ GUE random matrix are i.i.d. centered Gaussians random variables (complex above the diagonal and real on the diagonal) with variance $\frac{1}{4N}$. In this setting, Wigner's theorem states that the empirical measure of the (ordered) eigenvalues $\lambda_1\leq ...\leq \lambda_N$ of such a matrix converges to the semi-circle law $d\mu_{{\rm sc}}(x)=\frac{2}{\pi}\sqrt{1-x^2} dx$ supported on the interval $[-1,1]$.

Our main result is concerned with the fluctuations of the eigenvalues $\lambda_j$ around the quantiles of $\mu_{{\rm sc}}$: to describe this, we introduce the quantiles $\kappa_j\in(-1,1]$ through the relation
\begin{equation*}
\int_{-1}^{\kappa_j}d\mu_{{\rm sc}}= \frac{j}{N},  \qquad j=1,\dots,N.
\end{equation*}
With this notation in hand, we can state our main result (Theorem \ref{thm:rigidity} below) in the setting of the GUE: for any $\epsilon>0$
\begin{equation}\label{eq:GUErig}
\lim_{N\to\infty}\P\left(1-\epsilon<\max_{j=1,\dots,N}\left\{\frac{2N}{\log N}\sqrt{1-\kappa_j^2}|\lambda_j-\kappa_j|\right\} <1+\epsilon\right)=1.
\end{equation}	
The interpretation of the result is that with high probability, bulk eigenvalues can not fluctuate around the corresponding quantile more than (a constant times) $\log N/N$, but again with high probability, there do exist eigenvalues which fluctuate this much.	

Before turning to describing our results more precisely, let us briefly compare \eqref{eq:GUErig} to existing results about the fluctuations of the eigenvalues in the setting of the GUE, as well as offer a short discussion of our approach to the problem of rigidity estimates. Gustavsson \cite{Gustavsson} studied the typical deviation of an individual eigenvalue from its quantile and found that for any $\epsilon>0$ and for any $j\in (\epsilon N,(1-\epsilon)N)$, 
\begin{equation}\label{BulkGus}
2\sqrt{2}\sqrt{1-\kappa_j^2}\frac{N}{\sqrt{\log N}}(\lambda_j-\kappa_j)\to \mathcal N(0,1),
\end{equation}
in distribution as $N\to \infty$, where $\mathcal N(0,1)$ denotes a standard Gaussian law. This result states that  an eigenvalue typically  fluctuates on the scale $\sqrt{\log N}/N$, and not on the scale $\log N/N$, which according to \eqref{eq:GUErig} describes the maximal fluctuations. Note that \eqref{BulkGus} can be viewed as a lower bound for the maximal  fluctuations -- so from this point of view, \eqref{eq:GUErig} can be seen as an improvement of such a lower bound.

Concerning upper bounds for the fluctuations, we already mentioned that there exist results concerning rigidity estimates in various settings. The first results in this direction are due to  
Erd\H{o}s, Schlein and Yau, see e.g. \cite[Theorem 2.1]{ESY09}, which concerns Hermitian Wigner matrices, and gives large deviations estimates for the number of eigenvalues in suitable intervals. For a more direct comparison to \eqref{eq:GUErig}, we refer to \cite[Theorem 2.2]{EYY}, which holds in the setting of generalized Wigner matrices (including the GUE) and states that there exist constants $\alpha>\alpha'>0$ and $C,c>0$ such that
\begin{equation}\label{EYY}
\P\left( \max_{j=1,\dots,N}\left\{\sqrt{1-\kappa_j^2}|\lambda_j-\kappa_j|\right\}\ge \frac{(\log N)^{\alpha \log\log N}}{N} \right) \le C \exp\left(- c (\log N)^{\alpha' \log \log N} \right),
\end{equation}
which should be understood as stating that with overwhelming probability, the fluctuations are no larger than $(\log N)^{\alpha \log\log N}/N$. Thus \eqref{eq:GUErig} is an improvement of both \eqref{BulkGus} and \eqref{EYY} in that we identify the optimal order of the fluctuations, and actually go even further by proving a precise limit theorem.

Our approach to the proof of \eqref{eq:GUErig} is very different from those of \eqref{BulkGus} and \eqref{EYY},  and involves a combination of techniques. Our starting point in the proof of our rigidity estimates is note that such estimates are equivalent to precise asymptotics for the (centered) eigenvalue counting function -- a random function which counts the number of eigenvalues to the left of a given point $x$. A well known central limit theorem involving eigenvalues of random matrices, due to Johansson \cite[Theorem 2.4]{Johansson}, can be formulated as stating that in the large $N$ limit, this centered eigenvalue counting function behaves like a log-correlated stochastic process, namely a stochastic process whose covariance has a logarithmic singularity on the diagonal, and our problem is closely related to the extreme value theory of such stochastic processes.

While similar observations have been made and advantageously applied in the setting of random unitary matrices as well as  for circular $\beta$-ensembles \cite{ABB16,CMN16,HKC01}, we  offer a novel approach to the extreme value theory of  (asymptotically  Gaussian) log-correlated processes, which we believe to be of independent interest, as log-correlated processes appear in various branches of modern probability theory and mathematical physics, ranging from lattice models in statistical mechanics and probabilistic number theory to random geometry and constructive conformal field theory \cite{Kenyon,FK,Najnudel,SW16,AJKS,Sheffield,KRV}. In particular, when studying the extrema of the eigenvalue counting function, we make use of certain random fractal measures, described by the theory of multiplicative chaos. 

To make the connection to the theory of multiplicative chaos, we apply recent ideas introduced in \cite{LOS18}, which in our setting require understanding asymptotics of large Hankel determinants whose symbols have what are known as Fisher-Hartwig singularities. While such objects have been under intensive study in the past years \cite{ItsKrasovsky, DIK11, CF16, BCI16, Charlier}, not all of the estimates we require are available in the literature. In particular, we need to consider the situation where we have two merging jump singularities, a jump singularity approaching an edge, and a situation where we have two merging jump singularities in addition to a dimension-dependent impurity. We believe that the asymptotics for these Hankel determinants are also of independent interest.

We now turn to discussing more precisely the setting and statements of our main results.

\subsection{Main results on eigenvalue rigidity} \label{rigidity_results}
In this section, we describe in detail the ensembles of random Hermitian matrices which we study, and we state our main results concerning rigidity of the associated eigenvalues. We consider the ordered eigenvalues $\lambda_1\leq \lambda_2\leq \cdots\leq \lambda_N$ of a random Hermitian matrix $\mathbf M$, whose law is absolutely continuous with respect to the Lebesgue measure on the space of $N\times N$ Hermitian matrices, and has a density proportional to $e^{-N\tr V(\mathbf M)}=e^{- N\sum_{j=1}^N V(\lambda_j)}$. 
This class is known as unitary invariant ensembles of Hermitian random matrices.
We require that the confining potential $V$ is real--analytic and that it satisfies the growth condition
\begin{equation} \label{V}
\lim_{|x| \to\infty} \frac{V(x)}{\log |x|} = +\infty . 
\end{equation}
It is well--known that the joint law of the ordered eigenvalues of $\mathbf M$  has a density given by
\begin{equation} \label{Z}
\frac{1}{Z_N} \prod_{1\le i<j \le N} \big| \lambda_i  - \lambda_j\big|^2  \prod_{1\le j\le N}e^{-N V(\lambda_j)} , 
\end{equation}
where $0<Z_N<\infty$ is a normalization constant which depends on $V$, see e.g.\ \cite{Deift}. To connect to our discussion concerning the GUE in the previous section, we mention that with $V(x)=2x^2$, this produces precisely the probability distribution of the GUE and its eigenvalues.

Moreover, it is also known (see e.g. \cite[Section 2.6]{AGZ} and \cite{DKML98} as general references for the statements below) that there exists an \emph{equilibrium measure} $\mu_V$,  in the sense that for any bounded continuous test function $f$, 
\begin{equation} \label{lln}
 \frac{1}{N}\tr f(\mathbf{M}) =\frac{1}{N} \sum_{1\le j\le N} f(\lambda_j)  \ \to\ \int f d\mu_V ,
\end{equation}
in probability as $N\to\infty$. This equilibrium measure  $\mu_V$ is the unique minimizer of the functional
\[
\mu\mapsto \int_{\R\times \R} \log|x-y|^{-1}{d\mu(x)d\mu(y)}+\int_\R V(x){d\mu(x)}, 
\]
 over all Borel probability measures on $\R$ and, if $V$ is smooth, it is characterized uniquely by the Euler--Lagrange conditions: there exists a constant $\ell=\ell_V\in \R$ such that 
\begin{align}&\label{eq:el1}
2\int \log|x-y|d\mu_V(y)-V(x)=-\ell&&\text{for } x\in \mathrm{supp}(\mu_V),\\
&\label{eq:el2}
2\int \log|x-y|d\mu_V(y)-V(x)\leq -\ell && \text{for } x\notin \mathrm{supp}(\mu_V).
\end{align}

Moreover for real analytic $V$ satisfying the growth condition \eqref{V}, $\mu_V$ has  a compact support which consists of a union of finitely many  closed intervals. Some aspects of the asymptotic behavior of unitary invariant ensembles of random Hermitian matrices depend significantly on the number of these intervals (see e.g. \cite{DKMVZ2}), and we will focus on the simplest situation, namely the one-interval, or one-cut, case. 
By scaling and shifting the confining potential $V$, we can always assume that this interval is $\S$. To keep our approach as simple as possible, we assume that the  equilibrium measure is  \emph{regular}, which is a technical condition which holds for generic confining potentials \cite{KM}. Another simplifying assumption we make is to assume that in addition to $V$ being real analytic, it actually has an analytic continuation into some infinite horizontal strip. The precise assumptions are as follows.
 
 \begin{assumption}[One--cut regular potential] \label{ass:regular} 
Let $\epsilon>0$, and let
\begin{equation} \label{DefDomV}
\mathcal V=\{z:|\Im z|<\epsilon\}
\end{equation} 
be a strip in the complex plane. We assume that
$V$ is real on $\mathbb R$,  analytic on $\mathcal V$, that $V$ satisfies the growth condition \eqref{V}, and more generally that
\begin{equation}\label{eq:Vgrowth}
\frac{\Re(V(z))}{\log |z|}\to \infty,
\end{equation}
as $|z|\to \infty$, uniformly in $\mathcal V.$ Moreover the inequality \eqref{eq:el2} is strict and the associated equilibrium measure $\mu_V$ has support on $\S$ and it takes the form
\begin{equation}\label{eq:muV}
d \mu_V(x)=\psi_V(x)\sqrt{1-x^2}dx ,
\end{equation}
where the function $\psi_V>0$ on $\S$.
\end{assumption}
 We mention here that with some effort, one could presumably use ideas from \cite{DKMVZ2} to dispense of the regularity assumption and with further effort, possibly replace the assumption of analyticity in a strip by real analyticity, but we choose not to go into these technical issues in this article.

For instance, the GUE, which corresponds to the potential $V(x) = 2x^2$, satisfies Assumptions \ref{ass:regular} with the equilibrium measure being the Wigner semicircle law, $\psi_{V}(x)=2/\pi$. More generally,  it is known that  Assumptions~\ref{ass:regular} hold if in addition to satisfying the conditions concerning analyticity in $\mathcal V$ and the growth rate \eqref{eq:Vgrowth}, the confining potential $V$  is strictly convex on $\R$ and suitably normalized -- see  e.g. \cite[Theorem 11.2.3]{PS11}.

In analogy to the quantiles of the semi-circle law in the case of the GUE, we now define the quantiles $\varkappa_j\in (-1,1]$ of the equilibrium measure by 
\begin{equation} \label{quantile}
\int_{-1}^{\varkappa_j} d\mu_V = \frac{j}{N} , \qquad j=1,\dots, N. 
\end{equation}

We are now ready to state our main result, which we prove in Section \ref{sec:rigidity}.
\begin{theorem} \label{thm:rigidity}
Let $\lambda_1\leq \cdots\leq \lambda_N$ be the ordered eigenvalues in the ensemble \eqref{Z}, with $V$ satisfying Assumptions \ref{ass:regular}, and let $(\varkappa_j)_{j=1}^N$ be as in \eqref{quantile}.
Then for any $\epsilon>0$, 
\begin{equation} \label{rigidity}
\lim_{N\to\infty} \P \left[  \frac{1-\epsilon}{\pi} \frac{\log N}{N} \le \max_{  j = 1, \dots , N} \Big\{  \psi_V(\varkappa_j) \sqrt{1-\varkappa_j^2} |\lambda_j-\varkappa_j|  \Big\}\le \frac{1+\epsilon}{\pi} \frac{\log N}{N}\right]  =1. 
\end{equation}
\end{theorem}
We mention here that the maximum above is unlikely to be realized very close to the edge of the spectrum: see e.g. our Proposition \ref{prop:rigedge} below. Thus, the lower bound from Theorem \ref{thm:rigidity} should be viewed more as a statement concerning bulk eigenvalues with \eqref{rigidity} being a convenient way of stating the result.

Let us briefly elaborate on the comparison between Theorem \ref{thm:rigidity} and Gustavsson's result described in \eqref{BulkGus}. Gustavsson actually proved a stronger result than that of \eqref{BulkGus}: he proved that if we look at a fixed number of eigenvalues of the form $\lambda_{\lfloor a_1 N\rfloor },...,\lambda_{\lfloor a_k N\rfloor}$ with $a_i\in(0,1)$ distinct, then the fluctuations (normalized as in \eqref{BulkGus}), converge to i.i.d. standard Gaussians. Heuristically, this suggests that perhaps Theorem \ref{thm:rigidity} can be understood by treating $\frac{\pi \sqrt{2}N}{\sqrt{\log N}}\psi_V(\varkappa_j)\sqrt{1-\varkappa_j^2}(\lambda_j-\varkappa_j)$ as $N$ i.i.d. standard Gaussians. A basic fact from extreme value theory is that the maximum of $N$ i.i.d. standard Gaussians behaves like $\sqrt{2\log N}$, so this reasoning would indeed seem to give a heuristic justification for Theorem \ref{thm:rigidity}. Unfortunately, such a simple argument can not be turned into a proof. 
More precisely, the lack of independence plays no role in our proof of the upper bound, but to obtain the lower bound, one needs to take into account the correlations between the eigenvalues. The accurate picture comes from viewing the fluctuations of the eigenvalues as  behaving asymptotically like a logarithmically correlated Gaussian field for which  we can study the extremal behavior using the theory of Gaussian multiplicative chaos (GMC) that we describe in Section~\ref{gmc_results} and in further detail in Section~\ref{sect:background}.

The connection to logarithmically correlated Gaussian fields comes through the centered eigenvalue counting function 
\begin{equation} \label{h}
\h(x)  = \sqrt{2} \pi \bigg( \sum_{1\le j\le N} \1_{\lambda_j \le x}  - N \int_{-1}^x d\mu_V\bigg), \qquad  x\in \R .   
\end{equation}
The ad hoc normalization by $\sqrt{2} \pi$ is to be consistent with the convention used in  the GMC literature.
The connection between $h_N$ and the maximal fluctuations of the eigenvalues comes from the following observation: due to monotonicity properties, we have 
\[
\max_{x\in\R} \h(x)  = \max_{j=1,\dots , N}  \h(\lambda_j)  .
\]
For $N$ large, a simple calculation shows that 
\[
\h(\lambda_j) = \sqrt{2}\pi N\int_{\lambda_j}^{\varkappa_j} d\mu_V \approx  \sqrt{2}\pi N\psi_V(\varkappa_j)\sqrt{1-\varkappa_j^2}(\varkappa_j-\lambda_j)
\]  
provided that $\displaystyle\max_{j=1,\dots , N}|\varkappa_j-\lambda_j|$ is small (which we expect to be true with high probability), so at least heuristically,  we obtain with high probability that 
\begin{equation} \label{max_h}
\max_{x\in\R} \h(x)  \approx \sqrt{2}\pi N\max_{j=1,\dots,N}\bigg\{\psi_V(\varkappa_j)\sqrt{1-\varkappa_j^2} (\varkappa_j  - \lambda_j)\bigg\} . 
\end{equation}
A similar approximation holds for  $\min_{x\in\R} \h(x)$. Hence, in order to prove Theorem~\ref{thm:rigidity}, it is sufficient to obtain  precise upper and lower bounds for the random variables $\max_{x\in\R} \h(x)$ and $\min_{x\in\R} \h(x)$.  The following theorem, proven in Section \ref{sect:evcf}, gives such bounds.

\begin{theorem} \label{thm:max2}
Under the same conditions as in Theorem \ref{thm:rigidity}, we have that for any $\delta>0$, 
\[
\lim_{N\to\infty}\P\left[ (1-\delta) \sqrt{2}\log N  \le \max_{x\in\R}\hspace{-.1cm}\big\{\pm h_N(x) \big\} \le (1+\delta)\sqrt{2} \log N \right] =1  .
\]
\end{theorem}

\begin{remark}\label{rem:kolmo}
As a slightly different kind of interpretation of this result, recall that the Kolmogorov-Smirnov distance of two probability measures $\mu$ and $\nu$ on $\R$ is defined by $d_K(\mu,\nu)=\sup_{x\in \R}|\mu(-\infty,x]-\nu(-\infty,x]|$. Thus if we write $\mu^{(N)}=\frac{1}{N}\sum_{j=1}^N\delta_{\lambda_j}$ for the empirical distribution of the eigenvalues, we see that $\mu^{(N)}(-\infty,x]-\mu_V(-\infty,x]=\frac{1}{\sqrt{2}\pi N} \h(x)$, so Theorem \ref{thm:max2} can be reformulated as  
\[
\lim_{N\to\infty}\P\left[ \frac{(1-\delta)}{\pi} \frac{\log N}{N}  \le d_K(\mu^{(N)},\mu_V) \le \frac{(1+\delta)}{\pi}\frac{ \log N}{N} \right] =1  
\]	
for any $\delta>0$. 

The rate of convergence of the empirical distribution to the equilibrium measure has been studied e.g. in the setting of Wigner matrices -- see \cite{GT}, where an upper bound of the form $(\log N)^b/N$  with a unexplicit constant $b>0$ was obtained for the Kolmogorov distance. So in the setting of the GUE, Theorem \ref{thm:max2} can be seen as an improvement of \cite[Theorem 1.1]{GT} in that we establish the precise rate as well as identify the correct multiplicative constant.     	 \hfill $\blacksquare$ 
\end{remark}

The proof of Theorem~\ref{thm:max2} relies on the  log-correlated structure of the random field $\h$ and the connection to the theory of multiplicative chaos that we explain in the next section.

\subsection{Main results on Gaussian multiplicative chaos} \label{gmc_results}

In this section, we state our main results concerning multiplicative chaos and the fractal geometry of the log-correlated field $\h$.

As mentioned earlier, the connection between $\h$ and the theory of log-correlated fields can be seen from Johansson's CLT \cite[Theorem 2.4]{Johansson} for the linear statistics of eigenvalues. 
Namely, for compactly supported smooth functions $f$, as $N\to\infty$, 
\begin{equation}\label{eq:CLT}
\frac{-1}{\sqrt{2}\pi}\int_\R f'(x)\h(x)dx = 
\sum_{j=1}^N f(\lambda_j)-N\int_\R f d\mu_V
\ \Rightarrow\ \mathcal N\big(0,\sigma(f)^2\big),
\end{equation}
where the variance is given by\footnote{Usually, e.g. in \cite{Johansson}, the variance is given in a slightly different form. For a discussion about the equivalence of various representations of it, see Appendix \ref{app:logcor}.}
\begin{equation}\label{eq:cov}
\sigma(f)^2=\iint_{\S^2}f'(x)f'(y)\frac{\Sigma(x,y)}{2\pi^2}dxdy,\qquad 
\Sigma(x,y):=\log\left| \frac{1-xy+\sqrt{1-x^2}\sqrt{1-y^2}}{x-y} \right|.
\end{equation}

For our purposes, the most important  observation is that \eqref{eq:CLT}  and \eqref{eq:cov}  suggest that in the sense of generalized functions, as $N\to\infty$, $\h$ converges to a Gaussian process with covariance kernel $\Sigma$. 
Since $\Sigma(x,y)$ has a logarithmic singularity on the diagonal $x=y$, this means that any possible limit of $\h$ can not be an honest function, as the relevant Gaussian process would have an infinite variance. This is reflected in Theorem \ref{thm:max2} in that the maximum of $h_N$ blows up as $N\to\infty$, while an honest function would have a finite maximum.  See also Figure \ref{fig:simu1} for an illustration of what a realization of $\h$ looks like, and in particular, how it is a rough-looking object which will not have a meaningful limit as $N\to\infty$, at least in the sense of say continuous functions. Moreover, note that $\sqrt{2}\log 6000\approx 12.3$, so the global maximum and minimum in the figure seem to agree quite well with Theorem \ref{thm:max2}.
\begin{figure}
	\centering
		\includegraphics[width=0.9\textwidth]{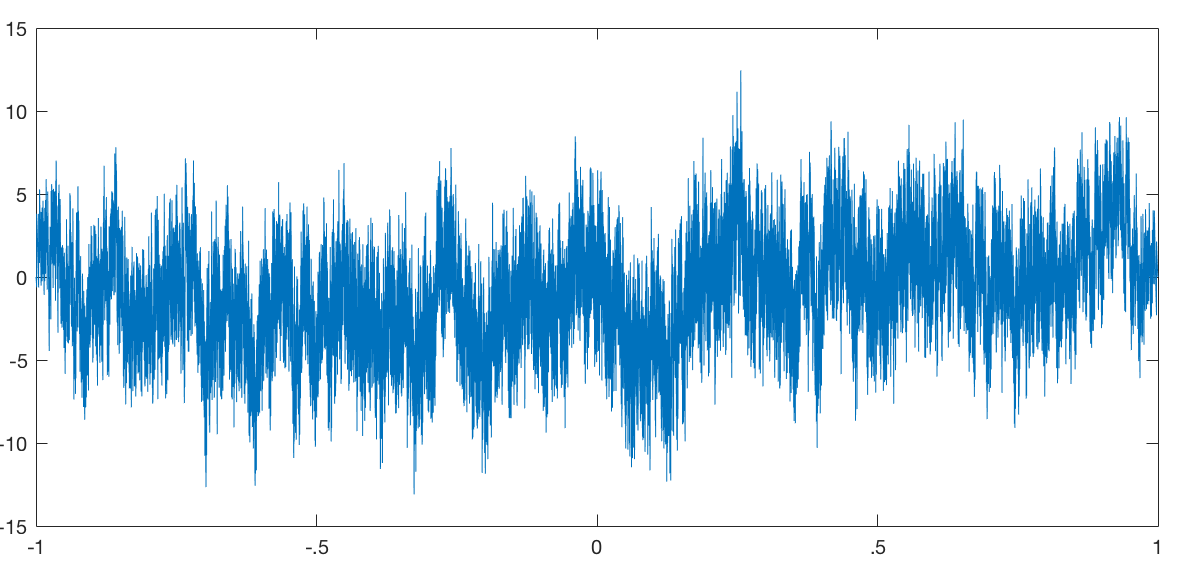} 
	\caption{A realization of $\h$ for $N=6000$ in the setting of the GUE.}\label{fig:simu1}
\end{figure}

\medskip

As is done in \cite{HKC01,FKS16} for related objects, one can make sense of the limiting object as a random element of a suitable space of generalized functions. Such stochastic processes, which are random generalized functions whose correlation kernels have a logarithmic singularity on the diagonal, are termed log-correlated fields. 
While questions about extrema are a priori ill-posed for generalized functions, tools have been developed for studying the extrema of log-correlated fields. For instance, if one smooths the field by convolving with a smooth bump function, then with a modified second moment method, one can extract asymptotics of the extrema of the smoothed field in the limit where the width of the bump function vanishes. We refer the interested reader to \cite{Kistler} and references therein for more details about such questions. 

One expects that some modification of this second moment method could work for studying the extrema of $\h$, which can be seen as another type of approximation of the log-correlated field. In related models, namely looking at similar questions for random unitary matrices and the Riemann zeta function in the neighborhood of a random point on the critical line, this type of approach was originally suggested in \cite{FHK12}, see also \cite{FK}, and has indeed been carried out in many related problems: in \cite{ABB16,CMN16,PZ16} this is done when studying the extrema of the eigenvalue counting function and a closely related object, namely the logarithm of the absolute value of the characteristic polynomial\footnote{The ``close relation" comes from noting that the logarithm of the absolute value of the characteristic polynomial is the real part of the logarithm of the characteristic polynomial, while the eigenvalue counting function can be viewed as the imaginary part of it -- if one chooses the branch of the logarithm suitably.} of the Circular Unitary Ensemble and the Circular $\beta$ Ensemble. This has also been done in the setting of unitary invariant ensembles of random Hermitian matrices (without the assumption of there being only one cut) in the work \cite{LP18}, where the authors consider the maximum of the logarithm of the characteristic polynomial. In \cite{HP18}, the authors study the extrema of the ``counting function" of the sine-$\beta$ process on a large interval. 

Instead of a modified second moment method, we have decided to take a different approach which reveals more of the fractal geometric nature of $\h$.  What we do is to establish a connection to a family of random fractal measures described by the theory of multiplicative chaos introduced by Kahane \cite{Kahane}, and leverage properties of these measures to get information about e.g. the extrema of $\h$. While such an approach is in a sense intuitively understood in the multiplicative chaos community, we are unaware of a reference  where such a study is carried out. We thus hope that also from this point of view, our work will be of independent interest as we expect that it can be generalized to other probabilistic models where log-correlated fields arise. In particular, some of the results of Section~\ref{sect:LB} have already been used in \cite{L19a,L19b}.

In the setting of log-correlated fields (and on an intuitive level) the main insight of Kahane was that if $X$ is some Gaussian log-correlated field, then one can construct from it a family of random measures, formally written as $e^{\gamma X(x)-\frac{\gamma^2}{2}\mathds E X(x)^2}dx$ for $\gamma\in \R$, and the fractal properties of these measures determine properties of extreme values of $X$, or more precisely, extreme values of say convolution approximations of $X$. We direct the reader to \cite[Section 4]{RV14} for precise statements of this flavor and to \cite{Berestycki15,RV14} for a general introduction to the subject. The rigorous construction of these measures goes through approximating the random generalized function e.g.\ by a convolution with a bump function, then exponentiating, normalizing, and removing the approximation. One might hope that $\h$ can serve as a sufficiently good approximation to the log-correlated field in order to carry out this approach. This is not a very far fetched idea, as similar arguments already exist in the literature, see e.g. \cite{FK} for original conjectures, \cite{LOS18,Webb15,Webb18} for related studies in the setting of random unitary matrices and the sine process, and \cite{BWW18} for a connection between the absolute value of the characteristic polynomial of random Hermitian matrices and multiplicative chaos. See also \cite{CZ, CN19,L19a,L19b} for further recent studies connecting multiplicative chaos to random matrix theory and $\beta$-ensembles.

To state our main theorem regarding the connection between $\h$ and multiplicative chaos, we introduce the following notation: for $\gamma\in \R$, as a measure on $[-1,1]$, let
\begin{equation}\label{def:muN}
d\mu_N^\gamma(x)=\frac{e^{\gamma h_N(x)}}{\mathds E e^{\gamma h_N(x)}}dx.
\end{equation}
We define the limiting object slightly informally here -- for a more precise discussion on how to construct it, see Section \ref{sect:gmc}. Let $X$ be the centered Gaussian field on $(-1,1)$ with covariance kernel $\Sigma$ given by \eqref{eq:cov}. Then write  for each $\gamma\in \R$
\begin{equation}\label{def:mu}
d\mu^\gamma(x)= `` \frac{ e^{\gamma X(x)}}{\E e^{\gamma X(x) }}dx"
\end{equation}
for the associated multiplicative chaos measure -- the quotation marks indicate that the proper definition of the measure requires a renormalization procedure as described above and in further detail in Section \ref{sect:gmc}. With this notation, our main result concerning multiplicative chaos and $\h$, proven in Section \ref{sect:gmc}, is the following.
\begin{theorem}\label{th:gmc}
Let $V$ satisfy Assumptions \ref{ass:regular}. Then for $\gamma\in (-\sqrt{2},\sqrt{2})$, as $N\to\infty$, the sequence of random measures $\mu_N^\gamma$ defined in \eqref{def:muN} converges in law, with respect to the topology of weak convergence, to the random measure $\mu^\gamma$. 
\end{theorem}

\begin{remark}
Let us point out here that we expect that with a very similar argument as our proof of Theorem \ref{th:gmc}, one should be able to prove a corresponding result for $\h$ replaced by $x\mapsto \sqrt{2}\sum_{j=1}^N \log |x-\lambda_j|-\sqrt{2}N\int \log |x-\lambda|d\mu_V(\lambda)$ and for say $\gamma\in(0,\sqrt{2})$ -- thus extending a result of \cite{BWW18}. Moreover, we expect that a modification of our approach to Theorem \ref{thm:max2} should then yield precise estimates for $\max_{x\in \R}(\sqrt{2}\sum_{j=1}^N \log |x-\lambda_j|-\sqrt{2}N\int \log |x-\lambda|d\mu_V(\lambda))$, though as mentioned, such results have already been obtained in \cite{LP18}.  \hfill $\blacksquare$
\end{remark}

\begin{remark}
The value $\gamma=\gamma_c:=\sqrt{2}$ is critical in the sense that if $|\gamma| >\sqrt{2}$, then $\mu_N^\gamma$ converges in distribution to the zero measure as $N\to+\infty$ -- this follows from Remark~\ref{rk:supercritical} below. \hfill $\blacksquare$ 
\end{remark}

Concerning the proof of Theorem \ref{th:gmc}, we point out that since the stochastic process $\h$ is not Gaussian and $\mu_N^\gamma$ is not a martingale, the proof is more involved than those needed for simply constructing $\mu^\gamma$ -- for an elegant one, see \cite{Berestycki15}. Our approach here is to apply recent ideas developed in \cite{LOS18} for proving in general how an object that is only asymptotically a Gaussian log-correlated field, such as $\h$, can give rise to (Gaussian) multiplicative chaos. This method relies on estimates for exponential moments of $\h$ and related quantities. These on the other hand can be formulated in terms of asymptotics of certain Hankel determinants, and establishing these estimates is the most technical part of this paper. We will state the relevant asymptotics of Hankel determinants in Section \ref{sec:hankel_intro} and make the connection to exponential moments in Section \ref{sec:special}, but before doing this, we discuss some consequences of Theorem \ref{th:gmc} and give a heuristic reasoning to why it implies Theorem \ref{thm:max2}.

The connection between the measures $\mu_N^\gamma$ and the extreme values of the field $\h$ comes through the notion of thick points. Intuitively, a thick point is a point where the  field takes an abnormally large value: one where it is of the order of its variance instead of the order of its standard deviation. 
In our setting, the variance of $\h$ is roughly $\log N$ in the bulk for $N$ large, and  we define for $\gamma>0$ the set of $\gamma$-thick points as 
\begin{equation}\label{eq:thick}
\mathscr{T}_N^{\pm\gamma}  = \left\{  x \in [-1,1]:  \pm \h(x)\ge \gamma \log N\right\}.
\end{equation}
The idea behind this definition is that the measure $\mu_N^\gamma$ essentially lives on the set $\mathscr T_N^\gamma$. In particular, these points are so sparse that in the limit $\mu^\gamma$ lives on a fractal set.  Moreover, this heuristic and the fact that $\mu^\gamma$ is a non-trivial random measure if and only if $|\gamma|<\sqrt{2}$ gives strong evidence that the leading order of the maximum of $\h$ should be $\sqrt{2} \log N$ when $N$ is large.See Figure \ref{fig:simu2} for what $\mu_N^\gamma$ looks like for various $\gamma$; note in particular the fractal-like or strongly fluctuating behavior of the density,  as well as the fact that  $\mu_N^\gamma$  concentrates on the extrema of the field $\h$ as $\gamma$ increases.

In the following theorem, which we prove in Section \ref{sec:special}, we provide precise estimates for the size of $\mathscr T_N^\gamma$.

\begin{theorem}\label{th:TP}
Under Assumptions \ref{ass:regular},  for any $\gamma\in (-\sqrt{2},\sqrt{2})\setminus\{0\}$,
\[
\lim_{N\to \infty}\frac{\log |\mathscr T_N^\gamma|}{\log N}=-\frac{\gamma^2}{2},
\]
where the convergence is in probability.
\end{theorem}
\begin{remark}
As we already explained, this result implies directly the lower bound in Theorem~\ref{thm:max2}. 
Indeed, if we set $\gamma=\sqrt{2}(1-\delta)$, the above result implies that the probability of $\mathscr{T}_N^{\pm\gamma}$ being non-empty converges to $1$ as $N\to\infty$. But it is clear from the definition that
$\mathscr{T}_N^{\pm\gamma}$ is non-empty if and only if $\max_{x\in\mathbb R} (\pm h_N(x))\geq (1-\delta)\sqrt{2}\log N$.
Similarly, if we would know that the probability of having $\gamma$-thick points for $\pm\gamma>\sqrt{2}(1+\delta)$ tends to $0$ as $N\to\infty$, this would imply the upper bound in Theorem \ref{th:gmc}. \hfill $\blacksquare$
\end{remark}

\begin{figure}[h]
\centering
\includegraphics[width=1\textwidth]{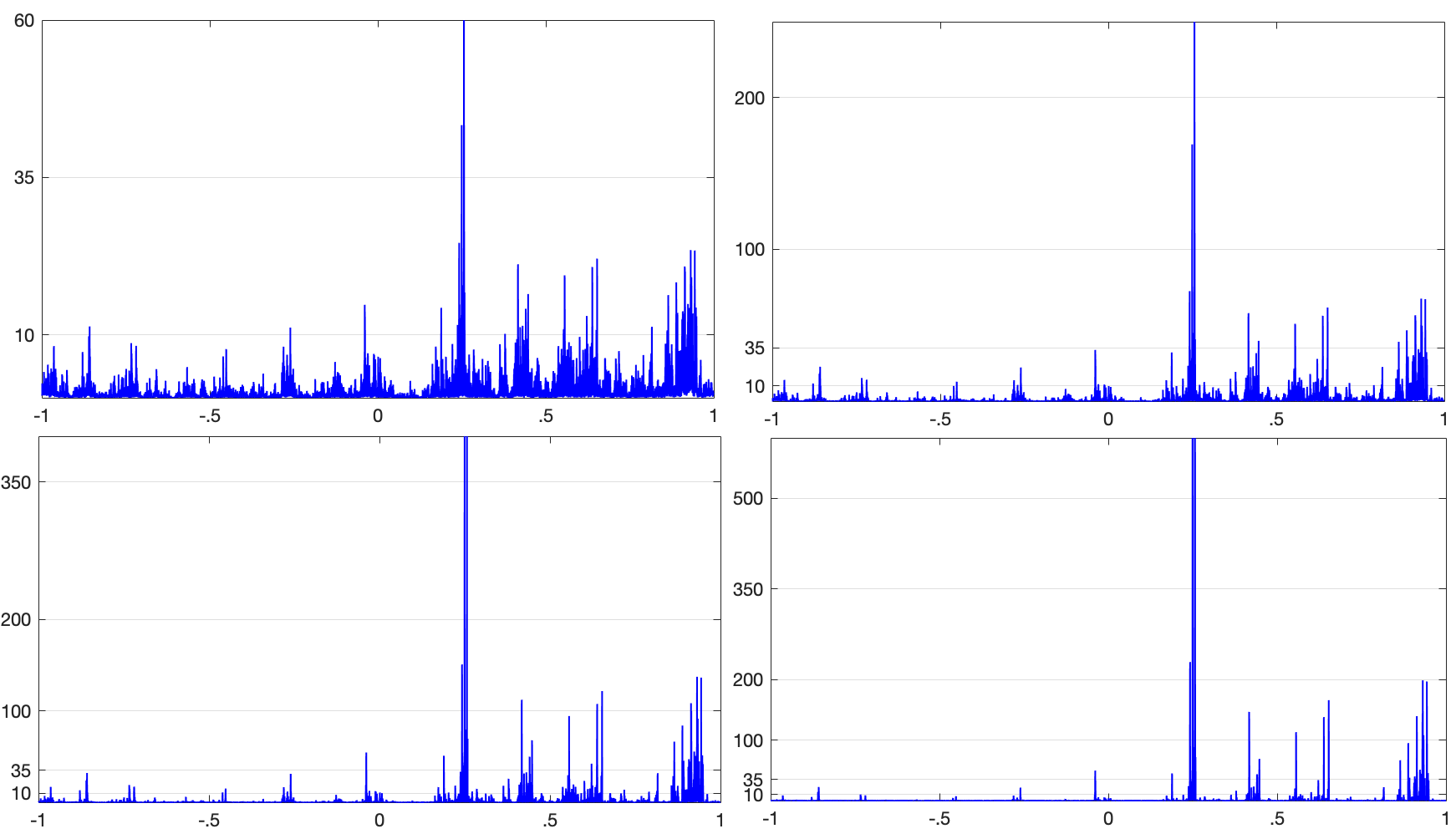} 
	\caption{A realization of the density of the measure $\mu_N^\gamma$ in the setting of the GUE for $N=6000$, $\gamma=0.4$ (top left), $\gamma=0.6$ (top right), $\gamma=.8$ (bottom left) and $\gamma=1.2$ (bottom right). The realization of the field $\h$ is the same as in Figure \ref{fig:simu1}.}\label{fig:simu2}
\end{figure}

In the physics literature \cite{FB08,FHK12,FLR12,FK}, it has been described that when interpreting the eigenvalue counting function $\h$ as
a \emph{random energy landscape}, the associated free energy exhibits a weak form of freezing, i.e. above the critical temperature $\gamma_c=\sqrt{2}$, the free energy becomes linear in $\gamma$. The asymptotics of the free energy are closely related to the asymptotics for the size of the set of $\gamma$-thick points, \cite{ABB16}. We demonstrate in Section~\ref{sect:freezing} that a proof of such behavior follows rather easily from Theorem \ref{th:gmc}. More precisely, we have the following result.

\begin{corollary} \label{thm:freezing}
Under Assumptions \ref{ass:regular},  for any $\gamma>0$, 
\[
\lim_{N\to\infty}\frac{1}{\log N}\log\left( \int_{-1}^1 e^{\gamma \h(x)} dx \right)  = \begin{cases}
\gamma^2/2 &\text{if } \gamma \le \sqrt{2} \\
\sqrt{2}\gamma -1 &\text{if } \gamma \ge \sqrt{2}
\end{cases},
\]
where the convergence is in probability.
\end{corollary}

To conclude this section, we mention that there is in fact a technically simpler approach to the lower bound of Theorem \ref{thm:max2}. The idea of this alternative approach is that a smoothed version of $\h$ can be used to obtain nearly the correct lower bound for the maximum of $\h$. Thus we can carry out the program of controlling extrema through multiplicative chaos for this smoothed version of the field. This is technically less demanding as the required Hankel determinant estimates are simpler to obtain. Such an approach does not yield optimal estimates for example for the thick points, and thus does not reveal  the full fractal geometric properties of $\h$, but for questions about extrema, we feel that this approach is preferable. 

More precisely, we define for arbitrary $z\in \C_+=\{w\in \C: \Im w>0\}$ the harmonic extension of $\h(x)$ as follows,
\begin{equation}\label{eq:hcomplex}
\h(z)=\sqrt{2}\Im\left(\sum_{j=1}^N \log(z-\lambda_j)-N\int\log(z-t)d\mu_V(t)\right),
\end{equation}
where the choice of the branch of the logarithm is chosen such that the cut is on the negative imaginary axis -- see Section \ref{sec:special} for details. We define the mesoscopic scale $\epsilon_N=N^{ -1+\alpha}$  for $\alpha\in(0,1)$ and a sequence of random measures on $[-1,1]$ by 
\begin{equation}\label{def:muNtilde}
d\widetilde \mu_N^{\gamma}(x)=\frac{e^{\gamma \h(x+\i \epsilon_N)}}{\mathds E e^{\gamma \h(x+\i \epsilon_N)}}dx.
\end{equation}
The counterpart to Theorem \ref{th:gmc}, which we will prove in Section \ref{sect:proof2}, is the following.
\begin{theorem}\label{th:gmc2}
Let $V$ satisfy Assumptions \ref{ass:regular}. Then for $\gamma\in (-\sqrt{2},\sqrt{2})$ and $\alpha\in(0,1)$, as $N\to\infty$, the sequence of random measures $\widetilde\mu_N^\gamma$ converges in law, with respect to the topology of weak convergence, to the random measure $\mu^\gamma$ from \eqref{def:mu}.
\end{theorem}

\subsection{Asymptotics of Hankel determinants}\label{sec:hankel_intro}

{
As discussed in Section \ref{gmc_results}, our proofs of Theorem \ref{th:gmc} and Theorem \ref{th:gmc2} (and hence the proof of Theorem \ref{thm:rigidity}) rely on asymptotics of exponential moments of the field $\h$. These can be expressed as Hankel determinants. We will shortly state the precise connection between Hankel determinants and random matrices, but the relevant exponential moment estimates will be addressed in Section \ref{sec:special}. 
In this section, we will state our results on asymptotics of Hankel determinants and briefly review the background to such questions.

For an integrable function $f:\R\to \R$ with fast enough decay at infinity, the $N\times N$ Hankel determinant with symbol $f$ is
\begin{equation}\label{eq:Hankel}
D_N(f)=\det\left(\int_\R \lambda^{i+j}f(\lambda)d\lambda\right)_{i,j=0}^{N-1}.
\end{equation}
The connection to the unitary invariant random matrix ensembles defined in \eqref{Z} comes from Heine's identity (see e.g. \cite[Proposition 3.8]{Deift}), which implies that for functions $w:\R\to \R$ with some mild regularity and not growing too rapidly at infinity and for $V:\R\to \R$ continuous satisfying the condition \eqref{V},
one has
\begin{equation}\label{eq:Heine}
\mathds E e^{\sum_{j=1}^N w(\lambda_j)}=\frac{1}{Z_N}D_N(e^{w-NV}),
\end{equation}
where the expectation is taken according to \eqref{Z}.
As mentioned, applying the approach of \cite{LOS18} for the proof of Theorem \ref{th:gmc} requires (among other things) exponential moment estimates for $h_N$. This corresponds to functions $w$ which have discontinuities.

Large $N$ asymptotics of Hankel determinants with such discontinuous symbols have been studied in quite some detail. They are a special case of symbols with  Fisher-Hartwig singularities. Hankel and Toeplitz determinants with Fisher-Hartwig singularities have a long history, and have (besides their prominent role in random matrix theory) a variety of applications in statistical physics, in particular the 2d Ising model \cite{WMTB} and the impenetrable Bose gas \cite{VT, FF, FFGW}. We refer to \cite{DIK-Ising} for a historical overview.
A general real-valued symbol with Fisher-Hartwig singularities can be written as 
\[
f_{\mathrm{FH}}(\lambda)=\prod_{j=1}^k |\lambda-x_j|^{2\alpha_j}e^{\sqrt{2}\pi\gamma_j \mathbf{1}_{\{\lambda\leq x_j\}}} e^{w(\lambda)}e^{-NV(\lambda)},
\]
where the pairwise distinct singularities $x_j$ and the parameters $\gamma_j$ are real, $\alpha_j>-1/2$ for all $j$, and where $w$ is sufficiently smooth.
In what follows, we will assume $w$ for simplicity to be real analytic.

If the points $x_j\in(-1,1)$, the parameters $\alpha_j> -\frac{1}{2}$ and $\gamma_j\in \R$, and the function $w$ are independent of $N$, and if $V$ satisfies Assumptions \ref{ass:regular}, the asymptotics for $D_N(f_{\mathrm{FH}})$ have been established recently by Charlier in \cite{Charlier} using Riemann-Hilbert (RH) methods pioneered by Deift, Its, and Krasovsky in the setting of Toeplitz determinants with Fisher-Hartwig singularities \cite{DIK11}, and we will use these results in the cases $k=1$ and $k=2$ with $\alpha_1=\alpha_2=0$. 
In addition, to apply the approach of \cite{LOS18} in the proofs of Theorem \ref{th:gmc} and Theorem \ref{th:gmc2}, we also need estimates for the Hankel determinants associated to $f_{\rm FH}$ in situations where the singularities $x_1, x_2$ depend on $N$ and in particular, may merge as $N\to\infty$, and also in situations where $w$ depends on $N$ and may develop a singularity in the limit $N\to\infty$. Moreover, to obtain the upper bound in Theorem \ref{thm:max2}, we also need an estimate in the case where $k=1$ and the singularity $x_1$ tends to the boundary $\pm 1$ as $N\to\infty$. While analogous questions have been studied to some degree -- see e.g. \cite{CK} for Toeplitz determinants and \cite{CF16} for Hankel determinants with two merging singularities, \cite{BWW18, CIK} for smooth symbols which develop singularities in the large $N$ limit, and \cite{XuZhao, BCI16} for symbols with singularities which approach the edge of the spectrum -- none of these works quite apply in our setting. For this reason, a large part of this article is concerned with deriving such asymptotics.
Our methods could be used to prove rather general results (for instance for symbols with an arbitrary number of singularities, but just two merging or one approaching the edge; both jump and root type singularities) for Hankel determinants, but we choose to restrict to the simplest situations that are sufficient for our purposes, in order to limit the technicality and length of the paper.

Let us introduce some further notation. For points $x_1,x_2\in (-1,1)$, $\gamma_1,\gamma_2\in \R$, and any function $w:\R\to \R$ with not too bad growth at infinity, we write 
\begin{equation}\label{eq:FH}
D_N(x_1,x_2;\gamma_1,\gamma_2;w)=\det\left(\int_\R \lambda^{i+j}e^{\sqrt{2}\pi\gamma_1
\mathbf{1}_{\{\lambda\leq x_1\}}+\sqrt{2}\pi\gamma_2\mathbf{1}_{\{\lambda\leq x_2\}}}e^{w(\lambda)-NV(\lambda)}d\lambda\right)_{i,j=0}^{N-1}.
\end{equation}
We will also need notation for the case where there is only a single singularity: we write
\[
D_N(x;\gamma;w)=\det\left(\int_\R \lambda^{i+j}e^{\sqrt{2}\pi\gamma\mathbf{1}_{\{\lambda\leq x\}}}e^{w(\lambda)-NV(\lambda)}d\lambda\right)_{i,j=0}^{N-1}.
\]

To describe what we require of the function $w$, we introduce a suitable domain. Given $0<\alpha<2/3$, we set $\epsilon_N=N^{-1+\alpha}$ and $\delta_N=N^{-\alpha/2}$, and we define the domain $\mathcal S_N$ as follows,
\begin{equation}\label{domain}
\mathcal S_N:= \big(\big\{ z\in\C : |\Re z| \le 1  - 3\delta_N,   |\Im z | <\epsilon_N/2  \big\} \cup
\big\{ z\in\C : |\Re z| \ge 1 - 3\delta_N,   |\Im z | <3\delta_N  \big\}\big).
\end{equation}

Our first result describes how to compare $D_N(x_1,x_2;\gamma_1,\gamma_2;w)$ to  $D_N(x_1,x_2;\gamma_1,\gamma_2;0)$, when $w$ may depend on $N$.
\begin{theorem}\label{th:Hankel1}
Let $\gamma_1,\gamma_2\in \R$ and let $-1<x_1<x_2<1$.
Let $V$ satisfy Assumptions \ref{ass:regular}.
 Assume that $w=w_N$ is a sequence of functions which are real valued on $\R$, analytic  and uniformly bounded on a domain $\mathcal S_N$ as in  \eqref{domain}. 
As $N\to\infty$, we have
\[\begin{aligned}
\log D_N(x_1,x_2;\gamma_1,\gamma_2;w_N)
&=\log D_N(x_1,x_2;\gamma_1,\gamma_2;0)\\
&\quad+N\int w d\mu_V+\frac{1}{2}  \sigma(w)^2
+\sum_{j=1}^2\frac{\gamma_j}{\sqrt{2}}\sqrt{1-x_j^2}\mathcal U w(x_j)+o(1),
\end{aligned}\]
where the error is uniform in $(x_1,x_2)$ in any fixed compact subset of $(-1,1)^2$ and for any fixed $R>0$, it is uniform in $w$ in the set $\{\xi:\mathcal S_N\to \C \ \text{analytic},\ \xi(\R)\subset \R,\ \sup_{z\in \mathcal S_N}|\xi(z)|\leq R\}$.   Here $\sigma(w)^2$ is as in \eqref{eq:CLT} and $\mathcal Uw$ denotes the finite Hilbert transform$:$
\begin{equation} \label{U_transform}
(\mathcal Uw)(x)=\frac{1}{\pi}\mathrm{P.V.}\int_{-1}^1 \frac{w(t)}{x-t}\frac{dt}{\sqrt{1-t^2}},
\end{equation}
with $\mathrm{P.V.}\int$ the Cauchy principal value integral.
\end{theorem}
This result will be proven in Section \ref{section: RHseparated} by relying on the RH analysis from Section \ref{sec:merging} and Section \ref{section: RHseparated}.

Our next step is to compare $D_N(x_1,x_2;\gamma_1,\gamma_2;0)$ to $D_N(x_1;\gamma_1+\gamma_2;0)$. Together these two steps will be sufficient for carrying out the program of \cite{LOS18}. This comparison is highly non-trivial in that the asymptotic behavior of the Hankel determinants $D_N(x_1,x_2;\gamma_1,\gamma_2;0)$ as the singularities $x_1$ and $x_2$ merge  is subtle, and a sharp transition in the nature of the asymptotics takes place when $x_2-x_1$ is of order $1/N$. On this scale, the asymptotics are described in terms of a Painlev\'e V transcendent.
A similar phenomenon was observed in \cite{CK} for Toeplitz determinants, in a more general setting where general Fisher-Hartwig singularities of root and jump type are allowed.
Another related result was obtained in \cite{CF16} for Hankel determinants with merging root type singularities.
We do not need the precise details of this Painlev\'e transition, but we do need good control of the behavior of the Hankel determinants in this regime.

We will prove the following result in Section \ref{sec:merging}.
\begin{theorem}\label{th:Hankel2}
Let $\gamma_1,\gamma_2\in \R$ and let $-1<x_1<x_2<1$.
Let $V$ satisfy Assumptions \ref{ass:regular}.  Then, as $N\to\infty$, we have
\begin{equation}\label{eq:asDn1}
\begin{aligned}
\log D_N(x_1,x_2;\gamma_1,\gamma_2;0) &=\log D_N(x_1;\gamma_1+\gamma_2;0)\\
& +\sqrt{2}\pi\gamma_2 N\int_{x_1}^{x_2}d\mu_V-\gamma_1\gamma_2\max\{0,\log (|x_1-x_2|N) \} +\mathcal O(1), 
\end{aligned}
\end{equation}
where the error term is uniform for $-1+\delta<x_1<x_2<1-\delta$, $0<x_2-x_1<\delta$ for $\delta$ sufficiently small.
\end{theorem}
\begin{remark}
Like in \cite{CK} and \cite{CF16}, the $\mathcal O(1)$ term can be expressed explicitly in terms of a solution $\sigma$ to the so-called Jimbo-Miwa-Okamoto $\sigma$-form of the Painlev\'e V equation.\hfill $\blacksquare$
\end{remark}

Finally, to obtain the upper bound for Theorem \ref{thm:max2}, we need to consider a situation with only one singularity, namely $D_N(x;\gamma;0)$, but in which we allow $x$ to approach the edge $\pm 1$ as $N\to\infty$. The precise result is the following, and will be proven in Section \ref{section: RHedge}.

\begin{theorem}\label{th:Hankel3}
Let $\Gamma>0$. There exists a constant $m>0$ such that
\[
\log \frac{D_N(x;\gamma;0)}{D_N(x;0;0)}=\sqrt{2}\pi\gamma N\int_{-1}^xd\mu_V(\xi) + \frac{\gamma^2}{2}\log N+\frac{3\gamma^2}{4}\log(1-x^2)+ \mathcal O(1),
\]
as $N\to\infty$, with the error term uniform for all $|x|\le 1- m N^{-2/3}$ and all $\gamma \in [-\Gamma, \Gamma]$. 
\end{theorem}
This concludes our preliminary discussion of Hankel determinants and we move on to an outline of the remainder of the article. 

\subsection{Outline of the article and acknowledgements}\label{sec:ackno}

In Section \ref{sect:gmc}, we first review the classical theory of Gaussian Multiplicative Chaos and the main result of \cite{LOS18}, then in a rather general setting, show that a random field, satisfying suitable exponential moment assumptions, can be used to construct a multiplicative chaos measure. Using this general approach, we prove Theorem \ref{th:gmc} and Theorem \ref{th:gmc2} assuming the results from Section \ref{sec:hankel_intro}. 

In Section \ref{sect:LB} and Section \ref{sec:rigidity}, we discuss the applications of Theorems \ref{th:gmc} and \ref{th:gmc2}  to eigenvalue statistics.  Namely, in Section \ref{sect:LB}, we prove our estimates for the extrema of $\h$, that is  Theorem \ref{thm:max2}, and in Section \ref{sec:rigidity}, we prove our optimal rigidity result given by Theorem \ref{thm:rigidity}. 

The latter half of the article, Section \ref{section: diffid} to Section \ref{section: RHedge}, is where we carry out our RH analysis to prove the results stated in Section \ref{sec:hankel_intro}: Theorem \ref{th:Hankel2} is proven in Section \ref{sec:merging}, Theorem \ref{th:Hankel1} is proven in Section \ref{section: RHseparated}, and Theorem \ref{th:Hankel3} is proven in Section \ref{section: RHedge} -- Section \ref{section: diffid} and Section \ref{sec:model} cover preliminary results needed for these proofs.

We also record some auxiliary material concerning the log-correlated Gaussian field relevant for our purposes in Appendix \ref{app:logcor} as well as some model RH problems in Appendices \ref{section:appendixAiry} and \ref{section:appendixCFH}.

\medskip

In the following, $C>0$ is a constant which does not depend on the dimension~$N$ and may  vary from line to line. Sometimes,  we may write $C_\epsilon>0$ to emphasize that this constant depends for instance on the parameter $\epsilon$. We write $A_N=\mathcal O(B_N)$ to indicate that there exists a constant $C$ independent of $N$ such that for large enough $N$, $|A_N|\leq C|B_N|$. If we allow $C$ to depend on some further parameter say $\epsilon>0$, we may write $A_N=\mathcal O_\epsilon(B_N)$. We will also make use of the closely related Vinogradov notation $A_N\ll B_N$ to indicate that there exists a constant $C>0$ independent of $N$ such that $0\leq A_N\leq C B_N$. If we allow this constant to depend on a further parameter, say $\epsilon$, we may again write $A_N\ll_\epsilon B_N$.

\medskip

{\bf Acknowledgements:}
T.C. was supported by 
 the Fonds de la Recherche Scientifique-FNRS
under EOS project O013018F.
G. L. was supported by the University of Zurich grant FK-17-112 and by the Swiss National Science Foundation  Ambizione  grant S-71114-05-01. 
 C. W. was supported by the Academy of Finland grant
308123. B.F. was supported by the G\"oran Gustafsson Foundation (UU/KTH) and by the Leverhulme Trust research programme grant RPG-2018-260.

\section{Gaussian Multiplicative Chaos} \label{sect:gmc}

In this section,  we prove our results concerning multiplicative chaos and the eigenvalue counting function as described in Section \ref{gmc_results}, assuming our results concerning asymptotics of Hankel determinants from Section \ref{sec:hankel_intro}. Before going into the actual proofs, we pick up our general discussion on log-correlated fields and multiplicative chaos from Section \ref{gmc_results}. As we hope that our approach to proving convergence to multiplicative chaos will find applications in other models as well, we offer a proof of convergence to multiplicative chaos which is valid for general log-correlated fields -- in particular, we will not restrict ourselves to one spatial dimension. What we prove is that any object for which suitable exponential moment estimates can be established (in our case this follows from the results of Section \ref{sec:hankel_intro}), converges to a multiplicative chaos measure. In Section \ref{sect:LB}, we prove that under similar assumptions, we have estimates for the thick points, the maximum of the field, etc. 
 
Before going into proofs, to be presented in Sections \ref{sect:main}--\ref{sec:special}, we will elaborate on our discussion in Section \ref{gmc_results}, and now rigorously review some of the standard theory of multiplicative chaos. After that, in Section~\ref{sec:asyGaus} we will describe some recent ideas from \cite{LOS18} which provide a tool for proving that a random process, not necessarily Gaussian, can give rise to a Gaussian multiplicative chaos measure if the process satisfies certain exponential moment estimates. We then move on to Sections \ref{sect:main}--\ref{sect:extension}, where we prove our main result about convergence to multiplicative chaos under certain assumptions on exponential moments. Finally in Section \ref{sec:special}, we specialize to the case of relevance to random matrices, namely prove Theorem \ref{th:gmc} and Theorem \ref{th:gmc2} using our general results and exponential moment estimates following from the results of Section \ref{sec:hankel_intro}.

We will use the following notation throughout our discussion: for $\Omega\subset \R^d$ open, we write $C_c(\Omega)$ for the space of continuous functions with compact support inside of $\Omega$ and $C_c^\infty(\Omega)$ for the space of infinitely differentiable functions with compact support in $\Omega$. We will also briefly mention the standard $L^2$-based Sobolev spaces of negative regularity. More precisely, if $\mathcal S'(\R^d)$ denotes the space of tempered distributions\footnote{Recall that for tempered distributions, the Fourier transform is a well defined operation. Our convention for the Fourier transform is that if $\varphi\in \mathcal S(\R^d)$ -- the Schwartz space of smooth functions with rapid decay -- then $\widehat \varphi(\xi)=\int_{\R^d}e^{-2\pi i x\cdot \xi}\varphi(x)dx$.} then we write for any $\epsilon>0$
\[
H^{-\epsilon}(\R^d)=\left\{\varphi \in \mathcal S'(\R^d): \|\varphi\|_{H^{-\epsilon}(\R^d)}^2:= \int_{\R^d}(1+|\xi|^2)^{-\epsilon}|\widehat \varphi(\xi)|^2d\xi<\infty\right\}.
\] 
It is a standard fact that $H^{-\epsilon}(\R^d)$ is in fact a separable Hilbert space when one defines the inner product in the natural way. We will find it convenient to abuse notation slightly and write for a generalized function $\varphi\in H^{-\epsilon}(\R^d)$ and an honest function $f\in H^{\epsilon}(\R^d)$, $\int \varphi(x)f(x)dx$ to denote the dual pairing of $\varphi$ and $f$, despite pointwise values $\varphi(x)$ having no meaning.

\subsection{Gaussian multiplicative chaos background} \label{sect:background} 
In this section we discuss rigorously what log-correlated fields are and what multiplicative chaos is. For a more extensive discussion of the subject, we direct the reader to the comprehensive review \cite{RV14} and the concise construction of multiplicative chaos measures in \cite{Berestycki15}. As mentioned, we wish to keep our discussion on a rather general level in hope of its use in other applications, but readers wishing for concreteness should keep in mind that our main goal is to consider a log-correlated field on the interval $\Omega=(-1,1)$ with a covariance of the form 
\begin{equation}\label{eq:Cspecial}
\Sigma(x,y)=\log\left( \frac{1-xy+\sqrt{1-x^2}\sqrt{1-y^2}}{|x-y|} \right)
\end{equation}
 and we wish to define for certain values of $\gamma\in \R$, a measure $\mu^\gamma$, which  in some sense is proportional to the ``exponential" of $\gamma$ times such a field.

We now discuss briefly what log-correlated fields are -- the discussion follows \cite[Section 2]{JSW18} to some degree. Let $\Omega\subset \R^d$ be an open set and for simplicity, let us assume that it is simply connected, bounded and has a smooth boundary. Let us assume further that we  are given a symmetric positive definite kernel $\Sigma$ on $\Omega\times \Omega$ of the form:
\begin{equation}\label{eq:logcorC}
\Sigma(x,y)=\log |x-y|^{-1}+g_\Sigma(x,y),
\end{equation}
where $g_\Sigma$  is continuous, bounded from above and say in $L^2(\Omega\times \Omega)$ -- note that we are not requiring boundedness from below or that $g_\Sigma$ extends continuously to $\overline{\Omega\times\Omega}$. 
We also point out that our assumptions cover in particular the case of the special $\Sigma$ from \eqref{eq:Cspecial}.

Making use of the Karhunen-Lo\`eve expansion, it is proven in \cite[Proposition 2.3]{JSW18} that for any $\epsilon>0$, there exists a centered Gaussian process on $H^{-\epsilon}(\R^d)$  with covariance kernel $\Sigma$ given by \eqref{eq:logcorC}. In fact, by standard properties of Gaussian processes, there is a unique Gaussian measure on $H^{-\epsilon}(\R^d)$ with such a covariance kernel. This Gaussian process is what is referred to as a log-correlated field with covariance kernel $\Sigma$ and we will write $\X$ for such an object. 

We mention that while in general, the Karhunen-Lo\`eve expansion is slightly abstract (it follows from applying the spectral theorem to the compact operator with integral kernel $\Sigma$ acting on $L^2(\R^d)$), the log-correlated field with covariance \eqref{eq:Cspecial} can be constructed  explicitly: mimicking e.g. the proof of \cite[Proposition 2.3]{JSW18}, one can show that  the series
\begin{equation}\label{eq:series}
\sum_{k=0}^\infty \frac{\sqrt{2}\xi_k}{\sqrt{k+1}}U_{k}(x)\sqrt{1-x^2}\mathbf{1}_{(-1,1)}(x),
\end{equation}
where $\xi_k$ are i.i.d. standard Gaussians and $U_{k}$ the $k$th Chebyshev polynomial of the second kind,\footnote{That is, the unique polynomial satisfying $U_k(\cos\theta)=\frac{\sin (k+1)\theta}{\sin \theta}$ for all $\theta\in[0,\pi]$.} converges in $H^{-\epsilon}(\R)$ for any $\epsilon>0$ and the covariance kernel of the resulting process is given by \eqref{eq:Cspecial} -- for further discussion about this field, see Appendix \ref{sec:cov}.

\medskip

Having understood what log-correlated fields are, as well as having seen a concrete series representation for the one most relevant to us, let us move on to constructing multiplicative chaos measures. As repeatedly mentioned, these are objects that can be formally written as
$\mu^\gamma(dx)=``e^{\gamma \X(x)}/\E e^{\gamma \X(x)}dx."$
There are two immediate problems with making sense of such an object: 
the first one being that one can not naively exponentiate an element of $H^{-\epsilon}(\R^d)$, which is after all a generalized function. The second is that even formally $\E e^{\gamma \X(x)}=+\infty$ for any $\gamma \neq 0$ for a Gaussian process with a covariance of the form \eqref{eq:logcorC}. In fact, these two issues \emph{compensate each other} and one should think formally of $\E e^{\gamma \X(x)}$ as the correct way to renormalize $e^{\gamma \X(x)}$.
Then, the rigorous construction of the random measure $\mu^\gamma$ consists in approximating $\X$ by a continuous function $\X_\epsilon$ and taking a limit as $\epsilon\to0$ which corresponds to removing the approximation.

There are various ways to approximate a log-correlated field $\X$ by a continuous function. One option is to simply truncate the Karhunen-Lo\`eve expansion. While such an approximation is useful for some purposes (see e.g. our proof of Proposition \ref{pr:gmcprops} below), for proving the existence of a non-trivial limiting object, another useful approximation is to smooth the field $\X$ by convolving it with a smooth bump function. More precisely, one can check that if $\varphi:\R^d\to [0,\infty)$ is a radially symmetric, compactly supported $C^\infty$-function with mean one $\int_{\R^d}\varphi(x)=1$, and if we write $\varphi_\epsilon(x)=\epsilon^{-d}\varphi(x/\epsilon)$, then for each $\epsilon>0$
\begin{equation}\label{eq:Xeps}
\X_\epsilon(x)=(\X*\varphi_\epsilon)(x)\mathbf 1_{\Omega}(x)
\end{equation}
defines almost surely a random continuous function on $\Omega$. More importantly, these assumptions imply for example that 
$\displaystyle\lim_{\epsilon\to 0}\E \X_\epsilon(x)\X_\epsilon(y)=\Sigma(x,y)$ for $x\neq y$ and 
\begin{equation} \label{cov_estimate}
\E \X_\epsilon(x)\X_\delta(y)=\log \frac{1}{\max(|x-y|,\epsilon,\delta)}+\mathcal O(1)
\end{equation}
where the implied constant is independent of $\epsilon,\delta>0$ and is uniform for all $x,y$ in any given compact subset of $\Omega$ -- see e.g. \cite[Proposition 2.7]{JSW18}.

As mentioned above, $\X_\epsilon$ is almost surely continuous, which implies that we can exponentiate $\X_\epsilon$ to construct a  well defined random measure $\mu^\gamma_\epsilon(dx) = e^{\gamma \X_\epsilon(x)} / \E e^{\gamma \X_\epsilon(x)} dx$. 
Moreover, it is proved\footnote{We will carry out a very similar argument in the coming sections, so we will not elaborate further on the ideas of \cite{Berestycki15} here. Also we mention that in \cite{Berestycki15}, instead of considering integrals against compactly supported continuous functions, the author considers the measure of an arbitrary set. This is a cosmetic difference and the reader will have no trouble translating the proofs of \cite{Berestycki15} to our setting.} in \cite{Berestycki15}, 
using a modified second moment method, that for any $\gamma\in(-\sqrt{2d},\sqrt{2d})$,\footnote{We mention here that this regime is known as the sub-critical regime. For $|\gamma|\geq \sqrt{2d}$, the limiting measure constructed in this way is the trivial zero  measure. There is a way to construct non-trivial limiting measures in this regime through other kinds of renormalization schemes, but we will not discuss this further here.} any continuous function $f$ with compact support in $\Omega$, and any sequence $\epsilon_n$ tending to zero as $n\to\infty$, $\int_\Omega e^{\gamma \X_{\epsilon_n}(x)-\frac{\gamma^2}{2}\E\X_{\epsilon_n}(x)^2}f(x)dx$ is a Cauchy sequence in $L^1(\P)$.
Thus $\int f d\mu^\gamma_\epsilon = \int_\Omega e^{\gamma \X_{\epsilon}(x)-\frac{\gamma^2}{2}\E\X_{\epsilon}(x)^2}f(x)dx$ converges in $L^1(\P)$ to a random variable which turns out to be independent of the mollifying function $\varphi$. We denote the limiting random variable by $\int f d\mu^\gamma$, since through a standard argument in measure theory, one can show that this convergence implies that there exists a unique random measure $\mu^\gamma$ for which the above notation makes sense.\footnote{For readers wishing to understand what this says about the convergence of $\mu^\gamma_\epsilon$ to $\mu^\gamma$,
we suggest \cite[Chapter 4]{Kallenberg17} as a general reference for convergence questions of random measures.}  This random measure $\mu^\gamma$ is our multiplicative chaos measure associated to the log-correlated field $\X$, and it gives rigorous meaning to $``e^{\gamma \X(x)}/\E e^{\gamma \X(x)}dx"$ for $\gamma\in(-\sqrt{2d},\sqrt{2d})$.

Before moving to constructing multiplicative chaos from asymptotically Gaussian processes,  let us record some basic properties of $\mu^\gamma$ that are well known to experts on the topic, yet do not exist in the literature exactly in our setting.

\begin{proposition}\label{pr:gmcprops}
For $\gamma\in(-\sqrt{2d},\sqrt{2d})$ and any non-negative $f\in  L^1(\Omega)$ with $\|f\|_{L^1(\Omega)}\neq 0$, $\int_\Omega f(x)d\mu^\gamma(x)$ is almost surely finite and almost surely positive. 
\end{proposition}
\begin{proof}
We begin by mentioning that by standard approximation arguments, it is sufficient to prove the result for $f\in L^1(\Omega)\cap C(\Omega)$. We assume  this from now on.

Almost sure finiteness is  a straightforward first moment estimate, which we present for the convenience of the reader. If $f$ is compactly supported and non-negative, almost sure finiteness is a trivial consequence of the $L^1(\P)$-convergence in the construction of $\mu^\gamma$: one has 
\begin{equation}\label{eq:Ebound}
\E\left[ \int_\Omega f(x)d\mu^\gamma(x)\right] \leq \int_\Omega f(x)dx<\infty.
\end{equation}
Let $g_k\in C_c(\Omega)$ be an increasing sequence of non-negative functions such that $g_k(x)=1$ for $d(x,\partial \Omega)\geq \frac{1}{k}$. Then for non-negative $f\in  L^1(\Omega)$, we have by the monotone convergence theorem $\int_\Omega f(x)d\mu^\gamma(x)=\lim_{k\to\infty}\int_\Omega g_k(x)f(x)d\mu^\gamma(x)$ and another application of the monotone convergence theorem (now applied to $\P$) combined with \eqref{eq:Ebound}  yields 
\[
\E\left[\int_\Omega f(x)d\mu^\gamma(x) \right]=\lim_{k\to\infty}\E\left[ \int_\Omega g_k(x)f(x)d\mu^\gamma(dx) \right] 
\leq \int_\Omega f(x)dx.
\]
Thus for each non-negative $f\in L^1(\Omega)$, $\E \int_\Omega f(x)d\mu^\gamma(x)<\infty$ which implies that $\int_\Omega f(x)d\mu^\gamma(x)$ is almost surely finite.

\medskip

Proving almost sure positivity is a little more involved. This is quite similar to a uniqueness argument given in \cite[Section 5]{Berestycki15} so we shall only sketch the proof and direct the interested reader to \cite{Berestycki15}. 
Applying the Karhunen-Lo\`eve expansion, we can decompose $\X(x)=\sum_{k=1}^\infty \sqrt{\lambda_k} \xi_k e_k(x)$, where $\xi_k$ are i.i.d. standard Gaussians, $\lambda_k>0$ are the positive eigenvalues of the integral operator with kernel $\Sigma$ and $e_k\in L^2(\Omega)$ are the corresponding (normalized) eigenfunctions:  $\int_\Omega \Sigma(x,y)e_k(y)dy=\lambda_k e_k(x)$. 

Due to our assumptions on the kernel $\Sigma$, namely \eqref{eq:logcorC}, the eigenequation implies that $e_k$ are continuous on $\Omega$. Thus $\X^{(n)}(x):=\sum_{k=1}^n \sqrt{\lambda_k }\xi_k e_k(x)$  defines a random continuous function on $\Omega$. Since $\xi_k$ are  i.i.d. standard Gaussians, for any non-negative $f\in C_c(\Omega)$, $\int_\Omega e^{\gamma \X^{(n)}(x)-\frac{\gamma^2}{2}\E \X^{(n)}(x)^2}f(x)dx$ is a non-negative martingale with respect to the filtration generated by $\xi_1, \xi_2, \dots$. The martingale convergence theorem implies that  $\int_\Omega e^{\gamma \X^{(n)}(x)-\frac{\gamma^2}{2}\E \X^{(n)}(x)^2}f(x)dx$ converges almost surely to a random variable as $n\to\infty$. In \cite[Section 5]{Berestycki15} it is proven that this random variable actually is $\int_\Omega f d\mu^\gamma $.

Almost sure positivity now follows from the Kolmogorov 0-1 law: for non-negative $f\in C_c(\Omega)$, the event that $\int_\Omega f d\mu^\gamma $ is positive is  a tail event -- it does not depend on the value of $\xi_1,...,\xi_n$ for any fixed $n$. Thus by Kolmogorov's 0-1 law, the probability that $\int_\Omega f d\mu^\gamma=0$ is either zero or one. By the $L^1(\P)$-convergence proven in \cite{Berestycki15},
we have that $\E \int_\Omega f d\mu^\gamma =\int_\Omega f(x)dx$, so  by our 0-1 argument, we conclude that for any non-trivial non-negative function $f\in  C_c(\Omega)$, $\int_\Omega f d\mu^\gamma >0$ almost surely.
\end{proof}

This concludes our review of the standard theory of multiplicative chaos.
In the next section, we turn to discussing under which assumptions and how to prove that a sequence of non-Gaussian random fields defined on possibly different probability spaces 
can give rise to a multiplicative chaos measure.

\subsection{Constructing multiplicative chaos from asymptotically Gaussian processes}\label{sec:asyGaus}

In this section we review some basic results from \cite{LOS18} about how to prove that non-Gaussian processes  which approximate a log-correlated field in a specific sense explained below, can give rise to a multiplicative chaos measure. Our interest lies in the eigenvalue counting function $\h$, but unfortunately, the results of \cite{LOS18} are not directly applicable to this case due to the highly non-trivial local structure of $\h$ -- a problem that we address  in Section \ref{sect:main}. Nevertheless, our approach will rely heavily on \cite{LOS18}, and we review now the basic ideas from there.

Another unfortunate fact is that in the type of questions we shall study, the notation becomes slightly heavy quite rapidly. We will try to alleviate this problem slightly by viewing the random variable $\X$ as a ``canonical process" or ``canonical function" and studying it under different probability measures $\P_N$. The picture to keep in mind for $\P_N$ is the law of the random $N\times N$ matrix and $\X= \h$   is a sequence of functions with some mild regularity.

 We will also suppose that under  $\P_N$,  the law of this canonical process $\X$ approximates the law of a Gaussian log-correlated field. We collect in the following definition some of our main assumptions.

\begin{definition}[Notational conventions and assumptions]\label{def:EN}
Let $\P$ denote the law of the log-correlated Gaussian field $\X$ with covariance $\Sigma$ from \eqref{eq:logcorC} --  $\E$  denotes the expectation with respect to  $\P$. We also assume that for each $N\geq 1$, we have some probability measure $\P_N$ $($similarly $\E_N$ denotes the expectation with respect to $\P_N)$ under which the canonical process $\X$ is a random function which is integrable, bounded from above $($though possibly with a random upper bound depending on $N)$, upper semi-continuous on $\Omega$ and satisfies for any  $\alpha>0$, $x\in \Omega$, and $f\in C_c^\infty(\Omega)$
\[
\E_N e^{\alpha\X(x)}<\infty \qquad \text{and} \qquad \E_N e^{\int \X(x)f(x)dx}<\infty . 
\]
Moreover, the law of $\X$ under $\P_N$ approximates the law of $\X$ under $\P$ in the sense that for each $f\in C^\infty_c(\Omega)$, the law of $\int_\Omega f(x)\X(x)dx$ under $\P_N$ converges to that of it under $\P$.
\end{definition}

In this language, we wish to prove that under $\P_N$, 
$\int_\Omega f(x)e^{\gamma \X(x)}/ \E_N e^{\gamma \X(x)} dx$ 
converges in law to $\int_\Omega f d\mu^\gamma$. The basic idea is to mimic the proof of existence of multiplicative chaos, but to use suitable exponential moment estimates to prove that we can approximate the Gaussian fields by the asymptotically Gaussian ones. As we saw in Section \ref{sect:background}, a central role is played by convolution approximations of the log-correlated Gaussian field $\X$, so our first ingredient is to establish that if we regularized the canonical process using a fixed mollifier, then  $\X_\epsilon$ converges to a Gaussian process as $N\to+\infty$ in a strong enough sense.
Such a result has been proven in \cite[Lemma 2.2]{LOS18}.

\begin{lemma}[\cite{LOS18}]\label{le:wcvg}
Let $\varphi\in C_c^\infty(\R^d)$ be a non-negative, rotation invariant function satisfying $\int_\Omega \varphi(x)dx=1$, and for each $\epsilon>0$ define $\varphi_\epsilon(x)=\epsilon^{-d}\varphi(x/\epsilon)$, and  let $\X_\epsilon=\X*\varphi_\epsilon$.
In addition to the assumptions of Definition \ref{def:EN}, assume that for each fixed $\epsilon>0$ and $\gamma\in \R$, as $N\to\infty$,
\begin{equation}\label{eq:1expmom}
\E_N e^{\gamma \X_\epsilon(x)}=(1+o(1))\E e^{\gamma \X_\epsilon(x)},
\end{equation}
where the implied constant is uniform in $x$ in any compact subset of $\Omega$. Then for any fixed $f\in C_c(\Omega)$ and $\epsilon>0$, as $N\to\infty$, the law of $\int_\Omega f(x)\frac{e^{\gamma \X_\epsilon(x)}}{\E_N e^{\gamma \X_\epsilon(x)}}dx$ under $\P_N$ tends to the law of $\int_\Omega f(x)\frac{e^{\gamma \X_\epsilon(x)}}{\E e^{\gamma \X_\epsilon(x)}}dx$  under $\P$.  
\end{lemma}

While our assumptions are slightly different than in \cite[Lemma 2.2]{LOS18}, looking at the proof there, one sees that it carries through word for word in our case -- more precisely, the assumptions of Definition \ref{def:EN} combined with \eqref{eq:1expmom} imply that $e^{\gamma \X_\epsilon}/ \E_N e^{\gamma \X_\epsilon}$ under $\P_N$ converges to $e^{\gamma \X_\epsilon-\frac{\gamma^{2}}{2}\E \X_\epsilon^2}$ under $\P$ in the sense of finite dimensional distributions while tightness of $\int_\Omega f(x)\frac{e^{\gamma \X_\epsilon(x)}}{\E_N e^{\gamma \X_\epsilon(x)}}dx$ under $\P_N$ follows  by construction, so by Prokhorov's theorem, one finds the claim.

\medskip

The role of this result is that if we add a second level of approximation, namely smooth our ``approximately Gaussian" field by a fixed amount, then we can use it to construct a multiplicative chaos measure by first letting $N\to\infty$ and then $\epsilon\to0$.
Recall that our goal is to prove that this can be done without the smoothing. While not quite doing this, in \cite{LOS18} the authors provided a condition under which a  \emph{mesoscopic smoothing} is sufficient, i.e.  one can take $\epsilon=\epsilon_N$ tending to zero as $N\to\infty$. 
Mesoscopic smoothing refers to the fact that one regularizes the process at a small scale $\epsilon_N$ tending to 0 as $N\to+\infty$ but which does not see the \emph{microscopic} or \emph{local} scale of the model. 
Bearing in mind the example of the eigenvalue counting function, this means e.g. $\epsilon_N  = N^{\alpha-1}$  for some $0<\alpha<1$ so that  the regularized process $\h * \varphi_{\epsilon_N}$ cannot distinguish between individual eigenvalues  in the bulk of the spectrum. The result of \cite{LOS18} concerning mesoscopic smoothing is the following.

\begin{theorem}[\cite{LOS18}, Theorem 1.7] \label{thm:LOS}
In addition to the assumptions of Definition \ref{def:EN}, suppose that the following conditions hold$:$  there exists a sequence $\epsilon_N>0$ tending to zero as $N\to\infty$ such that for any $n\in \N$, $\boldsymbol{t}\in \R^n $, we have as $N\to \infty$,
\begin{equation} \label{exp_moment}
\E_N\bigg[ \exp\bigg( \sum_{k=1}^n t_k \X_{\eta_k}(\mathrm{x}_k) \bigg)\bigg]=
\E\bigg[ \exp\bigg( \sum_{k=1}^n  t_k \X_{\eta_k}(\mathrm{x}_k) \bigg)\bigg]   \big( 1+ \o(1) \big), 
\end{equation}
where the error term is uniform in $\boldsymbol{\eta} \in [\epsilon_N ,1]^n$ and $\mathbf{x}$ in a fixed compact subset of $\Omega^n$. Then for any $-\sqrt{2d}<\gamma <\sqrt{2d}$ and $f\in C_c(\Omega)$, the law of $\int_\Omega \frac{e^{\gamma \X_{\epsilon_N}(x)}}{\E_N e^{\gamma \X_{\epsilon_N}(x)}}f(x)dx$ under $\P_N$ tends to that of $\int_\Omega f(x)d\mu^\gamma(x)$ $($under $\P)$ as $N\to\infty$.
\end{theorem}

As mentioned, for the proof of Theorem \ref{th:gmc} we cannot directly apply Theorem \ref{thm:LOS}. Indeed, if e.g. $n=2$ and  $\eta_1=\eta_2=0$,  the condition \eqref{exp_moment} does not hold and in general, if the model has some non-trivial behavior on the local scale, the uniform estimates required by Theorem~\ref{thm:LOS} are not likely to hold.  The idea of the proof of  Theorem~\ref{th:gmc} is to keep the general approach of \cite{LOS18}, but to work under slightly different exponential moment assumptions, which are more of a kin to the approach of \cite{Webb18}.  Some of the technical details also become more involved. For this reason, we will formulate precisely our assumptions and explain the core of the proof in Section~\ref{sect:main}, while the details of how to modify the arguments of \cite{LOS18} will be given in Section~\ref{sect:L2}. 

\subsection{Main result concerning convergence to multiplicative chaos}\label{sect:main}

We  work  in the general context of Section~\ref{sect:background}. Let us fix some $\gamma\in[0,\sqrt{2d})$ -- the extension to $\gamma\in(-\sqrt{2d},0)$ is simply a matter of notational changes by applying our results to $-\X$.

We begin this section by formulating a set of general assumptions on asymptotics of exponential moments of the fields $\X$ and $\X_\epsilon$ with respect to the measure $\P_N$, under which we will be able to prove our main result about convergence to multiplicative chaos.  After this, we state some key lemmas, and prove our main result assuming these lemmas. We conclude our general discussion with the proof of these key lemmas in Section \ref{sect:lem} and a discussion of test functions which are not necessarily compactly supported in Section \ref{sect:extension}. Finally, in Section \ref{sec:special}, we will verify that in the setting of $\h$, these assumptions follow from our results stated in Section \ref{sec:hankel_intro}.

The assumptions we shall require for proving our main result concerning convergence to multiplicative chaos are the following.

\begin{assumption} \label{A:exp_moment}
We suppose that there exists a sequence $\epsilon_N$ $($possibly depending on $\gamma)$ tending to zero as $N\to\infty$ such that the following assumptions hold$:$
\begin{itemize}
\item
For any fixed {$\epsilon,\epsilon' \ge 0$}, 
\begin{equation} \label{A31}
	\lim_{N\to \infty}\frac{\E_N e^{\gamma \X_\epsilon(x)+\gamma \X_{\epsilon'}(y)}}{\E_N e^{\gamma \X_\epsilon(x)}\E_N e^{\gamma \X_{\epsilon'}(y)}} = e^{ \gamma^2
		\E \X_\epsilon(x)\X_{\epsilon'}(y)}
	\end{equation}
uniformly for all $(x,y)$ in any fixed compact subset of $\{(u,v)\in \Omega^2:u\neq v\}$.

\item
For any fixed $\epsilon\geq 0$, $\epsilon'>0$, and $K\subset \Omega$ compact, there exists a positive constant $C=C(\epsilon,\epsilon',K)<\infty$ such that 
\begin{equation}\label{A32}
\sup_{N\geq 1}\sup_{x,y\in K}\frac{\E_N e^{\gamma \X_\epsilon(x)+\gamma\X_{\epsilon'}(y)}}{\E_N e^{\gamma \X_\epsilon(x)}\E_N e^{\gamma \X_{\epsilon'}(y)}}\leq C.
\end{equation}

\item
One has
\begin{equation} \label{A33}
	\frac{\E_N e^{\gamma \X(x)+\gamma \X(y)}}{\E_N e^{\gamma \X(x)}\E_N e^{\gamma \X(y)}}  \ll  |x-y|^{-\gamma^2}
	\end{equation}
uniformly for all $(x,y)$ in a fixed compact subset of  $\Omega^2$ intersected with $\{(x,y)\in \Omega^2: |x-y| \ge \epsilon_N\}$.

 \item For any fixed $\rho>0$,  $\epsilon, \epsilon' \ge 0$,  $n\in \N$ and $\boldsymbol{t}\in \R^n$, one has as $N\to +\infty$,
\begin{align} \label{A1}
	\frac{\E_N e^{\gamma \X_\epsilon(x)+\gamma \X_{\epsilon'}(y)}e^{\sum_{k=1}^n  t_k \X_{\eta_k}(\mathrm{z}_k) } }{\E_N e^{\gamma \X_\epsilon(x)+\gamma \X_{\epsilon'}(y)}} &=
	\big(1+ \o_{\rho,\epsilon,\epsilon'}(1) \big) e^{\sum_{k=1}^n \gamma t_k(\E \X_\epsilon(x)\X_{\eta_k}(\mathrm{z}_k)+\E \X_{\epsilon'}(y)\X_{\eta_k}(\mathrm{z}_k))}\\
	& \notag \qquad \times e^{\frac{1}{2}\E (\sum_{k=1}^n t_k \X_{\eta_k}(\mathrm{z}_k))^2}	
	\end{align}
uniformly for all $\boldsymbol{\eta}\in (\epsilon_N,1]^n$,  $\mathbf{z}$ in any fixed compact subset of $\Omega^n$, and $(x,y)$ in a compact subset of $\{(u,v)\in \Omega^2: |u-v|\geq \rho\}$.

\item
For any $\lambda\in\R$,  one has
 \begin{equation} \label{A2}
\frac{\E_N\big[ e^{\gamma \X(x) +\lambda \X_{\epsilon}(x) }\big]}{\E_N\big[ e^{\gamma \X(x)}\big]} \ll \epsilon^{-\lambda(\lambda+2\gamma)/2} , 
\end{equation}
uniformly for all $\epsilon\ge \epsilon_N$ and for all $x$ in compact subsets of $\Omega$.

\item For any fixed $\epsilon \ge 0$ and $\epsilon'>0$ such that $\epsilon'\ge\epsilon$ and $\lambda\in\R$, 
\begin{equation} \label{A41}
	\limsup_{N\to+\infty} \frac{\E_N e^{\gamma \X_\epsilon(x)+\gamma \X_{\epsilon'}(y)+\lambda\X_\eta(y)}}{\E_N e^{\gamma \X_\epsilon(x)+\gamma \X_{\epsilon'}(y)}}\ll e^{\frac{\lambda^2}{2}\E X_\eta(y)^2+\lambda\gamma \E \X_\epsilon(x)\X_\eta(y)+\lambda\gamma \E \X_{\epsilon'}(y)\X_\eta(y)}	, 
	\end{equation}
uniformly for all  $\eta \ge \epsilon'$ and $(x,y)$ in any fixed compact subset of of $\Omega^2$. Moreover the implied constant does not depend on $(\epsilon,\epsilon')$.

\item For any fixed $\lambda\in\R$, we have
\begin{equation} \label{A42}
	\frac{\E_N e^{\gamma \X(x)+\gamma \X(y)+\lambda\X_\eta(x)}}{\E e^{\gamma \X(x)+\gamma \X(y)}}\ll e^{\frac{\lambda^2}{2}\E X_\eta(x)^2+\lambda\gamma \E \X(x)\X_\eta(x)+\lambda\gamma \E \X(y)\X_\eta(x)}, 
	\end{equation}
uniformly for all  $\eta \ge \epsilon_N$ and $(x,y)$  in any fixed compact subset of $\Omega^2$. 

\item
Finally, we suppose that there exists a small parameter $0<\tau<\min(\frac{(d-\frac{\gamma^2}{2})^2}{4d},\frac{\gamma^2}{2})$ such that for any compact set $A\subset \Omega^2$,  
\begin{equation} \label{A5}
\iint_{A} \1_{|x-y| \le \epsilon_N} \frac{\E_N e^{\gamma \X(x)+\gamma \X(y)}}{\E_N e^{\gamma \X(x)}\E_N e^{\gamma \X(y)}}dxdy 
\underset{A}{\ll}  \epsilon_N^{- (\gamma^2-d)_+ - \tau}.  
\end{equation}	
\end{itemize}
Even though it was not emphasized,  the implied constants above may  depend on $\gamma>0$, the dimension $d\in\N$, as well as $n,\rho, \boldsymbol{t},\lambda$ in \eqref{A1}, \eqref{A2}, \eqref{A41} and \eqref{A42}. Also the convergence in \eqref{A31} is not assumed to be uniform in $\epsilon,\epsilon'$.
\end{assumption}
\medskip

Let us briefly comment on the nature of these assumptions here. First of all, they could be relaxed to a degree. E.g. in the $L^2$-phase -- $0<\gamma<\sqrt{d}$, the argument is simpler (see e.g. \cite{LOS18,Webb15}). Also we do not use in \eqref{A1} in quite this strong a form, but this is a convenient way to write the assumptions. Next we note that by setting $\gamma=0$, \eqref{A1} implies for example the assumptions of Theorem \ref{thm:LOS}  so these are stronger assumptions than \eqref{exp_moment}. Next, we point out that typically in models where log-correlated Gaussian fields play a role, these objects arise on mesoscopic (and global) scales. Thus $\epsilon_N$ should be viewed as a \emph{mesoscopic cutoff} and  conditions like $|x-y|\geq \epsilon_N$ or $\eta\geq \epsilon_N$ are needed to ensure Gaussian behavior. From an asymptotic viewpoint, these cases should  be simpler to analyze.  What is expected to be harder to control is a situation where a possible local structure of the model comes to play, e.g. situations where $|x-y|<\epsilon_N$ or $\epsilon= \epsilon'=0$, and that it is important to obtain uniform estimates in these cases as well. 

\medskip

Our main result concerning convergence to multiplicative chaos is that these assumptions imply such convergence in distribution. More precisely, we have the following theorem.

\begin{theorem} \label{thm:chaos}
If Assumptions~\ref{A:exp_moment} hold, then for any $0<\gamma <\sqrt{2d}$ and $f\in C_c(\Omega)$, the law of 
\begin{equation} \label{integral}
\int_\Omega f(x)\frac{e^{\gamma \X(x)}}{\E_N e^{\gamma \X(x)}}dx 
\end{equation}
under $\P_N$ converges to the law of $\int_\Omega f(x)d\mu^\gamma(x)$, where $\mu^\gamma$ is as in Section \ref{sect:background}.
\end{theorem}

\medskip

The main idea of the proof which originates from \cite{DS11, Berestycki15} is to split the integral  \eqref{integral}
into two parts  depending on the values of the field $\X$. 
Typically we expect that the random measure $\mu^\gamma_N(dx) := \frac{e^{\gamma \X(x)}}{\E_N e^{\gamma \X(x)}}dx$ lives on the set of $\gamma$-thick points (see \eqref{eq:thick} and e.g. Proposition~\ref{prop:TP} below)
and the idea is to introduce the following barrier events: 
  for any $\ell , L \in \N$ with $\ell<L$ and all $x\in \Omega$,  let
\begin{equation} \label{eventA}
A_{\ell,L}(x) = \bigcap_{\ell \le k \le L} \big\{ \X_{e^{-k}}(x) < (\gamma +\varkappa) k \big\} , 
\end{equation}
where $\varkappa>0$ is a small fixed parameter which depends only on $\gamma$ and we will choose later on. We also denote by $A_{\ell,L}^*(x)$ the complement of this event. 
As we already emphasized, the event $A_{\ell,L}(x)$ will turn out to be typical for most $x\in \Omega$ sampled according to the random measure $\mu^\gamma_N$ since these points are not $(\gamma+\varkappa)$-thick.

We split our proof of Theorem \ref{thm:chaos} into two parts. First we get rid of the contribution from the points $x$ for which the event $A_{\ell,L}^*(x)$ occurs using a simple first moment estimate. Then, using the Gaussian estimates from Assumptions~\ref{A:exp_moment}, we will show that if we restrict the integral \eqref{integral} to the set of points $x$ for which the event $A_{\ell,L}(x)$ occurs, then it converges in law a to $\int_\Omega f d\mu^\gamma$ as the various parameters $N, L$ and $\ell$ converge to infinity in a suitable way. 

\begin{lemma} \label{lem:barrier}
Fix two integers $\ell <L$ independent of $N$ and let $\eta_L=e^{-L}$. Then for any $\varkappa>0$ and  any $f\in C_c(\Omega)$,  under Assumptions \ref{A:exp_moment} $($actually, one needs only \eqref{A1}$)$
 \begin{equation*}
 \limsup_{N\to\infty} \E_N \bigg[ \int_{\Omega}|f(x)| \1_{A^*_{\ell,L}(x)}  \frac{e^{\gamma\X_{\eta_L}(x)}}{\E_N e^{\gamma \X_{\eta_L}(x)}}dx \bigg]  \ll_{\varkappa} e^{- \varkappa^2 \ell/2} . 
 \end{equation*}
 Let us also write $M_N=\lfloor \log \epsilon_N^{-1}\rfloor$. Then, we have $($under Assumptions \ref{A:exp_moment}$)$
 \begin{equation*}
 \limsup_{N\to\infty} \E_N \bigg[ \int_{\Omega} |f(x)|\1_{A^*_{\ell,M_N}(x)} \frac{e^{\gamma \X(x)}}{\E_N e^{\gamma \X(x)}}dx \bigg]  \ll_{\varkappa} e^{- \varkappa^2 \ell/2}   .
\end{equation*}
\end{lemma}

\begin{proof}
The arguments for the two cases are very similar, and in fact exactly the same as in \cite[Lemma~2.7]{LOS18}, so we only prove the second estimate -- the proof of the first one would differ simply by applying \eqref{A1} (with $\gamma=0$ and $n=2$ as well as $n=1$) instead of \eqref{A2}. 

Observe that by a union bound, $\1_{A^*_{\ell,M_N}(x)} \le \sum_{k=\ell}^{M_N} e^{\varkappa \X_{e^{-k}}(x)} e^{-\varkappa(\gamma+\varkappa) k} $. 
Using the assumption \eqref{A2} with $\epsilon=0$ and $\lambda= \varkappa$, this shows that uniformly in $x\in\mathrm{supp}(f)$, 
\[
\E_N \bigg[ \1_{A^*_{\ell,M_N}(x)}  \frac{e^{\gamma \X(x)}}{\E_N e^{\gamma \X(x)}} \bigg]  \ll  \sum_{k=\ell}^{M_N} e^{ \varkappa(\varkappa+2\gamma)k/2 -\varkappa(\gamma+\varkappa) k}   \le  \frac{e^{- \varkappa^2 \ell/2}}{1-e^{- \varkappa^2/2}} . 
\]
Since $f\in C_c(\Omega)$, this completes the proof. 
\end{proof}

The harder ingredient of the proof of Theorem \ref{thm:chaos} is the following lemma, whose proof we postpone to Section \ref{sect:L2}. This result is also where the main differences between our approach and that of \cite{LOS18} lie, and where we need the full extent of Assumptions \ref{A:exp_moment}.

\begin{lemma} \label{prop:L2}
Let $\tau$ be as in \eqref{A5}. Suppose that $\sqrt{2\tau}< \varkappa< \min( \frac{d- \gamma^2/2}{\sqrt{2d}} ,\frac{\gamma}{2})$. Then, writing again $\eta_L=e^{-L}$ and $M_N=\log\lfloor \epsilon_N^{-1}\rfloor$, we have under Assumptions \ref{A:exp_moment}, for any $f\in C_c(\Omega)$
\begin{equation*}
\lim_{\ell\to+\infty}\limsup_{L\to+\infty}\limsup_{N\to+\infty}
\E_N \bigg[ \Big|  \int_\Omega f(x) \1_{A_{\ell,M_N}(x)} \frac{e^{\gamma \X(x)}}{\E_N e^{\gamma \X(x)}}dx - \int_\Omega f(x) \1_{A_{\ell,L}(x)} \frac{e^{\gamma \X_{\eta_L}(x)}}{\E_N e^{\gamma \X_{\eta_L}(x)}} \Big|^2\bigg]
=0 .
\end{equation*}
\end{lemma}

Let us now present the proof of Theorem \ref{thm:chaos}, assuming Lemma \ref{prop:L2}.

\begin{proof}[Proof of  Theorem~\ref{thm:chaos}]
Let us choose $\varkappa$ as in the statement of Lemma \ref{prop:L2}. By the triangle inequality, for any $\ell , L\in\N$ such that $\ell<L<M_N$, 
 \begin{align}
\notag \E_N &\bigg[ \Big|  \int_\Omega f(x)\frac{e^{\gamma \X(x)}}{\E_N e^{\gamma \X(x)}}dx   - \int_\Omega f(x) \frac{e^{\gamma \X_{\eta_L}(x)}}{\E_N e^{\gamma \X_{\eta_L}(x)}}  dx \Big|\bigg]  
\\
&  \label{E1}
\le \E_N \bigg[  \int_\Omega |f(x)|  \1_{A^*_{\ell,M_N}(x)}  \frac{e^{\gamma \X(x)}}{\E_N e^{\gamma \X(x)}}dx  + \int_\Omega |f(x)|   \1_{A^*_{\ell,L}(x)} \frac{e^{\gamma \X_{\eta_L}(x)}}{\E_N e^{\gamma \X_{\eta_L}(x)}} \bigg]\\
&\quad \label{E2}
+ \E_N \bigg[ \Big|  \int_\Omega f(x) \1_{A_{\ell,M_N}(x)} \frac{e^{\gamma \X(x)}}{\E_N e^{\gamma \X(x)}}dx - \int_\Omega f(x) \1_{A_{\ell,L}(x)} \frac{e^{\gamma \X_{\eta_L}(x)}}{\E_N e^{\gamma \X_{\eta_L}(x)}} dx \Big|\bigg] . 
\end{align}
By Lemma \ref{lem:barrier}, we see that as we let $N\to +\infty$, $L\to+\infty$ and then $\ell\to+\infty$, \eqref{E1} converges to 0. Similarly, by Lemma~\ref{prop:L2}, \eqref{E2} also converges to 0.  This shows that 
\begin{equation}  \label{E6}
\lim_{L\to \infty}\limsup_{N\to+\infty}
\E_N \bigg[ \Big|  \int_\Omega f(x) \frac{e^{\gamma \X(x)}}{\E_N e^{\gamma \X(x)}}dx - \int_\Omega f(x) \frac{e^{\gamma \X_{\eta_L}(x)}}{\E_N e^{\gamma \X_{\eta_L}(x)}} dx \Big|\bigg]
=0 . 
\end{equation}
Fix $\xi\in\R$ and let $\Phi(u) = e^{\i \xi u}$ for all $u\in\R$.  
Since $\Phi$ is $|\xi|$-Lipschitz continuous, by the triangle inequality for any (large) $L$,
 \begin{align}
&\left| \E_N\bigg[ \Phi\Big(  \int_\Omega f(x)  \frac{e^{\gamma \X(x)}}{\E_N e^{\gamma \X(x)}}dx\Big)\bigg] - \E \bigg[\Phi\Big(  \int_\Omega f(x)  d\mu^\gamma(x)\Big)  \bigg]\right|\notag \\
&\notag
\qquad \le  |\xi|\ \E_N \bigg[ \Big|  \int_\Omega f(x) \frac{e^{\gamma \X(x)}}{\E_N e^{\gamma \X(x)}}dx - \int_\Omega f(x) \frac{e^{\gamma \X_{\eta_L}(x)}}{\E_N e^{\gamma \X_{\eta_L}(x)}} dx \Big|\bigg]
\\
&\label{E4}
\qquad \quad + \left| \E_N\bigg[ \Phi\Big(  \int_\Omega f(x)  \frac{e^{\gamma \X_{\eta_L}(x)}}{\E_N e^{\gamma \X_{\eta_L}(x)}}\Big)\bigg] - {\E} \bigg[\Phi\Big(  \int_\Omega f(x) \frac{e^{\gamma \X_{\eta_L}(x)}}{\E e^{\gamma \X_{\eta_L}(x)}} dx\Big)  \bigg]\right| \\
& \qquad \quad \label{E5}
 + |\xi|\ \E\bigg[ \Big|  \int_\Omega f(x) \frac{e^{\gamma \X_{\eta_L}(x)}}{\E e^{\gamma \X_{\eta_L}(x)}}dx - \int_\Omega f(x)  d\mu^\gamma(x) \Big|\bigg] . 
 \end{align}
By Lemma~\ref{le:wcvg} (whose use is justified by \eqref{A1} with $\gamma=0$ and  $n=1$), \eqref{E4} converges to 0 as $N\to+\infty$ and, by the very definition of multiplicative chaos, as reviewed in Section \ref{sect:background}, \eqref{E5} converges to 0 as $L\to\infty$. Combining these facts with \eqref{E6}, we obtain
\[
\limsup_{N\to+\infty} \left| \E_N\bigg[ \Phi\Big(  \int_\Omega f(x)  \frac{e^{\gamma\X(x)}}{\E_N e^{\gamma \X(x)}}dx\Big)\bigg] - {\E} \bigg[\Phi\Big(  \int_\Omega f(x)  d\mu^\gamma(x)\Big)  \bigg]\right| =0 . 
\]
As the pointwise convergence of the characteristic function of a sequence of random variables implies convergence in law, this completes the proof.
\end{proof}
We now turn to the proof of Lemma \ref{prop:L2},

\subsection{Proof of Lemma \ref{prop:L2}}\label{sect:L2}
Before turning to the proof of Lemma \ref{prop:L2}, we find it convenient to introduce some further notation. More precisely, we introduce certain biased probability measures. Similar objects play an important role in the construction of multiplicative chaos measures in \cite{Berestycki15}.
	
\begin{definition}\label{def:bias}
	For any $x,y\in \Omega$, $N\in \N$ and $\epsilon,\epsilon'\geq 0$, let
	\begin{equation}\label{eq:Wdef}
	\mathscr{W}^{N}_{\epsilon, \epsilon'}(x,y): = \frac{\E_N\big[ e^{\gamma \X_\epsilon(x) +\gamma \X_{\epsilon'}(y) }\big]}{\E_N\big[ e^{\gamma \X_\epsilon(x)}\big]\E_N\big[ e^{\gamma\X_{\epsilon'}(y) }\big]} 
	\end{equation}
	and define a new probability measure  $\q_{N,\epsilon,\epsilon'}^{(x,y)}$ $($on the same sample space and $\sigma$-algebra as $\P_N)$ by
	\begin{equation}\label{eq:Qdef}
	\frac{d\q_{N,\epsilon,\epsilon'}^{(x,y)}}{d \P_N} :=  \frac{e^{\gamma \X_\epsilon(x)+ \gamma \X_{\epsilon'}(y)  }}{\E_N\big[e^{\gamma \X_\epsilon(x) +\gamma \X_{\epsilon'}(y) }\big]}.
	\end{equation}
	Moreover for any $x,y\in \Omega$ and $\epsilon,\epsilon'>0$ define also the probability measure 
	\begin{equation}\label{eq:QGdef}
	\frac{d\q_{\epsilon,\epsilon'}^{(x,y)}}{d \P} :=  \frac{e^{\gamma \X_\epsilon(x)+ \gamma \X_{\epsilon'}(y)  }}{\E\big[e^{\gamma \X_\epsilon(x) +\gamma \X_{\epsilon'}(y) }\big]}.
	\end{equation}
	By Girsanov's theorem, the law of $\X(\cdot)$ under $\q_{\epsilon,\epsilon'}^{(x,y)}$ is simply the law of $\X(\cdot)+\gamma\E \X(\cdot)\X_\epsilon(x)+\gamma\E \X(\cdot)\X_{\epsilon'}(y)$. Using this representation, we extend the definition of $\q_{\epsilon,\epsilon'}^{(x,y)}$ to $\epsilon=0$ or $\epsilon'=0$ $($or $\epsilon=\epsilon'=0)$.
\end{definition}

Note that in this notation, Assumptions \ref{A:exp_moment} can be formulated in a concise way. For example \eqref{A31} becomes a statement about $\mathscr W_{\epsilon,\epsilon'}^N(x,y)$ converging to $\E e^{\gamma^2\E \X_\epsilon(x)\X_{\epsilon'}(y)}$. Similarly e.g. \eqref{A1} can be written as 
\[
\q_{N,\epsilon,\epsilon'}^{(x,y)}\big[ e^{\sum_{k=1}^n  t_k \X_{\eta_k}(\mathrm{z}_k) } \big] =
{\q}_{\epsilon,\epsilon'}^{(x,y)}\big[ e^{\sum_{k=1}^n  t_k \X_{\eta_k}(\mathrm{z}_k) } \big] \big(1+ \o_{\rho,\epsilon,\epsilon'}(1) \big), 
\]	
where we have used the shorthand notation $\q(F)$ for the expectation of the random variable $F$ under the measure $\q$.

Other than this, we work with the notation and under the assumptions  of Section~\ref{sect:main}. To simplify notation, we will write $\eta=e^{-L}$ and $M=\lfloor \log\epsilon_N^{-1}\rfloor$ instead of $\eta_L$ and $M_N$. Also we will constantly assume Assumptions \ref{A:exp_moment} without further statements, and we will assume as in Lemma \ref{prop:L2} that 
\begin{equation} \label{kappaalpha}
\sqrt{2\tau}<\varkappa<\min\left(\frac{d-\frac{\gamma^2}{2}}{\sqrt{2d}},\frac{\gamma}{2}\right),
\end{equation} 
where $\tau$ is as in \eqref{A5}. Throughout this section, we fix $\ell\in\N$ and let $\delta = e^{-\ell}$ -- we will pass to the limit as $\ell\to+\infty$ at the very last step. 

For clarity reasons,  we split the proof of Lemma~\ref{prop:L2} into three lemmas: Lemmas \ref{lem:trunc1}, \ref{lem:trunc2}, and \ref{lem:cvg}, which will be proved in Section~\ref{sect:lem}. 
Compared with the proof of Theorem~\ref{thm:LOS}, that is \cite[Theorem 1.7]{LOS18}, 
we will use  Lemma~\ref{lem:trunc1} to remedy the issue that the asymptotics of $\E_N e^{\gamma \X(x)+\gamma\X(y)}$ for $|x-y|\leq \epsilon_N$ might not be Gaussian -- e.g. due to a complicated local structure in the model. In particular, the assumptions \eqref{A42} and \eqref{A5} will be crucial in the proof of Lemma~\ref{lem:trunc1}. 
Then, the proof of Lemma~\ref{lem:cvg} is essentially the same as that of \cite[Lemma 2.9]{LOS18} and relies only on the asymptotics \eqref{A1}.

We introduce now some notation that will allow expressing some of the relevant quantities in our proof in a simple way.  For any  $g\in L^\infty(\Omega\times \Omega)$,  $\epsilon, \epsilon' \ge 0$, $ L, L' \in\N$ with $L,L' > \ell$, we define
\[
\Upsilon^{\epsilon, \epsilon'}_{N,L,L'}(g) := \iint_{\Omega\times \Omega}  g(x,y)\ \q_{N,\epsilon,\epsilon'}^{(x,y)}\big[A_{\ell, L}(x) \cap A_{\ell, L'}(y) \big]\
\mathscr{W}^N_{\epsilon, \epsilon'} (x,y) dxdy . 
\]
We also write for $f\in C_c(\Omega)$, 
\[
F_\rho(x,y) = f(x)f(y) \1_{|x-y| \ge \rho}
\] 
for all $0\le \rho\le 1$. Using the definition of $\q_{N,\epsilon,\epsilon'}^{(x,y)}$ and $\mathscr W_{\epsilon,\epsilon'}^N(x,y)$, a simple calculation shows that we can write the object we are interested in the following way:
 \begin{align} \notag
\Delta^N_{\ell,L}  :&= 
\E_N \bigg[ \Big|  \int_\Omega f(x) \1_{A_{\ell,M}(x)} \frac{e^{\gamma \X(x)}}{\E_N e^{\gamma \X(x)}}dx - \int_\Omega f(x) \1_{A_{\ell,L}(x)}\frac{e^{\gamma \X_{\eta}(x)}}{\E_N e^{\gamma \X_{\eta}(x)}}dx \Big|^2\bigg] \\
&\label{Delta}
=  \Upsilon^{0,0}_{N,M,M}(F_0) -2 \Upsilon^{0, \eta}_{N,M,L}(F_0)+\Upsilon^{\eta, \eta}_{N,L,L}(F_0) . 
\end{align}
The first step is to show that we may replace $F_0$ by $F_{\delta}$  in formula \eqref{Delta} by paying only a small price. In fact, removing the diagonal from the integrals will allow using e.g. the precise asymptotics of \eqref{A1}.  The  statement about replacing $F_0$ by $F_{\delta}$ is the following.

\begin{lemma} \label{lem:trunc1}
If $(L' , \epsilon) = (L,\eta) \text{ or }(L',\epsilon)=(M,0)$, then for each $f\in C_c(\Omega)$, 
\begin{equation}  \label{trunc0}
\limsup_{N\to+\infty} \Upsilon^{\epsilon, \eta}_{N,L',L}(f(x)f(y)\1_{|x-y|\le \delta}) \ll \delta^{\frac{\varkappa^2}{2}} .
\end{equation}
Similarly, if $\tau$ is as in \eqref{A5}, then
\begin{equation}  \label{trunc1}
\limsup_{N\to+\infty} \Upsilon^{0, 0}_{N,M,M}(f(x)f(y)\1_{|x-y|\le \delta}) \ll \delta^{\frac{\varkappa^2}{2}-\tau} .
\end{equation}
where the implied constants in \eqref{trunc0} and \eqref{trunc1} depend only on the test function $f$. 
\end{lemma}

The second step  is to replace $M$  by $L$ which is independent of the parameter $N$ in formula \eqref{Delta} (with $F_0$ replaced by $F_{\delta}$). This is justified by the following result.

\begin{lemma}\label{lem:trunc2}
For any $L\in\N$ such that $\ell <L<M$, we have
\begin{equation}  \label{trunc2}
\limsup_{N\to+\infty} \big| \Upsilon_{N,L,L}^{0,\eta}(F_{\delta})  - \Upsilon^{0, \eta}_{N,M,L}(F_{\delta}) \big|
\ll_{\ell} \eta^{\varkappa^2/2} ,
\end{equation}
and
\begin{equation}  \label{trunc3}
\limsup_{N\to+\infty}  \big| \Upsilon^{0,0}_{N,L,L}(F_{\delta})  - \Upsilon^{0, 0}_{N,M,M}(F_{\delta}) \big|
\ll_{\ell} \eta^{\varkappa^2/2}, 
\end{equation}
where the implied constants are independent of $L$.
\end{lemma}

The final step will be to compute the limit of $\Upsilon^{\eta, \eta}_{N,L,L}(F_{\delta})$ and similar  quantities allowing us to control the asymptotics of \eqref{Delta} as $N\to+\infty$ and then $L\to+\infty$. 

\begin{lemma} \label{lem:cvg}
For any $\ell\in\N$ and $x,y\in \Omega$, the following limit exists
\[
q_\ell(x,y) =  \lim_{L\to+\infty}  {\q}_{0,0}^{(x,y)}\big[A_{\ell, L}(x) \cap A_{\ell, L}(y) \big] . 
\]
Moreover, if $\epsilon= \eta\text{ or } \epsilon=0$, we have
\begin{equation} \label{E7}
\lim_{L\to+\infty} \lim _{N\to+\infty}
\Upsilon^{\epsilon, \eta}_{N,L,L}(F_{\delta}) 
=
\iint F_{\delta}(x,y) q_\ell(x,y) e^{\gamma^2 {\E}[\X(x) \X(y)]} dxdy .  
\end{equation}
and
\begin{equation} \label{E8}
\lim_{L\to+\infty}\lim _{N\to+\infty}\Upsilon^{0,0}_{N,L,L}(F_{\delta})
= 
\iint F_{\delta}(x,y) q_\ell(x,y) e^{\gamma^2 {\E}[\X(x) \X(y)]} dxdy . 
\end{equation}
\end{lemma}

Observe that since $\q_{0,0}^{(x,y)}$ is a family of Gaussian measures, we have $q_\ell\in L^\infty(\Omega\times\Omega)$  and  the integrals on the RHS of \eqref{E7}--\eqref{E8} are finite since $f\in C_c(\Omega)$ and 
\[
F_{\delta}(x,y) \le \|f\|_{L^\infty}^2 \1_{|x-y| \ge \delta} \1_{\mathrm{supp}(f)\times \mathrm{supp}(f)}(x,y).
\] 
We are now ready to give the proof of Lemma~\ref{prop:L2} assuming Lemmas \ref{lem:trunc1}, \ref{lem:trunc2}, and \ref{lem:cvg}.

\begin{proof}[Proof of  Lemma~\ref{prop:L2}]
By formula \eqref{Delta} and Lemma~\ref{lem:trunc1},  we have 
\begin{equation*}
\Delta^N_{\ell,L} 
=  \Upsilon^{0,0}_{N,M,M}(F_{\delta}) -2 \Upsilon^{0, \eta}_{N,M,L}(F_{\delta})+\Upsilon^{\eta, \eta}_{N,L,L}(F_{\delta}) + \O(\delta^{\frac{\varkappa^2}{2}-\tau }).
\end{equation*}
Then, by Lemma~\ref{lem:trunc2},  this implies that as $N\to+\infty$, 
\begin{equation}\label{Delta2}
\Delta^N_{\ell,L} 
 =  \Upsilon^{0,0}_{N,L,L}(F_{\delta}) -2 \Upsilon^{0, \eta}_{N,L,L}(F_{\delta})+\Upsilon^{\eta, \eta}_{N,L,L}(F_{\delta}) + \O_\ell(\eta^{\varkappa^2/2}) + \O(\delta^{\frac{\varkappa^2}{2}-\tau}) . 
\end{equation}
By Lemma \ref{lem:cvg}, all of the $\Upsilon$-terms on the RHS of \eqref{Delta2} converge to the same finite limit if we let first $N\to+\infty$ and then $L\to+\infty$. 
Hence, since $\eta\to0$ as $L\to+\infty$ and $\delta\to0$ as $\ell\to+\infty$, we conclude that
\[
\lim_{\ell\to+\infty}\limsup_{L\to+\infty}\limsup_{N\to+\infty} \Delta^N_{\ell,L}  =0,
\]
which completes the proof.
\end{proof}
To finish our proof of Theorem \ref{thm:chaos}, we turn to the proof of our key lemmas in the next section.

\subsection{Proof of Lemmas~\ref{lem:trunc1}, \ref{lem:trunc2} and~\ref{lem:cvg}} \label{sect:lem}
Recall that we are assuming that $0<\gamma <\sqrt{2d}$, 
$\varkappa$ satisfies the condition \eqref{kappaalpha} where $\tau>0$ is as in \eqref{A5} and we write $\eta=\eta_L=e^{-L}$, $\delta=e^{-\ell}$ with $\ell\leq L$, and $M=\lfloor \log \epsilon_N^{-1}\rfloor$. We begin with the proof of Lemma \ref{lem:trunc1}. 

\begin{proof}[Proof of Lemma~\ref{lem:trunc1}]
We give the proof of the estimate \eqref{trunc0} in the case where $(L', \epsilon)= (M,0)$ -- the other case is analogous. 
Recall that by definition,
\[
\Upsilon^{0, \eta}_{N,M,L}(f(x)f(y)\1_{|x-y|\le \delta}) = \iint_{\Omega^2} f(x)f(y)\1_{|x-y|\le \delta}\ \q_{N,0,\eta}^{(x,y)}\big[A_{\ell, M}(x) \cap A_{\ell, L}(y) \big]\
\mathscr{W}^N_{0, \eta} (x,y) dxdy . 
\]
If $k =\min( \lfloor \log |x-y|^{-1} \rfloor ,L )$ and  $|x-y|\le \delta$, we have by Markov's inequality:
\[\begin{aligned}
 \q_{N,0,\eta}^{(x,y)}\big[A_{\ell, L}(y) \big]
 & \le  \q_{N,0,\eta}^{(x,y)}\big[ \X_{e^{-k}}(y) < (\gamma+\varkappa) k\big]  \\
 &\le e^{(\gamma^2-\varkappa^2) k} {\q}_{N,0,\eta}^{(x,y)}\big[ e^{(\varkappa-\gamma)\X_{e^{-k}}(y)} \big] , 
\end{aligned}\]
and using \eqref{A41} (which is readily translated into a statement about $\q_{N,0,\eta}^{(x,y)}$),  we obtain for some constant $C$ independent of $\ell,L,x,$ and $y$,
\[
\limsup_{N\to+\infty}  \q_{N,0,\eta}^{(x,y)}\big[A_{\ell, L}(y) \big] \le C e^{(\gamma^2-\varkappa^2) k} {\q}_{0,\eta}^{(x,y)}\big[ e^{(\varkappa-\gamma)\X_{e^{-k}}(y)} \big] , 
\]
for all  $x,y\in\mathrm{supp}(f)$ such that $|x-y|\le \delta$.
Recalling that under $\P$, the covariance of the Gaussian process $\X$ satisfies \eqref{cov_estimate}, a simple Gaussian calculation shows that by our choice of $k$, 
\[
 {\q}_{0,\eta}^{(x,y)}\big[ e^{(\varkappa-\gamma)\X_{e^{-k}}(y)} \big]  
 \ll e^{(\gamma -\varkappa)^2k/2 -2\gamma(\gamma-\varkappa)k}
 = e^{-(\frac{3}{2}\gamma^2 - \gamma \varkappa - \frac{\varkappa^2}{2})k} ,
\]
which implies that for all $(x,y)\in\{(u,v)\in \mathrm{supp}(f)^2:|u-v|\leq \delta\}$, 
\[
\limsup_{N\to+\infty} \q_{N,0,\eta}^{(x,y)}\big[A_{\ell, L}(y) \big] \ll e^{-(\gamma-\varkappa)^2k/2}
 = \big(\max(|x-y|,\eta)\big)^{(\gamma-\varkappa)^2/2},
\]
the implied constant being independent of the parameters $\ell, L$. Then \eqref{A32} allows us to use the reverse Fatou lemma to take the limit under the integral, and combining \eqref{A31} with the above estimates, we see that with the required uniformity:
\[\begin{aligned}
\limsup_{N\to+\infty}  \Upsilon^{0, \eta}_{N,M,L}(f(x)f(y)\1_{ |x-y|\le \delta}) 
&\ll \iint_{\substack{\Omega^2 \\ |x-y|\le \delta}}  f(x)f(y)\big( \max(|x-y|,\eta)\big)^{-\gamma^2 +(\gamma-\varkappa)^2/2 } dxdy \\
&\ll \delta^{\frac{\varkappa^2}{2}} .
\end{aligned}\]
Note that we used here our assumption that $d > \frac{\gamma^2}{2}+\sqrt{2d}\varkappa$.

\medskip

We now turn to the proof of the estimate \eqref{trunc1} which is quite similar. On the one hand, by using the asymptotics \eqref{A42} (again, this is readily seen to be equivalent to a statement about $\q_{N,0,0}^{(x,y)}[e^{\lambda \X_\eta(y)}]$) like we used \eqref{A41} in the first part of the proof, since $M= \lfloor \log \epsilon_N^{-1}\rfloor$,  we obtain that  uniformly for all $(x,y)\in\{(u,v)\in \mathrm{supp}(f)^2:|u-v|\leq \delta\}$,
\begin{equation}\label{bd3}
 \q_{N,0,0}^{(x,y)}\big[A_{\ell, M}(x) \big] 
 \ll \big(\max(\epsilon_N ,|x-y|) \big)^{(\gamma-\varkappa)^2/2} . 
\end{equation}
Then, by \eqref{A33} (translated into a statement about $\mathscr W_{0,0}^N(x,y)$), this implies that
\begin{equation} \label{bd1}
\Upsilon^{0, 0}_{N,M,M}(f(x)f(y)\1_{\epsilon_N\le |x-y|\le \delta}) \ll \delta^{\frac{\varkappa^2}{2}} ,
\end{equation} 
with the implied constants in \eqref{bd3} and \eqref{bd1} being independent of the parameters $N$ and $\ell$. On the other hand, by using the asymptotics \eqref{bd3} and \eqref{A5} (again, this can be readily written in terms of $\mathscr W_{0,0}^N(x,y)$), we obtain
\[
\Upsilon^{0, 0}_{N,M,M}(f(x)f(y)\1_{|x-y|\le \epsilon_N}) 
 \ll \epsilon_N^{(\gamma-\varkappa)^2/2} \epsilon_N^{-(\gamma^2-d)_+ -\tau}   
\]
where the implied constant depends only on the test function $f$. 
Using the condition \eqref{kappaalpha}, we obtain 
\begin{equation} \label{bd2}
\Upsilon^{0, 0}_{N,M,M}(f(x)f(y)\1_{|x-y|\le \epsilon_N}) 
\ll \epsilon_N^{\frac{\varkappa^2}{2}-\tau} .
\end{equation}
Combining \eqref{bd1} and \eqref{bd2} completes the proof. 
\end{proof}

The next step is the proof of Lemma \ref{lem:trunc2}.

\begin{proof}[Proof of Lemma~\ref{lem:trunc2}]
We begin by proving the estimate \eqref{trunc2}. By definition, we have
\[ \begin{aligned}
 \left| \Upsilon^{0, \eta}_{N, L,L}(F_{\delta}) - \Upsilon^{0,\eta}_{N, M,L}(F_{\delta}) \right| 
&\leq  \iint |F_{\delta}(x,y)| \q_{N,0,\eta}^{(x,y)}\big[A_{\ell, L}(x) \cap A_{\ell, L}(y) \cap A^*_{L+1, M}(x) \big] \mathscr{W}^N_{0, \eta} (x,y)  dxdy \\
& \le  \| f\|_{L^\infty}^2   \iint_{\supp(f)^2}  \1_{|x-y|\ge \delta} \q_{N,0,\eta}^{(x,y)}\big[ A_{L+1, M}^*(x) \big]  \mathscr{W}^N_{0, \eta} (x,y)  dxdy . 
\end{aligned}\] 
Moreover, by a union bound and Markov's inequality,
\begin{align*}
 \q_{N,0,\eta}^{(x,y)}\big[ A_{L+1, M}^*(x) \big] 
 & \le \sum_{L< k\le M} \q_{N,0,\eta}^{(x,y)}\big[ \X_{e^{-k}}(x) \ge  (\gamma+\varkappa)k\big] \\
 & \le \sum_{L<k\le M} e^{- k\varkappa^2} \q_{N,0,\eta}^{(x,y)} \big[ e^{\varkappa( \X_{e^{-k}}(x) -\gamma k)} \big] 
\end{align*}

From \eqref{A1} with $\epsilon=0, \epsilon'=\eta$ (translated into a statement about $\q_{N,0,\eta}^{(x,y)}[e^{\varkappa \X_{e^{-k}}(x)}]$), if $N$ is sufficiently large, ``sufficiently large" depending on $\delta$ and $\mathrm{supp}(f)$,  we have  for all $(x,y)\in \mathrm{supp}(f)^2$ with $|x-y|\geq \delta$ and  for all $k = L, \dots , M$, 
\[
\q_{N,0,\eta}^{(x,y)} \big[ e^{\varkappa( \X_{e^{-k}}(x) -\gamma k)} \big] \le 2 \q_{0,\eta}^{(x,y)} \big[ e^{\varkappa( \X_{e^{-k}}(x) -\gamma k)} \big] . 
\]
By \eqref{cov_estimate}, for $|x-y|\geq \delta$, we see that under the measure  ${\q}_{0,\eta}^{(x,y)}$, $\X_{e^{-k}}(x)$ is a Gaussian random variable with variance $k+\O(1)$
and mean $\gamma \E \X(x)\X_{e^{-k}}(x)+ \gamma \E \X_{\eta}(x)\X_{e^{-k}}(y) = \gamma k + \O_\delta(1)$.
This implies that 
\begin{align} \notag
 \q_{N,0,\eta}^{(x,y)}\big[ A_{L+1, M}^*(x) \big] 
 & \ll_\delta \sum_{k>L} e^{- k\varkappa^2/2} \\
 & \label{trunc4}
  \ll_\delta \eta^{\varkappa^2/2} , 
\end{align}
where the implied constant is independent of the relevant $x,y$, $N$, and $L$.
Using \eqref{A32} to justify taking the limit under the integral and then combining \eqref{A31} with \eqref{trunc4}, we obtain
\[
\limsup_{N\to+\infty} \big| \Upsilon^{0, \eta}_{N, L,L}(F_{\delta}) - \Upsilon^{0,\eta}_{N, M,L}(F_{\delta})  \big|
 \ll_{\delta}  \eta^{\varkappa^2/2}   \delta^{-\gamma^2+d}\ll_\delta \eta^{\varkappa^2/2} .
\]

\smallskip

We proceed in a similar way to prove the estimate \eqref{trunc3}. By symmetry of the integrand, we check that  
\[ \begin{aligned}
\left| \Upsilon^{0, 0}_{N, L,L}(F_\delta) - \Upsilon^{0,0}_{N, M,M}(F_{\delta}) \right|
& \le 2  \| f\|_{L^\infty}^2   \iint_{\mathrm{supp}(f)^2}  \1_{|x-y|\ge \delta} \q_{N,0,0}^{(x,y)}\big[ A_{L+1, M}^*(x) \big] \mathscr{W}^N_{0, 0} (x,y)  dxdy \\
& \ll_{\delta}  \eta^{\varkappa^2/2}   \delta^{-\gamma^2+d} ,
\end{aligned}\]
where the argument is essentially identical to the previous case apart from us having $\q_{N,0,0}^{(x,y)}$ instead of $\q_{N,0,\eta}^{(x,y)}$ (which changes nothing in our argument) and we made use of \eqref{A33} instead of \eqref{A32}. This concludes the proof. 
\end{proof}
To conclude this section, we prove Lemma \ref{lem:cvg}.
\begin{proof}[Proof of Lemma~\ref{lem:cvg}]
Let us begin by proving that the function $q_\ell$ exists. This is an immediate consequence of the monotone convergence theorem as $\1_{A_{\ell,L}(x)\cap A_{\ell,L}(y)}$ is decreasing in $L$. Let us turn to the other statements --
the simplest one being \eqref{E8}.
  
For any $L \in\N$, the asymptotics \eqref{A1}  imply that, for any $(x,y)\in\{(u,v)\in \mathrm{supp}(f)^2:|u-v|\geq \delta\}$,  the Laplace transform of the law of
$\big(\X_{e^{-k}}(x), \X_{e^{-k}}(y) \big)_{k=1}^L$ under $\q_{N,0,0}^{(x,y)}$  converges to that of a multivariate Gaussian vector on $\R^{2L}$, with a mean and covariance inherited from the law $\q_{0,0}^{(x,y)}$. This implies that the law of $\big(\X_{e^{-k}}(x), \X_{e^{-k}}(y) \big)_{k=1}^L$ under $\q_{N,0,0}^{(x,y)}$ converges to a suitable multivariate normal distribution.
Since such a Gaussian measure is absolutely continuous with respect to the Lebesgue measure on $\R^{2L}$, this directly implies (e.g. by the portmanteau theorem) that  
\begin{equation} \label{lim1}
\lim_{N\to+\infty} \q_{N,0,0}^{(x,y)}\big[A_{\ell, L}(x) \cap A_{\ell, L}(y) \big] =
{\q}_{0,0}^{(x,y)}\big[A_{\ell, L}(x) \cap A_{\ell, L}(y) \big]   .
\end{equation}
Moreover, the asymptotics \eqref{A31} also imply that 
\begin{equation} \label{lim2}
\lim_{N\to+\infty}\mathscr{W}^N_{0,0}(x,y) =e^{\gamma^2 {\E}[\X(x) \X(y)]}  ,
\end{equation}
uniformly in $(x,y)\in\mathrm{supp}(f)^2$ with $|x-y| >\delta$. Hence, by combining \eqref{lim1} and \eqref{lim2}, 
we have by the dominated convergence theorem that
\begin{equation} \label{lim3}
\lim _{N\to+\infty}\Upsilon^{0,0}_{N,L,L}(F_{\delta})
= 
\iint F_{\delta}(x,y) {\q}_{0,0}^{(x,y)}\big[A_{\ell, L}(x) \cap A_{\ell, L}(y) \big] e^{\gamma^2 {\E}[\X(x) \X(y)]} dxdy . 
\end{equation}
Then, \eqref{E8} follows from  \eqref{lim3} and the definition of the function $q_\ell$ by the monotone convergence theorem.

\medskip

Let us now turn to the proof of \eqref{E7} in the case where $\epsilon =0$ -- the argument when $\epsilon=\eta$ is identical.
Like in the first part of the proof, we have 
\begin{equation*}
\lim_{N\to+\infty} \q_{N,0,\eta}^{(x,y)}\big[A_{\ell, L}(x) \cap A_{\ell, L}(y) \big] =
{\q}_{0,\eta}^{(x,y)}\big[A_{\ell, L}(x) \cap A_{\ell, L}(y) \big]   ,
\end{equation*}
for all $(x,y)\in\mathrm{supp}(f)^2$ with $|x-y| >\delta$.
Then, by \eqref{A31}
$\mathscr{W}^N_{0,\eta}(x,y) \to e^{\gamma^2 {\E}[\X(x) \X_{\eta}(y)]}$ (uniformly in $(x,y)\in \mathrm{supp}(f)^2$ with $|x-y|\geq \delta$) as $N\to+\infty$, this implies that 
\[
\lim _{N\to+\infty}\Upsilon^{0,\eta}_{N,L,L}(F_{\delta})
= 
\iint F_{\delta}(x,y) {\q}_{0,\eta}^{(x,y)}\big[A_{\ell, L}(x) \cap A_{\ell, L}(y) \big] e^{\gamma^2 {\E}[\X(x) \X_\eta(y)]} dxdy . 
\]
First of all, ${\E}[\X(x) \X_{\eta}(y)] \to {\E}[\X(x) \X(y)]$ (uniformly in $(x,y)\in \mathrm{supp}(f)^2$ with $|x-y|\geq \delta$) as $L\to\infty$ by basic convolution estimates. Moreover as ${\E}[\X(x) \X_\eta(y)]  \ll_\ell 1$  with the implied constant being uniform in $(x,y) \in\mathrm{supp}(f)^2$ with $|x-y| \ge \delta$, the dominated convergence theorem implies that to complete the proof, it remains to show that
\begin{equation} \label{lim4}
\lim_{L\to+\infty}  {\q}_{0,\eta}^{(x,y)}\big[A_{\ell, L}(x) \cap A_{\ell, L}(y) \big] = q_\ell(x,y).
\end{equation}
To do this,  fix an integer $L'>\ell$ and define for all $\epsilon\ge 0$ the vectors  $\mathbf{m}_\epsilon , \mathbf{m}_\epsilon' \in \R^{L'}$ by
\[ \begin{cases}
\mathrm{m}_{\epsilon,k}  = \gamma \E \X(x)\X_{e^{-k}}(x)+ \gamma \E \X_{e^{-k}}(x)\X_\epsilon(y) ,   \\
\mathrm{m}'_{\epsilon,k}   = \gamma \E\X(x)\X_{e^{-k}}(y) + \gamma \E \X_{e^{-k}}(y)\X_\epsilon(y), 
\end{cases} \qquad k=1,\dots L' .  \]
Observe that by Girsanov's theorem for any $\epsilon\ge 0$,
\[
{\q}_{0,\eta}^{(x,y)}\big[A_{\ell, L'}(x) \cap A_{\ell, L'}(y) \big] 
= {\P}{ \left[ \begin{array}{l}
\X_{e^{-k}}(x) \le {(\gamma+\varkappa)k} -  \mathrm{m}_{\epsilon,k}  \\
\X_{e^{-k}}(y) \le {(\gamma+\varkappa)k} -  \mathrm{m}'_{\epsilon,k}
\end{array},\
 k =\ell, \dots L'\right]} . 
\]
Hence, e.g. by combining the absolute continuity of  a Gaussian measure with respect to the Lebesgue measure on $\R^{2L'}$ with the continuity
of the functions $\epsilon \mapsto \mathbf{m}_\epsilon , \mathbf{m}_\epsilon' $, we see that for any fixed integer $L'>\ell$,
\begin{equation} \label{lim5}
\lim_{L\to+\infty}{\q}_{0,\eta}^{(x,y)}\big[A_{\ell, L'}(x) \cap A_{\ell, L'}(y) \big]   = {\q}_{0,0}^{(x,y)}\big[A_{\ell, L'}(x) \cap A_{\ell, L'}(y)\big] . 
\end{equation}

Finally, if $\ell<L'<L$, we have once again by Girsanov and Markov's inequality that for $|x-y|\geq \delta$, 
\[ \begin{aligned} &0\le
{\q}_{0,\eta}^{(x,y)}\big[A_{\ell, L'}(x) \cap A_{\ell, L'}(y) \big]  - {\q}_{0,\eta}^{(x,y)}\big[A_{\ell, L}(x) \cap A_{\ell, L}(y) \big]  \\
&\hspace{3cm}
\le {\q}_{0,\eta}^{(x,y)}\big[A^*_{L'+1, L}(x) \big]  + {\q}_{0,\eta}^{(x,y)}\big[A^*_{L'+1, L}(y) \big]   \\
&\hspace{3cm}
\le \sum_{L'<k\le L} {\P}\big[\X_{e^{-k}}(x) > {(\gamma+\varkappa)k} -  \mathrm{m}_{\eta,k} \big]
+  {\P}\big[\X_{e^{-k}}(y) > {(\gamma+\varkappa)k} -  \mathrm{m}'_{\eta,k} \big] \\
&\hspace{3cm}
\ll_\ell \sum_{L'<k\le L} e^{- k\varkappa^2/2}
\end{aligned}\]
since, under the law ${\P}$, $\X_{e^{-k}}(x)$ is a  centered Gaussian variable with variance ${k} + \O(1)$ and  $\mathrm{m}_{\eta, k} = {\gamma k} +\O(1)$ uniformly for all $k\le L$ and $x,y\in\supp(f)^2$ such that  $|x-y| \ge \delta$. 
Hence we have established that (with an implied constant being independent of $L',L$)
\begin{equation} \label{lim6}
\Big| {\q}_{0,\eta}^{(x,y)}\big[A_{\ell, L'}(x) \cap A_{\ell, L'}(y) \big]  - {\q}_{0,\eta}^{(x,y)}\big[A_{\ell, L}(x) \cap A_{\ell, L}(y) \big] \Big| 
\ll_\ell e^{-L' \varkappa^2/2} . 
\end{equation}
Combining the estimates \eqref{lim6} and \eqref{lim5}, we conclude that for any integer $L' >\ell$, 
\[
\lim_{L\to+\infty}{\q}_{0,\eta}^{(x,y)}\big[A_{\ell, L}(x) \cap A_{\ell, L}(y) \big]   = {\q}_{0,0}^{(x,y)}\big[A_{\ell, L'}(x) \cap A_{\ell, L'}\big] +
\O_\ell(e^{-L' \varkappa^2/2}) . 
\]
 By the monotone convergence theorem, the RHS converges to $q_\ell(x,y)$ as $L'\to+\infty$, which means that we obtain \eqref{lim4} and have completed the proof.
\end{proof}

This concludes our proof of Theorem \ref{thm:chaos}. To wrap up our general discussion of convergence to multiplicative chaos, we will now turn to extending our convergence to test functions which are not necessarily compactly supported. 

\subsection{Extension to all bounded test functions}\label{sect:extension}

The goal of this section is to prove the following result which is a simple extension of Theorem~\ref{thm:chaos}:

\begin{proposition} \label{prop:L1}
If Assumptions~\ref{A:exp_moment} hold, then for any $0<\gamma <\sqrt{2d}$ and any bounded $ f\in C(\Omega)$, the law of 
\begin{equation*} 
\int_\Omega f(x)\frac{e^{\gamma\X(x)}}{\E_N e^{\gamma \X(x)}}dx  
\end{equation*}
under $\P_N$ converges to that of $\int_\Omega f(x)d\mu^\gamma(x)$.
\end{proposition}

We point out that the reader might find Proposition \ref{prop:L1} conceptually satisfying as it follows from general theory that this result implies convergence of the law of $\frac{e^{\gamma \X(x)}}{\E_N e^{\gamma \X(x)}}dx$ under $\P_N$ to that of $\mu^\gamma$ with respect to the weak topology. For further discussion, see e.g. \cite[Chapter 4]{Kallenberg17}. The proof of Proposition \ref{prop:L1} is very similar to that of Proposition \ref{pr:gmcprops}.

\begin{proof}[Proof of Proposition~\ref{prop:L1}]
Recalling the notation of the proof of Proposition \ref{pr:gmcprops}, let us write $g_k\in C_c(\Omega)$ for an increasing sequence of functions, bounded by one, such that $g_k(x)=1$ for $d(x,\partial\Omega)\geq 1/k$. We also write for $x,\xi\in \R$, $\Phi(x)=e^{i\xi x}$. Once again, this is bounded by one and $|\xi|$-Lipschitz. We thus have 
\begin{align*}
&\left|\E_N \Phi\left(\int_\Omega f(x)\frac{e^{\gamma \X(x)}}{\E_N e^{\gamma \X(x)}}dx\right)-\E_N  \Phi\left(\int_\Omega g_k(x)f(x)\frac{e^{\gamma \X(x)}}{\E_N e^{\gamma \X(x)}}dx\right)\right|\\
&\leq |\xi|\E_N\left[\int_\Omega |f(x)||1-g_k(x)|\frac{e^{\gamma \X(x)}}{\E_N e^{\gamma \X(x)}}dx\right]\\
&\leq |\xi|\|f\|_{L^\infty(\Omega)}\int_\Omega |1-g_k(x)|dx
\end{align*}
which tends to zero as $k\to\infty$ (uniformly in $N$) say by the dominated convergence theorem (recall that $\Omega$ is bounded). Thus by Theorem \ref{thm:chaos} and (the proof of) Proposition \ref{pr:gmcprops}, we have
\begin{align*}
\limsup_{N\to\infty}&\left|\E_N \Phi\left(\int_\Omega f(x)\frac{e^{\gamma \X(x)}}{\E_N e^{\gamma \X(x)}}dx\right)-\E \Phi\left(\int_\Omega f(x)d\mu^\gamma(x)\right)\right|\\
&\leq \limsup_{k\to\infty}\limsup_{N\to\infty} \left|\E_N \Phi\left(\int_\Omega f(x)\frac{e^{\gamma \X(x)}}{\E_N e^{\gamma \X(x)}}dx\right)-\E_N\Phi\left(\int_\Omega g_k(x)f(x)\frac{e^{\gamma \X(x)}}{\E_N e^{\gamma \X(x)}}dx\right)\right|\\
&\quad +\limsup_{k\to\infty}\limsup_{N\to\infty}\left|\E_N \Phi\left(\int_\Omega g_k(x)f(x)\frac{e^{\gamma \X(x)}}{\E_N e^{\gamma \X(x)}}dx\right)-\E\Phi\left(\int_\Omega g_k(x)f(x)d\mu^\gamma(x)\right)\right|\\
&\quad +\limsup_{k\to\infty}\left|\E \Phi\left(\int_\Omega f(x)d\mu^\gamma(x)\right)-\E\Phi\left(\int_\Omega g_k(x)f(x)d\mu^\gamma(x)\right)\right|\\
&=0,
\end{align*}
which concludes the proof.\end{proof}

\subsection{Specializing to the setting of Theorem \ref{th:gmc} and Theorem \ref{th:gmc2}} \label{sec:special}

In this section, we use our results on asymptotics of Hankel determinants as described in Section \ref{sec:hankel_intro} to to prove Theorem \ref{th:gmc} and Theorem \ref{th:gmc2}  with our general approach to multiplicative chaos. Before going into the actual proofs, we will discuss how our problem relates to the discussion in Section \ref{sect:gmc}. After this, we will first prove Theorem \ref{th:gmc2} in Section \ref{sect:gmceasy} as its proof only relies on Theorem \ref{thm:LOS}, while the proof of Theorem \ref{th:gmc}, which we present in Section \ref{sect:gmchard}, relies on Theorem \ref{thm:chaos}. 

First of all, let us mention that to prove that a sequence of random measures $\mu_N$ to converges to a limiting measure $\mu$ in law with respect to the weak topology, it is sufficient to prove that for each bounded continuous test function $f$, $\int f d\mu_N$ converges in law to $\int f d\mu$ -- for details about this fact, see e.g. \cite[Chapter 4]{Kallenberg17}. Now based on Proposition \ref{prop:L1}, it is indeed sufficient to verify Assumptions \ref{A:exp_moment} in the setting of Theorem \ref{th:gmc} and those of Theorem \ref{thm:LOS} in the setting of Theorem \ref{th:gmc2}.

\medskip

To make contact with the discussion in beginning of this section, let $\P_N$ be the law of our $N\times N$ one-cut regular Hermitian random matrix from Section \ref{sec:intro}, let the law of $\X$ under $\P_N$ be the law of the eigenvalue counting function $\h$ \eqref{h},  and let $\P$ be the law of a centered Gaussian log-correlated field on $\Omega=(-1,1)$ with covariance kernel
\[
\Sigma(x,y)=\log \frac{1-xy+\sqrt{1-x^2}\sqrt{1-y^2}}{|x-y|}. 
\]

The mollifying function $\varphi$ used in e.g. \eqref{eq:Xeps} will be taken to be the Poisson kernel
\[
\varphi(x)=\frac{1}{\pi}\frac{1}{1+x^2}.
\]
Note that this is not a compactly supported function as we assumed in our preceding discussion, but as we are looking at processes on $(-1,1)$, convolving with $\varphi_\epsilon(x)=\epsilon^{-1}\varphi(x/\epsilon)$ is the same as convolving with (a scaled version of) $\varphi$ multiplied by a smooth compactly supported function which is $1$ in a suitable neighborhood  of $(-1,1)$. Thus we can ignore the fact that $\varphi$ is not compactly supported. 

As described in Section \ref{sect:background}, one can construct the limiting multiplicative chaos measure $\mu^\gamma$ through convolution approximations of the Gaussian field using the function $\varphi$. The reason for using this particular function for the convolution is that in the notation \eqref{eq:hcomplex}, we have for any $\epsilon>0$,
\[
\h(x+\i\epsilon )=(\varphi_\epsilon* \h)(x),
\]
which is a very natural approximation to $\h$. 
Moreover, if we write for $\lambda\in \R$
\begin{align*}
w_{\epsilon,x}(\lambda) &= \sqrt{2}\pi  (\varphi_\epsilon*\1_{(-\infty, x]})(\lambda),
\end{align*}
then 
\[
\h(x+\i\epsilon ) = \sum_{1\le j\le N}  w_{\epsilon,x}(\lambda_j) - N \int w_{\epsilon,x}  d\mu_V
\]
is a centered linear statistic.  

 We also point out that we can analytically continue $w_{\epsilon,x}$ by writing 
\begin{align}\label{eq:wdef}
w_{\epsilon,x}(\lambda)=\frac{\sqrt{2}}{2}\pi+\frac{\sqrt{2}}{2\i}\big(\log (\epsilon-\i(x-\lambda))-\log(\epsilon+\i(x-\lambda))\big) 
\end{align}
where the branches of the logarithms are the principal ones -- thus the cuts of $w_{\epsilon,x}$ are along $x\pm \i[\epsilon,\infty)$. $w_{\epsilon,x}$ having such an analytic continuation to $\C\setminus [(x+\i [\epsilon,\infty))\cup (x-\i[\epsilon,\infty))]$ will be vital for our RH analysis for proving the results of Section \ref{sec:hankel_intro}.

We will also need some estimates on the size of $w_{\epsilon,x}$ in a suitable neighborhood of $(-1,1)$. Essentially by definition, we have the a trivial bound of the form $|\mathrm{Re}(w_{\epsilon,x}(\lambda))|\leq 4 \pi$. For the imaginary part, we note that using the identity
\[
\mathrm{Im}(w_{\epsilon,x}(\lambda))=-\frac{1}{2\sqrt{2}}\log \frac{(\epsilon-\mathrm{Im}(\lambda))^2+(x-\mathrm{Re}(\lambda))^2}{(\epsilon+\mathrm{Im}(\lambda))^2+(x-\mathrm{Re}(\lambda))^2}.
\]
an elementary calculation shows that there exists a universal constant $C$ (in particular, it is independent of $\epsilon$ and $x$) such that
\begin{equation}\label{eq:wboud}
\sup\left\{|w_{\epsilon,x}(\lambda)|: |\mathrm{Im}(\lambda)|<\frac{\epsilon}{2} \quad \text{or} \quad |\mathrm{Re}(\lambda)-x|\geq |\mathrm{Im}(\lambda)|+\epsilon\right\}\leq C.
\end{equation}
The reason we need such a bound later on is that it shows that for $\mathcal S_N$ as in \eqref{domain}, $w_{\epsilon,x}$ is analytic and uniformly bounded in $\mathcal S_N$ for any $\epsilon\geq \epsilon_N$.

Finally before turning to the proof of Theorem \ref{th:gmc2}, we verify the assumptions of Definition \ref{def:EN} as these are common for both Theorem \ref{thm:LOS} and Theorem \ref{thm:chaos}. 

By construction, $\h$ (or equivalently $\X$ under $\P_N$) is piecewise continuous (in fact continuous from the right with limits from the left existing), and one has of course $|\h(x)|\leq 2\sqrt{2}\pi N$ for all $x$ and every realization of our randomness, so certainly $\X$ is integrable, bounded from above, upper semi-continuous, and even the exponential moment estimates are trivially satisfied. Thus to verify the conditions of Definition \ref{def:EN}, it is sufficient to show that for each $f\in C_c^\infty(\Omega)$, the law of $\int f(x)\X(x)dx$ under $\P_N$ converges to that of it under $\P$, but this follows immediately from Johansson's CLT -- see our discussion at the beginning of Section~\ref{gmc_results}.

We now turn to the more involved parts  of the proofs.

\subsubsection{Proof of Theorem \ref{th:gmc2}}\label{sect:gmceasy}

As discussed above, to prove Theorem \ref{th:gmc2}, it remains to verify the conditions of Theorem \ref{thm:LOS}. In particular, as we already checked the conditions of Definition \ref{def:EN}, it remains to verify \eqref{exp_moment}. For this, let $\alpha \in(0,1)$ be as in the statement of Theorem \ref{th:gmc2}, $n\in \N$, $t_1,...,t_n\in \R^n$, $x_1,...,x_n\in (-1,1)$, and $\eta_1,....,\eta_n \ge N^{-1+\alpha}$.
We then note that by our definitions
\[
\sum_{j=1}^n t_j\X_{\eta_j}(x_j)=\sum_{l=1}^N \sum_{j=1}^n t_j w_{\eta_j,x_j}(\lambda_l) -N \sum_{j=1}^n t_j \int_{-1}^1 w_{\eta_j,x_j}d\mu_V.
\]
Thus following our discussion from Section \ref{sec:hankel_intro} on the connection between random matrices and Hankel determinants, we see that the object of interest for Theorem \ref{thm:LOS} can be written as
\begin{align*}
\E_N e^{\sum_{j=1}^n t_j\X_{\eta_j}(x_j)}&= e^{ -N \int_{-1}^1 w d\mu_V} \frac{D_N(e^{w-NV})}{Z_N}
\end{align*}
where the function $w$ is given by 
\[
w(\lambda)=\sum_{j=1}^n t_j w_{\eta_j,x_j}(\lambda).
\]
Recall from our discussion of the beginning of this section that $w_{\eta_j,x_j}(\lambda)$ is analytic in the slit plane $\C\setminus[(x_j+\i [\eta_j,\infty))\cup(x_j-\i [\eta_j,\infty))]$ for any $x_j\in (-1,1)$. Moreover, from \eqref{eq:wboud}, we see that for some constant $C$ only depending on the points $t_j$, we have 
\[
\sup\left\{|w(\lambda)|: |\mathrm{Im}(\lambda)|\leq \frac{N^{-1+\alpha}}{2} \quad \text{or} \quad \min_j |\mathrm{Re}(\lambda)-x_j|\geq |\mathrm{Im}(\lambda)|+N^{-1+\alpha}\right\}\leq C.
\]
This shows that  the function $w$ is analytic and bounded by the above constant $C>0$ in the domain  $\mathcal S_N$ as in \eqref{domain} with parameters $\epsilon_N=N^{-1+\alpha'}$, $\delta_N=N^{-\alpha'/2}$ and $\alpha' = \min\{\alpha, 1/2\}$. Here, it is important that the points $x_j$ are in a fixed compact subset of $(-1,1)$ so that either $|\mathrm{Im}(\lambda)|\leq \frac{\epsilon_N}{2}$ or $\min_j|\mathrm{Re}(\lambda)-x_j|\geq \epsilon+|\mathrm{Im}(\lambda)|$ for all $\lambda\in \mathcal S_N$. Thus from Theorem \ref{th:Hankel1} with $\gamma_1=\gamma_2=0$ (note that in this case, $D_N(x_1,x_2;\gamma_1,\gamma_2;0)=Z_N$ and $D_N(x_1,x_2;\gamma_1,\gamma_2;w)=D_N(e^{w-NV})$), we find that as $N\to\infty$
\[
\E_N e^{\sum_{j=1}^n t_j\X_{\eta_j}(x_j)}=e^{\frac{1}{2}\sigma(w)^2}(1+o(1)),
\]
where the error $o(1)$ is uniform in the locations $x_j$ when restricted to a compact subset of $(-1,1)$, and by \eqref{cov2}, 
\begin{align*}
\sigma(w)^2&=  \sum_{i,j=1}^n t_i t_j  \sigma^2( w_{\eta_i,x_i}  ; w_{\eta_j,x_j} )  = \sum_{i,j=1}^nt_i t_j \E \X_{\eta_i}(x_i) \X_{\eta_j}(x_j) .
\end{align*}

 We thus conclude that as $N\to+\infty$,
\[
\E_N e^{\sum_{j=1}^n t_j\X_{\eta_j}(x_j)}=  (1+o(1))e^{\frac{1}{2}\sum_{i,j=1}^nt_i t_j \E \X_{\eta_i}(x_i) \X_{\eta_j}(x_j)}=(1+o(1))\E e^{\sum_{j=1}^n t_j\X_{\eta_j}(x_j)},
\]
with the required uniformity of Theorem \ref{thm:LOS}, so we have verified the assumptions of Theorem \ref{thm:LOS} and we see that for each $\gamma\in(-\sqrt{2},\sqrt{2})$ and $f\in C_c(\Omega)$, the law of 
\[
\int_\Omega f(x)\frac{e^{\gamma \X_{\epsilon_N}(x)}}{\E_N e^{\gamma \X_{\epsilon_N}(x)}}dx
\]
under $\P_N$, that is, in the notation of Theorem \ref{th:gmc2}, the law of $\int_\Omega fd\widetilde \mu_N^\gamma$ converges to that of $\int_\Omega f d\mu^\gamma$ under $\P$. By our discussion at the beginning of this section, this implies that the measures converge in law with respect to the weak topology. This concludes the proof of Theorem \ref{th:gmc2}.

\subsubsection{Proof of Theorem \ref{th:gmc}}\label{sect:gmchard}

As indicated above, our proof of Theorem \ref{th:gmc} relies on Theorem \ref{thm:chaos}, and consists of verifying the conditions of Assumptions \ref{A:exp_moment}. This is rather similar to the verification of the conditions of Theorem \ref{thm:LOS} from Section \ref{sect:gmceasy}, but more involved -- more precisely, it relies on applying Theorem \ref{th:Hankel1} (with various $w$ -- all analytic in $\mathcal S_N$ from \eqref{domain}), Theorem \ref{th:Hankel2}, and \cite[Theorem 1.1]{Charlier}. We will verify Assumptions~\ref{A:exp_moment} with the choice $\epsilon_N=N^{-1+\alpha(\gamma)}$, where we will not write this $\alpha(\gamma)>0$ explicitly, but we will deduce at the end of verifying \eqref{A5} the existence of one that suits our needs. We will typically obtain much stronger asymptotics than those required by Assumptions \ref{A:exp_moment}. As in Section \ref{sect:gmceasy}, we will rely heavily on the connection between moments of suitable random matrix objects and Hankel determinants as discussed in Section \ref{sec:hankel_intro}.

\medskip

\underline{Verification of \eqref{A31}:} consider first the case of $\epsilon=\epsilon'=0$. Using\footnote{More precisely, if we write in the notation of \cite{Charlier} $m=2$, $\i \beta_j=\frac{\gamma}{\sqrt{2}}$, $\alpha_j=0$, $W=0$, $t_1=x$, $t_2=y$, and note that $ \frac{\gamma}{\sqrt{2}}\mathbf 1_{\{\lambda\leq x\}}-\frac{\gamma}{\sqrt{2}} \mathbf 1_{\{\lambda>x\}}=\sqrt{2}\gamma \mathbf 1_{\{\lambda\leq x\}}-\frac{\gamma}{\sqrt{2}}$, an elementary calculation using the identity $\frac{1-xy+\sqrt{1-x^2}\sqrt{1-y^2}}{|x-y|}=\frac{|x-y|}{1-xy-\sqrt{1-x^2}\sqrt{1-y^2}}$  for $x,y\in(-1,1)$ yields the claim.} \cite[Theorem 1.1]{Charlier} shows that for $x\neq y$ 
\[
\lim_{N\to \infty}\frac{\E_N e^{\gamma \X(x)+\gamma \X(y)}}{\E_N e^{\gamma \X(x)}\E_N e^{\gamma \X(y)}}=\left(\frac{1-xy+\sqrt{1-x^2}\sqrt{1-y^2}}{|x-y|}\right)^{\gamma^2}=e^{\gamma^2 \Sigma(x,y)}=e^{\gamma^2 \E\X(x)\X(y)}
\]	
and the convergence is uniform in compact subsets of $\{(x,y)\in \Omega^2: x\neq y\}$. Thus for $\epsilon=\epsilon'=0$, the condition is verified.

Consider now the situation where $\epsilon=0$ and $\epsilon'>0$ (the case of $\epsilon>0$ and $\epsilon'=0$ follows by symmetry). This would follow again from \cite[Theorem 1.1]{Charlier}, but to avoid discussing the connection between the notations, we rely instead on Theorem \ref{th:Hankel1}. We actually use it twice, since we have
\[
\frac{\E_N e^{\gamma \X(x)+\gamma \X_{\epsilon'}(y)}}{\E_N e^{\gamma \X(x)}\E_N e^{\gamma \X_{\epsilon'}(y)}}=\frac{D_N(x;\gamma;w) D_N(x;0;0)}{D_N(x;\gamma;0)D_N(x;0;w)}
\]
with $w(\lambda)=\gamma w_{\epsilon',y}(\lambda)$ (see \eqref{eq:wdef} for the definition). Thus Theorem \ref{th:Hankel1} first with $\gamma_1=\gamma$, $\gamma_2=0$, $\mathcal S_N$ as in \eqref{domain} (with arbitrary $\epsilon_N\leq \epsilon'$), and then with $\gamma_1=\gamma_2=0$ yields
\[
\lim_{N\to\infty}\frac{\E_N e^{\gamma \X(x)+\gamma \X_{\epsilon'}(y)}}{\E_N e^{\gamma \X(x)}\E_N e^{\gamma \X_{\epsilon'}(y)}}=(1+o(1))e^{\frac{\gamma}{\sqrt{2}}(\mathcal U w)(x)\sqrt{1-x^2}},
\]	
where the error is uniform in $(x,y)$ in a fixed compact subset of $(-1,1)^2$ (note that this is stronger than required and relies on $w$ being bounded as given by \eqref{eq:wboud}). To see the connection to $\E \X(x)\X_{\epsilon'}(y)$, note from \eqref{cov6} that
\[
\frac{1}{\sqrt{2}}(\mathcal U w)(x)\sqrt{1-x^2}=\gamma\E \X(x)\X_{\epsilon'}(y).
\]

Finally the case of $\epsilon,\epsilon'>0$ is very similar. One now simply takes in Theorem \ref{th:Hankel1} $\gamma_1=\gamma_2=0$ and applies the theorem three times with $w=w_1=\gamma w_{\epsilon,x}$, $w=w_2=\gamma w_{\epsilon',y}$ and $w=w_3=\gamma w_{\epsilon,x}+\gamma w_{\epsilon',y}$ and finds that as $N\to\infty$ 
 \[
	\frac{\E_N e^{\gamma \X_\epsilon(x)+\gamma \X_{\epsilon'}(y)}}{\E_N e^{\gamma \X_\epsilon(x)}\E_N e^{\gamma \X_{\epsilon'}(y)}}=(1+o(1))e^{\frac{1}{2}(\sigma(w_3)^2 -\sigma(w_1)^2-\sigma(w_2)^2)}
	\]
and using \eqref{cov2}, one finds that
\[
\frac{1}{2}(\sigma(w_3)^2-\sigma(w_1)^2-\sigma(w_2)^2)=\gamma^2\E \X_\epsilon(x)\X_{\epsilon'}(y).
\]
Again the error is uniform in $x,y$ in compact subsets of $(-1,1)$. This concludes the verification of \eqref{A31}.

\medskip

\underline{Verification of \eqref{A32}:} This follows directly from our discussion in the verification of \eqref{A31}. In particular, we saw that uniformly in $x,y$ in a compact subset of $(-1,1)^2$, for $0\leq \epsilon<\epsilon'$
\[
\lim_{N\to \infty}\frac{\E_N e^{\gamma \X_\epsilon(x)+\gamma \X_{\epsilon'}(y)}}{\E_N e^{\gamma \X_\epsilon(x)}\E_N e^{\gamma \X_{\epsilon'}(y)}}=e^{\gamma^2 \E \X_{\epsilon}(x)\X_{\epsilon'}(y)}. 
\]
The claim is now simply an estimate for the covariance $\E \X_{\epsilon}(x)\X_{\epsilon'}(y)$, and the required estimate follows immediately from \eqref{cov_estimate}. Thus we can move on to the next assumption.

\medskip

\underline{Verification of \eqref{A33}:} For any fixed $\delta>0$, the bound for $|x-y|\geq \delta$ follows directly from \cite[Theorem 1.1]{Charlier} with similar considerations as in the verification of \eqref{A31}, so let us focus on the case $\epsilon_N\leq |x-y|\leq \delta$ for some small but fixed $\delta>0$. This follows from Theorem \ref{th:Hankel2} combined with \cite[Theorem 1.1]{Charlier}. More precisely, writing
\[
\frac{\E_N e^{\gamma \X(x)+\gamma \X(y)}}{\E_N e^{\gamma \X(x)}\E_N e^{\gamma \X(y)}}=\frac{D_N(x,y;\gamma,\gamma;0)}{D_N(x;2\gamma;0)}\frac{D_N(x;2\gamma;0)D_N(x;0;0)}{D_N(x;\gamma;0)D_N(y;\gamma;0)}
\]	
and applying Theorem \ref{th:Hankel2} for the first ratio and \cite[Theorem 1.1]{Charlier} for the second one, we see that for $\epsilon_N\leq |x-y|\leq \delta$ 
	\[
\frac{\E_N e^{\gamma \X(x)+\gamma \X(y)}}{\E_N e^{\gamma \X(x)}\E_N e^{\gamma \X(y)}}	= e^{\mathcal O(1)} e^{-\gamma^2 \max(0,\log[N|x-y|])}N^{\gamma^2}\ll|x-y|^{-\gamma^2}
	\]	
and the implied constants have the required uniformity -- this concludes the verification of \eqref{A33}. Note that once again, the asymptotics we obtain are much stronger than required -- we in fact found an upper bound of the form $\min(N^{\gamma^2},|x-y|^{-\gamma^2})$ valid for all $x,y$ (in a compact subset of $(-1,1)$). 

\medskip 

\underline{Verification of \eqref{A1}:} This is very similar to the verification of \eqref{exp_moment}, but with the difference of the $e^{\gamma \X_\epsilon(x)+\gamma \X_\epsilon'(y)}$-terms. Let us focus on the $\epsilon=\epsilon'=0$-case. The others are similar though require a change of the function $w$.

For $\epsilon=\epsilon'=0$, we write
\[
w(\lambda)=\sum_{j=1}^n t_j w_{\eta_j,x_j}(\lambda).
\]
and similarly to the proof of Theorem \ref{th:gmc2}, we find from Theorem \ref{th:Hankel1} (with $\mathcal S_N$ as in \eqref{domain} with the same $\epsilon_N=N^{-1+\alpha(\gamma)}$)
 \[
	\frac{\E_N e^{\gamma \X(x)+\gamma \X(y)}e^{\sum_{j=1}^n t_j \X_{\eta_j}(x_j)}}{\E_N e^{\gamma \X(x)+\gamma \X(y)}}=(1+o(1))e^{\frac{1}{2}\sigma(w)^2+\frac{\gamma}{\sqrt{2}}(\sqrt{1-x^2}(\mathcal Uw)(x)+\sqrt{1-y^2}(\mathcal Uw)(y))}
	\]
where the error is uniform in the $x_i$ as $w$ is bounded and uniform in $\eta_i\geq \epsilon_N$ -- recall \eqref{eq:wboud}. From \eqref{cov2} and \eqref{cov6} we find again that
\[
\sigma(w)^2=\sum_{i,j=1}^n t_i t_j \E X_{\eta_i}(x_i)\X_{\eta_j}(x_j)\qquad \text{and} \qquad \frac{\gamma}{\sqrt{2}}\sqrt{1-x^2}(\mathcal Uw)(x)=\sum_{i=1}^n  \gamma t_i \E \X(x)\X_{\eta_i}(x_i).
\]
Thus with the required uniformity (again, actually even stronger uniformity since we do not need to exclude the diagonal)
\begin{align*}
	&\frac{\E_N e^{\gamma \X(x)+\gamma \X(y)}e^{\sum_{j=1}^n t_j \X_{\eta_j}(x_j)}}{\E_N e^{\gamma \X(x)+\gamma \X(y)}}\\
	&\qquad \qquad =(1+o(1))e^{\frac{1}{2}\sum_{i,j=1}^nt_it_j \E \X_{\eta_i}(x_i)\X_{\eta_j}(x_j)+\sum_{i=1}^n \gamma t_i [\E \X(x)\X_{\eta_i}(x_i)+\E \X(y)\X_{\eta_i}(x_i)]},
	\end{align*}
which was precisely the claim for $\epsilon=\epsilon'=0$. For $\epsilon>0$ or $\epsilon'>0$, we replace $w$ by $w+\gamma w_{\epsilon,x}$ or $w+\gamma w_{\epsilon',y}$ (or $w+\gamma w_{\epsilon,x}+\gamma w_{\epsilon',y}$ if $\epsilon,\epsilon'>0$) and again apply Theorem \ref{th:Hankel1} with the corresponding $\gamma_i=0$. We omit further details.

\medskip

\underline{Verification of \eqref{A2}:} This is essentially identical to the verification of \eqref{A1}. One chooses $w=\lambda w_{\epsilon,x}$ and applies Theorem \ref{th:Hankel1} with $\gamma_1=\gamma$, $x_1=x$,  and $\gamma_2=0$. One finds from \eqref{cov6} that with the required uniformity 
\begin{align*}
\frac{\E_N e^{\gamma \X(x)+\lambda \X_{\epsilon}(x)}}{\E e^{\gamma \X(x)}}&=(1+o(1)) e^{\frac{1}{2}\sigma(w)^2 +\frac{\gamma}{\sqrt{2}}(\mathcal U w)(x)\sqrt{1-x^2}}\\
&=(1+o(1))e^{\frac{\lambda^2}{2}\E \X_{\epsilon}(x)^2+\gamma \lambda \E \X_\epsilon(x)\X(x)}\\
&=(1+o(1))\epsilon^{-\frac{\lambda^2}{2}-\lambda\gamma+\mathcal O(1)},
\end{align*}
where in the last step, we used the covariance estimate \eqref{cov_estimate}. This is precisely the required claim, so we are done.

\medskip

\underline{Verification of \eqref{A41}:} This is a special case of the verification of \eqref{A1} since due to our strong asymptotics, we did not need to exclude the diagonal in our verification of \eqref{A1}.

\medskip

\underline{Verification of \eqref{A42}:} This follows by an identical reasoning as \eqref{A41} did. 

\medskip

\underline{Verification of \eqref{A5}:} For this final claim, we note that we saw in the verification of \eqref{A33}, that for $|x-y|$ small enough (say less than some fixed small $\delta>0$), we have
 \[
	\frac{\E_N e^{\gamma \X(x)+\gamma \X(y)}}{\E_N e^{\gamma \X(x)}\E_N e^{\gamma \X(y)}}\ll \min(N^{\gamma^2},|x-y|^{-\gamma^2}), 
	\]
so we find
\begin{align*}
\int_{|x-y|\leq \epsilon_N} \frac{\E_N e^{\gamma \X(x)+\gamma \X(y)}}{\E_N e^{\gamma \X(x)}\E_N e^{\gamma \X(y)}}dxdy&\ll\int_{|x-y|\leq N^{-1}} N^{\gamma^2} dxdy+\int_{N^{-1}\leq |x-y|\leq \epsilon_N}|x-y|^{-\gamma^2}dxdy\\
&\ll_\gamma  N^{-1+\gamma^2}+\log N N^{ (\gamma^2-1)_+}\\
&=\log N N^{(\gamma^2-1)_+},
\end{align*}	
where the $\log N$-factor arises only for $\gamma=1$. It remains to show that $\log N N^{(\gamma^2-1)_+}\leq \epsilon_N^{-(\gamma^2-1)_+-\tau}$ for a suitable $\tau$ as required in \eqref{A5}. Recalling that we have $\epsilon_N=N^{-1+\alpha(\gamma)}$, we see that there is of course no issue for $\gamma^2\leq 1$. Thus we simply need to see that for some choice of $\alpha(\gamma)$
\[
\gamma^2-1\leq (1-\alpha(\gamma))(\gamma^2-1)+\tau
\]
for some $\tau\in(0,\min(\frac{(1-\frac{\gamma^2}{2})^2}{4},\frac{\gamma^2}{2}))$. This is true if we take $\alpha(\gamma)$ close enough to $0$. 
Thus, we have finished the verification of \eqref{A5}. 

\medskip

This concludes the verification of Assumptions \ref{A:exp_moment} and thus the proof of Theorem \ref{th:gmc} using Theorem \ref{thm:chaos}.  We now move on to extrema of the eigenvalue counting function, thick points, and freezing.

\section{Consequences for extreme values of log--correlated fields} \label{sect:LB}

The goal of this section is to prove Theorem~\ref{thm:max2},  Theorem~\ref{th:TP}, and Corollary~\ref{thm:freezing}. 
Even though the applications discussed in this article are concerned with eigenvalue statistics coming from Hermitian random matrices, like in Section \ref{sect:gmc}, we formulate our results  under general assumptions  in hope of our approach being of use in other contexts. We will use the same notations as in Section \ref{sect:gmc}.

Recall that $\Omega \subset \R^d$ is a bounded open set, simply connected, with a smooth boundary (in particular, its Lebesgue measure  is finite and positive: $0< |\Omega| <\infty$) and that under the measure $\P_N$, the canonical process $\X$ approximates   a Gaussian log-correlated field in the sense of Definition~\ref{def:EN}.  Moreover, since we assume $\X$ is almost surely bounded from above and  upper semi-continuous, for any compact $A\subset\Omega$, $\X$ attains its maximum in $A$ and 
$\max_A \X <+\infty$  $\P_N$ almost surely.
 In this section, we will be interested in the leading asymptotic behavior of these random variables as $N\to+\infty$, as well as that of $\sup_\Omega \X$ -- see Section \ref{sect:max}.

Like in Section \ref{sect:gmc}, we will be interested in approximations to multiplicative chaos measures: we write for $\gamma>0$
\[
\mu^{\gamma}_N(dx) =    \frac{e^{\gamma \X(x)}}{\E_N[e^{\gamma  \X(x) }] } dx . \footnote{Note the  inconsistency in our notation. In Section \ref{sect:gmc}, we wrote $\mu_\epsilon^\gamma$ for an approximation of a multiplicative chaos measure coming from smoothing the Gaussian field. We hope this will cause no confusion for the reader.
Moreover, the case $\gamma<0$ can be treated by considering $-\X$.}
\]
In addition to extrema of the field $\X$ under $\P_N$, we will study the fractal geometric properties of $\X$ under $\P_N$ by studying the thick points of $\X$ and freezing properties of $\mu^\gamma_N$ -- see Sections \ref{sect:tp} and \ref{sect:freezing}.
We will prove our results under certain general assumptions -- e.g. the validity of the results of Section \ref{sect:gmc}  -- and then in Section \ref{sect:evcf}, check that in the setting of random matrices, these assumptions hold. A fundamental assumption we will make throughout this section is that  the following basic exponential moment estimates hold.

\begin{assumption} \label{ass:1p}
For all $\gamma>0$ and $N\in \N$, there exist $R_\gamma>0$ and $\epsilon_N>0$ such that $\lim_{N\to \infty}\epsilon_N=0$ and for all $x\in \Omega$, 
\begin{equation} \label{1p2}
 \E_N \big[ e^{\gamma  \X(x)} \big]  \le R_\gamma \epsilon_N^{- \gamma^2/2} .  
\end{equation}
For any compact set $A \subset \Omega$, there exists a $C_{\gamma, A}>0$ such that for all  $x\in A$, 
\begin{equation} \label{1p1}
 \E_N \big[ e^{\gamma  \X(x)} \big]  \ge C_{\gamma, A}^{-1} \epsilon_N^{- \gamma^2/2} .  
\end{equation}
\end{assumption}

We prove in Section \ref{sect:evcf} that for the eigenvalue counting function $\X = h_N$  on $\Omega=(-1,1)$, in the one-cut regular case,  the asymptotics of Theorem~\ref{th:Hankel3} imply that Assumptions~\ref{ass:1p} hold with $\epsilon_N = N^{-1}$.  The interpretation of these assumptions that the reader should keep in mind is that under $\P_N$, to leading order, $\X(x)$ behaves like a centered Gaussian random variable of variance $\log \epsilon_N^{-1}+\mathcal O(1)$. 

Before discussing the maximum of $\X$ under $\P_N$, we begin by considering the thick points of the field.

\subsection{Thick points and the ``support" of the random measure $\mu^\gamma_N$}\label{sect:tp}

For any $\alpha\ge 0$, we define
$$
\mathscr{T}_N^\alpha = \left\{  x \in \Omega : \X(x) \ge \alpha \log \epsilon_N^{-1}\right\} . 
$$
This is usually called the set of $\alpha$-\emph{thick} points of $\X$.  
These are (relatively) closed  subsets of $\Omega$,  $\P_N$ almost surely.  The following lemma, which will play a role in the study of the maximum of $\X$, states essentially that the random measure $\mu^\gamma_N$ lives on $\mathscr{T}_N^\gamma$ -- see also Remark~\ref{rk:supp}.

\begin{lemma} \label{lem:supp}
If  Assumptions~\ref{ass:1p} hold, then for any compact set $A\subset \Omega$ and for any $0<\delta \le \gamma$, 
\[
\E_N \big[ \mu_N^\gamma\big( \{x\in A : x \in \mathscr{T}_N^{\gamma+\delta} \text{ or }x\notin  \mathscr{T}_N^{\gamma-\delta}\}\big) \big] \underset{A, \gamma , \delta}{\ll} \epsilon_N^{\delta^2/2} . 
\]
\end{lemma}

\begin{proof}
The proof is an elementary argument based on Markov's inequality. Using the estimates \eqref{1p2} and \eqref{1p1}, and the fact that $\1_{\X<a}\leq e^{\delta(a-X)}$, we have
\begin{align*}
\E_N\big[  \mu_{N}^\gamma(A\setminus\mathscr{T}_N^{\gamma-\delta} ) \big]  
& =   \int_A  \E_N\big[ \1_{\X(x) < (\gamma-\delta) \log \epsilon_N^{-1} } e^{\gamma \X(x)} \big]\frac{dx}{\E_N[  e^{\gamma  \X(x)}] } \\
&  \le C_{\gamma, A}  \int_A \E_N\big[ e^{(\gamma-\delta) \X(x)}\big]  \epsilon_N^{-\delta(\gamma-\delta) +\tfrac{\gamma^2}{2}} dx \\
&\le  R_{\gamma-\delta} C_{\gamma, A} |\Omega|  \epsilon_N^{\delta^2 / 2}. 
\end{align*}
Similarly, one has
\begin{align*}
\E_N\big[  \mu_{N}^\gamma(A\cap\mathscr{T}_N^{\gamma+\delta} ) \big]  
&=   \int_A  \E_N\big[ \1_{\X(x) \ge (\gamma+\delta) \log \epsilon_N^{-1} } e^{\gamma \X(x)} \big] \frac{dx}{\E_N\big[  e^{\gamma  \X(x)} \big] } \\
&\le C_{\gamma, A}  \int_A \E_N\big[  e^{(\gamma+\delta) \X(x)} \big] \epsilon_N^{\delta(\gamma+\delta) +\tfrac{\gamma^2}{2}} dx \\
&\le  R_{\gamma+\delta} C_{\gamma, A} |\Omega|  \epsilon_N^{\delta^2 / 2} 
\end{align*}
which completes the proof.
\end{proof}

\begin{remark} \label{rk:supp}
We point out here a simple consequence of this lemma that we shall make use of later on. First of all, directly from the definition of $\mu^\gamma_N$ one has that for any compact $A\subset\Omega$, $\E_N\big[\mu^\gamma_N(\Omega \setminus A) \big] =|\Omega \setminus A| $. The previous lemma combined with this remark then implies that for any such compact $A$
\begin{align*}
\limsup_{N\to+\infty}&\E_N \big[ \mu_N^\gamma\big(\{x\in \Omega: x \in \mathscr{T}_N^{\gamma+\delta} \text{ or }x\notin  \mathscr{T}_N^{\gamma-\delta}\}\big) \big] \\
&\leq \limsup_{N\to+\infty}\E_N \big[ \mu_N^\gamma\big(\{x\in A: x \in \mathscr{T}_N^{\gamma+\delta} \text{ or }x\notin  \mathscr{T}_N^{\gamma-\delta}\}\big) \big]  + \limsup_{N\to+\infty}\E_N \big[ \mu_N^\gamma\big(\Omega\setminus A\big) \big] \\
&= |\Omega \setminus A|, 
\end{align*}
and as $A$ is arbitrary, one concludes that
\[
\lim_{N\to+\infty}\E_N \big[ \mu_N^\gamma\big(\{x\in \Omega: x \in \mathscr{T}_N^{\gamma+\delta} \text{ or }x\notin  \mathscr{T}_N^{\gamma-\delta}\}\big) \big]=0.
\]
Thus by Markov's inequality, we see that for any  $\delta, \varepsilon > 0$, 
\begin{equation} \label{supp_event}
\lim_{N\to\infty}\P_N\left(\mu_N^\gamma\big(\{x\in \Omega: x \in \mathscr{T}_N^{\gamma+\delta} \text{ or }x\notin  \mathscr{T}_N^{\gamma-\delta}\}\big) \ge \varepsilon \right)=0,
\end{equation}
which should be interpreted as $\mu_N^\gamma$ living roughly on $\mathscr T_N^{\gamma}$.\hfill $\blacksquare$
\end{remark}
We turn now to studying the maximum of $\X$ under $\P_N$.

\subsection{Leading order of the maximum} \label{sect:max}
In this section, we will prove our estimates concerning the maximum of a log-correlated field. The basic idea is that as the multiplicative chaos measure lives on the set of $\gamma$-thick points, and the measure is non-trivial for $\gamma<\sqrt{2d}$, then there exist such $\gamma$-thick points for any $\gamma<\sqrt{2d}$. This gives a lower bound for the maximum. 
So, in addition to the exponential moment assumption formulated in Assumption \ref{ass:1p}, we will make some further assumptions about the convergence of the approximation to multiplicative chaos measure to get hold of the asymptotic behavior of the maximum of $\X$ under $\P_N$. 
An upper bound can also be formulated in terms of thick points and proven to hold under Assumptions \ref{ass:1p} and some minor regularity assumptions on the field.

\begin{assumption} \label{ass:gmc}
Let $\gamma_* = \sqrt{2d}$. We assume that for any $\gamma < \gamma_*$ and any fixed deterministic Borel set $A \subseteq \Omega$ such that $|A|>0$, the random variable $\mu_N^\gamma(A)$ under $\P_N$ converges in distribution as $N\to+\infty$ to a random variable
$\zeta_A^\gamma$ which satisfies $\P[ 0<\zeta_A^\gamma<+\infty]=1$. 
\end{assumption}

Note that by  Proposition \ref{pr:gmcprops} and Proposition \ref{prop:L1}, Assumptions~\ref{ass:gmc} are satisfied under Assumptions \ref{A:exp_moment}.
We are now able to prove a lower bound for the maximum.

\begin{theorem} \label{thm:LB}
If Assumptions~\ref{ass:1p} and~\ref{ass:gmc} hold, then for any compact  set $A \subset \Omega$  with $|A|>0$ and for any $\delta>0$, one has as $N\to\infty$,
\[
\P_N\big[ \max_{ A} \X \le  (\gamma_* - \delta)\log \epsilon_N^{-1}\big] \to 0. 
\]
\end{theorem}

\begin{proof}
On the one hand, by a union bound,   for any $\alpha>0$ and any $\varepsilon>0$, 
\begin{equation} \label{LB1}
\P_N\big[  \mu_{N}^\gamma(A \cap \mathscr{T}_N^\alpha ) \le \varepsilon \big]  \le \P_N\big[  \mu^\gamma_N(A)\le 2 \varepsilon \big]+ 
\P_N\big[  \mu_{N}^\gamma( A\setminus\mathscr{T}_N^\alpha ) \ge \varepsilon \big]   . 
\end{equation}
If $\alpha<\gamma$, by Lemma~\ref{lem:supp}, the second term on the RHS  of \eqref{LB1} converges to 0 as $N\to+\infty$ and, since the random variable $\mu_{N}^\gamma(A)$ converges in distribution to $\zeta_A^\gamma$, we obtain by the portmanteau theorem that 
\begin{align*}
\limsup_{N\to\infty}\P_N\big[  \mu_{N}^\gamma(A \cap \mathscr{T}_N^\alpha ) \le \varepsilon\big]  \le \P\big[  \zeta_A^\gamma \le 2 \varepsilon\big]  .
\end{align*}
On the other hand, observe that by the definition of $\alpha$-thick points, 
\[
  \P_N\big[  \mu_{N}^\gamma(A \cap\mathscr{T}_N^\alpha ) >\varepsilon\big]  \le \P_N\big[ \max_A \X\ge \alpha \log \epsilon_N^{-1} \big],
\]
which implies that
\[
\liminf_{N\to\infty} \P_N\big[ \max_A \X \ge \alpha \log \epsilon_N^{-1} \big] \ge 1-  \P\big[  \zeta_A^\gamma \le 2 \varepsilon\big]. 
\]
Since this holds for any $\varepsilon>0$, $\alpha<\gamma<\gamma_*$ and $ \zeta_A^\gamma>0$ almost surely for any $\gamma <\gamma_*$, we conclude that for any $\alpha<\gamma_*$, 
\[
\liminf_{N\to\infty} \P_N\big[ \max_A \X \ge \alpha \log \epsilon_N^{-1} \big] \ge 1 . 
\]
This completes the proof. \end{proof}

We will now formulate conditions that will ensure the correct upper bound for the maximum. We first formulate our assumption in terms of thick points, and later show that it follows from a regularity assumption on the field which may be easier to check in concrete situations.

\begin{assumption} \label{ass:max}
For any $\alpha> \gamma_*$, as $N\to\infty$
\[
 \P_N[ \mathscr{T}_N^{\alpha}  \neq \emptyset] \to 0 . 
\]
\end{assumption}
The interpretation of this assumption is that for $\alpha>\gamma_*$, $\alpha$-thick points are unlikely to occur. This implies that under $\P_N$, $\max_\Omega \X$ is unlikely to be much greater than $\gamma_*\log \epsilon_N^{-1}$. The precise statement about the maximum is the following.

\begin{corollary} \label{max:LLN}
If the random field $\X$ satisfies Assumptions~\ref{ass:1p},~\ref{ass:gmc} and~\ref{ass:max},  then 
\[
\frac{\sup_\Omega \X }{\log \epsilon_N^{-1}} \to  \gamma_*  , 
\]
in probability as $N\to+\infty$.
\end{corollary}
\begin{proof}
The claim is that for any $\delta>0$
\[
\lim_{N\to\infty}\P_N\left(\left|\frac{\sup_\Omega \X}{\log \epsilon_N^{-1}}-\gamma_*\right|>\delta\right)=0.
\]
To see how this follows from our assumptions, note that if $A\subset \Omega$ is any compact  set with positive Lebesgue measure, we have
\begin{align*}
\P_N\left(\left|\frac{\sup_\Omega \X}{\log \epsilon_N^{-1}}-\gamma_*\right|>\delta\right)&=\P_N\left(\sup_\Omega\X<(\gamma_*-\delta)\log \epsilon_N^{-1}\right)+\P_N\left(\sup_\Omega\X>(\gamma_*+\delta)\log \epsilon_N^{-1}\right)\\
&\leq \P_N\left(\sup_A\X<(\gamma_*-\delta)\log \epsilon_N^{-1}\right)+\P_N\left(\mathscr{T}_N^{\gamma_*+\delta}  \neq \emptyset\right).
\end{align*}
By Theorem \ref{thm:LB} (which holds under Assumptions \ref{ass:1p} and Assumptions \ref{ass:gmc}), the first term tends to zero as $N\to\infty$. Under Assumptions \ref{ass:max}, the second term tends to zero as $N\to\infty$. This concludes the proof.
\end{proof}

As already noted, Assumptions \ref{ass:max} might not be so easy to check directly, so we formulate another assumption which is perhaps simpler to check in practice and implies Assumptions \ref{ass:max}. This condition can be seen as a rather mild regularity condition on the realizations of $\X$ under~$\P_N$. Moreover, by simply adapting the proof we also obtain large deviation estimates for the maximum of the random field $\X$.

\begin{proposition} \label{prop:cont}
Suppose that $\X$ satisfies the condition \eqref{1p2} of Assumption \ref{ass:1p} and that there exist  deterministic constants ${\rm C,  c}>0$ such that $\P_N$ almost surely$:$ for any $x\in \Omega$, there is a $($possibly random$)$ compact $\mho_x^N \subset \Omega$ such that  $|\mho_x^N| \ge {\rm c} \epsilon_N^{d}$ and 
\[
\X(t) \ge \X(x) - {\rm C}
\qquad\text{for all } t\in \mho_x^N . 
\]
We emphasize here that we do not require that $x\in \mho_x^N$.

Then for any  $\alpha> \gamma_*$, 
\[
 \P_N[ \mathscr{T}_N^{\alpha}  \neq \emptyset] \to 0  ,
 \qquad\text{as }N\to\infty.
\]
Moreover, if $\varpi_N$ is any increasing sequence such that 
$\displaystyle\lim_{N\to+\infty}\frac{\varpi_N}{\log \epsilon_N^{-1}} =+\infty$, for any $\gamma>0$, there exists a constant $C_\gamma>0$ such that \begin{equation} \label{LDE2}
 \P_N\big[ \sup_\Omega \X \ge \varpi_N \big] \le C_\gamma \epsilon_N^{-\gamma^2/2-d} e^{-\gamma \varpi_N} . 
\end{equation}
\end{proposition}

\begin{proof}
Without loss of generality, let us assume that $\alpha = \gamma_*+\delta$ for some small $0<\delta < 2\gamma_*/3$.
Let $x\in \Omega$ be a random point such that
$\X(x) \ge \alpha \log \epsilon_N^{-1} $.  Then, by the definition of $\mho_x^N$ and \eqref{1p2}, we have that under the event of the existence of such a point (or in other words, the event that $\mathscr T_N^\alpha\neq \emptyset$), for any $\gamma>0$, 
\[ \begin{aligned}
\mu_{N}^\gamma(\Omega)
& \ge \epsilon_N^{-\gamma \alpha} \int_{\mho_x^N} \frac{e^{-\gamma {\rm C}}}{ \E_N \big[ e^{\gamma  \X(x)} \big]} dx\\
&\ge {\rm c} R_\gamma^{-1} e^{-\gamma {\rm C}}  \epsilon_N^{-\gamma\alpha + \gamma^2/2 +d} .
\end{aligned}\]
Now, if we choose  and $\gamma=\gamma_*-\delta$, since $\gamma_*^2 = 2d$,  we find for some deterministic constant $C_{\delta}$ independent of $N$ that 
\[
\mu_N^\gamma(\Omega)\geq C_{\delta}\epsilon_N^{- \gamma_* \delta + 3 \delta^2/2} .
\]
By Markov's inequality, since $\E_N\big[ \mu_{N}^\gamma(\Omega)\big]= |\Omega| <+\infty$,   this estimate implies that 
\[
\limsup_{N\to+\infty}\P_N\big[ \exists x\in \Omega ; \X(x) \ge \alpha \log \epsilon_N^{-1} \big] 
\le  \limsup_{N\to+\infty} \P_N\big[ \mu_{N}^\gamma(\Omega)   \geq C_\delta  \epsilon_N^{- \gamma_* \delta + 3 \delta^2/2}  \big]  = 0 . 
\]
 This completes the proof of the first claim. Using the same first moment computation, we  similarly obtain for any $\gamma>0$,
 \[
\P_N\big[ \sup_\Omega \X \ge \varpi_N \big]
\le  \P_N\big[  \mu_N^\gamma(\Omega) \geq C_\gamma^{-1}  e^{\gamma \varpi_N } \epsilon_N^{\gamma^2/2 +d} \big] \leq C_\gamma \epsilon_N^{-\gamma^2/2 -d} e^{-\gamma \varpi_N}. 
\]
 for some constant $C_\gamma$ independent of $N$. 
\end{proof}

This concludes our discussion about the leading order of the maximum and its connection to multiplicative chaos. Before turning to thick points and freezing, we make a brief comment about what can be inferred about the supercritical regime of multiplicative chaos.

\begin{remark}[Supercritical regime] \label{rk:supercritical}
Let us recall from our general discussion about multiplicative chaos from Section \ref{sect:background}, that for $\gamma > \gamma_*$, the standard construction of the multiplicative chaos measure through Gaussian fields yields the zero measure. One would expect this to be the case also for the measure $\mu_N^\gamma$. In this remark, we show that this is indeed the case. For $\gamma>\gamma_*$, we can choose $\delta>0$ so that $\gamma -\delta = \gamma_* + \delta$ and 
\[
\mu_N^\gamma(\Omega) = \mu_N^\gamma(\mathscr{T}_N^{\gamma_*+\delta}) +   \mu^\gamma_N(\Omega \setminus\mathscr{T}_N^{\gamma-\delta}) .
\]
Then
\[
\P_N\big[ \mu_N^\gamma(\Omega)  \ge \varepsilon \big] \le 
  \P_N[ \mathscr{T}_N^{\gamma_*+\delta}  \neq \emptyset] + 
 \P_N\left(\mu_N^\gamma\big(\{x: x \in \mathscr{T}_N^{\gamma+\delta} \text{ or }x\notin  \mathscr{T}_N^{\gamma-\delta}\}\big) \ge \varepsilon/2 \right).
\]
By Remark~\ref{rk:supp}, this implies that if Assumptions~\ref{ass:1p} and~\ref{ass:max} hold, then for any $\varepsilon >0$
\[
\limsup_{N\to+\infty}  \P_N\big[ \mu_N^\gamma(\Omega)  \ge \varepsilon \big]  =0, 
\]
i.e. that in the super-critical regime $\gamma >\gamma_*$, the random measure $\mu^\gamma_N$ converges to 0 in the sense that its total mass tends to zero in probability. \hfill $\blacksquare$
\end{remark}

We now turn to discussing further fractal properties of $\h$ through thick points and freezing.
\subsection{Size of the sets of thick points and freezing}\label{sect:freezing}

In this section, we will provide estimates about the size of the sets of $\gamma$-thick points for any $\gamma>0$ and then deduce a general freezing result.  The proofs rely on the connection between the multiplicative chaos measures and the sets of thick points.

\begin{proposition} \label{prop:TP}
Under Assumptions~\ref{ass:1p} and~\ref{ass:gmc}, for any $0 \le \gamma <\gamma_*$ and $\delta>0$,
\begin{equation*}
\lim_{N\to\infty} \P_N\left( \frac{\gamma^2}{2} + \delta \ge  \frac{\log | \mathscr{T}_N^\gamma|}{\log \epsilon_N} \ge  \frac{\gamma^2}{2} - \delta\right) =1 .
\end{equation*}
Hence under $\P_N$, the random variable $\frac{\log | \mathscr{T}_N^\gamma|}{\log \epsilon_N}$  converges in probability to $\gamma^2/2$. 
\end{proposition}

\begin{proof} 
Let us first point out that for $\gamma=0$, the lower bound is trivial since $|\Omega|<\infty$, while the upper bound follows from a simple symmetry argument making use of the fact that $|\{x\in \Omega: \X(x)\geq 0\}|\stackrel{d}{=}|\{x\in \Omega: \X(x)\leq 0\}|$. We omit further details and focus on the $\gamma>0$ case, whose proof we split into two parts.	

\medskip

{\bf Lower--bound.} 
The estimate \eqref{1p2} implies that
\[
\mu_N^\gamma(\Omega)\geq  \mu^\gamma_N\big( \mathscr{T}_N^\gamma ) \ge  R_\gamma^{-1}  \epsilon_N^{-\gamma^2/2}  | \mathscr{T}_N^\gamma| ,
\]
so that for any $\delta>0$, 
\[
\P_N \big[  | \mathscr{T}_N^\gamma| \ge \epsilon_N^{\gamma^2/2 -\delta}   \big]  \le  \P_N \big[  \mu^\gamma_N(\Omega) \ge R_\gamma^{-1} \epsilon_N^{-\delta} \big] .
\]
Since we assume that $\mu^\gamma_N(\Omega) \Rightarrow \zeta^\gamma_\Omega $ as $N\to+\infty$ and $ \zeta^\gamma_\Omega<+\infty$ almost surely, this implies that for any $\delta>0$
 \[
\limsup_{N\to+\infty} \P_N \big[  | \mathscr{T}_N^\gamma| \ge \epsilon_N^{\gamma^2/2 -\delta}   \big]  =0 , 
 \]
which is precisely the lower bound we are after. 
\medskip

{\bf Upper--bound.} Fix a compact set $A \subset \Omega$, choose $\delta'>0$ in such a way that  $\gamma\delta'+ \frac{\delta^{\prime2}}{2} \le \delta$    and define the event:
\[
\mathscr{A} = \Big\{ \big| A \cap \mathscr{T}_N^{\gamma-\delta'}\big| \le \epsilon_N^{(\gamma+\delta')^2/2}  \Big\} .
\]
We will show that $\P_N[\mathscr{A}] \to 0$ as $N\to+\infty$.
Since on the complement of $\mathscr A$, $\log |\mathscr T_N^{\gamma-\delta'}| \geq \frac 12 (\gamma+\delta')^2\log \epsilon_N$, the upper bound will follow  from replacing $\gamma$ by $\gamma+\delta'$ and using that $\gamma\delta'+ \frac{\delta^{\prime2}}{2} \le \delta$. 
On the one hand, let us observe that by \eqref{1p1}, conditionally on  $\mathscr{A}$,  we have for any $\varepsilon \ge 0$, 
\[ \begin{aligned}
\mu_N^\gamma\big(\{ x\in A : x\in  \mathscr{T}_N^{\gamma-\delta'} \setminus  \mathscr{T}_N^{\gamma+\delta'} \}\big) & \le C_{\gamma, A} \epsilon_N^{-\gamma(\gamma+\delta') + \gamma^2/2} \big| A \cap \mathscr{T}_N^{\gamma-\delta'} \big|\\
&\le C_{\gamma, A} \epsilon_N^{\delta^{\prime2}/4}  \le \varepsilon/2 , 
\end{aligned}\]
when the parameter $N$ is sufficiently large. 
On the other hand, since
\[
\mu_N^\gamma(A) = \mu_N^\gamma\big(\{ x\in A : x\in  \mathscr{T}_N^{\gamma-\delta} \setminus  \mathscr{T}_N^{\gamma+\delta}\} \big) 
+  \mu_N^\gamma\big(\{ x\in A : x \in \mathscr{T}_N^{\gamma+\delta} \text{ or }x\notin  \mathscr{T}_N^{\gamma-\delta}\}\big) ,
\]
this implies that for any $\varepsilon \ge 0$, if $N$ is large, 
\[
\P_N[ \mu_N^\gamma(A) > \varepsilon] \le \P_N[ \mathscr{A}^*] +  \P_N[  \mu_N^\gamma\big(\{ x\in A : x \in \mathscr{T}_N^{\gamma+\delta} \text{ or }x\notin  \mathscr{T}_N^{\gamma-\delta}\}\big)  \ge \varepsilon/2] ,
\]
where $\mathscr{A}^*$ denotes the complement of the event $\mathscr{A}$. 
 Hence, by Lemma~\ref{lem:supp} and Markov's inequality,  we obtain that for any $\varepsilon \ge 0$, 
\[ \begin{aligned}
\limsup_{N\to+\infty} \P_N[ \mathscr{A}] 
& \le \limsup_{N\to+\infty}\P_N[ \mu_N^\gamma(A) \le \varepsilon]. 
\end{aligned}\]
By Assumptions~\ref{ass:gmc},
$\displaystyle  \limsup_{N\to+\infty}\P_N[ \mu_N^\gamma(A) \le \varepsilon] \le \P[\zeta^\gamma_A \le \varepsilon]$, and $\P[\zeta^\gamma_A \le \varepsilon] \to0$ as $\varepsilon \to 0$ since $\zeta^\gamma_A>0$ almost surely. 
In the end, we conclude that $\P_N[ \mathscr{A}]  \to 0$ as $N\to+\infty$ which completes the proof.
\end{proof}

Our last proposition of this section shows that the free energy corresponding to the random field $\X$ exhibits  freezing.  
This can be seen as a corollary of Proposition~\ref{prop:TP}, see e.g. \cite[Section~4]{ABB16} for such a discussion, but we give a slightly  different proof for completeness.

\begin{proposition} \label{prop:freezing}
Under the Assumptions~\ref{ass:1p},~\ref{ass:gmc} and~\ref{ass:max},  for any $\gamma \ge 0$, as $N\to\infty$,
\[ 
\frac{1}{\log \epsilon_N^{-1}} \log\left( \int_\Omega e^{\gamma \X(x)} dx \right)  \to 
\gamma\big( \gamma\wedge \gamma_*\big) - \frac{\big( \gamma\wedge \gamma_*\big)^2}{2}
\]
in probability with respect to  $\P_N$.
\end{proposition}

\begin{proof}
The proofs for $\gamma<\gamma_*$ and $\gamma\geq \gamma_*$ are rather different. For $\gamma<\gamma_*$, the proof  is a direct consequence of convergence to multiplicative chaos, while for $\gamma\geq \gamma_*$, we rely on our estimate on the size of the sets of thick points.

\medskip

{\bf Case 1.} Let us first consider the sub-critical regime $0<\gamma< \gamma_*$. It follows from Assumptions \ref{ass:1p} that for any compact $A\subset \Omega$
\[
C_{\gamma, A}^{-1}  \epsilon_N^{-\gamma^2/2}  \mu^\gamma_N(A) \le  \int_Ae^{\gamma \X(x)} dx  \le \int_\Omega e^{\gamma \X(x)} dx  \le R_\gamma \epsilon_N^{-\gamma^2/2} \mu^\gamma_N(\Omega),
\]
or in other words, 
\begin{equation}\label{eq:contmap}
\frac{\log C_{\gamma,A}^{-1}+\log \mu_N^\gamma(A)}{\log \epsilon_N^{-1}}\leq \frac{\log \int_\Omega e^{\gamma \X(x)}dx}{\log \epsilon_N^{-1}}-\frac{\gamma^2}{2}\leq \frac{\log R_\gamma+\log \mu_N^\gamma(\Omega)}{\log \epsilon_N^{-1}}.
\end{equation}
As $\mu_N^\gamma(\Omega)$ and similarly $\mu_N^\gamma(A)$ converges in law to an almost surely positive random variable, the continuous mapping theorem (see e.g. \cite[Theorem 4.27]{Kallenberg02}) implies that 
the RHS and LHS of the inequalities \eqref{eq:contmap} converge to zero in law and in probability  by a routine argument.
This shows that under $\P_N$, $\frac{\log \int_\Omega e^{\gamma \X(x)}dx}{\log \epsilon_N^{-1}}$ converges to $\frac{\gamma^2}{2}$ in probability in the subcritical regime. 
\medskip

{\bf Case 2.}  Now, we suppose that $\gamma \ge \gamma_*$. 
Let $0<\varepsilon<1$ and $0<\delta < \frac{\varepsilon}{2(\gamma+\gamma_*)}$  be small fixed  parameters. 
First of all,  observe that by the definition of $\mathscr T_N^{\gamma_*-\delta}$, 
\[\begin{aligned}
\int_\Omega e^{\gamma \X(x)} dx   
\ge \int_{  \mathscr{T}_N^{ \gamma_*-\delta}} e^{\gamma \X(x)} dx    
\ge \epsilon_N^{-\gamma(\gamma_*-\delta)} \big|  \mathscr{T}_N^{ \gamma_*-\delta}\big| , 
\end{aligned}\]
and by Proposition~\ref{prop:TP}, and the fact that $\gamma\delta \le \varepsilon$, we obtain that for small enough $\varepsilon>0$ as $N\to\infty$, 
\begin{equation} \label{freez3}
\P_N \left[ \int_\Omega e^{\gamma \X(x)} dx  \ge \epsilon_N^{-\gamma\gamma_*+\gamma_*^2/2+ \varepsilon} \right] \geq \P_N \left[ \left|\mathscr T_N^{\gamma_*-\delta}\right|  \ge \epsilon_N^{(\gamma_*^2 -\delta)/2 + \delta} \right] \to 1. 
\end{equation}

Next we  will obtain a complementary upper-bound.
Let $K = \lfloor \frac{2 \gamma_* (\gamma_*-\delta)}{\varepsilon} \rfloor$ and let us consider the sequence  
$\delta_{k} = \delta+ k \frac{\varepsilon}{2\gamma_*}$ for $k\ge 0$ (note that  $\delta_K \le \gamma_*$).
One has the decomposition:
\[ \begin{aligned}
\int_\Omega e^{\gamma \X(x)} dx  & =
 \int_{  \mathscr{T}_N^{ \gamma_*+\delta}}  e^{\gamma \X(x)} dx 
+ \int_{  \mathscr{T}_N^{ \gamma_*-\delta} \setminus \mathscr{T}_N^{ \gamma_*+\delta}}  e^{\gamma \X(x)} dx \\
&\qquad +  
\sum_{k =0}^{K-1}
\int_{  \mathscr{T}_N^{ \gamma_*-\delta_{k+1}} \setminus \mathscr{T}_N^{ \gamma_*- \delta_{k}}}  e^{\gamma \X(x)} dx 
+  \int_{ \Omega \setminus \mathscr{T}_N^{\gamma_*- \delta_{K}}}   e^{\gamma \X(x)} dx  .  
\end{aligned}\] 
First, since $\delta_K \ge \gamma_* - \frac{\varepsilon}{2\gamma_*}$,  we have $\X(x) \le  \frac{\varepsilon}{2\gamma_*}  \log \epsilon_N^{-1}$ for all $x\notin \mathscr{T}_N^{ \gamma_*  - \delta_K}$ and
\[
\Big|  \int_{ \Omega \setminus \mathscr{T}_N^{\gamma_*-\delta_K}}   e^{\gamma \X(x)} dx  \Big| \le |\Omega| \epsilon_N^{- \frac{\gamma}{2\gamma_*}\varepsilon}  \le  |\Omega|  \epsilon_N^{ -\gamma\gamma_*+\frac{\gamma_*^2}{2}}. 
\]

Secondly, if $\mathscr{A} = \big\{ | \mathscr{T}_N^{ \gamma_*-\delta_{k}} |  \le \epsilon_N^{\frac{(\gamma_*-\delta_k)^2- \varepsilon}{2}} ,\ k=0, \dots, K  \} $,
we have conditionally on $\mathscr{A}$, for any $k=0, \dots , K-1$, 
\[\begin{aligned}
\int_{  \mathscr{T}_N^{ \gamma_*-\delta_{k+1}} \setminus \mathscr{T}_N^{ \gamma_*- \delta_{k}}}  e^{\gamma \X(x)} dx  &\le 
\big| \mathscr{T}_N^{ \gamma_*- \delta_{k+1}} \big|  \epsilon_N^{-\gamma(\gamma_*-\delta_k)} \\
&\le \epsilon_N^{ (\gamma_*- \delta_{k+1})^2/2  - \gamma(\gamma_*-\delta_k) - \varepsilon/2} \\
&\le \epsilon_N^{-\gamma\gamma_*+ \gamma_*^2/2 -\varepsilon} \epsilon_N^{ \delta_{k+1}^2/2} .
\end{aligned}\]
For the last step, we used that $\gamma \delta_k -\gamma_* \delta_{k+1}  \ge -\frac{\varepsilon}{2} $. Similarly, since $\delta < \frac{\varepsilon}{2(\gamma +\gamma_*)}$, we have conditionally on $\mathscr A$ (where we use the condition for $k=0$):
\[\begin{aligned}
 \int_{  \mathscr{T}_N^{ \gamma_*-\delta} \setminus \mathscr{T}_N^{ \gamma_*+\delta}}  e^{\gamma \X(x)} dx 
&\le \epsilon_N^{ (\gamma_*- \delta)^2/2- \gamma(\gamma_*+\delta) - \varepsilon/2} \\
&\le \epsilon_N^{-\gamma\gamma_*+ \gamma_*^2/2 - \varepsilon }\epsilon_N^{\delta^2/2} . 
\end{aligned}\] 
Putting together our estimates  shows that conditionally on $\mathscr{A} \cap \{  \mathscr{T}_N^{ \gamma_*+\delta} = \emptyset\}$,  
\[
\int_\Omega e^{\gamma \X(x)} dx  \le   \epsilon_N^{-\gamma\gamma_*+ \gamma_*^2/2-\varepsilon} \left( {\textstyle \sum_{k=0}^{K} }\epsilon_N^{\delta_k^2/2}  + |\Omega| \epsilon_N^{\varepsilon}\right).    
\]
If $N$ is sufficiently large, the previous parenthesis is bounded by $1$ and we have shown that for large enough $N$, 
\begin{equation} \label{freez4}
\P_N[\mathscr{A} \cap \{  \mathscr{T}_N^{ \gamma_*+\delta} = \emptyset\}] \le 
\P_N \left[\int_\Omega e^{\gamma \X(x)} dx  \le   \epsilon_N^{-\gamma\gamma_*+ \gamma_*^2/2-\varepsilon}   \right] . 
\end{equation}
By Proposition~\ref{prop:TP} and  Assumption~\ref{ass:max}, 
$\P_N[\mathscr{A} \cap \{  \mathscr{T}_N^{ \gamma_*+\delta} = \emptyset\}] \to1$ as $N\to+\infty$.
Hence, combining the estimates \eqref{freez3} and  \eqref{freez4}, we conclude that for any $\gamma\ge \gamma_*$,  as $N\to\infty$
\[
\frac{1}{\log \epsilon_N^{-1}} \log\left( \int_\Omega e^{\gamma \X(x)} dx \right) 
\to  \gamma\gamma_* - \gamma_*^2/2
\]
in probability. 
\end{proof}
This concludes our general discussion and we now move on to applying it to random matrix theory.

\subsection{Application to the eigenvalue counting function in the one-cut regular case} \label{sect:evcf}

The main goal of this section is to prove  Theorem~\ref{thm:max2} by applying  the results of Section~\ref{sect:max}. In fact, we will give two independent proofs of the lower bound.  For the upper bound, the results of Section \ref{sect:max} provide a tool for controlling the maximum over $\Omega=(-1,1)$. To extend the upper bound to all of $\R$, as required in Theorem \ref{thm:max2}, we will make use of monotonicity properties of the eigenvalue counting function.

In Section~\ref{sect:proof1}, we work directly with the eigenvalue counting function using Theorem~\ref{th:gmc} whose proof in Section \ref{sec:special} required Fisher-Hartwig asymptotics in the regime where two singularities are merging -- the full extent of Theorem \ref{th:Hankel1} and Theorem \ref{th:Hankel2}. 
The second proof for the lower bound in Section~\ref{sect:proof2} is also interesting because it relies on Theorem~\ref{th:gmc2} whose proof in Section \ref{sect:gmceasy} is based 
on exponential moment estimates which are much easier to control in the RH analysis.

In Section \ref{sect:proof1}, we also offer proofs of Theorem \ref{th:TP} and Corollary \ref{thm:freezing} using the results of Section \ref{sect:freezing}.

\subsubsection{Proof of Theorem~\ref{thm:max2} and  proofs of Theorem~\ref{th:TP} and Corollary~\ref{thm:freezing}.} \label{sect:proof1}

We will first focus on the maximum over $(-1,1)$ using our general approach and in the end of this section extend to the maximum over $\R$ using specific properties of the eigenvalue counting function. 

Let  $\P_N$  be the law of the eigenvalues $(\lambda_1,\dots ,\lambda_N)$ of random matrix of size $N$ coming from a one-cut regular potential (or to be more precise, a potential satisfying Assumptions \ref{ass:regular}) $V$ with equilibrium measure 
$d\mu_V = \psi_V(t)\sqrt{1-t^2} \1_{|t|<1}dt$ as in Section \ref{rigidity_results}. 
We let  $F(x) = \int_{-1}^{x} d\mu_V$ and recall that for all $x\in\R$, 
\[ 
\h(x) = \sqrt{2}\pi \#\{j:\lambda_j \le x\} - \sqrt{2}\pi N F(x) .
\]
Note that this function is bounded from above, integrable on $(-1,1)$, and upper semicontinuous on $\R$. 
As we have seen in Section \ref{gmc_results}, under  $\P_N$,   the function $\X=  \h$  approximates   a Gaussian log-correlated field on $\Omega=(-1,1)$  in the sense of Definition~\ref{def:EN} and the underlying (centered) Gaussian process has the covariance kernel: 
\[
\Sigma(x,y) = \log\left( \frac{1-xy+\sqrt{1-x^2}\sqrt{1-y^2}}{|x-y|} \right).
\] 

In order to control the maximum over $(-1,1)$, we first need to check the validity of Assumptions 3.1.  This will rely on Theorem \ref{th:Hankel3}.
 
\begin{lemma} \label{lem:1p}
The random fields $\X=  \pm \h$ restricted to $(-1,1)$ satisfy Assumptions~\ref{ass:1p} with $\epsilon_N = 1/N$.  
\end{lemma}

{\begin{proof}
We will give the proof in the case  $\X= \h$, the case $\X= -\h$ being analogous. Let us fix $\gamma>0$ and recall that in the  notation of Section \ref{sec:hankel_intro}, we have 
\[
\E_N e^{\gamma \h(x)}=\frac{D_N(x;  \gamma;0)}{D_N(0;0;0)}e^{-N \sqrt{2}\gamma\pi F(x)} ,
\]
By Theorem \ref{th:Hankel3}, this implies that there exists a constant $m \in \N$ such that as $N\to+\infty$,
\begin{equation} \label{Laplace1}
\log \E_N e^{\gamma \h(x)}=\frac{\gamma^2}{2}\log N+\frac{3\gamma^2}{4}\log(1-x^2)+ \mathcal O_\gamma(1),
\end{equation}
uniformly for all $|x|\le 1- m N^{-2/3}$. 
In particular,  this already implies that for any fixed compact set $A \subset (-1,1)$, there exists $C_{\gamma, A}$ such that for all $x\in A$,
\[
\E_N e^{\gamma  \h(x)} \ge C_{\gamma, A}^{-1} N^{ \gamma^2/2} .  
\]
This gives the estimate  \eqref{1p1} with $\epsilon_N = 1/N$ -- it remains to obtain the complementary bound   \eqref{1p2}.
Since we assume that the potential is regular, the equilibrium density vanishes like a square-root at the edges and there exists a constant $C_V>0$ such that  $\psi_V(x)(1-x^2)^{3/2} \le C_V F(x)$ for all $|x| \le 1$, 
 then by \eqref{Laplace1}, there exists a constant $R'_\gamma>0$  such that for all $|x|\le 1- m N^{-2/3}$,
\begin{equation} \label{Laplace2}
\E_N e^{\gamma  \h(x)} 
\le R'_\gamma  \big(NF(x)\big)^{\gamma^2/2}  . 
\end{equation}
To extend this estimate for $x\le -1+ m N^{-2/3}$, we use the deterministic bounds:
\[
 \h(x) \le \h(-1+ m N^{-2/3}) +  \sqrt{2}\pi N F(-1+ m N^{-2/3})
\qquad\text{and}\qquad  F(x) \le C_m N^{-1} ,  
\] 
where $C_m>0$ is a constant.  By \eqref{Laplace2}, this implies that for all $x\le -1+ m N^{-2/3}$, 
\begin{align} \notag
\E_N e^{\gamma  \h(x)}
&  \le e^{\sqrt{2}\pi  \gamma C_{m}} \E_N \big[ e^{\gamma  \h(-1+m N^{-2/3})} \big]   \\
& \label{Laplace_3} \le R'_\gamma  e^{\sqrt{2}\pi  \gamma C_{m}}  C_{m}^{\gamma^2/2} . 
\end{align} 
 Similarly, since  there exists a constant $C'_{m}>0$ such that for all $x\ge 1- m N^{-2/3}$, 
\[
\h(x) \le   \sqrt{2}\pi  N\left( 1 - \int^{1- m N^{-2/3}}_{-1} \hspace{-1cm}\psi_V(t) \sqrt{1-t^2} dt \right) \le C'_{m} .
\]
Combining this estimate and \eqref{Laplace_3}, we have shown that for any fixed $m>0$, 
\begin{equation} \label{Laplace4}
\E_N e^{\gamma  \h(x)} = \O_{m, \gamma}(1)
\end{equation}
uniformly for all $|x|\ge 1- m N^{-2/3}$.
With the estimate \eqref{Laplace2}, this completes the proof of~\eqref{1p2}. 
\end{proof}}

 By Theorem~\ref{th:gmc}, the random processes $\X=  \pm\h$   also satisfy the Assumptions~\ref{ass:gmc}. Thus to establish the proofs of Theorem \ref{thm:max2} (for the maximum over $(-1,1)$), Theorem \ref{th:TP}, and Corollary \ref{thm:freezing}, using Corollary \ref{max:LLN}, Proposition \ref{prop:TP}, and Proposition \ref{prop:freezing}, we simply need to verify Assumption \ref{ass:max}, or by Proposition \ref{prop:cont}, we simply need to verify the assumptions of Proposition \ref{prop:cont}. To do this, note that since the counting function $x\mapsto \#\{\lambda_j \le x\}$ is non-decreasing and the density of the equilibrium measure is bounded, $0 \le \psi_V(x)\sqrt{1-x^2} \le {\rm C}/\pi$, 
we have for any  $x\in [-1,1-\frac 1N]$,
\begin{equation*}
h_N(t) \ge h_N(x) - {\rm C} , \qquad \text{for all } t\in[x, x+\tfrac 1N] . 
\end{equation*}
Moreover, notice that if $x\in[1- \tfrac 1N , 1]$, we have the deterministic bound 
$h_N(x) \le \sqrt{2}\pi  N\big(1- F(1-\tfrac 1N)\big) \le C/\sqrt{N}$ where the constant $C>0$ depends only $F$. Similarly, since $h_N(t) \ge - \sqrt{2}\pi NF(\tfrac 1N)  \ge -  C/\sqrt{N}$ for all $t\in[-1,-1+\tfrac 1N]$, this implies that if $N$ is sufficiently large:
\[
h_N(t) \ge h_N(x) - {\rm C} , \qquad \text{for all } x\in[1-\tfrac 1N, 1] \text{ and } t\in[-1,-1+\tfrac 1N] . 
\]
This shows that the field $\h$ satisfies the assumption of Proposition \ref{prop:cont}, and this concludes the proof of Theorem \ref{thm:max2} (for the maximum over $(-1,1)$),  Theorem \ref{th:TP}, and Corollary \ref{thm:freezing} --  similar arguments also apply to the field $-\h$.  Note that this proof also gives the following estimate: for any compact set $A\subset(-1,1)$ such that $|A|>0$ and any $\delta>0$,
\begin{equation} \label{lower_bound_1}
\P_N\big[ (\sqrt{2} - \delta)\log  N \le  \max_{A} \h \le   \max_{[-1,1]} \h   \le  (\sqrt{2} +\delta)\log  N  \big] \to 1
\qquad\text{as }N\to+\infty. 
\end{equation}
By \eqref{LDE2}, we also obtain the following large deviation estimate: for any increasing sequence $\varpi_N$ such that 
$\displaystyle\lim_{N\to+\infty}\frac{\varpi_N}{\log N} =+\infty$, for any $\gamma>0$, there exists a constant $C_\gamma>0$ such that 
\begin{equation} \label{LDE3}
 \P_N\big[ \max_{[-1,1]} \h \ge \varpi_N \big] \le C_\gamma N^{\gamma^2/2+1} e^{-\gamma \varpi_N} . 
\end{equation}
 
To conclude the proof of Theorem \ref{thm:max2}, we need to control the maximum of $\h$ off of $(-1,1)$. For this, we note that immediately from the definition of $\h$, we have for $x<-1$, $h(x)\leq h(-1)$ and for $x\geq 1$, $h(x)\leq 0$. From this, $\max_\R \h=\max(\sup_{(-1,1)}\h,0)$ and the claim follows from our upper bound for the maximum over $(-1,1)$.

We now turn to our second proof of the lower bound for the maximum of the eigenvalue counting function.

\subsubsection{Second proof of the lower bound of Theorem~\ref{thm:max2}} \label{sect:proof2}

In this section we offer a proof of the lower bound part of Theorem \ref{thm:max2}, which requires milder estimates than the first proof we gave.  In particular, we will work with the smoothed eigenvalue counting function and rely on Theorem \ref{th:gmc2} instead of the actual eigenvalue counting function and Theorem \ref{th:gmc}. We focus on the lower bound since it seems hard to get an upper bound for $\h$ from an upper bound of the smoothed field. A direct approach for the upper bound for $\h$ would in any case rely on Theorem \ref{th:Hankel3} in some way and we already gave one proof of this type -- though see Lemma \ref{lem:MB} for a slightly different argument.

Let $\alpha>0$ and $A \subset (-1,1)$ be a compact  set.
For any $N\in\N$, we once again let $\P_N$ be the law of the eigenvalues of a  one-cut regular random matrix of size $N$ and now, as opposed to the previous section, we let $\epsilon_N = N^{\alpha-1}$.
 As a special case of Theorem~\ref{th:Hankel1} (see our verification of \eqref{exp_moment} based on Theorem \ref{th:Hankel1}, in Section \ref{sec:special}), we see that the random field $\X=  \h(\cdot + \i\epsilon_N)$  restricted to $A$ satisfies Assumptions~\ref{ass:1p}. Moreover, by Theorem~\ref{th:gmc2},  $\X$ also satisfies  Assumptions~\ref{ass:gmc}.
Hence, by Theorem~\ref{thm:LB}, we obtain for any $\delta>0$,
\[
\P_N\big[ \max_{x\in A} \h(x+\i \epsilon_N) \ge  (\sqrt{2} - \delta)\log \epsilon_N^{-1}\big] \to 1
\qquad\text{as }N\to+\infty.  
\]
Since $\h(\cdot + \i\epsilon_N) = \h*\phi_{\epsilon_N}$ where $\phi(u) = \frac{1}{\pi}\frac{1}{1+u^2}$, $\max_\R \h\geq \max_A h(\cdot+\i\epsilon_N)$ and this implies that
\[
\P_N\big[ \max_\R \h \ge  (\sqrt{2} - \delta) (1 - \alpha) \log N \big]  \to 1
\qquad\text{as }N\to+\infty.  
\]
Since $\alpha, \delta >0 $ are arbitrary, this establishes the lower bound in Theorem \ref{thm:max2} in the $+$-case. For the $-$-case, the argument is identical.

\section{Eigenvalue Rigidity} \label{sec:rigidity}

In this section we will finally prove our main rigidity result -- Theorem \ref{thm:rigidity}. 
The main content of the proof is given in Section~\ref{sec:rigbulk}.  Then, in Section~\ref{sec:rigest}  we establish some technical lemmas that we need and in Section \ref{sec:rigedge},  we prove rigidity near the edge of the spectrum. 

\medskip

Throughout this section, let  $F(x) = \int_{-1}^x \psi_V(t)\sqrt{1-t^2}dt$ and recall that we denote by $\varkappa_1, \dots , \varkappa_{N}\in(-1,1]$ the quantiles of the equilibrium measure which are given implicitly by $F(\varkappa_k)= k/N$ for all $k=1,\dots,N$. Before going into the details of the proof of Theorem \ref{thm:rigidity}, we point out that a slightly weaker version of Theorem \ref{thm:rigidity} follows almost immediately from Theorem \ref{thm:max2}. The upper bound from Theorem \ref{thm:max2} implies that 
\[\lim_{N\to\infty}\P_N\left[|h_N(\lambda_j)|\leq (1+\delta)\sqrt{2}\log N\mbox{ for all $j=1,\ldots, N$}\right]=1\]
for any $\delta>0$.
As $\h(\lambda_j)=\sqrt{2}\pi j-\sqrt{2}\pi N F(\lambda_j)$, the above identity is equivalent to
\[\lim_{N\to\infty}\P_N\left[F(\lambda_j)\in\left[\frac{j}{N}-(1+\delta)\frac{\log N}{\pi N}, \frac{j}{N}+(1+\delta)\frac{\log N}{\pi N}\right]\mbox{ for all $j=1,\ldots, N$}\right]=1,\]
and this implies
\[\lim_{N\to\infty}\P_N\left[\lambda_j\in\left[\kappa_{\lfloor j-(1+\delta)\frac{\log N}{\pi}\rfloor}, \kappa_{\lceil j+(1+\delta)\frac{\log N}{\pi}\rceil}\right]\mbox{ for all $j=1,\ldots, N$}\right]=1,\]
where we set $\kappa_j=-\infty$ for $j<0$ and $\kappa_j=+\infty$ for $j>N$.
This gives us a good upper bound on the fluctuations of the bulk eigenvalues, but unfortunately it does not allow us to bound the fluctuations of the
$\lfloor(1+\delta)\frac{\log N}{\pi}\rfloor$ smallest and largest eigenvalues. Indeed, much of the work in this section concerns bounding the fluctuations of the eigenvalues close to the edge.

Finally as another preliminary remark, note that for any $k=k_N$ for which $k/N\to 0$, we have as $N\to\infty$
\begin{equation} \label{qest}
\frac{\varkappa_k+ 1}{k^{2/3} N^{-2/3}} \to c_-:=\left(\frac{3}{2\sqrt{2}\psi_V(-1)}\right)^{2/3}
 \qquad\text{and}\qquad 
  \frac{1-\varkappa_{N-k}}{k^{2/3} N^{-2/3}} \to c_+:=\left(\frac{3}{2\sqrt{2}\psi_V(1)}\right)^{2/3}.
\end{equation}
Note that by Assumption \ref{ass:regular}, $c_\pm$ are finite and positive. We use these results throughout the coming sections.

\subsection{Proof of Theorem \ref{thm:rigidity}} \label{sec:rigbulk}
 As already indicated above,  the proof is divided in two parts: bulk and edge rigidity. 
We begin by establishing rigidity for bulk eigenvalues, Proposition~\ref{prop:brig}. After this, we state the main ingredients for edge rigidity and assuming these, we provide a proof of Theorem \ref{thm:rigidity}.
In order to explain the connection with Theorem~\ref{thm:max2}, we will make the following assumption. Let $\ell_N = \lfloor (\log\log N)^5\rfloor$ and define the event 
\begin{equation} \label{B_event}
\mathscr{B}_N = \big\{ |\lambda_k- \varkappa_k| \le N^{-2/3} \text{ for all } 
k = \ell_N, \dots , N-\ell_N\big\} . 
\end{equation}
We assume that $\P_N[\mathscr{B}_N] \to 1$ as $N\to+\infty$. This will be justified in Section~\ref{sec:rigest} -- see Lemma~\ref{lem:rig}.

\begin{proposition}[Optimal bulk Rigidity] \label{prop:brig}
Let  $\ell_N = \lfloor (\log\log N)^5\rfloor$, $\mathscr{B}_N$ be given by \eqref{B_event} and suppose that  $\P_N[\mathscr{B}_N] \to 1$ as $N\to+\infty$. Then, for any $\alpha>0$, we have as $N\to+\infty$,
\[
\P_N\bigg[  \frac{1-\alpha}{\pi} \frac{\log N}{N} \le \max_{k=\ell_N \dots, N-\ell_N}   F'(\varkappa_k)|\lambda_k- \varkappa_k|  \le \frac{1+\alpha}{\pi} \frac{\log N}{N} \bigg] \to 1 .
\]
\end{proposition}

\begin{proof}
Observe that conditionally on $\mathscr{B}_N$, by a Taylor expansion, we have  for any $k=\ell_N ,\dots , N-\ell_N$,
\[
\h(\lambda_k) =\sqrt{2}\pi N\big(F(\varkappa_k) - F(\lambda_k) \big)
= - \sqrt{2}\pi N F'(\varkappa_k)(\lambda_k - \varkappa_k) +\O(1).
\]
 We emphasize here that the above $\mathcal O(1)$-term is only defined on the event $\mathscr B_N$ and the implied constant is non-random and independent of $k$ -- let us write $C$ for this constant.
  This estimate comes from the fact that by \eqref{qest} we have on the event $\mathscr{B}_N$:
$\lambda_k , \varkappa_k \in [-1+ N^{-2/3}, 1-N^{-2/3}]$ for all relevant $k$ and there exists a constant $C'>0$ such that $\displaystyle\sup_{|x| \le 1-N^{-2/3}}\big| F''(x) \big| \le C' N^{1/3}$.

In particular, this implies that conditionally on $\mathscr{B}_N$,
\[ \begin{aligned} \sqrt{2}\pi N   \max_{\ell_N \le k\le  N-\ell_N}  \big| F'(\varkappa_k)(\lambda_k - \varkappa_k)  \big| 
&\le  C+ \max_{k=\ell_N, \dots , N-\ell_N} \big| \h(\lambda_k) \big|  \\
&\le C + \max_{x\in \R} |\h(x)|.
\end{aligned}\]
By Theorem~\ref{thm:max2}, this  implies that as $N\to+\infty$,
\begin{align} \notag
\P_N\left[ \left\{  \max_{\ell_N \le k\le N-\ell_N}  F'(\varkappa_k)\big|\lambda_k - \varkappa_k  \big|  \ge \frac{1+\alpha}{\pi} \frac{\log N}{N} \right\} \cap \mathscr{B}_N \right]
&\le \P\left[ \max_{x\in \R} |\h(x)|  \ge \sqrt{2}(1+\alpha)\log N   - C \right]\\
&\label{rig1}\to 0 . 
\end{align}

On the other  hand,  we also have conditionally on $\mathscr B_N$
\[ \begin{aligned} \sqrt{2}\pi N    \max_{\ell_N \le k\le  N-\ell_N} \big| F'(\varkappa_k)(\lambda_k - \varkappa_k)  \big| 
&\ge    \max_{\ell_N \le k\le  N-\ell_N}   \big| \h(\lambda_k) \big| -C.
\end{aligned}\]
Note that the eigenvalues are exactly the local maxima of the function $\h$ and conditionally on $\mathscr{B}_N$ :  $[-\tfrac 12, \tfrac 12] \subset [\lambda_{\ell_N}, \lambda_{N-\ell_N}]$ so that
\[
 \sqrt{2}\pi N    \max_{\ell_N \le k\le  N-\ell_N} \big| F'(\varkappa_k)(\lambda_k - \varkappa_k)  \big|  \ge \max_{|x|\le 1/2 } \h(x) -C
\]
By Theorem~\ref{thm:max2}, this  implies that as $N\to+\infty$,
\begin{equation} \label{rig2}
\P_N\left[ \left\{  \max_{\ell_N \le k\le N-\ell_N}  F'(\varkappa_k)\big|\lambda_k - \varkappa_k  \big|  \le \frac{1-\alpha}{\pi} \frac{\log N}{N} \right\} \cap \mathscr{B}_N \right] \to 0 .
\end{equation}
Since we assume that $\P[\mathscr{B}_N] \to 1$ as $N\to+\infty$, combining the estimates \eqref{rig1} and \eqref{rig2},  we complete the proof.
\end{proof}

Let us observe that assuming that the probability of $\mathscr B_N$ tends to one, Proposition~\ref{prop:brig} already implies the lower--bound  of Theorem \ref{thm:rigidity}.
So, to complete the proof of Theorem \ref{thm:rigidity}, it remains to verify this assumption concerning $\mathscr B_N$ and extend the upper--bound to the edge eigenvalues. Namely, we will show in Section~\ref{sec:rigedge} (see Proposition \ref{prop:rigedge}) that there exists a constant $C>0$  such that if $\ell_N = \lfloor (\log\log N)^5\rfloor$, we have as $N\to+\infty$,
\begin{equation} \label{rigedge}
 \P\left[ \max_{\substack{k\le \ell_N \\ k\ge N-\ell_N}} \left\{ F'(\varkappa_k) |\lambda_k - \varkappa_k|  \right\}  \le C \frac{\ell_N}{N} \right] \to 1 . 
\end{equation}
By combining Proposition~\ref{prop:brig}, Lemma~\ref{lem:rig} below and \eqref{rigedge}, we complete the proof of Theorem~\ref{thm:rigidity}.

\subsection{Estimates for the eigenvalue counting function near the edges}
\label{sec:rigest}

The main goal of this section is to show that if $\mathscr{B}_N$ is given by \eqref{B_event}, then $\P[\mathscr{B}_N] \to 1$ as $N\to+\infty$. 
To prove this, we will need estimates for the maximum of the eigenvalue counting function $\h$ which are more precise than those  of Theorem~\ref{thm:max2} near the edges of the spectrum. The following lemma does this and will be crucial in order to prove Proposition \ref{prop:rigedge} in the next section.

\begin{lemma} \label{lem:MB}
There exists $C>0$ such that  for any $\rho>0$  $($possibly depending on $N)$, 
\begin{align} 
& \label{MB1}
\sup_{n \le N}\P_N\big[ \max_{x\le \varkappa_n } \h(x) \ge \sqrt{2} \log(n+1) +\rho \big] \le C e^{-\sqrt{2} \rho} \\
& \label{MB2}
\sup_{n \le N}\P_N\big[ \max_{x\ge \varkappa_{N-n+1} } \h(x) \ge \sqrt{2} \log(n+1) +\rho \big] \le C e^{-\sqrt{2} \rho} . 
\end{align}
By symmetry, similar estimates hold for the random function  $-\h$. 
\end{lemma}

\begin{proof}[Proof of Lemma~\ref{lem:MB}]
We first prove the estimate \eqref{MB1} and then explain how to change the argument to obtain \eqref{MB2}. 
First, let us observe that the estimates \eqref{Laplace2} and \eqref{Laplace4}  from the Proof of Lemma~\ref{lem:1p}  imply that there exists a constant $R>0$ such that for all $j=1,\dots, N$,
\begin{equation} \label{Laplace5}
\E_N e^{\sqrt{2} \h(\varkappa_j)} 
\le R NF(\varkappa_j) = Rj  . 
\end{equation}

By the definition of $\h$, we have for any $j =2,\dots, N$ and all $ x\in[\varkappa_{j-1},  \varkappa_j]$, 
\[
\h(x) \le \h(\varkappa_j)  +\sqrt{2}\pi , 
\]
and for all $x\le \varkappa_1$, 
\[
\h(x)  \le \h(\varkappa_1) + \sqrt{2}\pi .
\]
These bounds imply that for any $n=1,\dots, N$,
\[
\max_{x\le \varkappa_n} \h(x) \le \max_{j \le n} \h(\varkappa_j)  +\sqrt{2}\pi,
\]
and by a union bound, 
\[
\P_N\big[ \max_{x\le \varkappa_n } \h(x) \ge \sqrt{2} \log(n+1)+\rho \big]  \le \sum_{j\le n} \P_N\big[   \h(\varkappa_j)  \ge  \sqrt{2} \log(n+1) +\rho -\sqrt{2}\pi \big] . 
\]
Using Markov's inequality and \eqref{Laplace5}, this implies that for any $n=1,\dots, N$,
\[
\P_N\big[ \max_{x\le \varkappa_n } \h(x) \ge \sqrt{2} \log(n+1) +\rho \big]  \le \frac{R}{(n+1)^2} e^{-\sqrt{2}(\rho- \sqrt{2}\pi)}\sum_{j\le n} j  \le \frac{e^{2\pi}R}{2} e^{-\sqrt{2}\rho} . 
\]

Similarly, by slightly adapting the Proof of Lemma~\ref{lem:1p}, we have that for a possibly larger $R>0$ and for all $j=1,\dots, N$,
\[
\E_N e^{\sqrt{2} \h(\varkappa_j)} 
\le R N\big(1-F(\varkappa_j)\big) = R(N-j)  . 
\]
Using this bound, the estimate $\displaystyle\max_{x \ge \varkappa_{N-n+1}} \h(x) \le \max\Big\{ 0 , \max_{j\ge N-n+1} \h(\varkappa_j)\Big\} +\sqrt{2} \pi,$ and a similar application of Markov's inequality, we obtain  that for any $n=1,\dots, N$,
\[\begin{aligned}
\P_N\big[ \max_{x\ge \varkappa_{N-n+1} } \h(x) \ge \sqrt{2} \log(n+1) +\rho \big] 
& \le  \frac{e^{2\pi}R}{(n+1)^2}  e^{-\sqrt{2}\rho} \sum_{j\le n} \left(1+ \E_N e^{\sqrt{2} \h(\varkappa_{N-j})}  \right)\\
& \le \frac{e^{2\pi}(n+2)R}{2(n+1)}e^{-\sqrt{2}\rho},
\end{aligned}\]
which completes the proof.  
\end{proof}

In the remainder of this section, we show that this implies the a priori estimate that we need for verifying the assumptions of Proposition~\ref{prop:brig}.

\begin{lemma} \label{lem:rig}
Let $\ell_N = \lfloor (\log\log N)^5 \rfloor$ and $\mathscr{B}_N$ be given by \eqref{B_event}, then $\P[\mathscr{B}_N] \to 1$ as $N\to+\infty$. 
\end{lemma}

\begin{proof} 
 The idea is to look at the probability of finding a $k$ such that $|\lambda_k-\varkappa_k|\geq N^{-2/3}$ and if $k$ is not too close $1$ or $N$, to argue using Theorem \ref{thm:max2} that this probability tends to zero, while if $k$ is close to the end points, then we rely on Lemma \ref{lem:MB}.  Let $L_N = \lfloor (\log N)^5 \rfloor$ be a quantity differentiating these cases, 
 and suppose  that there exists $k \in [L_N, N-L_N]$ such that $|\lambda_k-\varkappa_k| \ge N^{-2/3} $. Then,  we have either 
\[  \tfrac{1}{\sqrt{2}\pi N}\h(\varkappa_k- N^{-2/3})  
 \ge F(\varkappa_k) - F(\varkappa_k- N^{-2/3})  
 \]
or 
\[  \tfrac{1}{\sqrt{2}\pi N}\h(\varkappa_k+ N^{-2/3})  
 \le F(\varkappa_k) - F(\varkappa_k+ N^{-2/3})   .
 \]
By \eqref{qest}, we have $\varkappa_{L_N}-N^{-2/3}\geq -1+L_N^{7/12} N^{-2/3}$  and $\varkappa_{N-L_N}+N^{-2/3}\leq 1- L_N^{7/12}N^{-2/3},$ if $N$ is large enough, so that  the previous bounds imply that if such a $k$ exists, then
 \begin{equation} \label{R1}
\max_{[-1,1]} |\h|  \ge \sqrt{2}\pi N^{1/3}   \inf_{|x| \le 1- L_N^{7/12}N^{-2/3}} \hspace{-.2cm} |F'(x)| . 
 \end{equation}
Since $|F'(x)| \gg\sqrt{1-x^2}$  for all $|x| \in (-1,1)$, we deduce from \eqref{R1} that if $N$ is sufficiently large,
\[
\P_N\left[ |\lambda_k-\varkappa_k| \ge N^{-2/3}  \text{ for a }k\in [L_N, N-L_N] \right]
\le \P_N\Big[ \max_{[-1,1]} |\h|  \ge L_N^{1/4}\Big]
\]
By Theorem~\ref{thm:max2} this probability converges to 0 as $N\to+\infty$. 

\medskip

Similarly, by \eqref{qest}, since  $\varkappa_{L_N} +  N^{-2/3} \le \varkappa_{2L_N} $ if $N$ is large enough,  we see that if $|\lambda_k-\varkappa_k| \ge N^{-2/3} $ for a $k\in[\ell_N, L_N]$, then 
\[
\max_{x\le \varkappa_{2L_N}} |\h(x)| \ge   \sqrt{2}\pi N^{1/3}  \inf_{|x| \le 1- \ell_N^{7/12}/N^{2/3}} \hspace{-.2cm} |F'(x)| \ge \ell_N^{1/4} . 
\]
This bound implies that
\[
\P_N\left[ |\lambda_k-\varkappa_k| \ge N^{-2/3}  \text{ for a }k\in [\ell_N, L_N]   \right] \le  \P_N\Big[ \max_{x\le \varkappa_{2L_N}} |\h(x)|  \ge \ell_N^{1/4}
\Big]
\]
Since $\frac{\ell_N^{1/4}}{\log L_N} \to +\infty$ as $N\to+\infty$,  by Lemma~\ref{lem:MB}, this probability converges to 0 as $N\to+\infty$.
We show in an analogous fashion that 
\[
\P_N\left[ |\lambda_k-\varkappa_k| \ge N^{-2/3}  \text{ for a }k\in [N- L_N, N- \ell_N]   \right] \to 0 , 
\]
as $N\to+\infty$, which completes the proof. 
\end{proof}

\subsection{Edge rigidity}  \label{sec:rigedge}

It remains to establish rigidity of the eigenvalues near the edges of the spectrum. Our proof relies on Lemma~\ref{lem:MB} and the convergence in distribution of the smallest and largest eigenvalues (properly rescaled) to Tracy-Widom random variables.
More precisely, we use the fact that if $V$ satisfies the Assumptions~\ref{ass:regular}, then the smallest and largest eigenvalues $\lambda_1$ and $\lambda_N$ of the ensemble satisfy as $N\to+\infty$
\begin{equation}\label{eq:TWconv}
c_+^{-1}(\lambda_N-1) N^{2/3} \Rightarrow {\rm TW}_2 
\qquad\text{and}\qquad
- c_-^{-1}(\lambda_1+1) N^{2/3} \Rightarrow {\rm TW}_2 ,
\end{equation}
where ${\rm TW}_2$ denotes the GUE Tracy-Widom law \cite{TW}. 
This was proved in \cite[Corollary 1.3]{DeiftGioev} for polynomial $V$, by relying on asymptotics for orthogonal polynomials from \cite{DKMVZ} which allow to prove trace class convergence of the eigenvalue correlation kernel to the Airy kernel.
Although \eqref{eq:TWconv} has, to the best of our knowledge, not been stated explicitly for non-polynomial $V$, the proof of \cite{DeiftGioev} can be extended by relying in addition on the asymptotic analysis from \cite{DKMVZ2}, valid also in the non-polynomial case. 

Our result concerning rigidity at the edge is the following.

\begin{proposition}[Edge] \label{prop:rigedge}
 There exists a constant $C>0$ such that if $\ell_N = \lfloor (\log\log N)^5 \rfloor$, then as $N\to+\infty$, 
\begin{equation} \label{edge1}
 \P_N\left[ \max_{k\le \ell_N} \left\{ F'(\varkappa_k) |\lambda_k - \varkappa_k|  \right\}  \le C \frac{\ell_N}{N} \right] \to 1 
\end{equation}
and 
\begin{equation} \label{edge2}
 \P_N\left[ \max_{k\ge N- \ell_N} \left\{ F'(\varkappa_k) |\lambda_k - \varkappa_k|  \right\}  \le C \frac{\ell_N}{N} \right] \to 1 . 
\end{equation}
\end{proposition}

\begin{proof}
We only show the estimate \eqref{edge1} since the proof of  \eqref{edge2} follows from analogous arguments. 
Define the event 
 \[
\mathscr{W}_N = \big\{ \lambda_1 \ge -1 - \ell_N^{2/3} N^{-2/3}\big\} 
\]
and note that \eqref{eq:TWconv} implies that since $\ell_N \to +\infty$ as $N\to+\infty$,  we have 
\begin{equation} \label{W_event}
\P_N[\mathscr{W}_N ] \to 1. 
\end{equation}

On the one hand,  conditionally on $\mathscr{W}_N$, we have for all $k \le \ell_N$,  
\[
  \varkappa_k - \lambda_k  \le   \varkappa_{\ell_N}+1 + \ell_N^{2/3} N^{-2/3}.  
\]
Since $F'(\varkappa_{\ell_N}) \ll \sqrt{1+ \varkappa_{\ell_N}}$, by \eqref{qest}, we have
\begin{equation} \label{R5}
\max_{k\le \ell_N} \left\{ F'(\varkappa_k) \big(   \varkappa_k - \lambda_k  \big) \right\}  \ll \frac{\ell_N}{N} , 
\end{equation}
where the implied deterministic constant depends only on the equilibrium measure.

On the other hand, since $F''>0$ on $(-1, -1+\delta]$ for some small $\delta>0$, 
if $N$ is sufficiently large, we have for all $k\le \ell_N$, 
\[
  F'(\varkappa_k) (\lambda_k - \varkappa_k) \le  \big( F(\lambda_k) - F(\varkappa_k) \big)_+ \le \tfrac{1}{\sqrt{2}\pi N}  \big(\h(\lambda_k)  \big)_- .
\]
Then conditionally on the event $\mathscr{B}_N$ given by \eqref{B_event}, $\lambda_{\ell_N} \le \varkappa_{\ell_N} + N^{-2/3} \le  \varkappa_{2\ell_N}$ if $N$ is large enough, so that 
\begin{equation} \label{R6}
\max_{k\le \ell_N} \left\{   F'(\varkappa_k) (\lambda_k - \varkappa_k) \right\} 
\le  \tfrac{1}{\sqrt{2}\pi N}  \max_{x\le \varkappa_{2\ell_N}} \big(\h(\lambda_k)  \big)_- .
\end{equation}
 Combining the estimates \eqref{R5} and \eqref{R6}, we see that there exists a  (deterministic) constant $C>0$ such that if $N$ is sufficiently large, we have conditionally on $\mathscr{W}_N \cap \mathscr{B}_N$, 
\[
\max_{k\le \ell_N} \left\{ F'(\varkappa_k) |\lambda_k - \varkappa_k|  \right\} 
\le \frac{C}{N} \max\left\{ \ell_N ,  \max_{x\le \varkappa_{2\ell_N} }\big\{- \h(x)\big\} \right\} . 
\]
This shows that
\[
\P_N\left[ \max_{k\le \ell_N} \left\{ F'(\varkappa_k) |\lambda_k - \varkappa_k|  \right\}  \ge C \frac{\ell_N}{N} \right]
\le \P_N\left[  \max_{x\le \varkappa_{2\ell_N} }\big\{- \h(x)\big\}  \ge \ell_N  \right]  + \P_N\big[ (\mathscr{W}_N \cap \mathscr{B}_N)^c\big] .
\]
Since $\P_N\big[ (\mathscr{W}_N \cap \mathscr{B}_N)^c\big] \to 0$ as $n\to+\infty$ by Lemma~\ref{lem:rig} and \eqref{W_event}, we conclude that by Lemma~\ref{lem:MB},
\[
\limsup_{N\to+\infty} \P_N\left[ \max_{k\le \ell_N} \left\{ F'(\varkappa_k) |\lambda_k - \varkappa_k|  \right\}  \ge C \frac{\ell_N}{N} \right] =0 . 
\]
This completes the proof. 
\end{proof}

\section{Differential identities for Hankel determinants and RH problems}\label{section: diffid}

In this section, we begin our analysis of the Hankel determinants discussed in Section \ref{sec:hankel_intro}. We begin with a short outline of the remainder of the article which focuses on proving the results of Section \ref{sec:hankel_intro}.

The goal of this section is to review the basic connection between Hankel determinants and RH problems as well to transform the RH problems associated to the Hankel determinants $D_N(x_1,x_2;\gamma_1,\gamma_2;w)$ from Section \ref{sec:hankel_intro} into a form accessible for asymptotic analysis. Particularly important tools for this asymptotic analysis are certain differential identities we record in Section \ref{sec:DI}. The transformations we shall perform are detailed in Section \ref{sec:g} and Section \ref{sec:lens1}. After this, the RH problems will be in a form where it becomes important to search for approximate solutions to the problem -- such approximate solutions are often called parametrices in the RH-literature. These approximate solutions depend radically on the location and distance of the points $x_1,x_2$ appearing in Theorems \ref{th:Hankel1}, \ref{th:Hankel2}, and \ref{th:Hankel3}. In Section \ref{sec:merging}, we consider the situation where the points $x_1$ and $x_2$ are close to each other, but not close to the edge of the spectrum -- constructing the approximate solution will allow proving Theorem \ref{th:Hankel2} in this section. In Section \ref{section: RHseparated}, we treat the situation where the points are separated, and combining with the analysis of Section \ref{sec:merging}, this will yield the proof of Theorem \ref{th:Hankel1}. In Section \ref{section: RHedge}, we consider the case of $x_1=x_2$ being near the edge, and our analysis of the approximate problem will yield a proof of Theorem \ref{th:Hankel3}. The construction of all of these approximate solutions relies on certain model RH problems we consider in Section \ref{sec:model}.

Let us now turn to discussing the classical connection between Hankel determinants, orthogonal polynomials and RH problems in order to derive suitable differential identities for the Hankel determinants $D_N(x_1,x_2;\gamma_1,\gamma_2;w)$ defined in {\eqref{eq:FH}}. In what follows, we assume $-1<x_1\leq x_2<1$, $\gamma_1,\gamma_2\in\mathbb R$, and we assume that $w$ is a function analytic in the region $\mathcal S_N$ defined in \eqref{domain}.

\subsection{Differential identities}\label{sec:DI}

For notational convenience, define the weight function $h:\R\to \R$ as follows,
\begin{equation}\label{def:w}
h(\lambda)=h(\lambda;x_1,x_2;\gamma_1,\gamma_2;w;N)=e^{\sqrt{2}\pi\gamma_1\mathbf{1}_{(-\infty,x_1 ]}(\lambda)+\sqrt{2}\pi\gamma_2\mathbf{1}_{(-\infty,x_2]}(\lambda)}e^{-NV(\lambda)+w(\lambda)}.
\end{equation}

Let $p_0,p_1,p_2,\ldots$ be normalized orthogonal polynomials with respect to the weight $h$ on the real line with positive leading coefficients $\kappa_0,\kappa_1,\kappa_2,\ldots$, characterized by the orthogonality conditions
\begin{equation*}
\int_{\mathbb R}p_j(x)p_k(x)h(x)dx=\delta_{j,k},\qquad j,k\in\mathbb N\cup\{0\}.
\end{equation*} 
We define the Cauchy transform for $z\in \C\setminus\R$ by
\begin{equation*}
C_hg(z)=\frac{1}{2\pi i}\int_{\mathbb R}g(x)h(x)\frac{dx}{x-z},
\end{equation*}
and we let
\begin{equation}\label{def:Y}
Y(z)=Y_N(z;x_1,x_2;\gamma_1,\gamma_2;w)=\begin{pmatrix}\frac{1}{\kappa_N}p_N(z)& \frac{1}{\kappa_N}C_hp_N(z)\\
-2\pi i\kappa_{N-1}p_{N-1}(z)&-2\pi i\kappa_{N-1}C_hp_{N-1}(z)\end{pmatrix}.
\end{equation}
This $2\times 2$ matrix-valued function is analytic in $\mathbb C\setminus \mathbb R$ and it is the unique solution to the standard RH problem for orthogonal polynomials \cite{FokasItsKitaev}.
More precisely, $Y$ satisfies the following conditions.

\subsubsection*{RH problem for $Y$}
\begin{itemize}
\item[(a)] $Y:\mathbb C\setminus\mathbb R\to\mathbb C^{2\times 2}$ is analytic,
\item[(b)] $Y$ has continuous boundary values $Y_\pm(x)$ when approaching $x\in\mathbb R\setminus\{x_1,x_2\}$ from above ($+$) and below ($-$), and they are related by
\begin{equation*}
Y_+(x)=Y_-(x)\begin{pmatrix}1&h(x)\\
0&1\end{pmatrix},\qquad x\in\mathbb R\setminus\{x_1,x_2\},
\end{equation*}
\item[(c)] as $z\to\infty$, we have
\begin{equation*}
Y(z)=\left(I+\mathcal O(z^{-1})\right)z^{N\sigma_3},\qquad \sigma_3=\begin{pmatrix}1&0\\0&-1\end{pmatrix},
\end{equation*}
uniformly for $z\in\mathbb C\setminus\mathbb R$,
\item[(d)] as $z\to x_j$ for $j=1,2$, we have
\begin{equation*}
Y(z)=\mathcal O(\log(z-x_j)),
\end{equation*}
by which we mean that each entry of the matrix is of this order.
\end{itemize}
We note that since $h(x)>0$ for $x\in \mathbb R$, the orthogonal polynomials $p_j$ exist uniquely (under the condition that $\kappa_j>0$) for $j=0,1,2,\dots$. It follows that the RH problem for $Y$ has a solution for $N=1,2,\dots$.
We also note that the orthogonal polynomials $p_j$ and the above RH problem also make  sense if $x_1=x_2$.

The reason for introducing $Y$ is that on one hand we can express logarithmic derivatives of the Hankel determinants $D_N$ in terms of $Y$; on the other hand we can use the Deift/Zhou steepest descent method to obtain large $N$ asymptotics for $Y$ in various double scaling regimes.

Our next result contains differential identities with respect to the position of the singularities, which will be crucial for our asymptotic analysis as $x_2- x_1\to 0$ and as $x_1=x_2\to \pm 1$. Similar differential identities were used for instance in \cite{DeiftItsKrasovsky2, Charlier}.

\begin{proposition}\label{pr:di}
Let $\gamma_1,\gamma_2\in\mathbb R$, and set $w\equiv 0$. For $-1<x_1<x_2<1$, we have 
\begin{equation}\label{eq:diffid}
\frac{d}{dy}\log D_N(x_1,x_2=x_1+y;\gamma_1,\gamma_2;0)=-e^{-NV(x_2)}\frac{1-e^{\sqrt{2}\pi\gamma_2}}{2\pi i}\left(Y^{-1}(x_2)Y'(x_2)\right)_{2,1},
\end{equation}
where $Y(z)=Y_N(z;x_1,x_2;\gamma_1,\gamma_2;0)$ is as in \eqref{def:Y} and by the notation $(Y^{-1}(x_2)Y'(x_2))_{2,1}$, we mean $\lim_{z\to x_2}(Y^{-1}(z)Y'(z))_{2,1}$, where the limit is taken off of the jump contour, and does not depend on the side of the contour we approach the point from.
If $x_1=x_2=:x\in(-1,1)$, we have 
\begin{equation}\label{eq:diffidconfluent}
\frac{d}{dx}\log D_N(x,x;\gamma_1,\gamma_2;0)=-e^{-NV(x)}\frac{1-e^{\sqrt{2}\pi\gamma_1+\sqrt{2}\pi\gamma_2}}{2\pi i}\left(Y^{-1}(x)Y'(x)\right)_{2,1}.
\end{equation}
\end{proposition}

\begin{proof}
In this proof, we write $D_N(y):=D_N(x_1,x_2=x_1+y;\gamma_1,\gamma_2;0)$.
A well-known consequence of Heine's determinantal formula for orthogonal polynomials is that the Hankel determinant is given as
\begin{equation*}
D_N(y)=\prod_{j=0}^{N-1}\kappa_j^{-2},
\end{equation*}
where as above, $\kappa_j$ is the leading coefficient of $p_j$.
We note that both $D_N(y)$ and $\kappa_j$ are positive for each $j$. Moreover, they are smooth in $y$, since they are expressed, again via Heine's formula, in terms of sums and products of moments of the function $h$ and using \eqref{def:w}, one can check that these moments are smooth in $y$. We can thus take the logarithmic $y$-derivative of the above identity and then use the orthogonality conditions, which gives
\begin{equation*}
\partial_y\log D_N(y)=-2\sum_{j=0}^{N-1}\frac{\partial_y \kappa_j}{\kappa_j}=-\int_{\mathbb R}\partial_y\left(\sum_{j=0}^{N-1}p_j(x)^2\right)h(x)dx.
\end{equation*}
Again, using the definition \eqref{def:w} of the weight function $h$, we get
\[
\partial_y\log D_N(y)=-\partial_y \int_{\mathbb R}\left(\sum_{j=0}^{N-1}p_j(x)^2\right) h(x)dx-e^{-NV(x_2)}\left(1-e^{\sqrt{2}\pi\gamma_2}\right)\left(\sum_{j=0}^{N-1}p_j(x_2)^2\right)
.
\]
By orthonormality and the confluent form of the Christoffel-Darboux formula for orthogonal polynomials (confluent meaning that we evaluate the kernel on the diagonal), this becomes
\begin{align*}
\partial_y\log D_N(y)&=-e^{-NV(x_2)}\left(1-e^{\sqrt{2}\pi\gamma_2}\right)\left(\sum_{j=0}^{N-1}p_j(x_2)^2\right)\\
&=-\frac{\kappa_{N-1}}{\kappa_N}e^{-NV(x_2)}\left(1-e^{\sqrt{2}\pi\gamma_2}\right)\left(p_N'(x_2)p_{N-1}(x_2)-p_N(x_2)p_{N-1}'(x_2)\right).
\end{align*}
By \eqref{def:Y}, we get \eqref{eq:diffid} after a straightforward calculation in which we make use of the fact that $\det Y\equiv 1$.

To prove \eqref{eq:diffidconfluent}, it suffices to note that 
$D_N(x,x;\gamma_1,\gamma_2;0)=D_N(x_1,x;0,\gamma_1+\gamma_2;0)$ for any $x_1\in(-1,x)$  and to apply \eqref{eq:diffid} to this case.
\end{proof}

\medskip

We now consider a deformation of $w$, namely for $t\in[0,1]$, we define 
\begin{equation}\label{eq:Wtdef}
w(\lambda;t)=tw(\lambda),
\end{equation}
which interpolates between $w(\cdot;0)=0$ and $w(\cdot;1)=w$. Note that $w(\lambda;t)$ is bounded and analytic in the region $\mathcal S_N$, uniformly in $t\in[0,1]$, provided that $w(\lambda;1)=w(\lambda)$ is.

Let us write, for $t\in[0,1]$, $h_t$ for the weight function defined as in \eqref{def:w} with $w=w(\cdot;t)$ and similarly, let us write $Y(z;t)$ for the matrix defined as in \eqref{def:Y} with $h=h_t$. This of course satisfies for each $t\in[0,1]$ the corresponding RH problem. Again using Heine's identity and representing the orthogonal polynomials in terms of moments of $h_t$, we see that both $Y(z;t)$ and $\log D_N(x_1,x_2;\gamma_1,\gamma_2;w(\cdot;t))$ are smooth functions of $t$. Indeed, the point of introducing the deformation $w(\cdot;t)$, is that we can express $\partial_t \log D_N$ in terms of $Y(z;t)$ in a rather simple manner, which then allows calculating the asymptotics of $D_N$ from the asymptotics of $Y(z;t)$. The following result follows from a general differential identity proven in \cite[Section 3.2]{Charlier}, by noting that $\partial_t h_t(\lambda)=w(\lambda)h_t(\lambda)$.

\begin{proposition}\label{le:di1}
Let $\gamma_1,\gamma_2\in\mathbb R$, and let $-1<x_1\leq x_2<1$. If $w(\lambda;t)$ is given by \eqref{eq:Wtdef}, we have for $t\in(0,1)$,
$$
\partial_t \log D_N(x_1,x_2;\gamma_1,\gamma_2;w(\cdot;t))=\frac{1}{2\pi i }\int_\R \left[Y^{-1}(\lambda;t) Y'(\lambda;t)\right]_{21}w(\lambda)h_t(\lambda)d\lambda.
$$
\end{proposition}

Before turning our attention to a detailed study of the different double scaling limits, we perform two transformations of the RH problem for $Y$ which are standard in RH analysis \cite{Deift} and which we will need in both of the relevant double scaling regimes.

\subsection{The \texorpdfstring{$g$}{g}-function and a first transformation of the RH problem}\label{sec:g}

Our first transformation of the RH problem normalizes the behavior at infinity and transforms the jump matrices  in such a way that they become small except on $[-1,1]$.

Define the functions
\begin{align}
\label{def:g}&g(z)=\int\log(z-s)d\mu_V(s),& z\in\mathbb C\setminus (-\infty,1],\\
&\label{def:xi}\xi(z)=g(z)-\frac{1}{2}V(z)+\frac{1}{2}\ell,& z\in\mathbb C\setminus (-\infty,1],
\end{align}
where the principal value of the logarithm is taken, and where $\ell$ is given in terms of the Euler-Lagrange variational conditions \eqref{eq:el1} for the equilibrium measure $\mu_V$. These imply that there exists a constant $\ell$ such that
\begin{align}
&\label{eq:EL1}g_+(x)+g_-(x)-V(x)+\ell=0,& x\in[-1,1],\\
&\label{eq:EL2}g_+(x)+g_-(x)-V(x)+\ell<0,& x\in\mathbb R\setminus [-1,1],
\end{align}
if $V$ satisfies Assumptions \ref{ass:regular}.
We then have the identities
\begin{align}
&\label{eq:id g mu}g_+(x)-g_-(x)=\begin{cases}2\pi i \int_x^{1}\mu_V(dx),& x\in(-1,1),\\
2\pi i ,& x<-1,
\end{cases}
\\
&\label{eq:id xi g}\xi_\pm(x)=\frac{g_{\pm}(x)-g_{\mp}(x)}{2}, \qquad \qquad \quad  x\in(-1,1).\end{align}
It will be useful to note that by analytic continuation of \eqref{eq:id g mu} and \eqref{eq:id xi g} 
\begin{equation} \label{formulaxi}
\xi(z)=-\pi  \int_1^z \psi_V(w)(w^2-1)^{1/2}dw,
\end{equation}
where the branch of the root is such that the cut is on $[-1,1]$ and the root is positive on $(1,\infty)$, and $\psi_V(w)$ is the analytic continuation of $\psi_V$ to a neighbourhood of the real line which we assume contains $\mathcal S_N$ defined in \eqref{domain} without loss of generality. 
Define now for $z\in \C\setminus \R$
\begin{equation}\label{def:T}
T(z)=e^{\frac{N\ell}{2}\sigma_3}Y(z)e^{-Ng(z)\sigma_3}e^{-\frac{N\ell}{2}\sigma_3}.
\end{equation}
Using the RH problem for $Y$, it is straightforward to check that $T$ satisfies the following RH problem.

\subsubsection*{RH problem for $T$}
\begin{itemize}
\item[(a)] $T:\mathbb C\setminus\mathbb R\to\mathbb C^{2\times 2}$ is analytic,
\item[(b)] $T$ satisfies the jump relations 
\begin{align*}
&T_+(x)=T_-(x)\begin{pmatrix}1&e^{\sqrt{2}\pi\gamma_1+\sqrt{2}\pi\gamma_2} e^{w(x;t)}e^{2N\xi_+(x)}\\
0&1\end{pmatrix},& x<-1,\\
&T_+(x)=T_-(x)\begin{pmatrix}e^{-2N\xi_+(x)}&e^{\sqrt{2}\pi\gamma_1+\sqrt{2}\pi\gamma_2}e^{w(x;t)}\\
0&e^{2N\xi_+(x)}\end{pmatrix},& -1<x<x_1,\\
&T_+(x)=T_-(x)\begin{pmatrix}e^{-2N\xi_+(x)}&e^{\sqrt{2}\pi\gamma_2}e^{w(x;t)}\\
0&e^{2N\xi_+(x)}\end{pmatrix},& x_1<x<x_2,\\
&T_+(x)=T_-(x)\begin{pmatrix}e^{-2N\xi_+(x)}&e^{w(x;t)}\\
0&e^{2N\xi_+(x)}\end{pmatrix},& x_2<x<1,\\
&T_+(x)=T_-(x)\begin{pmatrix}1& e^{w(x;t)} e^{2N\xi(x)}\\
0&1\end{pmatrix},& x>1,
\end{align*}
\item[(c)] as $z\to\infty$, we have
\begin{equation*}
T(z)=I+\mathcal O(z^{-1}),
\end{equation*}
uniformly for $z\in\mathbb C\setminus\mathbb R$,
\item[(d)] as $z\to x_j$, $j=1,2$, we have
\begin{equation*}
T(z)=\mathcal O(\log(z-x_j)).
\end{equation*}
\end{itemize}

We note that the above RH problem makes sense also if $x_1=x_2$. In that case, the jump on $(x_1,x_2)$ can be ignored, and $T$ still blows up logarithmically near $x_1=x_2$.

By \eqref{def:xi} and \eqref{eq:EL2} we have for any  $x\in\mathbb R\setminus[-1,1]$ that $\mathrm{Re}(\xi_+(x))<0$. Combining this with the uniform boundedness of $w(\cdot,t)$ on $\mathbb R$, it follows that the jump matrix for $T$ is close to the identity matrix on $\mathbb R\setminus[-1,1]$ for large $N$.
For $x\in(-1,1)$, the diagonal entries of the jump matrix for $T$ oscillate rapidly for $N$ large. The standard remedy for this is to perform a transformation known as the opening of lenses.

\subsection{Opening of lenses}\label{sec:lens1}
We use the factorization
\begin{multline*}
\begin{pmatrix}e^{-2N\xi_+(x)}&e^{\sqrt{2}\pi\gamma} e^{w(x;t)}\\
0&e^{2N\xi_+(x)}\end{pmatrix}=J_1(x;\gamma)J_2(x;\gamma)J_3(x;\gamma):=\begin{pmatrix}1&0\\
e^{-\sqrt{2}\pi\gamma} e^{-w(x;t)}e^{-2N\xi_-(x)}&1\end{pmatrix}\\
\times\
\begin{pmatrix}0&e^{\sqrt{2}\pi\gamma} e^{w(x;t)}\\
-e^{-\sqrt{2}\pi\gamma} e^{-w(x;t)}&0\end{pmatrix}\begin{pmatrix}1&0\\
e^{-\sqrt{2}\pi\gamma} e^{-w(x;t)}e^{-2N\xi_+(x)}&1\end{pmatrix}
\end{multline*} valid for any $\gamma\in\mathbb R$ and
for $x\in (-1,1)$
to split the jump contour into lens-shaped curves and to transform the oscillatory jumps into exponentially decaying jumps.
To that end, we extend $J_1$ analytically to the part of the region $\mathcal S_N$ in the lower half plane (note that we know that $V$  and hence $\xi$ is also analytic in this neighbourhood for large enough $N$), and similarly for $J_3$ in the upper half plane:
\begin{align*}
&J_1(z;\gamma)=\begin{pmatrix}1&0\\
e^{-\sqrt{2}\pi\gamma} e^{-w(z;t)}e^{-2N\xi(z)}&1\end{pmatrix},&z\in\mathbb C^{-}\cap \mathcal S_N,\\
&J_3(z;s)=\begin{pmatrix}1&0\\
e^{-\sqrt{2}\pi\gamma} e^{-w(z;t)}e^{-2N\xi(z)}&1\end{pmatrix},&z\in\mathbb C^{+}\cap \mathcal S_N,
\end{align*}
and we define regions $\Omega$, $\overline{\Omega}$ which are both unions of several half-lens shaped regions as shown in Figure \ref{fig:lens2} and in Figure \ref{fig:lens}.
We will specify the precise shape of the lenses later on in each of the asymptotic regimes we consider. For now, we simply point out that for $w,t=0$ they need to be contained in $\mathcal V$ (recall from Assumptions \ref{ass:regular} that we defined this set to be a strip in which $V$ is analytic), while for $w,t\neq0$ they need to be contained in a region $\mathcal S_N'$ which is of  a similar  form as $\mathcal S_N$ but slightly smaller and bounded, namely

\begin{align}\label{domain2}
 \mathcal S_N'&=\bigg(\left\{z\in\mathbb C:|\Re z|\leq 1-\tfrac{3}{2}\delta_N, |\Im z|<\tfrac{3\epsilon_N}{8}\right\}\\
 &\qquad \notag \bigcup\left\{z\in\mathbb C:1-\tfrac{3}{2}\delta_N\leq |\Re z|\leq 1+\tfrac{3}{2}\delta_N,|\Im z|<\tfrac{3}{2}\delta_N\right\}\bigg).
\end{align}
 See Figure \ref{fig:SSprime} for an illustration of the sets $\mathcal S_N$ and $\mathcal S_N'$.
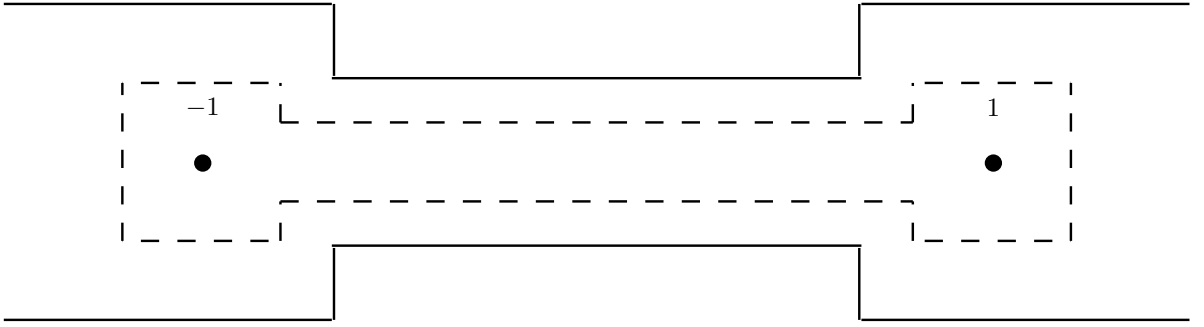
\begin{figure}
\begin{tikzpicture}[scale=2.3]
    \begin{axis}[ticks=none,axis line style={draw=none},
        unit vector ratio*=1 1 1,       
        xmin=-1.5,xmax=1.5,
        ymin=-.41,ymax=.41]
        \addplot [black,thin,domain=-147:-90] ({1+0.4*cos(147)}, {0.4*sin(x)});
        \addplot [black,thin,domain=90:147] ({1+0.4*cos(147)}, {0.4*sin(x)});
        \addplot [black,thin,domain=33:90] ({-1+0.4*cos(33)}, {0.4*sin(x)});
        \addplot [black,thin,domain=270:327] ({-1+0.4*cos(33)}, {0.4*sin(x)});

        \addplot [black,thin,domain=0.67:1.5]({x},{0.4});
        \addplot [black,thin,domain=0.67:1.5]({x},{-0.4});    
        \addplot [black,thin,domain=0.67:1.5]({-x},{0.4});
        \addplot [black,thin,domain=0.67:1.5]({-x},{-0.4});

        \addplot [black,dashed,domain=0.83:1.2]({x},{0.2});
        \addplot [black,dashed,domain=0.83:1.2]({x},{-0.2});    
        \addplot [black,dashed,domain=0.83:1.2]({-x},{0.2});
        \addplot [black,dashed,domain=0.83:1.2]({-x},{-0.2});    

        
        \addplot [black,dashed,domain=-0.2:-0.1]({-0.8},{x});
        \addplot [black,dashed,domain=0.1:0.2]({-0.8},{x});

        \addplot [black,dashed,domain=-0.2:0.2]({-1.2},{x});
        \addplot [black,dashed,domain=-0.2:0.2]({1.2},{x});

        \addplot [black,dashed,domain=-0.2:-0.1]({0.8},{x});
        \addplot [black,dashed,domain=0.1:0.2]({0.8},{x});


        \addplot [black,thin,domain=-0.67:0.67]({x},{0.212});
        \addplot [black,thin,domain=-0.67:0.67]({x},{-0.212});    

        \addplot [black,dashed,domain=-0.8:0.8]({x},{0.1});
        \addplot [black,dashed,domain=-0.8:0.8]({x},{-0.1});

    \end{axis}

\fill [color=black] (5.72,0.93) circle (0.05);
\node at (5.72,1.25) {$1$};
\fill [color=black] (1.15,0.93) circle (0.05);
\node at (1.15,1.25) { $-1$};
\end{tikzpicture}
\caption{The sets $\mathcal S_N$ (boundary drawn solid) and $\mathcal S_N'$ (boundary drawn dashed).}\label{fig:SSprime}
\end{figure}

If $x_1$ and $x_2$ converge slowly to each other as in Section \ref{section: RHseparated}, we will choose a lens configuration as in Figure \ref{fig:lens2}; if $x_1$ and $x_2$ approach each other at a fast rate as in Section \ref{sec:merging} or if they are equal (and merging to $\pm 1$) as in Section \ref{section: RHedge}, we will choose the configuration from Figure \ref{fig:lens}.
We note also that in the case where $\gamma_1=\gamma_2=0$, there is no need to close the lenses at $x_1$ and $x_2$, and then we could also work with a single lens connecting $-1$ with $1$. This would simplify the analysis and hence the proof of Theorem \ref{th:Hankel1} for $\gamma_1=\gamma_2=0$.
Recall that the proof of Theorem \ref{th:gmc2} relies only on this special case.

The total lens is delimited by the curves $\omega_{1}, \overline{\omega}_1$, $\omega_{2}, \overline{\omega}_2$, and $\omega_{3}, \overline{\omega}_3$, where $\omega_2,\overline{\omega}_2$ can possibly coincide with the interval $(x_1,x_2)$. We orient all the curves from left to right, and say that the top side of the contour is the $+$-side and the bottom side of it the $-$-side.

Define
\begin{equation}\label{def:S}
S(z)=\begin{cases}T(z)J_3(z;\gamma)^{-1},& z\in\Omega,\\
T(z)J_1(z;\gamma),& z\in\overline{\Omega},\\
T(z),& \mbox{elsewhere,}
\end{cases}
\end{equation}
where the value of $\gamma$ is taken equal to $\gamma_1+\gamma_2$ for $\Re z<x_1$, $\gamma_2$ for $x_1<\Re z<x_2$ and equal to $0$ for $\Re z>x_2$.

\begin{figure}
\begin{tikzpicture}[scale=2.3]
    \begin{axis}[ticks=none,axis line style={draw=none},
        unit vector ratio*=1 1 1,       
        xmin=-1.5,xmax=1.5,
        ymin=-.41,ymax=.41]
        \addplot [black,dashed,domain=-147:-90] ({1+0.4*cos(147)}, {0.4*sin(x)});
        \addplot [black,dashed,domain=90:147] ({1+0.4*cos(147)}, {0.4*sin(x)});
        \addplot [black,dashed,domain=33:90] ({-1+0.4*cos(33)}, {0.4*sin(x)});
        \addplot [black,dashed,domain=270:327] ({-1+0.4*cos(33)}, {0.4*sin(x)});

        \addplot [black,dashed,domain=0.67:1.33]({x},{0.4});
        \addplot [black,dashed,domain=0.67:1.33]({x},{-0.4});    
        \addplot [black,dashed,domain=0.67:1.33]({-x},{0.4});
        \addplot [black,dashed,domain=0.67:1.33]({-x},{-0.4});    

        \addplot [black,dashed,domain=-0.4:0.4]({1.33},{x});
        \addplot [black,dashed,domain=-0.4:0.4]({-1.33},{x});
        
        \addplot [black,dashed,domain=-0.67:0.67]({x},{0.212});
        \addplot [black,dashed,domain=-0.67:0.67]({x},{-0.212});    
        \addplot [black,thick,domain=-1.5:1.5]({x},{0});    
        \addplot [black,thick,domain=-1:-0.33]({x},{-0.9*x^2-1.2*x-0.3});
        \addplot [black,thick,domain=-1:-0.33]({x},{0.9*x^2+1.2*x+0.3});
        \addplot [black,thick,domain=-0.33:0.33]({x},{-0.9*x^2+0.1});
        \addplot [black,thick,domain=-0.33:0.33]({x},{0.9*x^2-0.1});
        \addplot [black,thick,domain=0.33:1]({x},{-0.9*x^2+1.2*x-0.3});
        \addplot [black,thick,domain=0.33:1]({x},{0.9*x^2-1.2*x+0.3});
    \end{axis}
\fill [color=black] (5.72,0.93) circle (0.05);
\node at (5.72,1.25) {$1$};
\fill [color=black] (1.15,0.93) circle (0.05);
\node at (1.15,1.25) { $-1$};
\fill [color=black] (2.67,0.93) circle (0.05);
\node at (2.67,1.25) {$x_1$};
\fill [color=black] (4.19,0.93) circle (0.05);
\node at (4.19,1.25) {$x_2$};

\node at (1.91,1.05) {$\Omega$};
\node at (1.91,0.82) {$\overline{\Omega}$};
\node at (1.91,1.25) {$\omega_1$};
\node at (1.91,0.55) {$\overline{\omega}_1$};

\node at (3.43,1.05) {$\Omega$};
\node at (3.43,0.82) {$\overline{\Omega}$};
\node at (3.43,1.25) {$\omega_2$};
\node at (3.43,0.55) {$\overline{\omega}_2$};

\node at (4.95,1.05) {$\Omega$};
\node at (4.95,0.82) {$\overline{\Omega}$};
\node at (4.95,1.25) {$\omega_3$};
\node at (4.95,0.55) {$\overline{\omega}_3$};

\end{tikzpicture}
\caption{A characterization of the jump contour (solid) contained in the set $\mathcal S_N'$ (dashed) in the separated regime.}\label{fig:lens2}
\end{figure}
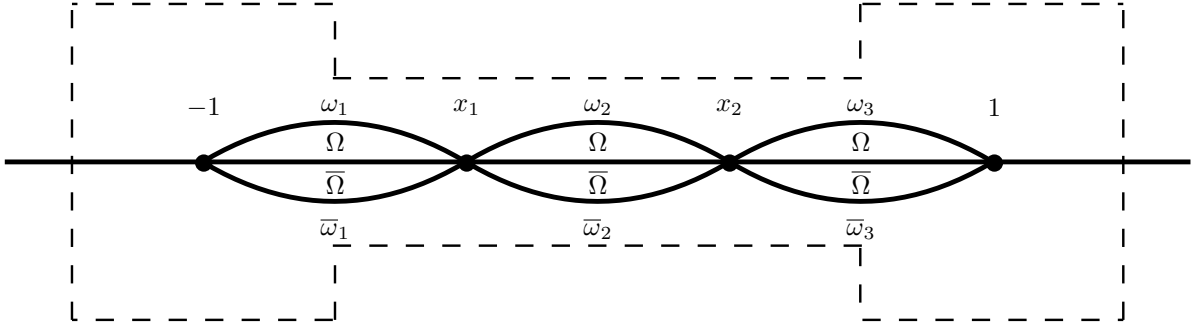
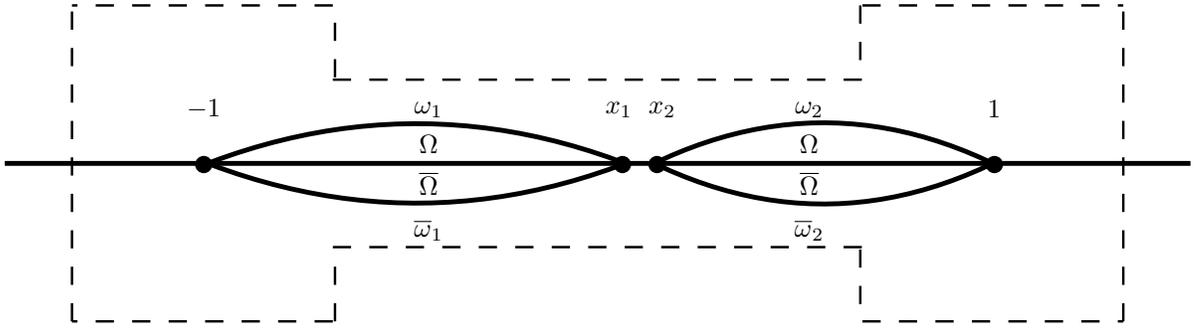
\begin{figure}
\begin{tikzpicture}[scale=2.3]
    \begin{axis}[ticks=none,axis line style={draw=none},
        unit vector ratio*=1 1 1,       
        xmin=-1.5,xmax=1.5,
        ymin=-.41,ymax=.41]
        \addplot [black,dashed,domain=-147:-90] ({1+0.4*cos(147)}, {0.4*sin(x)});
        \addplot [black,dashed,domain=90:147] ({1+0.4*cos(147)}, {0.4*sin(x)});
        \addplot [black,dashed,domain=33:90] ({-1+0.4*cos(33)}, {0.4*sin(x)});
        \addplot [black,dashed,domain=270:327] ({-1+0.4*cos(33)}, {0.4*sin(x)});

        \addplot [black,dashed,domain=0.67:1.33]({x},{0.4});
        \addplot [black,dashed,domain=0.67:1.33]({x},{-0.4});    
        \addplot [black,dashed,domain=0.67:1.33]({-x},{0.4});
        \addplot [black,dashed,domain=0.67:1.33]({-x},{-0.4});    

        \addplot [black,dashed,domain=-0.4:0.4]({1.33},{x});
        \addplot [black,dashed,domain=-0.4:0.4]({-1.33},{x});
        \addplot [black,dashed,domain=-0.67:0.67]({x},{0.212});
        \addplot [black,dashed,domain=-0.67:0.67]({x},{-0.212});    
        \addplot [black,thick,domain=-1:0.07]({x},{-0.36*x^2-0.33*x+0.025});    
        \addplot [black,thick,domain=0.15:1]({x},{-0.56*x^2+0.64*x-0.08});    
        \addplot [black,thick,domain=-1:0.07]({x},{0.36*x^2+0.33*x-0.025});    
        \addplot [black,thick,domain=0.15:1]({x},{0.56*x^2-0.64*x+0.08});    
        \addplot [black,thick,domain=-1.5:1.5]({x},{0});    

    \end{axis}
\fill [color=black] (3.77,0.93) circle (0.05);
\node at (3.8,1.25) {$x_2$};
\fill [color=black] (3.57,0.93) circle (0.05);
\node at (3.55,1.25) { $x_1$};
\fill [color=black] (5.72,0.93) circle (0.05);
\node at (5.72,1.25) {$1$};
\fill [color=black] (1.15,0.93) circle (0.05);
\node at (1.15,1.25) { $-1$};
\node at (2.45,1.05) {$\Omega$};
\node at (2.45,0.82) {$\overline{\Omega}$};
\node at (2.45,1.25) {$\omega_1$};
\node at (2.45,0.55) {$\overline{\omega}_1$};

\node at (4.65,1.05) {$\Omega$};
\node at (4.65,0.82) {$\overline{\Omega}$};
\node at (4.65,1.25) {$\omega_2$};
\node at (4.65,0.55) {$\overline{\omega}_2$};

\end{tikzpicture}
\caption{A characterization of the jump contour (solid) contained in the set $\mathcal S_N'$ (dashed) in the merging regime.}\label{fig:lens}
\end{figure}

If the contour setting is as in Figure \ref{fig:lens2}, the RH problem for $S$ now becomes the following.

\subsubsection*{RH problem for $S$}
\begin{itemize}
\item[(a)] $S:\mathbb C\setminus\left(\mathbb R\cup \omega_{1}\cup\overline{\omega}_{1}\cup \omega_{2}\cup\overline{\omega}_{2}\cup\omega_3\cup\overline{\omega}_3\right)\to\mathbb C^{2\times 2}$ is analytic,
\item[(b)]{   $S$ satisfies the jump relations
\begin{align*}
&S_+(x)=S_-(x)\begin{pmatrix}1&e^{{\sqrt{2}\pi}\gamma_1+{\sqrt{2}\pi}\gamma_2} e^{w(x;t)}e^{2N\xi_+(x)}\\
0&1\end{pmatrix},& x<-1,\\
&S_+(z)=S_-(z)J_1\left(z;\gamma\right),& z\in\overline{\omega}_{j},\,\, j=1,2,3,\\
&S_+(x)=S_-(x)J_2\left(x;\gamma\right),& -1<x<1,\\
&S_+(z)=S_-(z)J_3\left(z;\gamma \right),& z\in{\omega_{j}},\,\, j=1,2,3,\\
&S_+(x)=S_-(x)\begin{pmatrix}1& e^{w(x;t)} e^{2N\xi(x)}\\
0&1\end{pmatrix},& x>1,
\end{align*}
where the value of $\gamma$ is taken equal to $\gamma_1+\gamma_2$ for $\Re z<x_1$, $\gamma_2$ for $x_1<\Re z<x_2$ and equal to $0$ for $\Re z>x_2$.}
\item[(c)] as $z\to\infty$, we have
\begin{equation*}
S(z)=I+\mathcal O(z^{-1}),
\end{equation*}
uniformly for $z\in\mathbb C\setminus\mathbb R$,
\item[(d)] as $z\to x_j$, $j=1,2$, we have
\begin{equation*}
S(z)=\mathcal O(\log(z-x_j)).
\end{equation*} This is true both if $x_j$ is approached from inside the lens-shaped region and from outside.
Moreover, $S$ is bounded near $\pm 1$.
\end{itemize}
If the lenses are chosen as in Figure \ref{fig:lens}, the RH problem is the same, except for the fact that the jump conditions on $\omega_2, (x_1,x_2), \overline{\omega}_2$ are replaced by the single jump relation
\begin{equation}\label{ed:jumpSmod}
S_+(x)=S_-(x)\begin{pmatrix}e^{-2N\xi_+(x)}&e^{{\sqrt{2}\pi}\gamma_2} e^{w(x;t)}\\
0&e^{2N\xi_+(x)}\end{pmatrix},\qquad x\in(x_1,x_2).
\end{equation}
Like for the RH problem for $T$, the RH problem for $S$ also makes sense if $x_1=x_2$, in which case we simply ignore the jump matrix on $(x_1,x_2)$ and on the corresponding parts of the lenses.

We now prove a result that implies that the jump matrices of $S$ are close to the identity apart from on $(-1,1)$ and small neighborhoods of $\pm 1$ and $x_1,x_2$.

\begin{lemma}\label{pr:lensbound1}
Suppose that $w(z)$ is analytic and uniformly bounded for $z\in\mathcal S_N$. There exists a constant $\kappa>0$ such that
\begin{equation*}
\Re (w(z;t)+2N\xi(z))\geq N^{\kappa}
\end{equation*}
for $N$ sufficiently large and for all $t\in[0,1]$, for all $z\in\{\zeta\in\mathcal S_N': |\Re(\zeta)|<1, |\zeta\pm 1|\geq \delta_N/2, |\Im \zeta|\geq \frac{1}{16}\epsilon_N\}$ with  $\epsilon_N$ and $\delta_N$ as in \eqref{domain}.

Moreover, for $N$ sufficiently large, we have
\begin{equation*}
\Re(w(z;t)+2N\xi(z))<-N^{1/4}\log (|z|+1).
\end{equation*}
for  $|\Im z|<\epsilon_N$, $z\notin \mathbb R$, with $|z|\geq 1+\delta_N/2$, with $\delta_N$ as in \eqref{domain}. \end{lemma}

\begin{proof}
First of all, we recall that $w(\cdot;t)$  is uniformly bounded for $z\in\mathcal S_N$. It is thus sufficient to prove the inequalities in the case where we set $w(\cdot;t)=0$. To do this, we note that for $x\in (-1,1)$,
\[\partial_x\left.\Im\xi(x+iy)\right|_{y\to 0_+}=-\pi \sigma_V(x)<0,\] 
where $\sigma_V(x)=\psi_V(x)\sqrt{1-x^2}$, and that $\xi$ has an analytic continuation from the upper half plane to a full neighbourhood of $x$. Hence, we can apply the Cauchy-Riemann conditions to conclude that $\partial_y\left.\Re\xi(x+iy)\right|_{y\to 0_+}>0$. By uniform continuity, for any $\delta>0$, there exist $c,\eta>0$ such that 
$\partial_y\Re\xi(x+iy)>c$ for any $0<y<\eta$ and $x\in(-1+\delta, 1-\delta)$.
This allows us to conclude that there is a $c>0$ such that for $|\Re z|<1-\delta$, $0<\Im z<\eta$, 
$$
\Re (\xi(z))\geq c\,|\Im z|
$$
and a similar argument allows to conclude the same for $-\eta<\Im z<0$.
This yields the first claim for $|\Re z|<1-\delta$ for any fixed $\delta>0$. 

For $1-\delta\leq \Re z<1$ with $\delta>0$ sufficiently small, it follows from \eqref{eq:id g mu}, \eqref{eq:id xi g}, and Assumptions \ref{ass:regular} that 
$\xi(z)$ can be written as 
$\xi(z)=(z-1)^{3/2}f(z)$ where principal branches are taken with $f$ analytic in a neighborhood of $1$ and { $f(1)=-\frac{2\sqrt{2}\pi \psi_V(1)}{3}<0$}, and then \[\arg \xi(z)=\frac{3}{2}\arg(z-1)+\arg f(z),\qquad |\xi(z)|=|z-1|^{3/2}|f(z)|>N^{-1+\epsilon}\] 
for $z$ sufficiently close to $1$ and for some $\epsilon>0$, from which the claim follows easily assuming $\pi/3+\Delta<|\arg(z-1)|<\pi-\Delta$ for some fixed $\Delta>0$. 

 If $z=x+i\eta$, $\eta>0$, $1-\delta<x<1$,  then by \eqref{formulaxi}
\begin{equation}\begin{aligned}
\Re \xi(z)&=\Re\left(-\pi  \int_x^{x+i\eta}\psi_V(w)\sqrt{w^2-1}dw  \right)\\
&=\pi \eta \sqrt{1-x^2}\psi_V(x)(1+\mathcal O(\eta)+\mathcal O(\eta/(1-x)))
\end{aligned}
\end{equation}
as $\eta \to 0$, $\eta/(1-x)\to 0$. In particular, for $z$ satisfying $|z-1|\geq \delta_N/2$ and $|\Im z|\geq \epsilon_N/16$, we have for any fixed $\delta>0$ and for $0<1-\Re(z)<\delta$ the existence of a constant $C_\delta>0$ such that $N\Re \xi(z)\geq C_\delta N \delta_N^{1/2}\epsilon_N$.  A similar statement can be made for $\eta<0$  as $\eta \to 0$.  The case $-1<\Re z<-1+\delta$ is similar and recalling the definition of $\delta_N$ and $\epsilon_N$ above \eqref{domain}  concludes the proof of the first claim. 

For the second claim of the lemma, first recall from \eqref{eq:EL2} that on $\R\setminus[-1,1]$
$$
N\mathrm{Re}(\xi(x))=\frac{N}{2}\left(g_+(x)+g_-(x)-V(x)-\ell\right)<0.
$$
Moreover, combining the local behavior of $\xi$ near $1$ (that is $\xi(z)=(z-1)^{3/2}f(z)$ with $f$ analytic in a neighborhood of $1$ and $f(1)<0$) and its behavior at infinity (namely from \eqref{eq:Vgrowth} $\lim_{x\to \infty}\frac{V(x)}{\log x}=\infty$ while $g_\pm(x)=\mathcal O(\log x)$ as $x\to \infty$), one readily sees that for $x\geq 1+\delta_N/2$, \eqref{eq:EL2} can be upgraded to 
\[
N\mathrm{Re}(\xi(x))< -N^{1/4}\log (1+x).
\]
A similar argument works also for $x<-1-\delta_N/2$.

Using Assumptions \ref{ass:regular} to generalize this bound to $|z|\geq 1+\frac{\delta_N}{2}$ satisfying $|\Im z|<\epsilon_N$ and $\Re z\neq 0$ is a routine argument: for large $|z|$, the asymptotics are determined by $V$ through \eqref{eq:Vgrowth}; for $z$ close to $1$, we see that  $|\arg z|$ is small and the local behavior of $\xi$ near $1$ again yields the claim ($z$ close to $-1$ being similar); for intermediate values of $|z|$, the claim follows from a simple argument utilizing the uniform continuity of $V$ combined with the asymptotics on $\R$ derived above.  
\end{proof}

We now turn to the next step in the RH analysis, namely the construction of explicit approximations to $S$, called parametrices, in different regions of the complex plane.
The parametrices in small neighborhoods of the points $-1,x_1,x_2, 1$ will depend on the asymptotic regime and will be constructed later. Away from those points, we construct a global parametrix, which solves a RH problem which is obtained from the RH problem for $S$ after ignoring exponentially small jumps and small neighbourhoods of $-1,x_1,x_2, 1$. The fact that the parametrices are good approximations to $S$ will be shown at the end of the RH analysis in each of the different cases.

\subsection{Global parametrix}\label{sec:global} For $z$ not too close to $-1,x_1,x_2,1$ and $N$ large, we expect that $S$ will be well approximated by the solution $P^\infty$ to the following RH problem. 

\subsubsection*{RH problem for $P^\infty$}
\begin{itemize}
\item[(a)] $P^\infty:\mathbb C\setminus [-1,1]\to\mathbb C^{2\times 2}$ is analytic,
\item[(b)] $P^\infty$ satisfies the jump relations
\begin{align*}
&P^\infty_+(x)=P^\infty_-(x)J_2(x;\gamma_1+\gamma_2),& -1<x<x_1,\\
&P^\infty_+(x)=P^\infty_-(x)J_2(x;\gamma_2),& x_1<x<x_2,\\
&P^\infty_+(x)=P^\infty_-(x)J_2(x;0),& x_2<x<1,
\end{align*}
\item[(c)] as $z\to\infty$, we have
\begin{equation*}
P^\infty(z)=I+\mathcal O(z^{-1}),
\end{equation*}
\item[(d)] as $z\to x_j$ for $j\in \lbrace 1,2\rbrace$, 
$$
P^\infty(z)=\mathcal{O}(\log(z-x_j)),
$$
\item[(e)] as $z\to \pm 1$,  we have
\begin{equation*}
P^\infty(z)=\mathcal O(|z\mp 1|^{-1/4}).
\end{equation*}
\end{itemize}

It may appear strange at first sight that we allow $P^\infty$ to  be unbounded near $\pm 1$, since $S$ is bounded near those points. However, the RH problem for $P^\infty$ would not be solvable if we imposed boundedness near $\pm 1$.

The unique solution to this problem was constructed in \cite[Section 4.4]{Charlier}. While there are minor differences in our notation, it is a straightforward exercise to check that in our setting, the solution takes the form
\begin{equation}
\label{def:Pinfty}
P^{\infty}(z)=D_\infty^{\sigma_3}Q(z)D(z)^{-\sigma_3},
\end{equation}
where $Q$ is given by
\begin{equation}\label{eq:QQdef}Q(z)=\begin{pmatrix}\frac{1}{2}(a(z)+a(z)^{-1})&-\frac{1}{2i}(a(z)-a(z)^{-1})\\\frac{1}{2i}(a(z)-a(z)^{-1})&\frac{1}{2}(a(z)+a(z)^{-1})\end{pmatrix},\qquad a(z)=\left(\frac{z+1}{z-1}\right)^{1/4},\end{equation}
with $a(z)$ analytic off $[-1,1]$ and tending to $1$ as $z\to\infty$,
and with $D_\infty$ and $D(z)$ given by
\begin{align}\label{def:D}D(z)&=D_{w}(z)D_\gamma(z)\\
D_{w}(z)&=\exp\left(\frac{\sqrt{z^2-1}}{2\pi}\int_{-1}^1\frac{w(x;t)}{\sqrt{1-x^2}}\frac{dx}{z-x}\right),\\
D_\gamma(z)&=\exp\left(\frac{\sqrt{z^2-1}}{{\sqrt{2}}}\sum_{j=1}^2\gamma_j\int_{-1}^{x_j}\frac{1}{\sqrt{1-x^2}}\frac{dx}{z-x}\right)\\
D_\infty&=\lim_{z\to\infty}D(z)\nonumber.
\end{align}
The branch of $\sqrt{z^2-1}$ analytic off $[-1,1]$ and behaving like $z$ at infinity is chosen.
We can also write $D_{\gamma}$  as 
\begin{equation}
\label{expressionD}
{ D_\gamma (z)}=\prod_{j=1}^2 e^{\frac{\gamma_j{\pi}}{{\sqrt{2}}}}\left(\frac{zx_j-1+e^{-\frac{\pi i}{2}}\sqrt{(z^2-1)(1-x_j^2)}}{z-x_j}\right)^{\frac{\gamma_j}{{\sqrt{2}} i}},
\end{equation}
see e.g. \cite[Section 4.4]{Charlier}. 
Here, $\sqrt{(z^2-1)(1-x_j^2)}$ is analytic off $[-1,1]$ and positive for $z>1$, and the power $(.)^{\frac{\gamma_j}{\sqrt{2} i}}$ is interpreted as $|.|^{\frac{\gamma_j}{\sqrt{2} i}}e^{\frac{\gamma_j}{\sqrt{2}}\arg(.)}$, with the argument in $(-\pi,\pi)$.

Later on, we will need the following technical estimate.
\begin{lemma}\label{lemma:boundw}
Let $w:[-1,1]\to\mathbb R$ be a bounded and integrable function, and denote
\[\|w\|=\sup_{x\in[-1,1]}|w(x)|.\]
Then, there exists a universal constant $M$ such that
\[\left|\Re\left(\frac{\sqrt{z^2-1}}{2\pi}\int_{-1}^1\frac{w(\lambda)}{z-\lambda}\frac{d\lambda}{\sqrt{1-\lambda^2}}\right)\right|\leq M\|w\|\]
for all $z\in\mathbb C\setminus[-1,1]$, where $\sqrt{z^2-1}$ is chosen analytic for $z\in\mathbb C\setminus[-1,1]$ and positive for $z>1$.
\end{lemma}
\begin{proof} 
The proof is divided into two steps.  First, we establish the bound in the region $ |\Re z|+ |\Im z| \ge 1$. Then, we obtain it in the region $ |\Re z| + |\Im z| \le 1$.  

{\bf Step 1.} 
We first assume that $|\Re(z)|+|\Im(z)|>1$. Then if $z$ is bounded away from $\{-1,1\}$, the bound is clear. Now assume that $|z-1|=\delta<1/3$. Then, 
\begin{equation}\begin{aligned}
\left|  \frac{\sqrt{z^2-1}}{2\pi} \int_{-1}^1 \frac{w(t)}{z-t} \frac{dt}{\sqrt{1-t^2}}  \right|
&\leq 
  \frac{\sqrt{3\delta}\|w\|}{2\pi} \left(\int_{-1}^0+\int_0^{1-2\delta}+\int_{1-2\delta}^1 \right) \frac{dt}{|z-t|\sqrt{1-t^2}}  
\\& \leq \frac{ \sqrt{3\delta}\| w\|}{2\pi}\left(3+
\int_{0}^{1-2\delta}\frac{dt}{(1-t-\delta)^{3/2}}+
\frac{2}{\delta}\int_{1-2\delta}^1\frac{dt}{\sqrt{1-t}}
 \right)\\&<c_1\|w\|,
 \end{aligned}
\end{equation}
for some universal constant $c_1>0$. A similar argument can be made for $|z+1|<1/3$, and the result follows for $|\Re(z)|+|\Im(z)|>1$

{\bf Step 2.}
Let $z= x+i \eta$ with $0<\eta<1$ and $x\in[-1+\eta,1-\eta]$, so that $|\Re(z)|+|\Im(z)|<1$. If $\eta$ is bounded away from $0$, the result is clear, and we take the limit $\eta \to 0$. Then, since $|z-t| \ge \eta $ for all $t\in\S$ and $w$ is uniformly bounded, one has 
\begin{multline}\label{eqlemma}
\sqrt{z^2-1} \int_{\S} \frac{w(t)}{z-t} \frac{dt}{\sqrt{1-t^2}} 
= i \sqrt{1-x^2}\big(1+\O\left(\frac{\eta}{1-x^2}\right)\big) \int_{\S} \frac{w(t)}{z-t} \frac{dt}{\sqrt{1-t^2}} \\
= i \sqrt{1-x^2} \int_{\S} \frac{w(t)}{z-t} \frac{dt}{\sqrt{1-t^2}}   + \O\left(\frac{\eta \|w\|}{\sqrt{1-x^2}}\int_{\S}\frac{dt}{|z-t|\sqrt{1-t^2}}\right),
\end{multline}
as $\eta\to 0$. 
We now prove that the second term is actually $\O(\|w\|)$. We can restrict ourselves to the case where $x\geq 0$ by symmetry. Then,
\begin{align*}
&\frac{\eta \|w\|}{\sqrt{1-x^2}}\int_{\S}\frac{dt}{|z-t|\sqrt{1-t^2}}\leq
\frac{2\eta \|w\|}{\sqrt{1-x}}\int_{0}^1\frac{dt}{\sqrt{(x-t)^2+\eta^2}\sqrt{1-t}}\\
&\qquad \leq 2\frac{\eta}{1-x} \|w\|\int_{0}^{(1-x)^{-1}}\frac{du}{\sqrt{u}\sqrt{(u-1)^2+\left(\frac{\eta}{1-x}\right)^2}}\\
&\qquad \leq 2\left(\frac{\eta}{1-x}\right)^{3/4} \|w\|\int_{0}^{\infty}\frac{du}{\sqrt{u}|u-1|^{3/4}}
\leq c_1\|w\|,
\end{align*}
for some universal constant $c_1$,
where we applied the substitution $1-t=(1-x) u$ in the second line, and the inequalities $(u-1)^2+\left(\frac{\eta}{1-x}\right)^2\geq (u-1)^2$ and  $(u-1)^2+\left(\frac{\eta}{1-x}\right)^2\geq \left(\frac{\eta}{1-x}\right)^2$ in the last line.

For the first term in \eqref{eqlemma}, since $w$ is real,  
\begin{equation} \label{bound6}
\Re\left\{ i \sqrt{1-x^2} \int_{\S} \frac{w(t)}{z-t} \frac{dt}{\sqrt{1-t^2}} \right\}
= \eta \sqrt{1-x^2} \int_{\S} \frac{w(t)}{|z-t|^2} \frac{dt}{\sqrt{1-t^2}}. 
\end{equation}
Without loss of generality we assume again that $x>0$, and we split the region of integration into two regions $\left(-1,\frac{x+1}{2}\right)$ and $\left(\frac{x+1}{2},1\right)$, which we consider separately. In the first region we have 
\begin{equation}
\left|\int_{-1}^{\frac{x+1}{2}}\frac{w(t)}{|z-t|^2} \frac{dt}{\sqrt{1-t^2}}\right|\leq 2\pi\|w\|+\frac{{ 2}\|w\|}{\sqrt{1-x}}\int_{-1/2}^{\frac{x+1}{2}}\frac{dt}{(t-x)^2+\eta^2}\end{equation}
Taking the change of variables $t=u\eta +x$, it follows that 
\begin{equation} \int_{-1/2}^{\frac{x+1}{2}}\frac{dt}{(t-x)^2+\eta^2}\leq \frac{1}{\eta}\int_{-\infty}^{\infty}\frac{du}{u^2+1}, 
\end{equation} 
and thus
\begin{equation}\label{ineq1} 
\left|\int_{-1}^{\frac{x+1}{2}}\frac{w(t)}{|z-t|^2} \frac{dt}{\sqrt{1-t^2}}\right|\leq  \frac{c_2\|w\|}{\eta\sqrt{1-x}},
\end{equation}
for some universal constant $c_2$.
In the second region we have
\begin{equation}\label{ineq2}
\left|\int_{\frac{x+1}{2}}^1\frac{w(t)}{|z-t|^2} \frac{dt}{\sqrt{1-t^2}}\right|\leq \frac{4\|w\|}{(1-x)^2}
\int_{\frac{x+1}{2}}^1\frac{dt}{\sqrt{1-t}}\leq \frac{8\|w\|}{(1-x)^{3/2}}.
\end{equation}
Then our result follows upon substituting \eqref{ineq1}-\eqref{ineq2} into \eqref{bound6} and \eqref{eqlemma}, since $\eta<1-x$. 
\end{proof}

\section{Model RH problems}\label{sec:model}

In this section, we study properties of two model RH problems that we will need later on.

\subsection{The Painlev\'e V model RH problem}

We consider a RH problem which was studied in detail in \cite[Section 3.1]{CK}, and which we will need to build a local parametrix when two singularities approach each other. The general version of the problem depends on $5$ complex parameters $\alpha_1,\alpha_2,\beta_1,\beta_2, s$, but for our purposes only the special case where $\alpha_1=\alpha_2=0$ and $\beta_1,\beta_2$ purely imaginary such that
$-{\sqrt{2}} i\beta_1=\gamma_1$, $-{\sqrt{2}} i\beta_2=\gamma_2$, $s\in -i\mathbb R_+:=(-i\infty,0)$ will be relevant. We will often find it convenient to suppress the dependence of $\Psi$ on $s$. We will always suppress the dependence on $\gamma_1$ and $\gamma_2$. In this special case, the RH problem from \cite{CK} reads as follows.
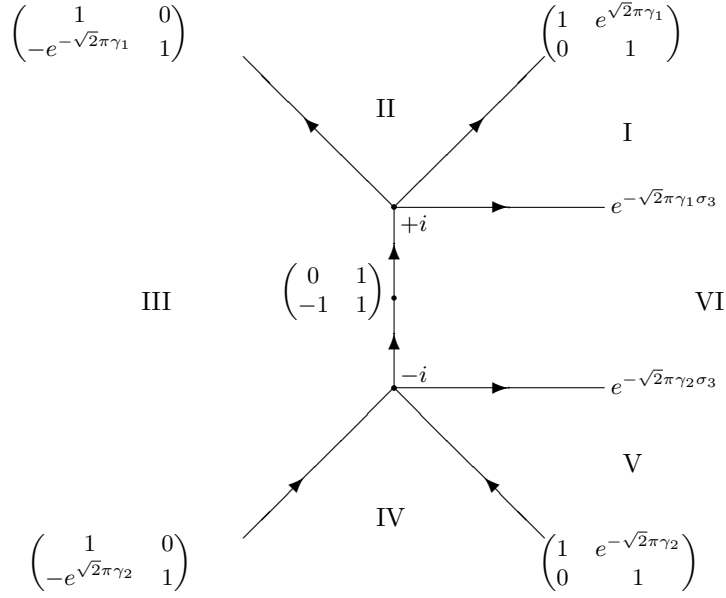
\begin{figure}[t]
\begin{center}
    \setlength{\unitlength}{0.8truemm}
    \begin{picture}(110,95)(-10,-7.5)
    \put(50,45){\thicklines\circle*{.8}}
    \put(50,60){\thicklines\circle*{.8}}
    \put(50,30){\thicklines\circle*{.8}}
    \put(51,56){\small $+i$}
    \put(51,31){\small $-i$}
    \put(50,60){\thicklines\circle*{.8}}
    \put(50,30){\thicklines\circle*{.8}}
    \put(69,60){\thicklines\vector(1,0){.0001}}
    \put(69,30){\thicklines\vector(1,0){.0001}}
    \put(50,60){\line(1,1){25}}
    \put(50,30){\line(1,-1){25}}
    \put(50,60){\line(-1,1){25}}
    \put(50,30){\line(-1,-1){25}}
    \put(50,30){\line(1,0){35}}
    \put(50,60){\line(1,0){35}}
    \put(50,30){\line(0,1){30}}
    \put(65,75){\thicklines\vector(1,1){.0001}}
    \put(65,15){\thicklines\vector(-1,1){.0001}}
    \put(50,39){\thicklines\vector(0,1){.0001}}
    \put(50,54){\thicklines\vector(0,1){.0001}}
    \put(35,75){\thicklines\vector(-1,1){.0001}}
    \put(35,15){\thicklines\vector(1,1){.0001}}

    \put(74,88){\small $\begin{pmatrix}1&e^{{\sqrt{2}\pi}\gamma_1}\\0&1\end{pmatrix}$}
    \put(-14,88){\small $\begin{pmatrix}1&0\\-e^{-{\sqrt{2}\pi}\gamma_1}&1\end{pmatrix}$}
\put(86,29){\small $e^{-{\sqrt{2}\pi}\gamma_2 \sigma_3}$} \put(86,59){\small
$e^{-{\sqrt{2}\pi}\gamma_1 \sigma_3}$}
\put(31,45){\small $\begin{pmatrix}0&1\\-1&1\end{pmatrix}$}
    \put(-11,0){\small $\begin{pmatrix}1&0\\-e^{{\sqrt{2}\pi}\gamma_2}&1\end{pmatrix}$}
    \put(74,0){\small $\begin{pmatrix}1&e^{-{\sqrt{2}\pi}\gamma_2}\\0&1\end{pmatrix}$}
\put(100,43){VI}
    \put(8,43){III}
    \put(88,71){I}
    \put(47,7){IV}
    \put(47,75){II}
    \put(88,16){V}
    \end{picture}
    \caption{The jump contour and jump matrices for $\Psi$.}
    \label{figure: Gamma}
\end{center}
\end{figure}

\subsubsection*{RH problem for $\Psi$}
\begin{itemize}
    \item[(a)] $\Psi=\Psi(\cdot;s):\mathbb C\setminus \Gamma \to \mathbb C^{2\times 2}$ is analytic, where
    \begin{align*}&\Gamma=\cup_{k=1}^7\Gamma_k,&& \Gamma_1=i+e^{\frac{i\pi}{4}}\mathbb R^+,
    &&&\Gamma_2=i+e^{\frac{3i\pi}{4}}\mathbb R^+,\\&\Gamma_3=-i+e^{\frac{5i\pi}{4}}\mathbb R^+,
    &&\Gamma_4=-i+e^{\frac{7i\pi}{4}}\mathbb R^+,&&&\Gamma_5=-i+\mathbb R^+,\\&\Gamma_6=i+\mathbb R^+,
    &&\Gamma_7=[-i,i],\end{align*}
    with the orientation chosen as in Figure \ref{figure: Gamma},
    \item[(b)] $\Psi$ satisfies the jump conditions
    \begin{equation}\label{jump Psi}\Psi_+(\zeta)=\Psi_-(\zeta)J_k,\qquad
    \zeta\in\Gamma_k,\end{equation}  where
                \begin{align}
                &J_1=\begin{pmatrix}1&e^{{\sqrt{2}\pi}\gamma_1}\\0&1\end{pmatrix},
                &&J_2=\begin{pmatrix}1&0\\-e^{-{\sqrt{2}\pi}\gamma_1}&1\end{pmatrix},\\
                &J_3=\begin{pmatrix}1&0\\-e^{{\sqrt{2}\pi}\gamma_2}&1\end{pmatrix},
                &&J_4=\begin{pmatrix}1&e^{-{\sqrt{2}\pi}\gamma_2}\\0&1\end{pmatrix},\\ \label{eq:jump56}
                &J_5=e^{-{\sqrt{2}\pi}\gamma_2\sigma_3},
                &&J_6=e^{-{\sqrt{2}\pi}\gamma_1\sigma_3},\\
                &J_7=\begin{pmatrix}0&1\\-1&1\end{pmatrix},
                \end{align}
    \item[(c)] there exist matrices $\Psi_1=\Psi_1(s), \Psi_2=\Psi_2(s)$ independent of $\zeta$ such that we have in all regions:
    \begin{equation}\label{Psi as}
    \Psi(\zeta)=\left(I+\frac{\Psi_1}{\zeta}+\frac{\Psi_2}{\zeta^2}+\bigO(\zeta^{-3})\right)
    \Psi^\infty(\zeta)e^{-\frac{is}{4}\zeta\sigma_3} \qquad  \mbox{ as $\zeta\to \infty$,}
    \end{equation}
    where
    \begin{equation}\label{Pinfty}
    \Psi^{\infty}(\zeta)=\Psi^\infty(\zeta;s)=
    \left(\frac{is}{2}\right)^{\frac{\gamma_1+\gamma_2}{{\sqrt{2}} i}\sigma_3}
    (\zeta-i)^{\frac{\gamma_1}{{\sqrt{2}} i}\sigma_3}(\zeta+i)^{\frac{\gamma_2}{{\sqrt{2}} i}\sigma_3}
     \end{equation}
                with the branches corresponding to
                arguments between $0$ and $2\pi$ (i.e.\ $z^{c\sigma_3}=|z|^{c\sigma_3}e^{ic\arg z \sigma_3}$ with $\arg z\in(0,2\pi)$),
                
                \item[(d)] as $\zeta \to x\in\{-i,i\}$, we have
                \begin{equation*}
                \Psi(\zeta;s)=\mathcal O(\log(\zeta-x)).
                \end{equation*}
    \end{itemize}

It was proven in \cite{CK} that the RH problem for $\Psi$ has a unique solution for each $\gamma_1,\gamma_2\in \mathbb R$ and $s\in -i\R_+$.

\medskip

It will turn out to be useful for us to consider a modification of this function and a corresponding modification of the RH problem. The modified function is given by
\begin{equation}\label{def:Psihat}
\widehat\Psi(\lambda;s)=\begin{cases}
e^{\frac{-3{\sqrt{2}\pi}\gamma_2-3{\sqrt{2}\pi}\gamma_1}{4}\sigma_3}e^{-\frac{s}{4}\sigma_3}\Psi(-2i\lambda+i;s)e^{-\frac{{\sqrt{2}\pi}\gamma_2}{2}\sigma_3},&\Im\lambda<0,\\
e^{\frac{-3{\sqrt{2}\pi}\gamma_2-3{\sqrt{2}\pi}\gamma_1}{4}\sigma_3} e^{-\frac{s}{4}\sigma_3}\Psi(-2i\lambda+i;s)\sigma_3\sigma_1e^{-\frac{{\sqrt{2}\pi}\gamma_2}{2}\sigma_3},&\Im\lambda>0, 0<\Re\lambda<1,\\
e^{\frac{-3{\sqrt{2}\pi}\gamma_2-3{\sqrt{2}\pi}\gamma_1}{4}\sigma_3} e^{-\frac{s}{4}\sigma_3}\Psi(-2i\lambda+i;s)\sigma_3\sigma_1  e^{-{\sqrt{2}\pi}\gamma_1\sigma_3}e^{-\frac{{\sqrt{2}\pi}\gamma_2}{2}\sigma_3},&\Im\lambda>0, \Re\lambda<0,\\
e^{\frac{-3{\sqrt{2}\pi}\gamma_2-3{\sqrt{2}\pi}\gamma_1}{4}\sigma_3}e^{-\frac{s}{4}\sigma_3}\Psi(-2i\lambda+i;s)\sigma_3\sigma_1 e^{\frac{{\sqrt{2}\pi}\gamma_2}{2}\sigma_3},&\Im\lambda>0, \Re\lambda>1,
\end{cases}
\end{equation}
where $\sigma_1=\begin{pmatrix}0&1\\1&0\end{pmatrix}$ and $\sigma_3=\begin{pmatrix}1&0\\0&-1\end{pmatrix}$ are the standard Pauli matrices.
To describe its RH problem, let us denote $\widehat\Gamma=\cup_{j=1}^7\widehat\Gamma_j$, where
 \begin{align*}& \widehat\Gamma_1=e^{\frac{i3\pi}{4}}\mathbb R^+,
    &&\widehat\Gamma_2=e^{-\frac{3i\pi}{4}}\mathbb R^+,&&&\widehat\Gamma_3=1+e^{\frac{-i\pi}{4}}\mathbb R^+,\\
&\widehat\Gamma_4=1+e^{\frac{i\pi}{4}}\mathbb R^+,&&\widehat\Gamma_5=[0,1],&&&\widehat\Gamma_6=(-\infty,0],\\
&\widehat\Gamma_7=[1,+\infty),
    \end{align*}
and where we choose the orientation as in Figure \ref{fig:hatPsi}. We claim that the RH problem for $\widehat{\Psi}$ then becomes the following.

\subsubsection*{RH problem for $\widehat\Psi$}
\begin{itemize}
    \item[(a)] $\widehat\Psi:\mathbb C\setminus \widehat\Gamma \to \mathbb C^{2\times 2}$ is analytic, 
    \item[(b)] $\widehat\Psi$ satisfies the jump conditions
    \begin{equation}\label{jump hatPsi}\widehat\Psi_+(\lambda)=\widehat\Psi_-(\lambda)\widehat J_k,\qquad
    \lambda\in\widehat\Gamma_k,\end{equation}  where
                \begin{align*}
                &\widehat J_1=\begin{pmatrix}1&0\\e^{-{\sqrt{2}\pi}(\gamma_1+\gamma_2)}&1\end{pmatrix},
                &&\widehat J_2=\begin{pmatrix}1&0\\e^{-{\sqrt{2}\pi}(\gamma_1+\gamma_2)}&1\end{pmatrix},\\
                &\widehat J_3=\begin{pmatrix}1&0\\1&1\end{pmatrix},
                &&\widehat J_4=\begin{pmatrix}1&0\\ 1&1\end{pmatrix},\\
                &\widehat J_5=\begin{pmatrix}1&e^{{\sqrt{2}\pi}\gamma_2}\\0&1\end{pmatrix},
&&\widehat J_6=\begin{pmatrix}0&e^{{\sqrt{2}\pi}\gamma_1+{\sqrt{2}\pi}\gamma_2}\\ -e^{-{\sqrt{2}\pi}\gamma_1-{\sqrt{2}\pi}\gamma_2}&0\end{pmatrix},\\
                &\widehat J_7=\begin{pmatrix}0&1\\-1&0\end{pmatrix},                
                \end{align*}
    \item[(c)] 
    \begin{equation}\label{Psi ashat}
    \widehat\Psi(\lambda)=\left(I+\frac{\widehat\Psi_1}{\lambda}+\frac{\widehat\Psi_2}{\lambda^2}+\bigO(\lambda^{-3})\right)
    \widehat\Psi^\infty(\lambda)e^{\pm \frac{s}{2}\lambda\sigma_3} \qquad  \mbox{ as $\lambda\to \infty$ \quad and \quad $\pm \mathrm{Im}\lambda>0$,}
    \end{equation}
    where
    \begin{equation}\label{Pinftyhat}
    \widehat\Psi^{\infty}(\lambda)=\widehat \Psi^\infty(\lambda;s)=\begin{cases}
   e^{-{\sqrt{2}\pi}\gamma_2\sigma_3}|s|^{\frac{\gamma_1+\gamma_2}{{\sqrt{2}} i}\sigma_3} \lambda^{\frac{\gamma_1}{{\sqrt{2}} i}\sigma_3} (1-\lambda)^{\frac{\gamma_2}{{\sqrt{2}} i}\sigma_3},&\Im\lambda<0,\\
   |s|^{\frac{\gamma_1+\gamma_2}{{\sqrt{2}} i}\sigma_3} \lambda^{\frac{\gamma_1}{{\sqrt{2}} i}\sigma_3} (1-\lambda)^{\frac{\gamma_2}{{\sqrt{2}}i}\sigma_3}\sigma_3\sigma_1,&\Im\lambda>0,
   \end{cases}
     \end{equation}
                with principal branches corresponding to
                arguments of $\lambda$ and $1-\lambda$ between $-\pi$ and $\pi$, for some matrices $\widehat\Psi_1, \widehat\Psi_2$. 
                \item[(d)] As $\lambda \to x\in\{0,1\}$, we have
                \begin{equation*}
                \widehat\Psi(\lambda;s)=\mathcal O(\log(\lambda-x)).
                \end{equation*}
    \end{itemize}

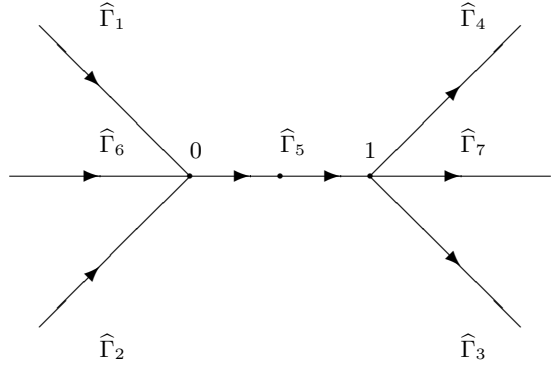
\begin{figure}[t]
\begin{center}
    \setlength{\unitlength}{0.8truemm}
    \begin{picture}(75,55)(5,10)
    \put(45,50){\thicklines\circle*{.8}}
    \put(60,50){\thicklines\circle*{.8}}
    \put(30,50){\thicklines\circle*{.8}}
    \put(55,50){\thicklines\vector(1,0){.0001}}
    \put(40,50){\thicklines\vector(1,0){.0001}}
    \put(15,50){\thicklines\vector(1,0){.0001}}
    \put(75,50){\thicklines\vector(1,0){.0001}}
    \put(0,50){\line(1,0){90}}
    \put(60,50){\line(1,1){25}}
    \put(30,50){\line(-1,1){25}}
    \put(60,50){\line(1,-1){25}}
    \put(30,50){\line(-1,-1){25}}
    \put(75,65){\thicklines\vector(1,1){.0001}}
    \put(15,65){\thicklines\vector(1,-1){.0001}}
    \put(75,35){\thicklines\vector(1,-1){.0001}}
    \put(15,35){\thicklines\vector(1,1){.0001}}

    \put(75,75){\small $\widehat{\Gamma}_4$}
    \put(15,75){\small $\widehat{\Gamma}_1$}
    \put(75,20){\small $\widehat{\Gamma}_3$}
    \put(15,20){\small $\widehat{\Gamma}_2$}
    \put(45,54){\small $\widehat{\Gamma}_5$}
    \put(15,54){\small $\widehat{\Gamma}_6$}
    \put(75,54){\small $\widehat{\Gamma}_7$}

    \put(30,53){\small $0$}
    \put(59,53){\small $1$}

    \end{picture}
\end{center}
    \caption{The jump contour for $\widehat{\Psi}$.}
    \label{fig:hatPsi}
\end{figure}

\begin{lemma}\label{pr:Psihat}
$\widehat\Psi$ defined in \eqref{def:Psihat} satisfies the above RH problem.
\end{lemma}

\begin{proof}
Since $\Psi(\zeta)$ is analytic on $\mathbb C\setminus\Gamma$, it follows that $\Psi(-2i\lambda+i;s)$ is analytic for $\lambda$ in $\mathbb C\setminus\left(\widehat\Gamma_1\cup\widehat\Gamma_2\cup\widehat\Gamma_3\cup\widehat\Gamma_4\cup [0,1]\cup [0, +i\infty)\cup [1, 1+i\infty)\right)$. By the definition \eqref{def:Psihat} of $\widehat\Psi$, it then follows directly that $\widehat\Psi$ is analytic on $\mathbb C\setminus\left(\widehat\Gamma_1\cup\widehat\Gamma_2\cup\widehat\Gamma_3\cup\widehat\Gamma_4\cup \mathbb R\cup [0, +i\infty)\cup [1, 1+i\infty)\right)$.
First of all, using \eqref{def:Psihat} and the jump matrices $J_5$ and $J_6$ in \eqref{eq:jump56}, we verify directly that $\widehat\Psi_+(\lambda)=\widehat\Psi_-(\lambda)$ for $\lambda\in (0,+i\infty)$ and for $\lambda\in (1,1+i\infty)$, hence we can conclude that $\widehat\Psi$ has an analytic continuation across those half-lines. Next, we can compute the jump matrix $\widehat J_j:=\widehat\Psi_-^{-1}(\lambda)\widehat\Psi_+(\lambda)$ on $\widehat\Gamma_j$ for $j=1,\ldots, 5$ by substituting \eqref{def:Psihat} and by using the fact that
$\Psi_-^{-1}(-2i\lambda+i){\Psi_+(-2i\lambda+i)}=J_j$.
Finally, again using \eqref{def:Psihat}, we compute the jump matrices $\widehat J_6$ and $\widehat J_7$ on $\mathbb R\setminus[0,1]$. This shows that conditions (a)--(b) of the RH problem are satisfied.

The asymptotic behaviour of $\widehat\Psi$ as $\lambda\to\infty$ is inherited from \eqref{Psi as} with $\zeta=-2i\lambda+i$. We recall that the arguments of $\zeta\pm i$ in \eqref{Pinfty} are chosen between $0$ and $2\pi$, whereas those of $\lambda$ and $1-\lambda$ in \eqref{Pinftyhat} are chosen between $-\pi$ and $\pi$.
This implies that, for $is>0$,
\[
\Psi^\infty({ -2i \lambda+i})=e^{\frac{3{\sqrt{2}\pi}\gamma_1+{\sqrt{2}\pi}\gamma_2}{4}\sigma_3}|s|^{\frac{\gamma_1+\gamma_2}{{\sqrt{2}} i}\sigma_3}\lambda^{\frac{\gamma_1}{{\sqrt{2}} i}\sigma_3}(1-\lambda)^{\frac{\gamma_2}{{\sqrt{2}} i}\sigma_3}
\ \times \ 
\begin{cases}1,&\Im\lambda<0,\\
1,&\Im\lambda>0, 0<\Re\lambda<1,\\
e^{-{{\sqrt{2}\pi}\gamma_1}\sigma_3},&\Im\lambda>0, \Re\lambda<0,\\
e^{-{\sqrt{2}\pi}\gamma_2\sigma_3},&\Im\lambda>0, \Re\lambda>1.
\end{cases}
\]
Substituting this together with \eqref{Psi as} in \eqref{def:Psihat}, we obtain \eqref{Psi ashat}--\eqref{Pinftyhat} after a straightforward calculation, and condition (c) of the RH problem is verified.

Finally, condition (d) is checked directly using the corresponding condition of the RH problem for $\Psi$ and the definition \eqref{def:Psihat} of $\widehat\Psi$.
\end{proof}

A basic fact about the functions $\Psi$ we will need in our asymptotic analysis is the following.

\begin{lemma}\label{le:psiasyat-i}
For $j=1,2$, the limit $\lim_{\epsilon\to 0^+}\Psi_{j,1}(-i-\epsilon;s)$ exists for all $s\in -i\R_+$, and it is bounded uniformly in $s\in -i\R_+$.
\end{lemma}

\begin{proof}
We begin by recalling from \cite[(3.17)]{CK} that we can write for $\epsilon>0$ and $s\in -i\R_+$
\begin{equation}\label{eq:psinear-i}
\Psi(-i-\epsilon;s)=F(-i-\epsilon;s)\begin{pmatrix}
1 & \frac{1-e^{-{\sqrt{2}\pi}\gamma_2}}{2\pi i}{\log( -\epsilon)}\\
0 & 1
\end{pmatrix},
\end{equation}
where for each $s\in -i\R_+$, $\zeta\mapsto F(\zeta;s)$ is analytic in a neighborhood of $-i$. From this, we see that the relevant limit exists and we have for $j=1,2$
$$
\lim_{\epsilon\to 0^+}\Psi_{j,1}(-i-\epsilon;s)=F_{j,1}(-i;s).
$$

It is a general fact\footnote{While well known to experts, a proof of such a statement unfortunately does not exist in the literature currently. For a proof of a similar result for a slightly different RH problem, see e.g. \cite[Appendix A]{FIKN}. For the convenience of a concerned reader, we briefly illustrate the main ideas behind a proof of such a statement. The first step is to transform the RH problem to one whose solution is normalized to be the identity matrix at infinity and whose only $s$-dependence is in the jump matrices (and this dependence is smooth). The solution to the transformed RH problem can then be expressed in terms of the resolvent of the Cauchy operator associated to the $s$-dependent jump matrix. From this, one can finally prove that the solution depends smoothly on $s$.} about RH-problems related to Painlev\'e equations (as is in the case at hand  -- see \cite{CK}) that quantities such as $F_{j,1}(-i;s)$ are meromorphic functions of $s\in \mathbb C$, with possible singularities for values of $s$ for which there is no solution to the RH problem of $\Psi(\cdot,s)$.

 Since the RH problem for $\Psi$ is solvable for all  $s\in -i\R_+$, it follows that $F_{j,1}(-i;s)$ is a smooth function for  $s\in -i\R_+$. In particular, if $K\subset -i\R_+$ is any compact set, then $\sup_{s\in K}|F_{j,1}(-i;s)|<\infty$. It thus remains to prove the claim for small $|s|$ and large $|s|$. For this, we need more precise information about $\Psi$ proven in \cite[Sections 4-6]{CK}. { In what follows below, one should keep in mind that we consider the parameters $\beta_1,\beta_2$ in \cite{CK}  to be purely imaginary and that $|||\beta|||=0$, with the notation of $|||\beta|||$ as in \cite{CK}.}

Let us begin with the situation where $s\to -i\infty$ -- this is mainly covered in \cite[Section 5]{CK}.  Using \cite[(5.1), (5,13), and (5.17)]{CK} (along with the discussion leading to these equations), we see that we can write for $\epsilon>0$ and large enough $|s|$, (namely $|s|>C$ for some fixed $C$)
$$
\Psi(-i-\epsilon;s)=\widetilde R (-i-\epsilon;s)E_2(-i-\epsilon;s)\widehat{M}\left(-\frac{|s|}{2}\epsilon\right),
$$
where by \cite[(5.18) and (5.25)]{CK} $\widetilde R(-i-\epsilon;s)=I+\mathcal O(|s|^{-1})$ (where the implied constant is uniform in $\epsilon>0$ small enough) by \cite[(5.14)]{CK}, $E_2(-i-\epsilon;s)=\mathcal O(1)$ (where the implied constant is uniform in $\epsilon>0$ small enough and $s\in -i\R_+$ with $|s|>C$)  and by \cite[(4.12) and (5.12)]{CK} 
$$
\widehat M\left(-\frac{|s|}{2}\epsilon\right)=\widetilde L\left(-\frac{|s|}{2}\epsilon\right)\begin{pmatrix}
1 & m(-\frac{|s|}{2}\epsilon)\\
0& 1
\end{pmatrix}{ e^{\frac{{\sqrt{2}\pi}\gamma_2}{4}\sigma_3}},
$$ 
where $\widetilde L$ is a $2\times 2$ matrix-valued function which is analytic in a neighborhood of zero and does not depend on $s$. The function $m$ is given in \cite[(4.13)]{CK}, but we don't need to know anything about it. Combining these remarks, we see that for large enough $|s|$,  $j=1,2$,
$$
\lim_{\epsilon\to 0^+}\Psi_{j,1}(-i-\epsilon;s)=\left(\widetilde R (-i;s)E_2(-i;s)\widetilde L\left(0\right)\right)_{j,1}{e^{\frac{{\sqrt{2}\pi}\gamma_2}{4}\sigma_3}},
$$
which is uniformly bounded for $s\in -i\R_+$ large enough (``large enough" depending only on $\gamma_1,\gamma_2$). This verifies the claim for large $s$.

It remains to check the claim for small $s$. For this, we make use of results of \cite[Section 6]{CK}. Using \cite[(6.2),(6.23),(6.26), and (6.28)]{CK}, we can write for $\epsilon>0$ and $s\in -i\R_+$ small enough
\begin{equation}\label{eq:psiat-i}
\Psi(-i-\epsilon;s)=e^{\frac{s}{4}\sigma_3}{ e^{-\frac{{\sqrt{2}\pi}\gamma_2-{\sqrt{2}\pi}\gamma_1}{4}\sigma_3}}H\left(s-\frac{\epsilon|s|}{2};s\right) { \widetilde L} \left(s-\frac{\epsilon|s|}{2}\right)\begin{pmatrix}
1 & \widehat m(\epsilon;s)\\
0 & 1
\end{pmatrix} { e^{-\frac{{\sqrt{2}\pi}\gamma_1-{\sqrt{2}\pi}\gamma_2}{4}\sigma_3}},
\end{equation}
where again $\widetilde L$ is a $2\times 2$ matrix-valued function which is analytic in a neighborhood of zero and independent of $s$, and while $\widehat m$ is explicit, we have no need for details about it. Arguing as in the discussion leading to \cite[(6.29)]{CK}, but now using \cite[(6.27)]{CK}, we have $H(s-\frac{\epsilon|s|}{2};s)=I+\mathcal O(|s\log |s||)$, where the implied constant is independent of $\epsilon$. We conclude that for $\epsilon>0$, $s\in -i\R_+$ small enough and $j=1,2$
$$
\Psi_{j,1}(-i-\epsilon;s)=\mathcal O(1)
$$
where the implied constant and ``small enough" depends only on $\gamma_1,\gamma_2$. This concludes the proof.
\end{proof}

Let us now recall the connection between the RH problem for $\Psi$ and the Painlev\'e V equation.
Consider the equation
\cite[Formula (2.8)]{ForresterWitte}
\begin{equation}\label{sigmaPV}
s^2\sigma''(s)^2=\left(\sigma(s)-s\sigma'(s)+2\sigma'(s)^2\right)^2
-4(\sigma'(s)-\theta_1)(\sigma'(s)-\theta_2)
(\sigma'(s)-\theta_3)(\sigma'(s)-\theta_4),
\end{equation}
where the parameters $\theta_1,\theta_2, \theta_3, \theta_4$ are given by
\begin{align}
&\theta_1=\theta_2=-\theta_3=-\theta_4=-\frac{\gamma_1+\gamma_2}{{2\sqrt{2}}i}.\label{theta1234}
\end{align}
It was shown in \cite{CK} that there exists a solution $\sigma$ which is real analytic for $s\in -i(0,+\infty)$ and has the asymptotic behavior
\begin{align}
&\sigma(s)=-\frac{\gamma_1-\gamma_2}{{2\sqrt{2}}} |s| +\frac{(\gamma_1-\gamma_2)^2}{{4}}+\mathcal O(|s|^{-1}),&&s\to -i\infty,\label{eq:sigmaasy1}\\
&\sigma(s)=\frac{(\gamma_1+\gamma_2)^2}{{4}}+\mathcal O(|s\ln|s||),&&s\to -i 0_+,\label{eq:sigmaasy2}
\end{align}
and which can be constructed in terms of the RH solution $\Psi$. We summarize this result in the following proposition.

\begin{proposition}\label{pr:psiderasy}
For $s\in -i\mathbb R^+$, we have the identity
\begin{align}
&\label{eq:sigmaid1}\lim_{\epsilon\to 0_+}\left(\Psi^{-1}(-i-\epsilon;s)\Psi'(-i-\epsilon;s)\right)_{2,1}=-\frac{\pi e^{{\sqrt{2}\pi}\gamma_2}}{e^{{\sqrt{2}\pi}\gamma_2}-1}\left(\sigma(s)-\frac{\gamma_1+\gamma_2}{{2\sqrt{2}} i}s-\frac{(\gamma_1+\gamma_2)^2}{{4}}\right).
\end{align}
where $\sigma$ is a smooth solution to the $\sigma$-form of the Painlev\'e V equation \eqref{sigmaPV} which has the asymptotic  behavior
\eqref{eq:sigmaasy1}--\eqref{eq:sigmaasy2}, and where $\Psi'(\zeta;s)$ denotes the derivative with respect to $\zeta$.
\end{proposition}
\begin{proof}
It was shown in \cite[Formula (3.17)]{CK} that $\Psi$ can be written in the following form for $\zeta$ close to $-i$ and in sector III shown in Figure \ref{figure: Gamma}:
\begin{equation}\label{eq:Psilocal}
\Psi(\zeta;s)=F(\zeta;s)\begin{pmatrix}1&\frac{1-e^{-{\sqrt{2}\pi}\gamma_2}}{2\pi i}\log(\zeta+i)\\0&1\end{pmatrix},
\end{equation} where $F$ is an analytic function near $-i$. We will prove \eqref{eq:sigmaid1} by proving first some identities for $F$.

 We use \eqref{eq:Psilocal} to expand $A(\zeta;s):=\Psi_\zeta(\zeta;s)\Psi^{-1}(\zeta;s)$ as $\zeta\to -i$ and obtain
\[A(\zeta;s)=\frac{1-e^{-{\sqrt{2}\pi}\gamma_2}}{2\pi i(\zeta+i)}F(\zeta;s)\begin{pmatrix}0&1\\0&0\end{pmatrix}F(\zeta;s)^{-1}+\mathcal O(1).\] Moreover, it was shown in \cite[Formula (3.21)]{CK} that $A$ takes the form
\[A(\zeta;s)=-\frac{is}{4}\sigma_3+\frac{A_1(s)}{\zeta-i}+\frac{A_2(s)}{\zeta+i}\]
for matrices $A_1(s)$ and $A_2(s)$ independent of $\zeta$.
Moreover, by \cite[Formulas (3.32), (3.45), (3.50), and Theorem 1.1]{CK} $\left(A_2(s)\right)_{1,1}$ is given by (recall that the connection between our notation and that of  \cite{CK} is $\gamma_j=-{\sqrt{2}} i\beta_j$)
\[\left(A_2(s)\right)_{1,1}=-\sigma_s+\frac{\gamma_1+\gamma_2}{{2\sqrt{2}}i},\]
where $\sigma$ is a certain smooth function satisfying the hypotheses of our proposition and $\sigma_s$ denotes its derivative.
It follows that
\begin{equation}\label{eq:idsigmaproof1}
\left(F(-i;s)\begin{pmatrix}0&1\\0&0\end{pmatrix}F(-i;s)^{-1}\right)_{1,1}
=\frac{2\pi i}{1-e^{-{\sqrt{2}\pi}\gamma_2}}\left(-\sigma_s+\frac{\gamma_1+\gamma_2}{{2\sqrt{2}}i}\right).
\end{equation}

To prove another identity for $F$, we recall from \cite[Formulas (3.20), (3.22), and (3.33)]{CK} that
$B(\zeta;s):=[\frac{d}{ds}\Psi(\zeta;s)]\Psi(\zeta;s)^{-1}$ takes the form
\[B(\zeta;s)=-\frac{i\zeta}{4}\sigma_3+\widetilde{B}(s),\]
where $\widetilde{B}(s)$ denotes a matrix independent of $\zeta$. 
Consequently, using \eqref{eq:Psilocal},
\[\frac{d}{ds}F(\zeta;s)=\left(-\frac{i\zeta}{4}\sigma_3+\widetilde{B}(s)\right)F(\zeta;s).\]
Expanding $F$ for $\zeta$ near $-i$ as
\[F(\zeta;s)=F_0(s)\left(I+F_1(s)(\zeta+i)+\mathcal O((\zeta+i)^2)\right),\]
we obtain the identities
\[\frac{d}{ds}F_0={ B(-i;s)}F_0,\qquad \frac{d}{ds}F_1= -\frac{i}{4}F_0^{-1}\sigma_3F_0.\]

It follows that \[\frac{d}{ds}\lim_{\epsilon\to 0_+}\left(\Psi^{-1}(-i-\epsilon;s)\Psi'(-i-\epsilon;s)\right)_{2,1}=\frac{d}{ds}F_{1, 21}(s)=-\frac{i}{4}\left(F_0(s)^{-1}\sigma_3F_0(s)\right)_{2,1}.\]
Now the right hand side can be expressed by \eqref{eq:idsigmaproof1} as
\begin{align*}
-\frac{i}{4}\left(F(-i;s)^{-1}\sigma_3F(-i;s)\right)_{2,1}&=-\frac{i}{2}\left(F(-i;s)\begin{pmatrix}0&1\\0&0\end{pmatrix}F(-i;s)^{-1}\right)_{1,1}\\
&=\frac{\pi }{1-e^{-{\sqrt{2}\pi}\gamma_2}}\left(-\sigma'(s)+\frac{\gamma_1+\gamma_2}{{2\sqrt{2}} i}\right).
\end{align*}

Integrating the identity 
\[\frac{d}{ds}\lim_{\epsilon\to 0_+}\left(\Psi^{-1}(-i-\epsilon;s)\Psi'(-i-\epsilon;s)\right)_{2,1}=\frac{\pi }{1-e^{-{\sqrt{2}\pi}\gamma_2}}\left(-\sigma'(s)+\frac{\gamma_1+\gamma_2}{{2\sqrt{2}} i}\right)\]
in $s$ proves \eqref{eq:sigmaid1}  up to the value of the integration constant. To calculate the value of this integration constant, we see from \eqref{eq:sigmaasy2} that it is sufficient to verify that $$\lim_{\epsilon\to 0_+}\left(\Psi^{-1}(-i-\epsilon;s)\Psi'(-i-\epsilon;s)\right)_{2,1}\to 0$$ as $s\to 0$ along $-i\R_+$. To do this, we recall \eqref{eq:psiat-i}, from which we find for small enough $\epsilon$ and $s$ (``small enough" depending only on $\gamma_1,\gamma_2$)

\begin{align*}
&\left(\Psi^{-1}(-i-\epsilon;s)\Psi'(-i-\epsilon;s)\right)_{2,1}\\
&\quad =\frac{|s|}{2}{ e^{-{\sqrt{2}\pi}\frac{\gamma_1-\gamma_2}{2}}}\left[{ \widetilde  L}\left(s-\frac{\epsilon|s|}{2}\right)^{-1}H\left(s-\frac{\epsilon|s|}{2};s\right)^{-1}H'\left(s-\frac{\epsilon|s|}{2};s\right){ \widetilde  L}\left(s-\frac{\epsilon|s|}{2}\right)\right]_{2,1}\\
&\qquad +\frac{|s|}{2}{ e^{{\sqrt{2}\pi}\frac{\gamma_1-\gamma_2}{2}}}\left[{ \widetilde  L}\left(s-\frac{\epsilon|s|}{2}\right)^{-1}{ \widetilde  L}'\left(s-\frac{\epsilon|s|}{2}\right)\right]_{2,1}.
\end{align*}

Let us recall from our discussion surrounding \eqref{eq:psiat-i} that it followed from \cite[Section 6]{CK} that $H(s-\frac{\epsilon}{2}|s|;s)=I+\mathcal O(s\log |s|)$. In fact, it follows from \cite[Section 6]{CK} (see in particular \cite[(6.27) -- (6.30)]{CK}) that $\lambda\mapsto H(\lambda;s)$ is analytic in a small enough but fixed neighborhood of zero and uniformly in $\lambda$ in this neighborhood, one has $H(\lambda;s)=I+\mathcal O(s\log |s|)$. Thus from Cauchy's integral formula (for derivatives), one finds that $H'(\lambda;s)=\mathcal O(s\log |s|)$ uniformly in a small enough fixed neighborhood of $0$. In particular, we have for small enough $s$, $H'(s-\frac{\epsilon}{2}|s|;s)=\mathcal O(s\log |s|)$. Moreover, the implied constants in the estimates for $H(s-\frac{\epsilon}{2}|s|;s)$ and $H'(s-\frac{\epsilon}{2}|s|;s)$ are uniform in $\epsilon>0$ small enough. Also as $\lambda\mapsto  \widetilde  L(\lambda)$ is independent of $s$, analytic at zero, and has determinant one, we see that ${ \widetilde  L}(s-\frac{\epsilon|s|}{2}), { \widetilde  L}(s-\frac{\epsilon|s|}{2})^{-1},{ \widetilde  L}'(s-\frac{\epsilon|s|}{2})=\mathcal{O}(1)$, with the implied constant uniform in $\epsilon>0$ and $|s|$ small enough, so we conclude that 
$$
\left(\Psi^{-1}(-i-\epsilon;s)\Psi'(-i-\epsilon;s)\right)_{2,1}=\mathcal O(|s|)
$$
where the implied constant is uniform in $\epsilon>0$ small enough (again with ``small enough" depending only on $\gamma_1,\gamma_2$). We conclude that $\lim_{\epsilon\to 0_+}\left(\Psi^{-1}(-i-\epsilon;s)\Psi'(-i-\epsilon;s)\right)_{2,1}\to 0$ as $s\to 0$ along $-i\R_+$, which concludes the proof.
\end{proof}

\subsection{A model RH problem for the local parametrix near $1$} \label{sect:paremetrix1}

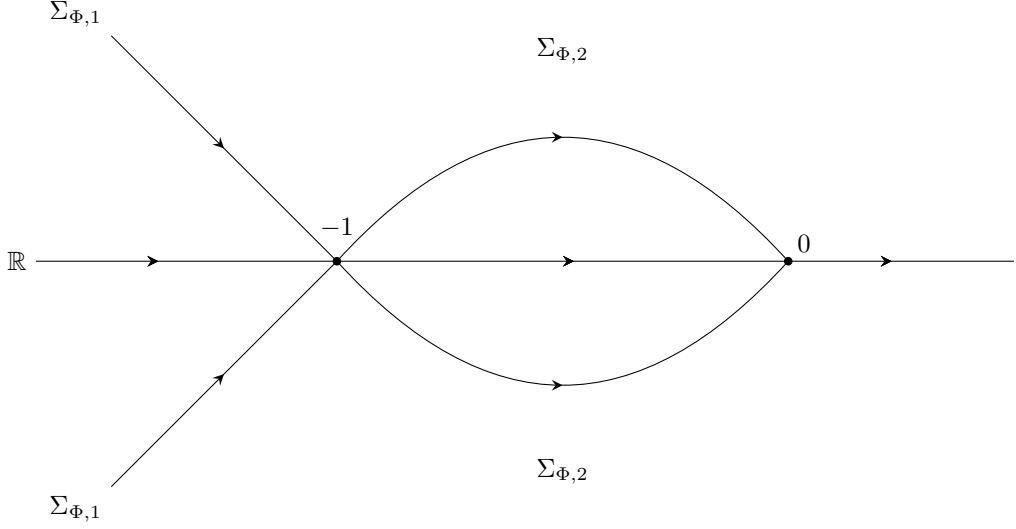
\begin{figure}
\begin{tikzpicture}
\draw [decoration={markings, mark=at position 0.125 with {\arrow[thick]{>}}},
        postaction={decorate}, decoration={markings, mark=at position 0.55 with {\arrow[thick]{>}}},
        postaction={decorate},decoration={markings, mark=at position 0.875 with {\arrow[thick]{>}}},
        postaction={decorate},] (-1,0) -- (12,0);
\draw [decoration={markings, mark=at position 0.5 with {\arrow[thick]{>}}},
        postaction={decorate}] (0,3) -- (3,0); 
\draw  [decoration={markings, mark=at position 0.5 with {\arrow[thick]{>}}},
        postaction={decorate}] (3,0) .. controls (5,-2.2) and (7,-2.2) .. (9,0);     
\draw  [decoration={markings, mark=at position 0.5 with {\arrow[thick]{>}}},
        postaction={decorate}] (0,-3) -- (3,0);
\draw  [decoration={markings, mark=at position 0.5 with {\arrow[thick]{>}}},
        postaction={decorate}] (3,0) .. controls (5,2.2) and (7,2.2) .. (9,0);  
\node [above left] at (0,3) {$\Sigma_{\Phi,1}$};
\node [above ] at (6,2.5) {$\Sigma_{\Phi,2}$};
\node [below left] at (0,-3) {$\Sigma_{\Phi,1}$};
\node [below ] at (6,-2.5) {$\Sigma_{\Phi,2}$};
\node [above right] at (9,0) {$0$};
\node [above ] at (3,0.2) {$-1$}; 
\draw[fill] (9,0) circle [radius=0.05];
\draw[fill] (3,0) circle [radius=0.05];  
\node [left] at (-1,0) {$\mathbb R$};    
\end{tikzpicture}
\caption{The jump contour $\Sigma_{\Phi}$}\label{ContourPhi}
\end{figure}

We now study another model RH problem, which we will need to construct a local parametrix in the case where a FH singularity  is close to the edges of the support $\pm 1$.
This RH problem did not appear in the literature in this form to the best of our knowledge, however it is closely related to model RH problems that appeared in \cite{ClaeysDoeraene,XuDai}.

\subsubsection*{RH Problem for $\Phi$}
\begin{itemize}
\item[(a)] $\Phi = \Phi(\cdot;u)$ is analytic on $\mathbb C\setminus (\mathbb R \cup \Sigma_{\Phi,1}\cup\Sigma_{\Phi,2})$. The contours $\Sigma_{\Phi,1}$ and $\Sigma_{\Phi,2}$ are as in Figure \ref{ContourPhi}: they consist of straight lines near $0$, and near $-1$ they will be specified below.
\item[(b)] On $\mathbb R \cup \Sigma_{\Phi,1}\cup\Sigma_{\Phi,2} \setminus \{-1,0\}$, $\Phi$ has the following jumps:
\begin{equation}
\begin{aligned}
\Phi_+(\lambda)&=\Phi_-(\lambda) \begin{pmatrix}
1&0\\e^{-{\sqrt{2}\pi}\gamma}e^{\frac{4}{3}(\lambda u)^{3/2}}&1
\end{pmatrix}	&& \textrm{for } \lambda\in \Sigma_{\Phi,1},
\\
\Phi_+(\lambda)&=\Phi_-(\lambda) \begin{pmatrix}
1&0\\e^{\frac{4}{3}(\lambda u)^{3/2}}&1
\end{pmatrix}	&& \textrm{for } \lambda\in \Sigma_{\Phi,2},
\\
\Phi_+(\lambda)&=\Phi_-(\lambda) \begin{pmatrix}
0&e^{{\sqrt{2}\pi}\gamma}\\-e^{-{\sqrt{2}\pi}\gamma}&0
\end{pmatrix}	&& \textrm{for } \lambda\in (-\infty,-1),
\\
\Phi_+(\lambda)&=\Phi_-(\lambda) \begin{pmatrix}
0&1\\-1&0
\end{pmatrix}	&& \textrm{for } \lambda\in (-1,0),
\\
\Phi_+(\lambda)&=\Phi_-(\lambda) \begin{pmatrix}
1&e^{-\frac{4}{3}(\lambda u)^{3/2}}\\0&1
\end{pmatrix}	&& \textrm{for } \lambda\in (0,\infty),\label{eq:jumpPhi}
\end{aligned}
\end{equation}
where principal branches are chosen, and where $u>0$ and { $\gamma \in \mathbb R$} are parameters.
\item[(c)] As $\lambda\to \infty$,
\begin{equation}\label{eq:Phiasy}
\Phi(\lambda)={  (I+\mathcal O(\lambda^{-1/2}))\begin{pmatrix}
1&0\\ i{\sqrt{2}}\gamma &1
\end{pmatrix}\lambda^{-\frac{1}{4}\sigma_3} A
e^{-{\sqrt{2}\pi}\frac{\gamma}{2}\sigma_3},}
\end{equation}
where principal branches are chosen, and $A=\frac{1}{\sqrt 2}\begin{pmatrix}
1&i\\i&1
\end{pmatrix}$.
\item[(d)] As $\lambda\to 0$, $\Phi(\lambda)$ remains bounded and as $\lambda\to -1$, $\Phi(\lambda)=\mathcal O(\log(\lambda+1))$.
\end{itemize}

Unlike for the first model RH problem, we do not need precise knowledge of the solution here. We only need the existence of a solution $\Phi=\Phi(\lambda;u)$ for large enough $u$, and explicit asymptotics for $\Phi$ as $u\to +\infty$.

Note that as long as the angles between the jump contour and $(-\infty,0)$ are strictly less than $\pi/2$ and strictly greater than zero, then as $u\to\infty$, the jumps of $\Phi$ tend uniformly to the identity except on $(-\infty,0)$ and in fixed neighbourhoods of $-1$ and $0$.
More precisely, for any $\delta>0$, we have for $\lambda\in \Sigma_{\Phi,1}\cup \Sigma_{\Phi,2}\cup (\delta,\infty)$ such that $|\lambda{+}1|>\delta$, $|\lambda|>\delta$,
\begin{equation}
\Phi_+(\lambda)=\Phi_-(\lambda)\left(I+\mathcal O\left(e^{-\epsilon u^{3/2}(|\lambda|^{3/2}+1)}\right)\right),\qquad u\to +\infty,
\end{equation}
for some $\epsilon>0$. Note that it follows from \eqref{eq:jumpPhi} that $\epsilon$ and the constant implied by the error term can be chosen independent of $\gamma\in[-\Gamma,\Gamma]$.

Away from $-1$ and $0$, it will turn out that $\Phi$ for large $u$ is well approximated by the function
\begin{equation}\label{def:M}
M(\lambda)=\lambda^{-\frac{1}{4}\sigma_3}A\left(\frac{1+e^{-\pi i /2}\sqrt{\lambda}}{\sqrt{\lambda+1}}\right)^{-i{\sqrt{2}}{\gamma}\sigma_3},
\end{equation}
where the matrix $A$ is as \eqref{eq:Phiasy} and with principal branches of the root functions.
Then it is straightforward to verify that $\left(\frac{1+e^{-\pi i /2}\sqrt{\lambda}}{\sqrt{\lambda+1}}\right)^{-i{\sqrt{2}}{\gamma}\sigma_3}$ is analytic on $\mathbb C\setminus (-\infty,0]$, and that
\[
\left(\frac{1+e^{-\pi i /2}\sqrt{\lambda}}{\sqrt{\lambda+1}}\right)^{-i{\sqrt{2}}{\gamma}\sigma_3}_+\left(\frac{1+e^{-\pi i /2}\sqrt{\lambda}}{\sqrt{\lambda+1}}\right)^{-i{\sqrt{2}}{\gamma}\sigma_3}_-=\begin{cases}
{ e^{-{\sqrt{2}\pi}\gamma\sigma_3}},&\mbox{on }(-\infty, -1),\\
I,&\mbox{on }(-1,0).
\end{cases}
\]
From this, one easily obtains that $M$ is analytic on $\mathbb C\setminus (-\infty,0]$, and that $M$ satisfies the same jump conditions as $\Phi$ on $(-\infty,-1)$ and on $(-1,0)$. Moreover, {a routine calculation shows that} we have
\begin{equation}\label{eq:asymptPhiM}
\lim_{\lambda\to\infty}\Phi(\lambda)M(\lambda)^{-1}=I.
\end{equation}

\medskip

Let $\Phi_{\textrm{Ai}}$ be the solution to the Airy model RH problem as defined in Appendix \ref{section:appendixAiry}. In a neighbourhood $U_0$ of $0$ {(not intersecting $(-\infty,-1]$)},  we define the local parametrix as
\begin{equation*}
P_0(\lambda)=\widehat E_0(\lambda)
\Phi_{\textrm{Ai}}(u\lambda)e^{\frac{2}{3}(\lambda u)^{3/2}\sigma_3}, \qquad \widehat E_0(\lambda)=M(\lambda)
A^{-1}(\lambda u)^{\frac{1}{4}\sigma_3},
\end{equation*}
where principal branches are chosen.
It is straightforward to verify that $\widehat E_0$ is analytic in $U_0$. Defined in this way, it follows from the RH conditions for $\Phi_{\rm Ai}$ (see in particular \eqref{RHP Psi: b1} -- \eqref{RHP Psi: b3}) that $P_0$ has the same jumps on $U_0\cap (\Sigma_{\Phi_2}\cup(-1,+\infty))$ as $\Phi$, and (from \eqref{RHP Psi: c}) that we have the matching condition
\begin{equation}\label{smallness1}
P_0(\lambda)M(\lambda)^{-1}=I+\mathcal O(u^{-3/2}),\qquad \lambda\in\partial U_0
\end{equation}
as $u\to \infty$, uniformly in $\lambda$ and in $\gamma\in[-\Gamma,\Gamma]$.  Moreover, one checks readily (from the boundedness of $\Phi_{\rm Ai}$ at the origin -- see property (d) in Appendix \ref{section:appendixAiry}) that $P_0$ is bounded at the origin.

\medskip

Next, we let $U_{-1}$ be a fixed open disk containing $-1$ {and not intersecting $U_0$}. 
Let
\begin{equation*}
h_u(\lambda)=\frac{4}{3}e^{-\pi i/2}\left((\lambda u)^{3/2}-e^{\frac{3\pi i}{2}}u^{3/2}\right),
\end{equation*}
with $\lambda^{3/2} $ positive as $\lambda\to +\infty$, and with a branch cut for $\arg \lambda=-\pi/2$. Then $h_u(-1)=0$, $h_u(\lambda)>0$ for $\lambda\in (-1,0)$, and $h_u$ is a conformal map in $U_{-1}$.
As $\lambda\to -1$,
\begin{equation}\label{asym:h}
h_u(\lambda)=2u^{3/2}(\lambda+1)-\frac{u^{3/2}}{2}(\lambda+1)^2
+\mathcal O((\lambda+1)^3).
\end{equation}
We can now specify our choice of contours $\Sigma_{\Phi,1}$ and $\Sigma_{\Phi,2}$ by requiring that $h_u$ maps $\left(\Sigma_{\Phi,1}\cup\Sigma_{\Phi,2}\right)\cap U_{-1}$ on the lines $e^{\pm \frac{\pi i}{4}}\mathbb R$.
To construct the local parametrix in $U_{-1}$,  we need another model RH problem, constructed in terms of confluent hypergeometric functions, similarly to \cite{DIK11, ItsKrasovsky, Charlier, FoulquieMartinezSousa}, and with jumps on $\mathbb R\cup e^{\pm \frac{\pi i}{4}}\mathbb R$. The relevant properties of this model RH problem are recalled in Appendix \ref{section:appendixCFH} and we write $\Phi_{\rm HG}$ for its solution.

We define the local parametrix $P_{-1}$ by
\begin{equation}\label{def:P-1}
P_{-1}(\lambda)=\widehat E_{-1}(\lambda)\Phi_{\rm HG}\left(h_u(\lambda);\beta=\frac{\gamma}{{\sqrt{2}} i}\right)e^{\frac{2}{3}(\lambda u)^{3/2}\sigma_3}e^{-{\sqrt{2}\pi}\frac{\gamma}{4}\sigma_3},
\end{equation} 
where principal branches are chosen, and $\widehat E_{-1}$ is analytic on $U_{-1}$, and is given by
\begin{equation}\label{def:E-1}
\widehat E_{-1}(\lambda)=
M(\lambda)\times\begin{cases}I, &\Im \lambda>0,\\
\begin{pmatrix}
0&1\\-1&0
\end{pmatrix},&\Im \lambda<0,\end{cases}\times h_u(\lambda)^{\frac{\gamma}{{\sqrt{2}} i} \sigma_3}e^{\frac{2}{3}i u^{3/2}\sigma_3},
\end{equation}
{where  $h_u(\lambda)^{\frac{\gamma}{{\sqrt{2}} i} \sigma_3}$ is analytic on $\mathbb C\setminus ([0,-i\infty)\cup (-\infty,-1])$}.
Then (using \eqref{eq:HGjump}) it is easily verified that $P_{-1}$ has the same jumps as $\Phi$ on $U_{-1}$, and (using \eqref{eq:HGasy}) that furthermore
\begin{equation}\label{smallness2}
P_{-1}(\lambda)M^{-1}(\lambda)=I+\mathcal O(u^{-3/2})
\end{equation}
as $u\to \infty$, uniformly in $\lambda\in \partial U_{-1}$ and $\gamma\in[-\Gamma,\Gamma]$.  Finally from the behavior of $\Phi_{\mathrm{HG}}$ near the origin (namely \eqref{eq:PhiHG-log}), one checks that $P_{-1}(\lambda)=\mathcal O(\log (\lambda+1))$ as $\lambda\to -1$. 

We now define the function
\begin{equation}\label{defR}
\Xi(\lambda)=\begin{cases}
\Phi(\lambda)M^{-1}(\lambda) &\textrm{for } \lambda\in \mathbb C \setminus \left(U_0\cup U_{-1}\right), \\
\Phi(\lambda)P_0^{-1}(\lambda) &\textrm{for } \lambda\in U_0,\\
\Phi(\lambda)P_{-1}^{-1}(\lambda) &\textrm{for } \lambda\in U_{-1}.
\end{cases}
\end{equation}

Then, {since we saw that the relevant jumps cancel and the possible isolated singularities at $-1$ and $0$ are not strong enough to be poles and thus must be removable, we deduce that} $\Xi$ satisfies the following RH problem. 
\subsubsection*{RH problem for $\Xi$}
\begin{itemize}
\item[(a)] $\Xi: \mathbb C \setminus \Sigma_\Xi \to \mathbb C$ is analytic, where $\Sigma_\Xi=\partial U_0\cup\partial U_1 \cup \Sigma_{\Phi,1}{ \cup \Sigma_{\Phi_2}}\cup  \{ (0,\infty)\setminus [U_0\cup U_{-1}]\}$, and where the orientations of $\partial U_0$ and $\partial U_1$ are taken clockwise, while the orientation of $\Sigma_{\Phi,1}{ \cup \Sigma_{\Phi_2}}\cup (0,\infty)\setminus [U_0\cup U_{-1}]$ is taken from left to right.
\item[(b)] On $\Sigma_\Xi$, $\Xi$ satisfies the jump relation
\begin{align}
\Xi_+(\lambda)&= \Xi_-(\lambda)\Delta(\lambda),
\end{align}
where 
\begin{equation}
\Delta(\lambda)=\begin{cases}
{ P_0(\lambda)M(\lambda)^{-1}} &\textrm{for } \lambda\in \partial U_0,\\
{ P_{-1}(\lambda)M(\lambda)^{-1}}&\textrm{for } \lambda\in \partial U_{-1},\\
M(\lambda) \begin{pmatrix}
1&0\\e^{-{\sqrt{2}\pi}\gamma}e^{\frac{4}{3}(\lambda u)^{3/2}}&1
\end{pmatrix}M(\lambda)^{-1}	&\textrm{for } \lambda\in \Sigma_{\Phi,1}\cup \setminus \overline{U_{-1}},
\\
M(\lambda) \begin{pmatrix}
1&0\\e^{\frac{4}{3}(\lambda u)^{3/2}}&1
\end{pmatrix}M(\lambda)^{-1}	&\textrm{for } \lambda\in \Sigma_{\Phi,2}\setminus \overline{U_0\cup U_{-1}},
\\
M(\lambda) \begin{pmatrix}
1&e^{-\frac{4}{3}(\lambda u)^{3/2}}\\0&1
\end{pmatrix}M(\lambda)^{-1}	&\textrm{for } \lambda\in (0,\infty)\setminus \overline{U_0}.
\end{cases}
\end{equation}

\item[(c)] As $\lambda \to \infty$,
\begin{equation} \Xi(\lambda)=I +\mathcal{O} (\lambda^{-1}).
\end{equation}
\end{itemize}
Then we can check, using \eqref{smallness1}, \eqref{smallness2}, and the behaviour of the exponential on the respective parts of the contours, that $\Delta(\lambda)$ is 
 close to the identity as $u\to +\infty$, uniformly in $\lambda$ on the jump contour and in $\gamma\in[-\Gamma,\Gamma]$. We have for instance the estimates
\begin{equation}
\Delta(\lambda)=\left(I+\mathcal O\left(e^{-\epsilon u^{3/2}(|\lambda|^{3/2}+1)}\right)\right),
\end{equation}
for { some $\epsilon>0$ and} $\lambda\in \Sigma_\Xi\setminus\left(\partial U_0\cup\partial U_{-1}\right)$,
 and 
\begin{equation}
\Delta(\lambda)=I+\mathcal O(u^{-3/2}),
\end{equation}
for $\lambda\in\partial U_0\cup\partial U_{-1}$.
It turns out that due to the normalization at infinity and the jump matrix being close to the identity, such problems can be solved approximately through a standard method (known as small norm analysis -- see e.g. \cite[Section 7]{DKMVZ} for details) and by this method, we have that the solution to the RH problem for $\Xi$ exists and is unique for $u$ large enough, and moreover
\begin{equation}\label{eq:smallnormPhi}
\Xi(\lambda)=I+\mathcal O((1+|\lambda|)^{-1}u^{-3/2}), \qquad \Xi'(\lambda)=\mathcal O((1+|\lambda|)^{-1}u^{-3/2}),
\end{equation}
as $u\to \infty$, uniformly for $\lambda\in \mathbb C\setminus \Sigma_\Xi$ and $\gamma\in[-\Gamma,\Gamma]$.

Recall that we have not proved yet that the solution to the RH problem for $\Phi$ exists for large enough $u$. This now follows from the fact that the RH problem for $\Xi$ has a solution
and from the fact that we can invert the transformation \eqref{defR}.

\section{Asymptotics for the Hankel determinants in the merging regime}
\label{sec:merging}

{The large $N$ asymptotics of $Y(z;t)$ relevant to {Proposition} \ref{pr:di} and {Proposition} \ref{le:di1} depend on how small $y=x_2-x_1$ is compared to $N$. We { consider in this section two situations} depending on whether $w=0$ or not. If $w=0$, we let $y<\epsilon$ for some {fixed} $\epsilon>0$, and in this section we prove Theorem \ref{th:Hankel2} for $y<\epsilon$ (the case where $y>\epsilon$ was proven in \cite{Charlier}). When $w\neq 0$, we {consider the situation} $y\leq \epsilon_N/8$, where $\epsilon_N$ is as in the definition of $\mathcal S_N$ from \eqref{domain} -- recall that in Theorem \ref{th:Hankel1} $w$ is assumed to be analytic in $\mathcal S_N$. This analysis will be important in Section \ref{section: RHseparated}, where we prove Theorem \ref{th:Hankel1}}.  

In \cite{Charlier}, large $N$ asymptotics for the RH solution $Y$ have been obtained for $y>0$ fixed (where the singularities are separated) and $w(\cdot;t)$ independent of $N$. The general strategy of our asymptotic analysis will be very similar to that in \cite[Section 4]{Charlier}, but the construction of a local parametrix near $x_1$ and $x_2$ will be different because of the fact that $y\to 0$ is allowed. The local construction is similar to the one in \cite{CK}.
An additional difficulty for us is that the function $w(\cdot;t)$ is not fixed but can depend on $N$, in particular, it can approximate a non-smooth function.

Our strategy throughout this section follows the standard approach to the asymptotic analysis of such RH problems. We perform a series of transformations which finally result in a  small-norm problem, which can be formulated as a suitable singular integral equation and solved approximately.

\subsection{Local parametrices near $\pm 1$}

Let $\widehat \delta= \delta_N$ when $w\neq 0$ and $\widehat \delta>0$ be some fixed but sufficiently small constant when $w=0$, where  $ \delta_N$ was given in the definition of the region 
$\mathcal S_N$ given by \eqref{domain}, such that the disks $B(\pm 1, \widehat \delta)$ are contained in $\mathcal S_N'\subset \mathcal S_N$ -- recall the definition of $\mathcal S_N'$ from \eqref{domain2}. 
Local parametrices in these disks can be constructed in terms of the Airy function. This is a standard construction, and we refer to \cite[Section 4.6 and Section 4.7]{Charlier} for details. We only need to pay attention to the fact that our disks shrink as $N\to\infty$.  All we need to know about the parametrices near $\pm 1$, is that they have the following properties.

\subsubsection*{RH conditions for $P$}
\begin{itemize}
\item[(a)] $P:B(\pm 1,\widehat \delta)\setminus\left(\mathbb R\cup \omega_1\cup\overline{\omega}_1\cup \omega_2\cup\overline{\omega}_2\right)\to\mathbb C^{2\times 2}$ is analytic,
\item[(b)] $P$ satisfies the same jump relations as $S$ on $B(\pm 1,\widehat \delta)\cap\left(\mathbb R\cup \omega_1\cup\overline{\omega}_1\cup \omega_2\cup\overline{\omega}_2\right)$,
\item[(c)] for $z\in\partial B(\pm 1,\widehat \delta)$, we have the matching condition
\begin{equation}\label{eq:matchingAiry}
P(z)P^\infty(z)^{-1}=P^\infty(z)\left(I+\O(\widehat \delta^{-3/2}N^{-1})\right){ P^{\infty}(z)^{-1}},
\end{equation}
as $N\to\infty$, where the error is uniform for all $|z\pm1| = \widehat \delta$,  
see \cite[equations  (4.58), (4.59), (8.2)]{Charlier} and also uniform in $w$ in the sense of Theorem \ref{th:Hankel1} and $P^{\infty}$ is the global parametrix (Section~\ref{sec:global}),
\item[(d)] as $z\to \pm 1$, $P(z)$ remains bounded.
\end{itemize}
In order to make the matching condition \eqref{eq:matchingAiry} more explicit, we make the important observation that {uniformly  in $w$ (in the same sense as above),} ${(P^\infty(z))^{\pm 1}}=\O(\widehat \delta^{-1/4})$ as $N\to\infty$, $|z\pm 1|=\widehat \delta$. This follows from \eqref{def:Pinfty}, \eqref{def:D}, and Lemma \ref{lemma:boundw}. We thus see that 
\begin{equation}\label{eq:matchingAiry2}
P(z)P^\infty(z)^{-1}=I+\mathcal O(\widehat \delta^{-2}N^{-1})
\end{equation}
 as $N\to\infty$, uniformly for $|z\pm1| = \widehat \delta$ and uniformly in the relevant $w$. Moreover, this is uniform in $x_1$ in a fixed compact subset of $(-1,1)$ and $0\leq y\leq \widehat \epsilon/8$,
 where we define $\widehat \epsilon=\epsilon_N$ for $w\neq 0$ and we let $\widehat \epsilon>0$ be fixed and sufficiently small for $w=0$.

Note that we have not explicitly specified our lenses inside $B(\pm 1,\widehat \delta)$; this is also done in a standard manner and we refer to \cite[Section 4.6 and Section 4.7]{Charlier} for further details.

\subsection{Parametrix near $x_1$ and $x_2$}

For the construction of the local parametrix, near $x_1$ and $x_2$, we will work in the disk  $B(x_1,\widehat \epsilon/4)$ which contains both $x_1$ and $x_2$ and which is contained in $\mathcal S_N'$. 
For the construction of the local parametrix, we follow the approach from \cite{CK, CF16}, and in particular we will use the function $\widehat\Psi$ given in \eqref{def:Psihat}.

We search for a parametrix of the form
\begin{equation}\label{def:P}
P(z)=E_{N,y}(z)\widehat\Psi(\lambda_y(z);s_{N,y})e^{-N\xi(z)\sigma_3}e^{-\frac{w(z;t)}{2}\sigma_3},
\end{equation}
where $\xi$ is given by \eqref{def:xi}; the matrix-valued function $E_{N,y}$ is analytic in $B({x_1},\widehat \epsilon/4)$ and will be constructed below; $s_{N,y}$ is given by
\begin{equation}\label{def:s}
s_{N,y}=2N(\xi_+(x_2)-\xi_+(x_1))=2Ny\xi_+'(x_1)+\mathcal O(Ny^2)=-2\pi iNy\sigma_V(x_1)+\mathcal O(Ny^2),
\end{equation}
uniformly in the limit where $N\to\infty$ and $0\leq y\leq\epsilon$ for $\epsilon>0$,
with $\sigma_V$ being the density of $\mu_V$, by \eqref{eq:id g mu}--\eqref{eq:id xi g}; 
$\lambda_y$ is a conformal map from a neighbourhood of $x_1$ to a neighbourhood of $0$.
More precisely, $\lambda_y$ maps $x_1$ to $0$ and $x_2$ to $1$, and it is given by
\begin{equation}\label{def:lambda}
\lambda_y(z)=\pm\frac{2N}{s_{N,y}}(\xi(z)-\xi_\pm(x_1)),\qquad \pm \Im z>0.
\end{equation}
First, it is straightforward to verify that $\lambda_y$ is continuous across the real line, and hence analytic in $B({x_1},\widehat \epsilon/4)$. Secondly, its derivative at $x_1$ is given by
\begin{equation}\label{eq:lambdader}
\lambda_y'(x_1)=\frac{2N}{s_{N,y}}\xi_{+}'(x_1)= \frac{1}{y}+\mathcal O(1), 
\end{equation}
again in the limit $N\to\infty$, $0\leq y\leq \epsilon$ for $\epsilon>0$.
This conformal map allows us to specify the lens-shaped jump contours ${ \omega_1,\omega_2}$ for $S$ in the vicinity of $x_1$ and $x_2$: we take them inside the disk $B({x_1},\widehat \epsilon/4)$ such that $\lambda_y$ maps them to the jump contour $\widehat\Gamma$ of $\widehat\Psi$.

The function $E_{N,y}$ is finally given by
\begin{equation}\label{def:E}
E_{N,y}(z)=
P^{\infty}(z)e^{\frac{w(z;t)}{2}\sigma_3}\widehat\Psi^\infty(\lambda_y(z))^{-1}e^{-N\xi_+(x_1)\sigma_3}.
\end{equation}
By the jump relations for $P^\infty$, the asymptotic behavior of $P^\infty$ at $x_1$ and $x_2$, and by \eqref{Pinftyhat}, it is straightforward to verify that $E_{N,y}$ has no jumps or non-removable singularities and hence is analytic in $B({x_1},\widehat \epsilon/4)$.

Before discussing the RH problem satisfied by $P$, we will state some facts about the asymptotics of $\widehat{\Psi}$ as well as ${ E_{N,y}}$. In particular, we will need these to check that $P$ agrees with $P^\infty$ to leading order on $\partial B(x_1,\widehat \epsilon/4)$. The asymptotics of ${ E_{N,y}}$ we need here are rather simple and those of $\widehat{\Psi}$ follow from asymptotics of $\Psi$ which have been studied extensively in \cite{CK}. In the proof of the following lemma, we will indicate where to find the relevant claims in \cite{CK}.

\begin{lemma}\label{le:psihatasy}
The following asymptotics hold uniformly as $N\to\infty$ and $0\leq y\leq \widehat \epsilon/8$.
\begin{itemize}
\item[$1.$] For $z\in \overline{B(x_1,\widehat \epsilon/4)}$, we have
$
E_{N,y}(z)^{\pm 1}=\mathcal O (1)
$
{uniformly in the relevant $w$ and uniformly in $0\leq y \leq \widehat \epsilon/8$ and $x_1$ in a fixed compact subset of $(-1,1)$.} \item[$2.$]   As $N\to \infty$,
$$
\widehat{\Psi}(\lambda_y(z);s_{N,y})e^{\mp \frac{s_{N,y}\lambda_y(z)}{2}\sigma_3}\widehat{\Psi}^\infty(\lambda_y(z);s_{N,y})^{-1}=I+\mathcal O\left(\frac{1}{N \widehat \epsilon}\right),
$$
 uniformly for $z\in \partial B(x_1,\widehat \epsilon/4)$, {$0\leq y \leq \widehat \epsilon/8$, and $x_1$ in a fixed compact subset of $(-1,1),$ with} $\pm\Im \lambda_y(z)>0$.
\end{itemize}
\end{lemma}
\begin{proof}
The first claim follows simply from noting that Lemma \ref{lemma:boundw} implies that $[P^\infty]^{\pm 1}(z)$ (see \eqref{def:D}) are bounded uniformly in all {the} relevant parameters {(including $w$)}. Moreover from the definition of $\widehat{\Psi}^\infty$ -- \eqref{Pinftyhat} -- and noting that $\xi_+(x_1)$ is imaginary (see e.g. \eqref{formulaxi}), we see that also the last two terms in \eqref{def:E} along with their inverses are bounded uniformly in everything relevant.

The second claim is more involved. Let us first translate it into a statement about $\Psi$.  As in the proof of Lemma \ref{pr:Psihat}, one can check using \eqref{def:Psihat} that the statement is equivalent to proving that 
$$
\Psi(-2i \lambda_y(z)+i;s_{N,y})e^{\frac{is_{N,y}}{4}(-2i\lambda_y(z)+i)\sigma_3}\Psi^\infty(-2i\lambda_y(z)+i)^{-1}=I+\mathcal O\left(\frac{1}{N \widehat \epsilon}\right).
$$

To do this, let us fix a sufficiently small positive constant $c$ and a sufficiently large positive constant $C$.
We first consider the situation where $Ny>C$, which translates into $|s_{N,y}|$ being large. Now it follows\footnote{Actually to see that this is equivalent to the asymptotics we are interested in, we must check that in the notation of \cite[Section 5]{CK}, for $z\in \partial B(x_1,\widehat \epsilon/4)$, $-2i\lambda_y(z)+i$ is not in $\overline{\mathcal U_1\cup\mathcal U_2}$ nor in the set delimited by $\Gamma_7'$ and $\Gamma_7''$. The condition of $\mathcal U_1$ and $\mathcal U_2$ is ensured by taking these sets small enough since for $z\in \partial B(x_1,\widehat \epsilon/4)$, $|\lambda_y(z)|\approx \frac{\widehat \epsilon}{4y}\geq 2$. The condition about not being between $\Gamma_7'$ and $\Gamma_7''$ can be checked with a similar estimate.} from \cite[Section 5]{CK} -- see in particular the discussion leading to \cite[(5.25)]{CK} -- that in this regime, 

\begin{align*}
\Psi(-2i \lambda_y(z)+i;s_{N,y})e^{\frac{is_{N,y}}{4}(-2i\lambda_y(z)+i)\sigma_3}&\Psi^\infty(-2i\lambda_y(z)+i)^{-1}\\
&=I+\mathcal O(|s_{N,y}|^{-1}(1+|\lambda_y|)^{-1})\\
&=I+\mathcal O(|s_{N,y}\lambda_y|^{-1})\\
&=I+\mathcal O((N\widehat \epsilon)^{-1}).
\end{align*}

In the situation where $c<Ny<C$, we have that $s=s_{N,y}$ is bounded and purely imaginary. It is a general fact\footnote{ Similarly to the fact about solutions of RH problems related to Painlev\'e equations being meromorphic in $s$, there unfortunately is not a general proof of such a claim in the literature. Again in \cite[Appendix A]{FIKN}, one can find a proof of a related fact. Moreover, similarly to the claim about smoothness in $s$, this can be reduced to analysis of the resolvent of a suitable Cauchy operator. Such an argument in a closely related problem (namely root singularities instead of jump singularities) can also be found in \cite[Appendix B]{Webb18}.} about RH problems related to Painlev\'e equations that the asymptotics in \eqref{Psi as} are actually uniform when $s$ is in compact subsets of the domain of analyticity of the corresponding Painlev\'e solution.

We can conclude that in this regime, 
$$
\Psi(-2i \lambda_y(z)+i;s_{N,y})e^{\frac{is_{N,y}}{4}(-2i\lambda_y(z)+i)\sigma_3}\Psi^\infty(-2i\lambda_y(z)+i)^{-1}=I+\mathcal O(|\lambda_y(z)|^{-1}).
$$
Now we have $|\lambda_y(z)|^{-1}=\mathcal O([\widehat \epsilon/y]^{-1})=\mathcal O((N\widehat \epsilon)^{-1})$ and we find the claim in this regime as well. 

Finally for $Ny<c$, or small $s$, we make use of results in \cite[Section 6]{CK}. Here we note that $|\lambda_y(z)|$ is uniformly at least of order $N^{\kappa}$ on $\partial B(x_1,\widehat \epsilon/4)$. Thus we are interested in the large $\lambda$ and small $|s|$ asymptotics of $\Psi(\lambda,s)$. Using in particular\footnote{To be precise, \cite[(6.29)]{CK} doesn't hold in our situation, since in the notation of \cite{CK}, we have $\alpha_1=\alpha_2=0$. Using \cite[(6.27)]{CK}, we see that to us, the only difference is that in \cite[(6.29)]{CK}, $H^{(1)}(\lambda)$ is replaced by a term that is $\mathcal O(|\lambda|^{-1}|s|\log |s|)$ and similarly one has logarithmic corrections to the subleading terms.} \cite[(6.28)]{CK}, \cite[(6.29)]{CK} and \cite[(4.6)]{CK}, we have for $z\in \partial B(x_1,\widehat \epsilon/4)$, 
\begin{align*}
\Psi(-2i \lambda_y(z)+i&;s_{N,y})e^{\frac{is_{N,y}}{4}(-2i\lambda_y(z)+i)\sigma_3}\Psi^\infty(-2i\lambda_y(z)+i)^{-1}\\
&=e^{\frac{s_{N,y}}{4}\sigma_3}(I+\mathcal O(|\lambda_y(z)s_{N,y}|^{-1}))\\
&=I+\mathcal O\left(\frac{1}{N \widehat \epsilon}\right).
\end{align*}
{All estimates satisfied the required uniformity, so this} concludes the proof.\end{proof}

We can now discuss the RH problem satisfied by $P$.

\begin{lemma}\label{pr:P}The function $P$ defined in \eqref{def:P} satisfies the following RH conditions:
\begin{itemize}
\item[$(a)$] $P:B({x_1},\widehat \epsilon/4)\setminus\left(\mathbb R\cup { \omega_{1}\cup\overline{\omega_{1}}\cup \omega_{2}\cup\overline{\omega_{2}}}\right)\to\mathbb C^{2\times 2}$ is analytic,
\item[$(b)$] $P$ satisfies the same jump relations as $S$ in $B({x_1},\widehat \epsilon/4)\setminus\left(\mathbb R\cup { \omega_{1}\cup\overline{\omega_{1}}\cup \omega_{2}\cup\overline{\omega_{{2}}}}\right)$,
\item[$(c)$] for $z\in\partial B({x_1},\widehat \epsilon/4)$, we have the uniform matching condition $($uniform in all the relevant parameters, including $w)$ 
\begin{equation*}
P(z)P^\infty(z)^{-1}=I+\mathcal O\left (\frac{1}{\widehat \epsilon N}\right), 
\end{equation*}
as $N\to\infty$ and  $y\leq \widehat \epsilon/8$
\item[$(d)$] as $z\to x_j$, $j=1,2$, we have
\begin{equation*}
P(z)=\mathcal O(\log(z-x_j)).
\end{equation*}
\end{itemize}
\end{lemma}

\begin{proof}
Condition (a) is valid because of the choice of the contours $\Omega_1, \Omega_2$ inside $B({x_1},\widehat \epsilon/4)$. The jump relations in condition (b) are directly verified using \eqref{def:P} along with the jump relations of $\widehat{\Psi}$ and $\xi$.
For the matching condition (c), we first note that using \eqref{def:E}, a short calculation (and recalling \eqref{eq:id xi g}) shows that for $\pm \Im \lambda_y(z)>0$
{
\begin{align*}
P(z)P^\infty(z)^{-1}&=E_{N,y}(z) \widehat{\Psi}(\lambda_y(z);s_{N,y})e^{-N(\xi(z)\mp \xi_+(x_1))\sigma_3}\widehat{\Psi}^\infty(\lambda_y(z);s_{N,y})^{-1} E_{N,y}(z)^{-1}\\
&=E_{N,y}(z) \widehat{\Psi}(\lambda_y(z);s_{N,y})e^{\mp\frac{1}{2}s_{N,y}\lambda_y(z)\sigma_3}\widehat{\Psi}^\infty(\lambda_y(z);s_{N,y})^{-1} E_{N,y}(z)^{-1}
\end{align*}
since $\xi_{+}(x_1)=-\xi_-(x_1)$.}
Thus an application of Lemma \ref{le:psihatasy} yields for $z\in \partial B(x_1,\widehat \epsilon/4)$

$$
P(z)P^\infty(z)^{-1}=E_{N,y}(z) \left(I+\mathcal O\left(\frac{1}{\widehat \epsilon N}\right)\right)E_{N,y}(z)^{-1}=I+\mathcal O\left(\frac{1}{\widehat \epsilon N}\right)
$$
{with the required uniformity. This} is precisely the matching condition (c).

Condition (d) is easily verified using \eqref{def:P} and condition (d) in the RH problem for $\widehat\Psi$.
\end{proof}
We now turn to solving asymptotically our initial RH problem.

\subsection{Small-norm RH problem}
We now define
\begin{equation*}
R(z)=\begin{cases}
S(z)P(z)^{-1},& z\in { B(x_1,\widehat \epsilon/4)\cup  B(-1,\widehat \delta)\cup B(1,\widehat \delta)},\\
S(z)P^{\infty}(z)^{-1},& z\in\mathbb C\setminus\left(B(x_1,\widehat \epsilon/4)\cup B(-1,\widehat \delta)\cup B(1,\widehat \delta)\right).
\end{cases}
\end{equation*}

It is straightforward to check that $R$ satisfies the following RH problem, using the RH conditions for $S$, $P$, and $P^\infty$. 
\subsubsection*{RH problem for $R$}
\begin{itemize}
\item[(a)] $R:\mathbb C\setminus\Sigma_R\to\mathbb C^{2\times 2}$ is analytic, with $\Sigma_R$ as in Figure \ref{fig:R},
\item[(b)] If we orient the disks in the clockwise manner and the remaining parts of $\Sigma_R$ from left to right, we see that for $z\in\Sigma_R$, $R$ satisfies jump conditions of the form $R_+(z)=R_-(z)J_R(z)$, where 
\begin{equation*}
J_R(z)=\begin{cases}
P(z)P^{\infty}(z)^{-1},&z\in { \partial B(x_1,\widehat \epsilon/4)\cup \partial B(-1,\widehat \delta)\cup \partial B(1,\widehat \delta)},\\
P^{\infty}(z)J_S(z)P^{\infty}(z)^{-1},&z\in \Sigma_R\setminus \left(B(x_1,\widehat \epsilon/4)\cup B(-1,\widehat \delta)\cup B(1,\widehat \delta)\right),
\end{cases}
\end{equation*}
\item[(c)] as $z\to\infty$, we have the asymptotics
\begin{equation*}
R(z)=I+\mathcal O(z^{-1}).
\end{equation*}
\end{itemize}

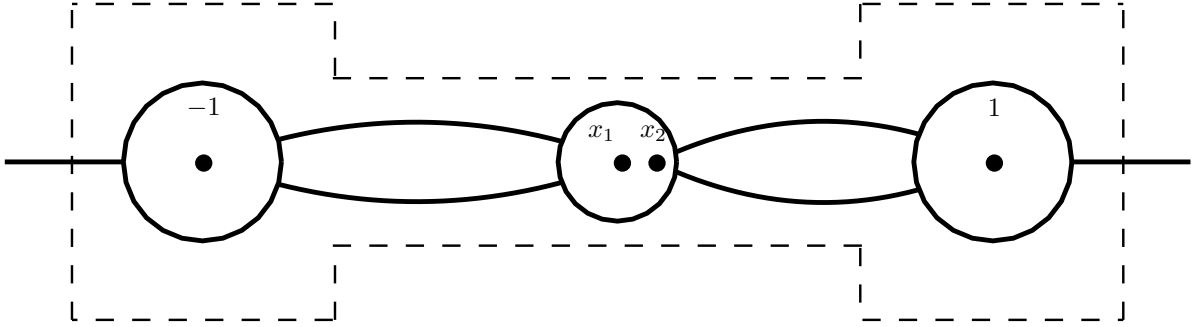
\begin{figure}
\begin{tikzpicture}[scale=2.3]
    \begin{axis}[ticks=none,axis line style={draw=none},
        unit vector ratio*=1 1 1,       
        xmin=-1.5,xmax=1.5,
        ymin=-.41,ymax=.41]

        \addplot [black,dashed,domain=-147:-90] ({1+0.4*cos(147)}, {0.4*sin(x)});
        \addplot [black,dashed,domain=90:147] ({1+0.4*cos(147)}, {0.4*sin(x)});
        \addplot [black,dashed,domain=33:90] ({-1+0.4*cos(33)}, {0.4*sin(x)});
        \addplot [black,dashed,domain=270:327] ({-1+0.4*cos(33)}, {0.4*sin(x)});

        \addplot [black,dashed,domain=-0.67:0.67]({x},{0.212});
        \addplot [black,dashed,domain=-0.67:0.67]({x},{-0.212});    

        \addplot [black,dashed,domain=0.67:1.33]({x},{0.4});
        \addplot [black,dashed,domain=0.67:1.33]({x},{-0.4});    
        \addplot [black,dashed,domain=0.67:1.33]({-x},{0.4});
        \addplot [black,dashed,domain=0.67:1.33]({-x},{-0.4});    

        \addplot [black,dashed,domain=-0.4:0.4]({1.33},{x});
        \addplot [black,dashed,domain=-0.4:0.4]({-1.33},{x});

        \addplot [black,thick,domain=-0.81:-0.09]({x},{-0.36*x^2-0.33*x+0.025});    
        \addplot [black,thick,domain=0.2:0.81]({x},{-0.56*x^2+0.64*x-0.08});    
        \addplot [black,thick,domain=-0.81:-0.09]({x},{0.36*x^2+0.33*x-0.025});    
        \addplot [black,thick,domain=0.2:0.81]({x},{0.56*x^2-0.64*x+0.08});    
        \addplot [black,thick,domain=-1.5:-1.2]({x},{0});    
        \addplot [black,thick,domain=1.2:1.5]({x},{0});

        \addplot [black,thick,domain=0:360] ({1+0.2*cos(x)}, {0.2*sin(x)});
        \addplot [black,thick,domain=0:360] ({-1+0.2*cos(x)}, {0.2*sin(x)});

        \addplot [black,thick,domain=0:360] ({0.05+0.15*cos(x)}, {0.15*sin(x)});

    \end{axis}
\fill [color=black] (3.77,0.93) circle (0.05);
\node at (3.75,1.1) {$x_2$};
\fill [color=black] (3.57,0.93) circle (0.05);
\node at (3.45,1.1) { $x_1$};
\fill [color=black] (5.72,0.93) circle (0.05);
\node at (5.72,1.25) {$1$};
\fill [color=black] (1.15,0.93) circle (0.05);
\node at (1.15,1.25) { $-1$};

\end{tikzpicture}
\caption{A characterization of the jump contour $\Sigma_R$ and the set $\mathcal S_N'$ (dashed).}\label{fig:R}
\end{figure}

We now verify that the jump matrices of $R$ are close to the identity for large $N$. For this, we need an a posteriori condition on the shape of the lenses, namely we require that
$|\Im z|>\frac{\widehat \epsilon}{8}$ for every $z\in\Sigma_R\setminus\left(\partial B(x_1,\widehat \epsilon/4)\cup \partial B(-1,\widehat \delta)\cup \partial B(1,\widehat \delta){\cup \R}\right)$. {Also for later purposes, we will find it convenient to require that also $|\Im z|\leq \frac{\widehat \epsilon}{4}$ for $z\in\Sigma_R\setminus\left(\partial B(-1,\widehat \delta)\cup \partial B(1,\widehat \delta)\right)$.} Note that {these requirements are} compatible with the size of the domain $\mathcal S_N'$ given in \eqref{domain2} for $w\neq 0$. 
{
We denote $\Sigma_R=\Sigma_{R,{x_1}}\cup \Sigma_{R,\pm 1}\cup \Sigma_{R,{\rm exp}}$, where 
\begin{equation} \Sigma_{R,x_1}=\partial B(x_1,\widehat \epsilon/4),\qquad \Sigma_{R,\pm 1}=\partial B(1,\widehat \delta)\cup \partial B(-1,\widehat \delta),
\end{equation}
and where $\Sigma_{R,{\rm exp}}$ is the remaining part of the contour, namely the union of the edges of the lenses outside the local parametrices and $(-\infty,-1-\widehat \delta)$ and $(1+\widehat \delta, \infty)$.

We denote 
\begin{equation}
\begin{cases}
J_R(z)=J_{R,x_1}(z)&{\rm for }\, z \in  \Sigma_{R,x_1},\\
J_R(z)=J_{R,\pm 1}(z)&{\rm for }\, z \in  \Sigma_{R,\pm 1},\\
J_R(z)=J_{R,{\rm exp}}(z)&{\rm for }\, z \in  \Sigma_{R,{\rm exp}}.\\
\end{cases}
\end{equation}

\begin{lemma}\label{prop:jump2}As $N\to\infty$ and $y\leq \widehat \epsilon/8$, 
\begin{equation}\label{eq:estimateJR}\begin{aligned}
J_{R,x_1}(z)&=I+\mathcal O\left(\frac{1}{N\widehat \epsilon}\right)\\
J_{R,\pm 1}(z)&=I+\mathcal O\left(\frac{1}{N\widehat \delta^{2}}\right)\\
J_{R,{\rm exp}}(z)&=
I+\mathcal O\left(\frac{1}{e^{N^\kappa}(|z|^2+1)}\right),&
\end{aligned}
\end{equation}
uniformly {in the relevant $w$, uniformly in $x_1$ in a fixed compact subset of $(-1,1)$, uniformly in $0\leq y\leq \widehat \epsilon/8$, and uniformly in} $z\in \Sigma_{R,{x_1}},\Sigma_{R,\pm 1}, \Sigma_{R,{\rm exp}}$ respectively,  for some $\kappa>0$.
\end{lemma}
}

\begin{proof}

Let us consider first a point $z\in \Sigma_{R,{\rm exp}}$. Since we have chosen the lenses such that $|\Im z|>\widehat \epsilon/8$, we can use Proposition \ref{pr:lensbound1} to conclude that on this part of $\Sigma_R$, the jump matrix of $S$ is $I+\mathcal O(e^{-N^\epsilon \log|z|})$ for some $\epsilon>0$  as $N\to\infty$. As $[P^\infty]^{\pm 1}(z)=\mathcal O(\widehat \delta^{-1/4})$ (using Lemma \ref{lemma:boundw} in a similar way as before) on this part of the contour, 
we have {(uniformly in all of the relevant parameters)} \[J_R(z)=I+\mathcal O\left(\widehat \delta^{-1/2}e^{-N^\epsilon \log|z|}\right)=I+\mathcal O\left(\frac{1}{e^{-N^{\kappa}}(|z|^2+1)}\right)\]
as $N\to\infty$ for some $\kappa>0$.

Next, we consider $z\in \partial B(\pm 1,\widehat \delta)$. Here the result follows from \eqref{eq:matchingAiry2}. 

Finally, for $z\in\partial B(x_1,\widehat \epsilon/4)$ the result follows from the matching condition (c) in the RH problem for $P$, see Lemma \ref{pr:P}.
\end{proof}

As already briefly mentioned at the end of the previous section, if a RH problem, normalized to $I$ at infinity, has jumps which are uniformly close to the identity matrix on jump contours which are fixed, 
there are standard arguments to conclude that the solution to the RH problem is also uniformly close to $I$ -- see e.g. \cite[Section 7]{DKMVZ}. Since we are able to choose fixed (independent of $N$, $x_1$, $y$) jump contours for $t=0$ (equivalent to $w=0$), these standard arguments allow us to conclude the following -- we omit further details.

\begin{proposition}\label{prop:R0}For $t=0$, or equivalently $w=0$,  we have
\begin{equation}\label{eq:asR0}
R(z)=I+\mathcal O\left(\frac{1}{N(|z|+1)}\right),\qquad R'(z)=\mathcal O\left(\frac{1}{N(|z|+1)}\right),
\end{equation}
uniformly for $z\in\mathbb C\setminus\Sigma_R$ {and uniformly in $x_1$ in a fixed compact subset of $(-1,1)$ and uniformly in $y\leq \widehat \epsilon/8$} as $N\to\infty$.
\end{proposition}

For $t> 0$, the situation is more complicated because the domain $\mathcal S_N$ may shrink as $N\to\infty$, and consequently the jump contour $\Sigma_R$ may depend on $N$. For that reason, the standard arguments to conclude that $R$ is close to the identity matrix do not apply directly, and we need to adapt them to our situation.

We will only need asymptotics for $R(z)$ and $R'(z)$ for $z$ outside the domain $\mathcal S_N'$. For the $t> 0$ case, the estimate we have is the following.

\begin{proposition}\label{prop:R}
For $t\in(0,1]$ and $z\in \mathbb C\setminus \left(\mathcal S_N' \cup\Sigma_R\right)$,  we have
\begin{equation}\label{eq:asR}\begin{aligned}
R(z)=I+\mathcal O\left(N^{-1}\left(|z-x_1|^{-1}+\delta_N^{-1}|z-1|^{-1}+\delta_N^{-1}|z+1|^{-1}\right)\right),\\ 
 R'(z)=\mathcal O\left(\frac{1}{N \epsilon_N }\left(|z-x_1|^{-1}+\delta_N^{-1}|z-1|^{-1}+\delta_N^{-1}|z+1|^{-1}\right)\right),
\end{aligned}
\end{equation}
{where the implied constants are uniform in $z,t,x_1,y$ and $w$.}
\end{proposition}

\begin{proof}As stated above, the basic idea of the proof is similar to the case of fixed contours. The proof consists of three steps. The first step of the proof, which  relies on the Sokhotski-Plemelj formula as in the fixed contour case, is to express the solution of the RH problem in terms of the boundary values of the solution and the jump matrices. The second step, which does differ from the fixed contour case, is to obtain uniform estimates for the boundary values -- this is done through bounds on the jump matrices combined with an adaptation of a well known contour deformation argument introduced in \cite[Section 7]{DKMVZ}.  Once these estimates are obtained, the final step is to use them to obtain estimates for the whole solution.

\underline{Step 1: expressing the solution in terms of the boundary values.} The RH problem for $R$ implies that
\[R_+(z)-R_-(z)=R_-(z)(J_R(z)-I),\qquad z\in\Sigma_R,\]
and by the Sokhotski-Plemelj formula, it follows that we have the integral equation
\begin{equation}\label{eq:Rint}R(z)=I+\frac{1}{2\pi i}\int_{\Sigma_R}R_-(\xi)(J_R(\xi)-I)\frac{d\xi}{\xi-z},\qquad z\in\mathbb C\setminus\Sigma_R.\end{equation}

\underline{Step 2: estimates for the boundary values.} This second step is the most involved one. 
We begin by defining a sufficiently narrow neighborhood $\mathcal U\subset \mathcal S_N$ of the jump contour $\Sigma_R$  which contains all $z{\in \mathcal S_N}$ which are such that ${\rm dist}(z,\Sigma_R)<\frac{1}{16}\epsilon_N$. Recall that {we required that $\frac{\epsilon_N}{8}<|\Im z|\leq \frac{\epsilon_N}{4}$} for all $z\in({ \omega_{1}\cup\omega_{2}\cup\overline{\omega}_1\cup\overline{\omega}_2})\setminus(B(-1,\delta_N)\cup B(1,\delta_N)\cup B(x_1,\epsilon_N/4))$.

The point of introducing this neighborhood is that we are able to deform the integration contour in \eqref{eq:Rint} within this domain. More precisely, we first note that looking at the jumps of the various parametrices as well as those of $S$, one readily verifies that the jump matrices $J_{R, x_1}$, $J_{R,\pm 1}$ and $J_{N,\exp}$ can be extended to analytic functions in small $\epsilon_N/16$ neighborhoods $\mathcal U_{x_1}, \mathcal U_{\pm 1}$ and $\mathcal U_{\exp}$ of their respective contours $\Sigma_{R,x_1}$, $\Sigma_{R,\pm 1}$, and $\Sigma_{R,\exp}$.  To be able to deform the contour, we also introduce analytic continuations of $R_-$: let the symbol $\star$ denote $x_1, \pm 1$, or $\exp$, and define $R_{-,\star}(\xi)=R(\xi)$ for $\xi$ at the negative side of the contour $\Sigma_{R,\star}$, and $R_{-,\star}(\xi)=R(\xi)J_{R,\star}(\xi)^{-1}$ for $\xi$ at the positive side, for $\xi\in\mathcal U_\star$.

\medskip

For any $z\in\mathcal U$, by analyticity of $R$ and $J_R$, we can deform the integration over $\Sigma_R$ in \eqref{eq:Rint} to integration over a slightly deformed contour $\widetilde\Sigma_R(z)$ near $z$, which lies in $\mathcal U$, and which we can always choose in such a way that
 \begin{equation}\label{eq:deformedcontour}
{\rm dist}(z,\widetilde \Sigma_{R}(z))\geq \frac{1}{32}\epsilon_N,\end{equation}
and in such a way that $\Sigma_R - \widetilde\Sigma_R(z)$ does not wind around $z$.
For instance, we can take
$\widetilde\Sigma_R(z)$ equal to $\Sigma_R\setminus B\left(z,\frac{1}{32}\epsilon_N\right)$ complemented if necessary with part of the circle $\partial B\left(z,\frac{1}{32}\epsilon_N\right)$ -- see Figure \ref{fig:deform}.  
Similarly to $\Sigma_R$, we can decompose the deformed contour $\widetilde\Sigma_R$ in four parts: three 
closed curves $\widetilde\Sigma_{R,x_1}$, $\widetilde\Sigma_{R,\pm 1}$ and the remaining parts $\widetilde\Sigma_{R,\exp}$ which lie outside these closed curves.

The integral equation \eqref{eq:Rint} can now be rewritten as
\begin{equation}\label{eq:Rintegral2}
R(z)=I+\frac{1}{2\pi i}\sum_{\star\in\{-1,1, x_1, \exp\}}\int_{\widetilde\Sigma_{R,\star}(z)}R_{-,\star}(\xi)(J_{R,\star}(\xi)-I)\frac{d\xi}{\xi-z}.
\end{equation}

The point of the representation \eqref{eq:Rintegral2} is that we automatically have an upper bound for the denominator, and  tracing through the proofs, one finds that the estimates for $J_{R,x_1}$ and $J_{R,\pm 1}$ from Proposition \ref{prop:jump2} remain valid on such a deformed contour $\widetilde\Sigma_R(z)$, while the deformed version of the contour $\Sigma_{R,\exp}$ is such that we can still apply Lemma \ref{pr:lensbound1}, which implies that the estimate for $J_{R,\exp}$ in Proposition \ref{prop:jump2} also holds on the deformed contour. Let us now see how this yields the estimates for $R_-$ that we were after in this part of the proof.

\begin{figure}
\begin{tikzpicture}[scale=2.3]
    \begin{axis}[ticks=none,axis line style={draw=none},
        unit vector ratio*=1 1 1,       
        xmin=-1.5,xmax=1.5,
        ymin=-.41,ymax=.41]

       \addplot [thick,domain=120:240] ({1+0.4*cos(x)}, {0.4*sin(x)});
       \addplot [thick,domain=-60:60] ({-1+0.4*cos(x)}, {0.4*sin(x)});
      
       \addplot [thick,domain= 0.46:0.61] ({x},{-0.5*x^2+0.3});
       \addplot [thick,domain=-0.61:0.35] ({x},{-0.5*x^2+0.3});
       \addplot [thick,domain=-0.61:0.61] ({x},{0.5*x^2-0.3});

       \addplot [thick,domain=166:331] ({0.41+0.06*cos(x)}, {0.23+0.06*sin(x)});

       \addplot [dashed,domain=120:145] ({1+0.5*cos(x)}, {0.5*sin(x)});
       \addplot [dashed,domain=215:240] ({1+0.5*cos(x)}, {0.5*sin(x)});

       \addplot [dashed,domain=-60:-35] ({-1+0.5*cos(x)}, {0.5*sin(x)});
       \addplot [dashed,domain=35:60] ({-1+0.5*cos(x)}, {0.5*sin(x)});

       \addplot [dashed,domain=-7:7] ({-1+0.5*cos(x)}, {0.5*sin(x)});
       \addplot [dashed,domain= 173:187] ({1+0.5*cos(x)}, {0.5*sin(x)});

       \addplot [dashed,domain=120:240] ({1+0.3*cos(x)}, {0.3*sin(x)});
       \addplot [dashed,domain=-60:60] ({-1+0.3*cos(x)}, {0.3*sin(x)});

       \addplot [dashed,domain=-0.55:0.55] ({x},{-0.5*x^2+0.4});
       \addplot [dashed,domain=-0.55:0.55] ({x},{0.5*x^2-0.4});

       \addplot [dashed,domain=-0.5:0.5] ({x},{-0.5*x^2+0.2});
       \addplot [dashed,domain=-0.5:0.5] ({x},{0.5*x^2-0.2});

    \end{axis}
\fill [color=black] (4.37,1.5) circle (0.03);
\node at (4.45,1.5) {$z$};
\fill [color=black] (5.72,0.93) circle (0.03);
\node at (5.72,1.1) {$1$};
\fill [color=black] (1.6,0.93) circle (0.03);
\node at (1.6,1.1) { $x_2$};
\end{tikzpicture}
\caption{A characterization of the procedure of deforming the contour (thick solid) near the point $z\in \mathcal{U}$ (boundary dashed).}\label{fig:deform}
\end{figure}
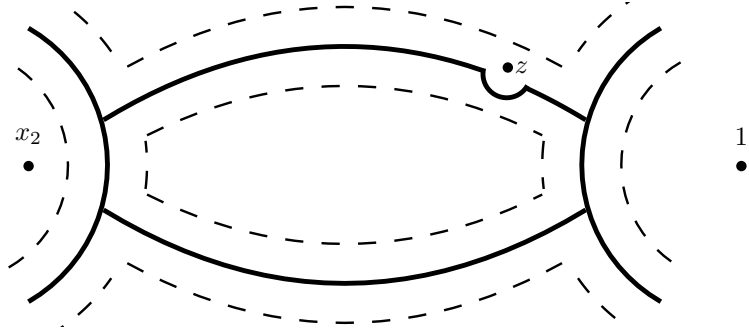

Taking a  sub-multiplicative matrix norm (normalized so that $\|I\|=1$,) of \eqref{eq:Rintegral2}  and the supremum over all $z\in\mathcal U\setminus\Sigma_R$, we get
\[\sup_{z\in\mathcal U\setminus\Sigma_R}\|R(z)\|\leq \|I\|+\frac{1}{2\pi}\sup_{z\in\mathcal U}\|R_-(z)\| \ \sup_{z\in\mathcal U}\sum_{\star\in\{-1,1, x_1, \exp\}}\int_{\widetilde\Sigma_{R,\star}(z)}\frac{\|J_{R,\star}(\xi)-I\|}{|\xi-z|}d|\xi|,\]
where we interpret 
$\sup_{z\in\mathcal U}\|R_-(z)\|$ as the maximum over $\star\in\{-1,+1,x_1,\exp\}$ of the suprema of the analytic continuations $R_{-,\star}(z)$ in $\mathcal U_\star$.

Since $J_{R,\star}$ converges uniformly {(in $z,t,$ and $w$)} to $I$ in $\mathcal U_\star$ as $N\to\infty$, we have (from the definition of the analytic continuations $R_{-,\star}(\xi)$) for $N$ sufficiently large that 
\[\sup_{z\in\mathcal U}\|R_-(z)\|\leq 2\sup_{z\in\mathcal U\setminus\Sigma_R}\|R(z)\|, \]
and we get
\[\sup_{z\in\mathcal U\setminus\Sigma_R}\|R(z)\|\leq \|I\|+\frac{1}{\pi}\sup_{z\in\mathcal U\setminus\Sigma_R}\|R(z)\| \ \sup_{z\in\mathcal U}\sum_{\star\in\{-1,1, x_1, \exp\}}\int_{\widetilde\Sigma_{R,\star}(z)}\frac{\|J_{R,\star}(\xi)-I\|}{|\xi-z|}d|\xi|.\]
 Thus,
\[\sup_{z\in\mathcal U\setminus\Sigma_R}\|R(z)\|\leq  \|I\|/\left(1-\frac{1}{\pi} \ \sup_{z\in\mathcal U}\sum_{\star\in\{-1,1, x_1, \exp\}}\int_{\widetilde\Sigma_{R,\star}(z)}\frac{\|J_{R,\star}(\xi)-I\|}{|\xi-z|}d|\xi|\right)
.\]
By \eqref{eq:deformedcontour}, $|\xi-z|=\mathcal O(\epsilon_N)$ as $\epsilon_N\to 0$. We may apply Lemma  \ref{prop:jump2} to the piecewise analytic continuation of $J_R$ (in particular, since $\epsilon_N<\delta_N$ the estimates for $J_{R,\pm1}$ also hold). Thus it follows that
\begin{equation}\label{eq:estimateintegralJR}\sup_{z\in\mathcal U}\sum_{\star\in\{-1,1, x_1, \exp\}}\int_{\widetilde\Sigma_{R,\star}(z)}\frac{\|J_{R,\star}(\xi)-I\|}{|\xi-z|}d|\xi|=\mathcal O(N^{-\kappa}),\end{equation}
uniformly {(in everything relevant)} as $N\to\infty$, for some $\kappa>0$, and it follows that $\sup_{z\in\mathcal U\setminus\Sigma_R}\|R(z)\|\leq 2$ for sufficiently large $N$. In particular, we have this bound for our boundary values $R_-$, and this concludes our second step of the proof.

\underline{Step 3: estimates for $R$.} We substitute our estimates for $R_-$ into \eqref{eq:Rint}, to find
\begin{equation}\nonumber \begin{aligned}
\|R(z)-I\|&\leq \int_{\Sigma_R}\frac{\|J_R(\xi)-I\|}{|\xi-z|}d|\xi|\\
&=\int_{\Sigma_{R,x_1}}\frac{\|J_{R,x_1}(\xi)-I\|}{|\xi-z|}d|\xi|
+\int_{\Sigma_{R,\pm1}}\frac{\|J_{R,\pm 1}(\xi)-I\|}{|\xi-z|}d|\xi|\\&
\quad +\int_{\Sigma_{R,{\rm exp}}}\frac{\|J_{R,{\rm exp}}(\xi)-I\|}{|\xi-z|}d|\xi|.
\end{aligned}
\end{equation}
For $z\in \mathbb C\setminus {\mathcal S_N',}$ with $|\Re z|\leq 1+\frac{3}{2}\delta_N$it follows {from the definition of $\mathcal S_N'$ and the fact that we required that $|\Im \xi|\leq \frac{\widehat \epsilon}{4}$ on $\Sigma_R\setminus (B(1,\widehat \delta)\cup B(-1,\widehat \delta))$,} that
\begin{equation}\nonumber
\begin{aligned}
\|R(z)-I\|&\leq \frac{3}{|z-x_1|}\int_{\Sigma_{R,x_1}}\|J_{R,x_1}(\xi)-I\|d|\xi|
\\&+\left(\frac{3}{|z-1|}+\frac{3}{|z+1|}\right)\int_{\Sigma_{R,\pm1}}
\|J_{R,\pm 1}(\xi)-I\|d|\xi|
+\frac{16\epsilon_N^{-1}}{|z|+1}\int_{\Sigma_{R,{\rm exp}}}\|J_{R,{\rm exp}}(\xi)-I\|d|\xi|.
\end{aligned}
\end{equation}
as $\epsilon_N,\delta_N\to 0$. A similar bound for $z\in \C\setminus \Sigma_R$ with $|\Re z|\geq 1+\frac{3}{2}\delta_N$ follows from substituting the bound $\sup_{z\in\mathcal U\setminus\Sigma_R}\|R(z)\|\leq 2$ into \eqref{eq:Rintegral2}. The length of $ \Sigma_{R,x_1}$ is of order $\mathcal O(\epsilon_N)$ as $\epsilon_N\to 0$, and the length of $\Sigma_{R,\pm1}$ is $\mathcal O(\delta_N)$ as $\delta_N\to 0$. 
 Thus the proof of the first part of the  proposition follows upon application of Lemma \ref{prop:jump2}.

For the estimate for $R'$, we use Cauchy's formula \[R'(z)=\frac{1}{2\pi i}\int_{\gamma} \frac{R(\xi)-I}{(z-\xi)^2}d\xi,\]
where we integrate over a circle $\gamma$  around $z$ with radius at least $c\epsilon_N$ for some fixed $c>0$, and if needed, we employ the same contour deformation argument involving $\widetilde \Sigma(z)$ as before. One finds from the bounds on $R$ that \[\|R'(z)\|\leq \frac{2\pi c\epsilon_N}{2\pi c^2\epsilon_N^2} \sup_{\xi\in\gamma}\|R(\xi)-I\|,\]
and the proposition follows.
\end{proof}

We now apply these asymptotics to studying the asymptotics of the $y$-differential identity.

\subsection{Asymptotics of the \texorpdfstring{$y$}{y}-differential identity}

We will now use our representation of $Y$ in terms of our transformations, local parametrix and $R$ to study the large $N$ asymptotics of the $y$-differential identity \eqref{eq:diffid}. Later on, we will be able to prove Theorem \ref{th:Hankel2} by integrating these asymptotics in $y$.
 Note that here we only care about the situation where $w=0$. This means that we can use the stronger estimate from Proposition \ref{prop:R0}.
In our analysis of the differential identity, we will need  further information about the behavior of $\widehat{\Psi}(\lambda_y(z);s_{N,y})$ and $E_{N,y}(z)$ as we let $z\to x_2$.

The result about $E_{N,y}$ that we will make use of is the following (recall that $E_{N,y}$ is analytic at $x_2$ so the statement makes sense).

\begin{lemma}\label{le:Ebound}
As $N\to\infty$, we have for $w=0$ and $|y|<c$ with a small enough $c>0$ $($which is independent of $N)$
$$
E_{N,y}^{-1}(x_2)E_{N,y}'(x_2)=\mathcal O(1),
$$
where the implied constant is uniform in $|y|<c$ and $x_2$ in a fixed compact subset of $(-1,1)$ and $E_{N,y}$ is as in \eqref{def:E}. 
\end{lemma}

\begin{proof}
To begin, note that by analyticity of $E_{N,y}$, we can calculate the quantity at $z$ in the lower half plane and then let $z\to x_2$. By \eqref{def:E} (in the case $w=0$), we have for $\Im z<0$ and $z$ sufficiently close to $x_2$ that
\begin{align*}
E_{N,y}^{-1}(z)E_{N,y}'(z)&={ e^{N\xi_+(x_1)\sigma_3}}
\left(\widehat\Psi^\infty(\lambda_y(z);s_{N,y})\frac{d}{dz}\widehat\Psi^\infty(\lambda_y(z);s_{N,y})^{-1} \right. \\
&\qquad \left. +\widehat\Psi^\infty(\lambda_y(z);s_{N,y})
P^{\infty}(z)^{-1}\frac{dP^{\infty}(z)}{dz}
\widehat\Psi^\infty(\lambda_y(z);s_{N,y})^{-1}
\right)
e^{-N\xi_+(x_1)\sigma_3}.
\end{align*}
We also have, by \eqref{Pinftyhat} (still for $\Im z<0$),
\[\widehat\Psi^\infty(\lambda_y(z))\frac{d}{dz}\widehat\Psi^\infty(\lambda_y(z))^{-1} = -\frac{\gamma_1}{{\sqrt{2}} i \lambda_y(z)}\lambda_y'(z)\sigma_3-\frac{\gamma_2}{{\sqrt{2}} i (\lambda_y(z)-1)}\lambda_y'(z)\sigma_3.\]
By \eqref{def:Pinfty}, it follows that
\[P^{\infty}(z)^{-1}\frac{d}{dz}P^{\infty}(z)=-(\log D(z))'\sigma_3 + D(z)^{\sigma_3}Q^{-1}(z)Q'(z)D(z)^{-\sigma_3}.\]

By \eqref{Pinftyhat} we also have (for $\Im z<0$)
\[\widehat\Psi^\infty(\lambda_y(z))
\sigma_3
\widehat\Psi^\infty(\lambda_y(z))^{-1}=\sigma_3.\]

Substituting the last three expressions into the one for $E_{N,y}^{-1}E_{N,y}'$, we obtain (for $\Im z<0$)
\begin{align*}
E_{N,y}^{-1}(z)E_{N,y}'(z)&={ e^{N\xi_+(x_1)\sigma_3}}
\left(\left[\frac{-\gamma_1 \lambda_y'(z)}{{\sqrt{2}} i \lambda_y(z)}+\frac{-\gamma_2 \lambda_y'(z)}{{\sqrt{2}} i (\lambda_y(z)-1)}  -(\log D)'(z)\right]
\sigma_3
\right.\\
&\qquad +\widehat\Psi^\infty(\lambda_y(z);s_{N,y})
D(z)^{\sigma_3}Q^{-1}(z)Q'(z)D(z)^{-\sigma_3}
\widehat\Psi^\infty(\lambda_y(z);s_{N,y})^{-1}
\Bigg)\\
&\qquad
\times\ e^{-N\xi_+(x_1)\sigma_3}.
\end{align*}
Using the fact that $\xi_+(x_1)$ is purely imaginary and the uniform (in $y,x_2$) boundedness of $[\widehat\Psi^\infty]^{\pm 1}$, $D^{\pm 1}$, $Q^{-1}$, and $Q'$ for $z$ near $x_2$, we obtain after letting $z\to x_2$,
\begin{multline}\label{eq:Enest}
E_{N,y}^{-1}(x_2)E_{N,y}'(x_2)=\left[\frac{-\gamma_1 \lambda_y'(x_2)}{{\sqrt{2}} i }+\lim_{z\to x_2} \left(\frac{-\gamma_2 \lambda_y'(z)}{{\sqrt{2}} i (\lambda_y(z)-1)}-(\log D)'(z)\right)\right]
\sigma_3 +\mathcal O(1),
\end{multline}
as $N\to\infty$, uniformly in all relevant parameters.
Using \eqref{def:D}, one can verify that {(once again uniformly in everything relevant)} 
\[(\log D)'(z)=-\frac{\gamma_1}{{\sqrt{2}} i(z-x_1)}-\frac{\gamma_2}{{\sqrt{2}} i(z-x_2)}+\mathcal O(1),\]
{ as $z\to x_2$, for $\Im z<0$.} Using the definition of $\lambda$, that is \eqref{def:lambda}, and Assumptions \ref{ass:regular} along with \eqref{formulaxi} we find that as $z\to x_2$
\[
\frac{\lambda_y'(z)}{\lambda_y(z)-1}=\frac{1}{z-x_2}+\mathcal O(1),
\]
where the implied constant is uniform in $x_1,x_2$. Thus
\[
\lim_{z\to x_2} \left(\frac{-\gamma_2 \lambda_y'(z)}{{\sqrt{2}} i (\lambda_y(z)-1)}-(\log D)'(z)\right)=\frac{\gamma_1}{\sqrt{2}i (x_2-x_1)}+\mathcal O(1)
\]	
again with an implied constant that is uniform in $x_1,x_2$. Finally using \eqref{eq:lambdader} (or more precisely, \eqref{def:lambda} to obtain the same bound at $x_2$), we see that
\[
\frac{-\gamma_1 \lambda_y'(x_2)}{{\sqrt{2}} i }+\lim_{z\to x_2} \left(\frac{-\gamma_2 \lambda_y'(z)}{{\sqrt{2}} i (\lambda_y(z)-1)}-(\log D)'(z)\right)=\mathcal O(1)
\]
with the required uniformity. Plugging this into \eqref{eq:Enest} yields the claim.
\end{proof}

We are now in a position to study the asymptotics of our differential identity and integrating it.
For simplicity, let us write $D_N(y)=D_N(x_1, x_1+y;\gamma_1,\gamma_2;0)$ in the remaining part of this section, with $D_N(x_1, x_1+y;\gamma_1,\gamma_2;0)$ defined as in \eqref{eq:FH}.
\begin{proposition}\label{pr:diasy}{ Let $w=0$.}
For $\delta>0$ small enough but fixed and $0<y<\delta$, we have as $N\to\infty$
$$
\frac{d}{dy}\log D_N(y)=\lambda_y'(x_1+y)\left(\sigma(s_{N,y})-\frac{\gamma_1+\gamma_2}{{2\sqrt{2}} i}s_{N,y}-\frac{(\gamma_1+\gamma_2)^2}{{4}}\right)+\mathcal O(1),
$$
where the implied constant is uniform in $0<y<\delta$ and $x_1\in(-1+2\delta, 1-2\delta)$, $\lambda_y$ is as in \eqref{def:lambda}, $\sigma$ is as in Proposition \ref{pr:psiderasy}, and $s_{N,y}$ is as in \eqref{def:s}.
\end{proposition}

\begin{proof}
 Unravelling our transformations, we have for $z\in B(x_1,\widehat \epsilon/4)$ 
\begin{equation*}
Y(z)=e^{-\frac{N\ell}{2}\sigma_3}R(z)P(z){ e^{Ng(z)\sigma_3}}
e^{\frac{N\ell}{2}\sigma_3}.
\end{equation*}
By the formula \eqref{def:P} for $P$ (recalling that we are considering now $w=0$), this becomes
\begin{equation*}
Y(z)=e^{-\frac{N\ell}{2}\sigma_3}R(z)E_{N,y}(z)\widehat\Psi(\lambda_y(z);s_{N,y}){ e^{N\xi(z)\sigma_3}}e^{-Ng(z)\sigma_3}
e^{\frac{N\ell}{2}\sigma_3}.
\end{equation*}
By definition \eqref{def:xi} of $\xi$, this becomes
\begin{equation*}
Y(z)=e^{-\frac{N\ell}{2}\sigma_3}R(z)E_{N,y}(z)\widehat\Psi(\lambda_y(z);s_{N,y})e^{\frac{N}{2}V(z)\sigma_3}.
\end{equation*}
This allows us to compute $\left(Y^{-1}(x_2)Y'(x_2)\right)_{2,1}$, namely the quantity appearing in the differential identity of Proposition \ref{pr:di}, identically in terms of $E_{N,y}, R, \Psi, \lambda_y, s_{N,y}$. 
We obtain
\begin{align*}
e^{-NV(x_2)}&\left(Y^{-1}(x_2)Y'(x_2)\right)_{2,1}\\
&=\lim_{z\to x_2}\left(\widehat \Psi(\lambda_y(z);s_{N,y})^{-1}E_{N,y}^{-1}(z)R^{-1}(z)R'(z)E_{N,y}(z)\widehat \Psi(\lambda_y(z);s_{N,y})\right)_{2,1}\\
&\quad +\lim_{z\to x_2}\left(\widehat \Psi(\lambda_y(z);s_{N,y})^{-1}E_{N,y}^{-1}(z)E_{N,y}'(z)\widehat \Psi(\lambda_y(z);s_{N,y})\right)_{2,1}\\
&\quad +\lambda_y'(x_2)\lim_{\lambda\to 1}\left(\widehat\Psi^{-1}(\lambda;s_{N,y})\widehat \Psi'(\lambda;s_{N,y})\right)_{2,1}.
\end{align*}

Let us approach the point $x_2$ from the lower half plane so that in terms of $\Psi$, by 
\eqref{def:Psihat}
\begin{align*}
e^{-NV(x_2)}&\left(Y^{-1}(x_2)Y'(x_2)\right)_{2,1}\\
&=\lim_{z\to x_2}e^{-{\sqrt{2}\pi}\gamma_2}\left(\Psi(-2i\lambda_y(z)+i;s_{N,y})^{-1}\mathcal O(1)\Psi(-2i\lambda_y(z)+i;s_{N,y})\right)_{2,1}\\
&\quad -2i\lambda_y'(x_2)e^{-{\sqrt{2}\pi}\gamma_2}\lim_{\epsilon\to 0^+}\left({\Psi^{-1}(-i-\epsilon;s_{N,y}) \Psi'(-i-\epsilon;s_{N,y})}\right)_{2,1},
\end{align*}
where we used Lemma \ref{le:psihatasy}, Proposition \ref{prop:R0} (recall that we are considering the case $w=0$, or equivalently $t=0$, so we do not need to rely on Proposition \ref{prop:R}, where the error term might blow up if $y$ were small enough), and Lemma \ref{le:Ebound}, along with the fact that $s$ is imaginary, to obtain the $\mathcal O(1)$-term{ -- it follows from the uniformity in the quoted lemmas that the implied constant is uniform in everything relevant.} Now noting that independently of the $\mathcal O(1)$-matrix, the first quantity here will only involve entries of the first column of $\Psi$ (note that $\det\Psi=1$). Thus applying Lemma \ref{le:psiasyat-i}, we see that the first quantity here is $\mathcal O(1)$ {(uniformly in everything relevant)}.  For the second one, we have Lemma \ref{pr:psiderasy} which shows that {(once again uniformly in everything relevant)} $$
e^{-NV(x_2)}\left(Y^{-1}(x_2)Y'(x_2)\right)_{2,1}=\frac{2\pi i\lambda_y'(x_2)}{e^{{\sqrt{2}\pi}\gamma_2}-1}\left(\sigma(s_{N,y})-\frac{\gamma_1+\gamma_2}{{2\sqrt{2}} i}s_{N,y}-\frac{(\gamma_1+\gamma_2)^2}{{4}}\right)+\mathcal O(1)
$$
for a smooth function $\sigma$ with asymptotics described by Proposition \ref{pr:psiderasy}. The differential identity from Proposition \ref{pr:di} thus becomes (using the definition of $\lambda_y$ \eqref{def:lambda})
\begin{align*}
\frac{d}{dy}\log D_N(y)&=\lambda_y'(x_2)\left(\sigma(s_{N,y})-\frac{\gamma_1+\gamma_2}{{2\sqrt{2}}i}s_{N,y}-\frac{(\gamma_1+\gamma_2)^2}{{4}}\right)+\mathcal O(1)
\end{align*}
{with the required uniformity -- this is precisely the claim.} 
\end{proof}
We are finally in a position to prove Theorem \ref{th:Hankel2}.

\subsection{Proof of Theorem \ref{th:Hankel2}}

Using Proposition \ref{pr:diasy}, we find as $N\to\infty$ with $y$ small enough,
$$
\log D_N(y)=\log D_N(0)+\int_0^y \lambda_{u}'(x_1+u)\left(\sigma(s_{N,u})-\frac{\gamma_1+\gamma_2}{{2\sqrt{2}} i}s_{N,u}-\frac{(\gamma_1+\gamma_2)^2}{{4}}\right)du+\mathcal O(1)
$$
{uniformly in the relevant $x_1$ and $y$.}  Let us fix a constant $C>0$ and split the integral into two parts: $Nu<C$ and $Nu\geq C$ (if $Ny<C$, then the integral over $Nu\geq C$ is understood as zero).

 In the first regime, we have $|s_{N,u}|=\mathcal O(1)$ so using the asymptotics \eqref{eq:sigmaasy1} along with \eqref{def:s}, we have
$$
\sigma(s_{N,u})-\frac{\gamma_1+\gamma_2}{{2\sqrt{2}} i}s_{N,u}-\frac{(\gamma_1+\gamma_2)^2}{{4}}=\mathcal O(||s_{N,u}|\log |s_{N,u}||)=\mathcal O((Nu)\log (Nu)), 
$$
and  by \eqref{eq:lambdader} (or more precisely, again by a similar argument one can prove the same bound at $x_1+u$) $\lambda_{u}'(x_1+u)=\frac{1}{u}+\mathcal O(1)$, so we see that 
\begin{equation}\label{star1}\begin{aligned}
&\int_0^{\min(y,C/N)}\lambda_{u}'(x_1+u)\left(\sigma(s_{N,u})-\frac{\gamma_1+\gamma_2}{{2\sqrt{2}} i}s_{N,u}-\frac{(\gamma_1+\gamma_2)^2}{{4}}\right)du\\
&\qquad=\mathcal O\left(\int_0^C\big|\log |u|\big|du\right)=\mathcal O(1),
\end{aligned}
\end{equation}
{where all of the estimates were uniform in $x_1,y$.}

Consider then the situation when $y>\frac{C}{N}$, and the integral over $Nu\geq C$ might not vanish.  In this regime $|s_{N,u}|=is_{N,u}$ is bounded from below and may tend to infinity so we have by \eqref{eq:sigmaasy2}, \eqref{def:s}

\begin{align*}
\sigma&(s_{N,u})-\frac{\gamma_1+\gamma_2}{{2\sqrt{2}} i}s_{N,u}-\frac{(\gamma_1+\gamma_2)^2}{{4}}\\
&=-\frac{\gamma_1-\gamma_2}{{2\sqrt{2}}}|s_{N,u}|+\frac{(\gamma_1-\gamma_2)^2}{{4}}+\frac{\gamma_1+\gamma_2}{{2\sqrt{2}} }|s_{N,u}|-\frac{(\gamma_1+\gamma_2)^2}{{4}}+\mathcal O(|s_{N,u}|^{-1})\\
&=\frac{\gamma_2}{{\sqrt{2}}}|s_{N,u}|-{\gamma_1\gamma_2}+\mathcal O(|s_{N,u}|^{-1}),
\end{align*}
as $s_{N,u}\to \infty$.
Thus by \eqref{def:lambda} we find

\begin{align*}
\int_{C/N}^y &\lambda_{N,u}'(x_1+u)\left(\sigma(s_{N,u})-\frac{\gamma_1+\gamma_2}{{2\sqrt{2}} i}s_{N,u}-\frac{(\gamma_1+\gamma_2)^2}{{4}}\right)du\\
&=2N\int_{C/N}^y \frac{|\xi_+'(x_1+u)|}{|s_{N,u}|}\left(\frac{\gamma_2}{{\sqrt{2}}}|s_{N,u}|-{\gamma_1\gamma_2}+\mathcal O(|s_{N,u}|^{-1})\right)du.
\end{align*}

Now note that from \eqref{eq:id g mu} and \eqref{eq:id xi g}, we find $|\xi_+'(x_1+u)|=\pi \sigma_V(x_1+u)$, where $\sigma_V=\frac{d\mu_V}{dx}$. Also as $s_{N,u}=2Nu\xi_+'(x_1+u)(1+\mathcal O(u))$, we find

\begin{equation}\label{star2}\begin{aligned}
\int_{C/N}^y &\lambda_{N,u}'(x_1+u)\left(\sigma(s_{N,u})-\frac{\gamma_1+\gamma_2}{{2\sqrt{2}} i}s_{N,u}-\frac{(\gamma_1+\gamma_2)^2}{{4}}\right)du\\
&=N{\sqrt{2}\pi}\gamma_2\int_{C/N}^y \sigma_V(x_1+u)du -{\gamma_1\gamma_2}\int_{C/N}^{y}\frac{1+\mathcal O(u)}{u}du+\mathcal O\left(\frac{1}{N}\int_{C/N}^{y}\frac{1}{u^2}du\right)\\
&=N{\sqrt{2}\pi}\gamma_2\mu_V([x_1,x_2])-\gamma_1\gamma_2\log N-\gamma_1\gamma_2\log y+\mathcal O(1).
\end{aligned}\end{equation}
{Again, the estimates were uniform in $x_1,y$ so} \eqref{star1} and \eqref{star2} prove Theorem \ref{th:Hankel2}.

\section{Asymptotics for the Hankel determinants in the separated regime}\label{section: RHseparated}
{ 
In this section, we consider $w\neq 0$ and prove Theorem \ref{th:Hankel1}. For $|x_1-x_2|<\epsilon_N/8$, we rely on the analysis of the previous section. When  $|x_1-x_2|>\frac{1}{8}\epsilon_N$ we construct two different local parametrices in two disjoint disks around $x_1$ and $x_2$, say $B(x_1,\frac{\epsilon_N}{32})$ and $B(x_2,\frac{\epsilon_N}{32})$, which will approximate the RH solution $S$ from Section \ref{section: diffid}.
 As illustrated in Figure \ref{fig:lens2} and as already explained in Section \ref{section: diffid}, we  open lenses on three intervals $(-1,x_1)$, $(x_1,x_2)$ and $(x_2,1)$. As in the merging case, the contours near $x_1,x_2$ and $\pm 1$ will be chosen such that suitable conformal maps map the contours  to certain straight lines{, but we still require the lenses to stay within $\mathcal S_N'$ from \eqref{domain2}}. Further away from these points, one can for instance choose the contours to be straight line segments. We recall the definition of the functions $T$ and $S$ in \eqref{def:T} and \eqref{def:S}, and the construction of the global parametrix $P^\infty$ as in \eqref{def:Pinfty}.

We construct 4 local parametrices, each containing one of the points $-1,1,x_1,x_2$. The construction of local parametrices in the disks $B(\pm 1,\delta_N)$ near $\pm 1$ is exactly the same as in Section \ref{sec:merging}. It satisfies in particular the matching condition \eqref{eq:matchingAiry2} at the boundary of the disks.}

\medskip

The construction of the local parametrices near $x_1$ and $x_2$ in the separated regime differs significantly from the one in the previous section. 
The difference is that we surround $x_1$ and $x_2$ separately with disks of radius 
$\frac{1}{32}\epsilon_N$ and inside these we define the local parametrix $P$ (we choose to denote all local parametrices by $P$, in each of the asymptotic regimes, to keep notations consistent and simple) in terms of Kummer's confluent hypergeometric functions, using the RH solution $\Phi_{\rm HG}$ from Appendix \ref{section:appendixCFH}. This construction is almost identical to the one in \cite{Charlier}, which was inspired by earlier constructions in \cite{DIK11, ItsKrasovsky, FoulquieMartinezSousa}.
We omit the precise details and refer to \cite[Section 4.5 and Section 8]{Charlier} and references therein.
We note that our situation corresponds to the case where $m=2$, ${\sqrt{2}} i\beta_1=\gamma_1$, ${\sqrt{2}} i\beta_2=\gamma_2$, $\alpha_1=\alpha_2=0$, $t_1=x_1$, $t_2=x_2$ in \cite{Charlier}, { and furthermore with $W$ as in \cite{Charlier}, $W(z)=w+\gamma_1{\pi/\sqrt{2}}+\gamma_2{\pi/\sqrt{2}}$}.
We require the local parametrix to satisfy a matching condition on the boundary of the small disk surrounding the singularities $x_1$, $x_2$. To verify this, we begin with the following identity (which for us could be seen as the definition of $P$, but for further details, see \cite[Section 4.5]{Charlier}): for $z\in B(x_j,\frac{1}{32}\epsilon_N)$
$$
P(z)P^\infty(z)^{-1}=E(z)\Phi_{\rm HG}(Nf(z))e^{-N\xi(z)\sigma_3}M(z)e^{N\xi_+(x_j)\sigma_3}(Nf(z))^{\frac{\gamma_j}{{\sqrt{2}} i}\sigma_3}E(z)^{-1},
$$
where $f$ is a conformal map defined by 

$$
f(z)=\mp 2(\xi(z)-\xi_\pm(x_j)), \qquad \pm \Im z>0,
$$
and $M$ is given by
$$
M(z)=\begin{cases}
e^{-{\sqrt{2}\pi}\frac{\gamma_j}{2}\sigma_3}, & \Re f(z)<0,\\
\sigma_3\sigma_1, & \Re f(z)>0.
\end{cases}
$$
$E$ is given by
$$
E(z)=P^\infty(z)\Omega_j(z)^{\frac{\sigma_3}{2}}{ e^{{\sqrt{2}\pi}\frac{\gamma_j}{2}\sigma_3}  }e^{\frac{w(z;t)}{2}\sigma_3}M(z)e^{N\xi_+(x_j)\sigma_3}(Nf(z))^{\frac{\gamma_j}{{\sqrt{2}} i}\sigma_3},
$$
with
$$
\Omega_j(z)=e^{{\sqrt{2}\pi}\gamma_{j^*}\mathbf{1}_{\lbrace \Re z<x_{j^*}\rbrace}}, 
$$
where $j^*=2$ if $j=1$ and $j^*=1$ if $j=2$. It follows from the discussion in \cite[Section 4.5]{Charlier} that $E$ is analytic in $B(x_j,\frac{1}{32}\epsilon_N)$. Moreover,
$\Phi_{{\rm HG}}$ is a $2\times 2$-matrix-valued function which is analytic in $\C\setminus \cup_{0\leq k<4}e^{k\frac{\pi i}{4}}\R$ and which is constructed in terms of confluent hypergeometric functions, but for which we only need its asymptotic behavior: as $\lambda\to \infty$ with $\lambda\in\C\setminus \cup_{0\leq k<4}e^{k\frac{\pi i}{4}}\R$, we have
$$
\Phi_{\rm HG}(\lambda)M(\lambda)e^{\frac{\lambda}{2}\sigma_3}\lambda^{\frac{\gamma_j}{{\sqrt{2}} i}\sigma_3}=I+\mathcal O(\lambda^{-1}).
$$

Using the definitions and Lemma \ref{lemma:boundw}, one can check that $E$ is bounded uniformly in $B(x_j,\frac{1}{32}\epsilon_N)$. Another straightforward calculation shows that $\Re f(z)<0$ is equivalent to $\Im z>0$ from which one can deduce that 
$$
e^{-N\xi(z)\sigma_3}M(z) e^{N\xi_+(x_j)\sigma_3}=M(z)e^{\frac{1}{2}Nf(z)\sigma_3}.
$$
Thus, using the fact that on $\partial B(x_j,\frac{1}{32}\epsilon_N)$, $|Nf(z)|\geq c  N^{\alpha}$ for some uniform constant $c$ along with the asymptotics of $\Phi_{{\rm HG}}$, we have a matching condition: for $z\in \partial B(x_j,\frac{1}{32}\epsilon_N)$ 
\begin{equation}\label{eq:match2}
P(z)P^\infty(z)^{-1}=E(z)(I+\mathcal O(N^{-\alpha}))E(z)^{-1}=I+\mathcal O(N^{-\alpha}),
\end{equation}
{where the implied constant is uniform in the relevant $t,w,x_1,x_2$.}  Again, this allows us to transform the RH problem for $S$ into a small-norm  problem.

\subsection{Small-norm RH problem}
We now define
\begin{equation*}
R(z)=\begin{cases}
S(z)P(z)^{-1},& z\in B(x_1,\frac{1}{32}\epsilon_N)\cup B(x_2,\frac{1}{32}\epsilon_N)\cup B(-1,\delta_N)\cup B(1,\delta_N),\\
S(z)P^{\infty}(z)^{-1},& z\in\mathbb C\setminus\left(B(x_1,\frac{1}{32}\epsilon_N)\cup B(x_2,\frac{1}{32}\epsilon_N)\cup B(-1,\delta_N)\cup B(1,\delta_N)\right)
\end{cases}
\end{equation*}
and the relevant RH problem becomes the following.

\subsubsection*{RH problem for $R$}
\begin{itemize}
\item[(a)] $R:\mathbb C\setminus\Sigma_R\to\mathbb C^{2\times 2}$ is analytic, with $\Sigma_R$ as in Figure \ref{fig:R2},
\item[(b)] if we orient the disks in the clockwise manner and the remaining parts of $\Sigma_R$ from left to right, we see that for $z\in\Sigma_R$, $R$ satisfies jump conditions of the form $R_+(z)=R_-(z)J_R(z)$, where 
\begin{equation*}
J_R(z)=\begin{cases}
P(z)P^{\infty}(z)^{-1},&z\in \partial B(x_1,\frac{1}{32}\epsilon_N)\cup \partial B(x_2,\frac{1}{32}\epsilon_N)\cup \partial B(\pm 1,\delta_N),\\
P^{\infty}(z)J_S(z)P^{\infty}(z)^{-1},&z\in \Sigma_R\setminus \left(B(x_1,\frac{1}{32}\epsilon_N)\cup B(x_2,\frac{1}{32}\epsilon_N)\cup B(\pm 1,\delta_N)\right),
\end{cases}
\end{equation*}
where $P$  on the boundary of the disks is interpreted as its boundary value as the boundary is approached from the inside of the circles,
\item[(c)] as $z\to\infty$, we have the asymptotics
\begin{equation*}
R(z)=I+\mathcal O(z^{-1}).
\end{equation*}
\end{itemize}

\begin{figure}
\begin{tikzpicture}[scale=2.3]
    \begin{axis}[ticks=none,axis line style={draw=none},
        unit vector ratio*=1 1 1,       
        xmin=-1.5,xmax=1.5,
        ymin=-.41,ymax=.41]

        \addplot [black,dashed,domain=-147:-90] ({1+0.4*cos(147)}, {0.4*sin(x)});
        \addplot [black,dashed,domain=90:147] ({1+0.4*cos(147)}, {0.4*sin(x)});
        \addplot [black,dashed,domain=33:90] ({-1+0.4*cos(33)}, {0.4*sin(x)});
        \addplot [black,dashed,domain=270:327] ({-1+0.4*cos(33)}, {0.4*sin(x)});

        \addplot [black,dashed,domain=-0.67:0.67]({x},{0.212});
        \addplot [black,dashed,domain=-0.67:0.67]({x},{-0.212});    

        \addplot [black,dashed,domain=0.67:1.33]({x},{0.4});
        \addplot [black,dashed,domain=0.67:1.33]({x},{-0.4});    
        \addplot [black,dashed,domain=0.67:1.33]({-x},{0.4});
        \addplot [black,dashed,domain=0.67:1.33]({-x},{-0.4});    

        \addplot [black,dashed,domain=-0.4:0.4]({1.33},{x});
        \addplot [black,dashed,domain=-0.4:0.4]({-1.33},{x});

        \addplot [black,thick,domain=-0.81:-0.25]({x},{-0.36*x^2-0.33*x+0.025});    
        \addplot [black,thick,domain=0.4:0.81]({x},{-0.56*x^2+0.64*x-0.08});    
        \addplot [black,thick,domain=-0.81:-0.25]({x},{0.36*x^2+0.33*x-0.025});    
        \addplot [black,thick,domain=0.4:0.81]({x},{0.56*x^2-0.64*x+0.08});    
        \addplot [black,thick,domain=-1.5:-1.2]({x},{0});    
        \addplot [black,thick,domain=1.2:1.5]({x},{0});    

        \addplot [black,thick,domain=-0.13:0.28]({x},{0.075});
        \addplot [black,thick,domain=-0.13:0.28]({x},{-0.075});

        \addplot [black,thick,domain=0:360] ({1+0.2*cos(x)}, {0.2*sin(x)});
        \addplot [black,thick,domain=0:360] ({-1+0.2*cos(x)}, {0.2*sin(x)});

        \addplot [black,thick,domain=0:360] ({0.35+0.1*cos(x)}, {0.1*sin(x)});

        \addplot [black,thick,domain=0:360] ({-0.2+0.1*cos(x)}, {0.1*sin(x)});

    \end{axis}
\fill [color=black] (4.22,0.93) circle (0.05);
\node at (4.2,1.05) {$x_2$};
\fill [color=black] (3.0,0.93) circle (0.05);
\node at (2.98,1.05) { $x_1$};
\fill [color=black] (5.72,0.93) circle (0.05);
\node at (5.72,1.25) {$1$};
\fill [color=black] (1.15,0.93) circle (0.05);
\node at (1.15,1.25) { $-1$};


\end{tikzpicture}
\caption{A characterization of the jump contour $\Sigma_R$ in the set $\mathcal S_N'$ (dashed).}\label{fig:R2}
\end{figure}
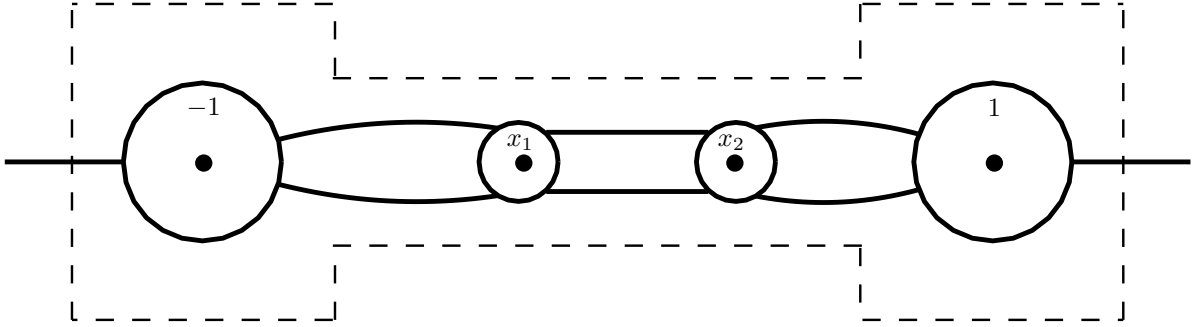
Similarly to the merging regime (see Lemma \ref{prop:jump2}), we now have jump matrices which are uniformly close to the identity matrix. { We let $\Sigma_{R,x_1,x_2}$ be the union of $\partial B(x_1,\frac{1}{32}\epsilon_N)$ and $\partial B(x_2,\frac{1}{32}\epsilon_N)$, we let $\Sigma_{R,\pm1}=\partial B(1,\delta_N)\cup\partial B(-1,\delta_N)$, and denote by $\Sigma_{R,\exp}$ the remaining part of the contour $\Sigma_R$, consisting of the sections of the lenses which are not in the local parametrices, and $(-\infty,-1-\delta)\cup(1+\delta, +\infty)$.
We denote the jump on
$\Sigma_{R,x_1,x_2}$ by $J_{R,x_1,x_2}$, the jump on
$\Sigma_{R,\pm1}$ by  $J_{R,\pm 1}$, and the jump on
$\Sigma_{R,\exp}$ by $J_{R,{\rm exp}}$.

\begin{lemma}\label{prop:sepjump1}
As $N\to\infty$ and $y{\geq }\widehat \epsilon/8$,  
\begin{equation}\label{eq:sepboundJR1}
\begin{aligned}
J_{R,x_1,x_2}(z)&=I+\mathcal O\left(\frac{1}{N\widehat \epsilon}\right)\\
J_{R,\pm 1}(z)&=I+\mathcal O\left(\frac{1}{N\widehat \delta^{2}}\right)\\
J_{R,{\rm exp}}(z)&=
I+\mathcal O\left(\frac{1}{e^{N^\kappa}(|z|^2+1)}\right),&
\end{aligned}
\end{equation}
uniformly for $z\in \Sigma_{R,{x_1,x_2}},\Sigma_{R,\pm 1}, \Sigma_{R,{\rm exp}}$ {and uniformly in the relevant $x_1,x_2,t,w$} respectively, for some $\kappa>0$.
\end{lemma}}
\begin{proof}
This follows from the matching condition (c) in the RH problem for $P$, see Lemma \ref{pr:P}, similarly as in the proof of Lemma \ref{prop:jump2}. We omit further details.
\end{proof}

Now arguing as in the proof of Proposition \ref{prop:R} (we omit details, as the argument is essentially identical), we find the following result.

\begin{proposition}\label{prop:R2}
For $t\in[0,1]$ and $z\in\mathbb C\setminus\left(\mathcal S_N'\cup\Sigma_R\right)$, we have
{
\begin{equation}\label{eq:asR2}\begin{aligned}
R(z)=I+\mathcal O\left(N^{-1}\left(|z-x_1|^{-1}+|z-x_2|^{-1}+\delta_N^{-1}|z-1|^{-1}+\delta_N^{-1}|z+1|^{-1}\right)\right),\\ 
 R'(z)=\mathcal O\left(\frac{1}{N \epsilon_N }\left(|z-x_1|^{-1}+|z-x_2|^{-1}+\delta_N^{-1}|z-1|^{-1}+\delta_N^{-1}|z+1|^{-1}\right)\right).
\end{aligned}
\end{equation}}
uniformly {in the relevant $z,x_1,x_2,t,w$} as $N\to\infty$, with  $y\geq \frac{\epsilon_N}{8}$.
\end{proposition}
This allows us to study the asymptotics of our second differential identity.

\subsection{Asymptotics of the \texorpdfstring{$t$}{t}-differential identity and the proof of Theorem \ref{th:Hankel1}} 
As in \cite[Section 7]{Charlier} (and a similar argument in \cite[Section 5.1]{BWW18}), we can use the jump conditions of $Y$ and a contour deformation argument to write the integral from Proposition \ref{le:di1} as follows:

\begin{align}
&\label{eq:diffidas}\partial_t \log D_N(x_1,x_2;\gamma_1,\gamma_2;w(.;t))=\frac{1}{2\pi i}\int_\R\big[Y^{-1}Y'\big]_{21}(\lambda;t){ w(\lambda)h_t(\lambda)}d\lambda\notag \\
& =\frac{1}{2\pi i}\int_{\mathbb R\setminus[-1-\epsilon,1+\epsilon]}\big[Y^{-1}Y'\big]_{21}(\lambda;t){ w(\lambda)h_t(\lambda)}d\lambda +\frac{1}{2\pi i}\int_{\mathcal C_N}\big[Y^{-1}Y'\big]_{11}(\lambda;t){ w(\lambda)}d\lambda,
\end{align}
where $\mathcal C_N$ is a contour oriented in the counter-clockwise manner, enclosing $[-1,1]$ and passing through the points $\pm (1+\epsilon)$ for some fixed $\epsilon>0$. We can now take $\mathcal C_N$ to be be a contour in $\mathcal S_N$, but outside of $\mathcal S_N'$ and hence outside the lenses and the local disks in which local parametrices were constructed. Tracing through our transformations in the RH analysis, we are able to express $Y$ identically in terms of $R$ and the global parametrix $P^\infty$ from Section \ref{sec:global} as follows:
\[Y(\lambda)=e^{-\frac{N\ell}{2}\sigma_3}R(\lambda)P^\infty(\lambda)e^{Ng(\lambda)\sigma_3}e^{\frac{N\ell}{2}\sigma_3}\]
for $\lambda$ on $\mathcal C_N\cup\mathbb R\setminus[-1-\epsilon, 1+\epsilon]$.
In addition to this, we can use the asymptotics for $R$. Since we want to prove Theorem \ref{th:Hankel1} both in the merging regime and in the separated regime, we will need to use the estimates from Proposition \ref{prop:R}
as well as those from Proposition \ref{prop:R2}.

Let us first focus on the first term at the right hand side of \eqref{eq:diffidas}, which gives
\begin{multline*}\frac{1}{2\pi i}\int_{\mathbb R\setminus[-1-\epsilon,1+\epsilon]]}\big[Y^{-1}Y'\big]_{21}(\lambda;t) w(\lambda) h_t(\lambda)d\lambda\\=\int_{\mathbb R\setminus[-1-\epsilon,1+\epsilon]}\big[(P^\infty)^{-1}(\lambda;t)\big((P^\infty)'(\lambda;t)+R_+^{-1}(\lambda;t)R_+'(\lambda;t)P^\infty(\lambda;t)\big)\big]_{21}\\
\quad \times  e^{2Ng_+(\lambda)+N\ell }w(\lambda)h_t(\lambda)d\lambda.\end{multline*}
Note that from \eqref{def:w}, our assumption of $w$ being uniformly bounded in $\mathcal S_N$, and  \eqref{eq:id g mu}, one finds that $e^{2Ng_+(\lambda)+N\ell }w(\lambda)h_t(\lambda)=\mathcal O(e^{N(g_+(\lambda)+g_-(\lambda)+\ell-V(\lambda))})$, where the implied constant is uniform in everything relevant. By \eqref{eq:el2} and the fact that as $|\lambda|\to \infty$, $g_\pm =\mathcal O(\log |\lambda|)$, while $V(\lambda)/\log |\lambda|\to \infty$, we find e.g. that for some fixed $c>0$, $e^{2Ng_+(\lambda)+N\ell }w(\lambda)h_t(\lambda)=\mathcal O(\lambda^{-c N})$ for $|\lambda|>1+\epsilon$. Combining this with the fact that Proposition \ref{prop:R} and Proposition \ref{prop:R2} imply that $R^{-1}R'$ is uniformly bounded on the integration contour, as are $(P^\infty)^{\pm 1},P^{\infty \prime}$, we find that for any fixed $\epsilon>0$,
\begin{align*}
\frac{1}{2\pi i}&\int_{\mathbb R\setminus[-1-\epsilon,1+\epsilon]]}\big[Y^{-1}Y'\big]_{21}(\lambda;t) w(\lambda) h_t(\lambda)d\lambda=\mathcal O\left(\int_{1+\epsilon}^\infty \lambda^{-Nc}d\lambda\right)=o(1)
\end{align*}
as $N\to\infty$ (in fact, the quantity is exponentially small, but this is not important to us).

\medskip

Thus our task is to understand the large $N$ asymptotics of 
\begin{align*}
\frac{1}{2\pi i}&\int_{\mathcal C_N}\big[Y^{-1}Y'\big]_{11}(\lambda;t){ w(\lambda)}d\lambda\\
&=N\int_{\mathcal C_N}g'(\lambda)w(\lambda)\frac{d\lambda}{2\pi i}+\int_{\mathcal C_N}[(P^\infty)^{-1}(\lambda;t)(P^\infty)'(\lambda;t)]_{11}w(\lambda)\frac{d\lambda}{2\pi i}\\
&\quad +\int_{\mathcal C_N}[(P^\infty)^{-1}(\lambda;t)R^{-1}(\lambda;t)R'(\lambda;t)P^\infty(\lambda;t)]_{11}w(\lambda)\frac{d\lambda}{2\pi i}.
\end{align*}

Recalling the definition of $g$ from \eqref{def:g}, Cauchy's integral formula yields immediately (recall the orientation of $\mathcal C_N$) that 
\[
N\int_{\mathcal C_N}g'(\lambda)w(\lambda)\frac{d\lambda}{2\pi i}=N\int w(x)d\mu_V(x).
\]
For the last term, combining Proposition \ref{prop:R} and Proposition \ref{prop:R2} with the fact that $P^\infty$ has at worst fourth root singularities at $\pm 1$ and a logarithmic one at $x_j$, yields (after a routine calculation) that 
\begin{align*}
&\int_{\mathcal C_N}[(P^\infty)^{-1}(\lambda;t)R^{-1}(\lambda;t)R'(\lambda;t)P^\infty(\lambda;t)]_{11}w(\lambda)\frac{d\lambda}{2\pi i}\\
&=\mathcal O(N^{-1}\epsilon_N^{-1} \delta_N^{-3/2})+\mathcal O(N^{-1}\epsilon_N^{-1} (\log \epsilon_N^{-1})^2)=o(1)
\end{align*}
as $N\to\infty$ -- in the last step, we also made use of the definition of $\epsilon_N$ and $\delta_N$ from \eqref{domain}.

For the remaining term, we first note that by the definitions of Section \ref{sec:global}, we find
\begin{align*}
[(P^\infty)^{-1}(\lambda;t)(P^\infty)'(\lambda;t)]_{11}&=(Q^{-1}(\lambda)Q'(\lambda))_{11}-\frac{D'(\lambda;t)}{D(\lambda;t)}\\
&=-\partial_\lambda \log D_\gamma(\lambda)-\partial_\lambda\log D_{w}(\lambda;t),
\end{align*}
where we used the fact that the $Q$-term vanishes as one readily checks from \eqref{eq:QQdef}. It follows that
\begin{align*}
&\int_{\mathcal C_N}[(P^\infty)^{-1}(\lambda;t)(P^\infty)'(\lambda;t)]_{11}w(\lambda)\frac{d\lambda}{2\pi i}\\
&=-\int_{\mathcal C_N} w(\lambda) \partial_\lambda\log D_\gamma(\lambda)\frac{d\lambda}{2\pi i}+\int_{\mathcal C_N}\log D_{w}(\lambda;t)w'(\lambda)\frac{d\lambda}{2\pi i} 
\end{align*}
The quantity $\partial_\lambda \log D_\gamma(\lambda)$ has been evaluated in \cite[(6.11)]{Charlier}: one has 
\[
\partial_\lambda\log D_\gamma(\lambda)=\sum_{j=1}^2\frac{\gamma_j\sqrt{1-x_j^2}}{\sqrt{2}\sqrt{\lambda^2-1}(\lambda-x_j)},
\]
from which one finds by contour deformation
\[
-\int_{\mathcal C_N} w(\lambda) \partial_\lambda\log D_\gamma(\lambda)\frac{d\lambda}{2\pi i}=\sum_{j=1}^2 \frac{\gamma_j}{\sqrt{2}}\sqrt{1-x_j^2}\mathcal U w(x_j),
\]
where we used the notation \eqref{U_transform}.

To conclude, we note that from \eqref{def:D}, one finds with contour deformation and \eqref{cov3}
\[
\int_{\mathcal C_N}\log D_{w}(\lambda;t)w'(\lambda)\frac{d\lambda}{2\pi i}=\int_{\mathcal C_N}\frac{\sqrt{\lambda^2-1}}{2\pi}\int_{-1}^1 \frac{t w(x)}{\sqrt{1-x^2}}\frac{dx}{\lambda-x}w'(\lambda)\frac{d\lambda}{2\pi i}=t \sigma(w)^2.
\]

Putting everything together, we find that
\begin{align*}
\log \frac{D_N(x_1,x_2;\gamma_1,\gamma_2;w)}{D_N(x_1,x_2;\gamma_1,\gamma_2;0)}= N\int w d\mu_V+\sum_{j=1}^2\frac{\gamma_j}{\sqrt{2}}\sqrt{1-x_j^2}\mathcal U w(x_j)+\sigma(w)^2 \int_0^1 t dt+o(1),
\end{align*}
with the required uniformity. This concludes the proof of Theorem \ref{th:Hankel1} .

\section{Asymptotics for the Hankel determinants in the edge regime}\label{section: RHedge}
 
In this section, we will prove Theorem \ref{th:Hankel3}, namely we consider the case where $t=0$ (or $w\equiv 0$), $x_1=x_2=x$, and we denote $\gamma=\gamma_1+\gamma_2$.
We consider the case where for some fixed $\epsilon>0$, $|x\pm 1|\leq \epsilon$ and $N\to \infty$ in such a way that for some large but fixed $m$, ${N}^{2/3}|x\pm 1|{\geq m}$. We also suppose that $\gamma\in[-\Gamma, \Gamma]$  for some fixed $\Gamma>0$,  as required for Theorem \ref{th:Hankel3}. 
Uniform asymptotics for $|x\pm 1|\geq \epsilon$ were established in \cite{Charlier}, so we do not need to treat this case. We can assume without loss of generality that $x$ is close to $1$, the case of $-1$ is similar or can be treated when considering the potential $V(-x)$ instead of $V(x)$.

With the lenses chosen as in Figure \ref{fig:lens} with $x_2=x_1$ and $\mathcal S_N$ taken to be independent of $N$ as in Section \ref{sec:merging} with $w=0$, 
we define disks $B(-1,\delta)$ and $B(1,\delta)$ around $\pm 1$ with radius $\delta>\epsilon$ which is sufficiently small but independent of $N$  -- we could fix e.g. $\delta=2\epsilon$. The disk $B(1,\delta)$ will then contain both $1$ and $x$.
The local parametrix in $B(-1,\delta)$ is constructed exactly in the same way as before in the merging and separated regime, and satisfies the same RH conditions and matching condition with the global parametrix -- uniformly in the relevant $x$  and $\gamma\in[-\Gamma, \Gamma]$. Throughout the section, all error terms are uniform for $\gamma\in[-\Gamma, \Gamma]$.

Our approach in this section will be to integrate the differential identity  \eqref{eq:diffidconfluent} from $1-\epsilon$ to $x$ so that combining the asymptotics of \cite{Charlier} with our RH analysis for the asymptotics of the differential identity will yield the result of Theorem \ref{th:Hankel3}. Throughout this section, $x$ will be used both for the dummy integration variable -- or in other words, the variable appearing in \eqref{eq:diffidconfluent} -- as well as the end point of the integral -- or in other words, the variable appearing in Theorem \ref{th:Hankel3}. We hope this will cause no confusion for the reader.

\subsection{Asymptotic analysis of $Y$}\label{SectionY}

We start from the RH problem for $S$ given in Section \ref{section: diffid}, and construct a local parametrix containing both the edge of the support at $1$ and the singularity at $x$.
Let $f$ and $\lambda_x$ be conformal maps from a sufficiently small neighbourhood $B(1,\delta)$ of $1$ to a neighbourhood of $0$ defined by 
\begin{equation} \label{def:f1}f(z)=\left(\frac{3\pi}{2}\int_1^z \psi(s)\sqrt{s^2-1}ds\right)^{2/3},\qquad \lambda_x(z)=-\frac{f(z)}{f(x)},
\end{equation}
with the branches in the definition of $f$ taken such that $f$ is analytic on $B(1,\delta)$ and positive on $(1,1+\delta)$.

Define $u=u_{N,x}$ as
\begin{equation}\label{def:u}
u_{N,x}=-N^{2/3}f(x)=\left(\frac{3N}{2}|\xi_+(x)|\right)^{2/3},
\end{equation}
such that $u_{N,x}$ will be a large parameter, with 
\begin{equation}\label{uNxlarge}
u_{N,x}=\mathcal O(N^{2/3}(1-x)),
\end{equation}
as $N\to \infty$, uniformly for $x\in(1-\epsilon,1)$. We also choose $m$ large enough for the asymptotics \eqref{eq:smallnormPhi} to be valid for $\Xi(\lambda,u_{N,x})$ uniformly for all $x\in(1-\epsilon,1-m N^{-2/3})$.

Since we assumed that $x$ is in $B(1,\delta)$, and possibly approaches $1$ as $N\to\infty$, we think of $|f(x)|$ being small compared to $|f(z)|$ with $z\in \partial B(1,\delta)$, implying that $|\lambda_x(z)|$ on the other hand should be thought of being large.

By definition, 
\begin{equation}\label{largelambda}
\lambda_x(z)=\mathcal O((1-x)^{-1}), \end{equation}
 uniformly  for $z\in B(1,\delta)$.
More precisely, we have 
\begin{equation}
\label{eq:lambdx0}
\lambda_x(z)=-1+c_1(z-x)+c_2(z-x)^2+\mathcal O((z-x)^3),\qquad z\to x,
\end{equation}
with
\begin{equation}\label{derivativezeta1}
c_1=\lambda_x'(x)=\frac{1}{1-x}(1+\mathcal O(1-x)),\qquad 
c_2=\mathcal O\left(\frac{1}{1-x}\right),
\end{equation}
as $x\to 1$.

In $B(1,\delta)$, we define
\begin{equation}\label{def:P1}
P(z)=E(z)\Phi\left(
\lambda_x(z)
;u_{N,x}\right),
\end{equation} 
where $\Phi$ is the solution to the model RH problem from Section \ref{sect:paremetrix1}, and where
\begin{equation}\label{def:E1}
E(z)=P^{\infty}(z)M(\lambda_x(z))^{-1},
\end{equation}
with $M$ the global parametrix for $\Phi$ defined in \eqref{def:M} and  $P^{\infty}$ as in \eqref{def:Pinfty}.
We also define the jump contour $\Sigma_S$ for $S$ near $1$ in such a way that $\lambda_x$ maps $\Sigma_S\cap B(1,\delta)$ to a subset of the jump contour $\Sigma_\Phi$ for $\Phi$.

Since $P^{\infty}(z)$ and $M(\lambda_x(z))$ satisfy the same jump relations in $B(1,\delta)$, {and the possible isolated singularities at $x,1$ are not strong enough to be poles,} $E$ is analytic in $B(1,\delta)$.
Moreover, from the asymptotic analysis for $\Phi$ done in Section \ref{sec:model},  in particular from \eqref{defR} and \eqref{eq:smallnormPhi}, we know that $\Phi(\lambda;u)M(\lambda)^{-1}=I+\mathcal O(u^{-3/2}\lambda^{-1})$ for $\lambda$ sufficiently large as $u\to\infty$.

Hence, we have
\[
\Phi(\lambda_x(z);u_{N,x})M(\lambda_x(z))^{-1} = I+\mathcal O(u_{N,x}^{-3/2}\lambda_x(z)^{-1}) 
\]
{as $N\to \infty$, uniformly for $z\in\partial B(1,\delta)$ and $x$.} 
By \eqref{def:f1} and \eqref{def:u}, we obtain the uniform matching condition
\begin{equation*}
P(z)P^{\infty}(z)^{-1}=P^\infty(z)M^{-1}(\lambda_x(z))\left(I+\mathcal O(|1-x|^{-1/2}N^{-1})\right)M(\lambda_x(z))P^\infty(z)^{-1}\nonumber 
\end{equation*}
as $N\to \infty$, for $z\in\partial B(1,\delta)$. By \eqref{largelambda} and the definition of $M$ in \eqref{def:M}, we have $M(\lambda_x(z))^{\pm 1}=\mathcal O(|1-x|^{-1/4})$ as $x\to 1$, and so
\begin{equation}
P(z)P^{\infty}(z)^{-1}=I+\mathcal O\left(|1-x|^{-1}N^{-1}\right)=I+{\mathcal O}(N^{-1/3}),\label{smallP1Pinfty} \end{equation}
uniformly for $z\in B(1,\delta)$ as $N\to \infty$, where the last estimate follows from the fact that we assume that $N^{2/3}|1-x|\geq m$ for some fixed $m$.

The local parametrix near $-1$ is constructed in terms of the Airy model RH problem as before, and satisfies the matching condition
\begin{equation}\label{smallP-1Pinfty}
P(z)P^{\infty}(z)^{-1}=I+\mathcal O(N^{-1}),
\end{equation}
as $N\to \infty$, for $z\in\partial B(-1,\delta)$, with $\delta>0$ sufficiently small but independent of $N$, and the implied constant being uniform in $z$ and $x$. 
 
One  readily checks that the jumps of $S$ from Section \ref{sec:lens1} cancel inside of $B(1,\delta)$ and the only possible singularity, which is at $1$, is removable so $SP^{-1}$ has an analytic continuation to $B(1,\delta)$. A similar construction at $-1$ yields an analytic continuation into $B(-1,\delta)$. This motivates the following definition.  Let $R$ be given by
\begin{equation}\label{defhatR}
R(z)=\begin{cases}
 S(z)P^{\infty}(z)^{-1} &\textrm{for } z\in \mathbb C \setminus  (\Sigma_S\cup B( 1,\delta)\cup B(-1,\delta)), \\
 S(z)P(z)^{-1} &\textrm{for } z\in B(1,\delta)\cup B(-1,\delta).
\end{cases}
\end{equation}

By the RH problem for $S$ in Section~\ref{sec:lens1}, \eqref{smallP1Pinfty}, and \eqref{smallP-1Pinfty}, $R$ satisfies the following small norm RH problem:
\subsubsection*{RH problem for $R$}
\begin{itemize}
\item[(a)] $R: \mathbb C \setminus \Sigma_{R} \to \mathbb C$ is analytic, where $\Sigma_{R}$ is the jump contour of $R$, consisting of the two disks $\partial B(\pm 1,\delta)$ and the part of {$\Sigma_S\setminus[-1,1]$} outside these disks.
\item[(b)] On $\Sigma_{R}$, {with an orientation inherited from $\Sigma_S$ as well as the disks oriented in the clockwise manner,} $R$ has the following jumps as  $N\to \infty$,
\begin{align}
R_+(z)&= R_-(z)\left(I+{\mathcal O}(N^{-1/3}|z|^{-2})\right),
\end{align}
uniformly in $z\in\Sigma_R$ and $N^{2/3}|x-1|\geq m$.

\item[(c)] As $z \to \infty$,
\begin{equation*} R(z)=I +\mathcal{O} (z^{-1}).
\end{equation*}
\end{itemize}
By standard small norm analysis (see e.g. \cite{Deift}), it follows that
 \begin{equation}\label{order:R}
R(z)=I+ {\mathcal O}\left(N^{-1/3}\right), \qquad  R'(z)={\mathcal O}\left(N^{-1/3}\right),
\end{equation}
as $N\to\infty$, uniformly for $z\in \mathbb C \setminus \Sigma_{R}$  and uniformly in the relevant $x$ and $\gamma\in [-\Gamma, \Gamma]$. 
 
\subsection{Asymptotic analysis of $\log D_N(x,x;\gamma,0;0) $}

Reverting the transformations $Y\mapsto T\mapsto S\mapsto R$ and using the large $N$ asymptotics for $R$, we can now obtain large $N$ asymptotics for  the right hand side of the differential identity from Proposition \ref{pr:di}.
By \eqref{def:S}, \eqref{def:T} and Proposition \ref{pr:di}, we have
\begin{equation}\label{diffid2}
\frac{d}{dx}\log D_N(x,x;\gamma,0;0)=\frac{i}{2\pi}
\left(1-e^{{\sqrt{2}\pi}\gamma}
\right)e^{2N\xi_+(x)}(S_+^{-1}S_+')_{21}(x),
\end{equation}
where $S_+$ should here be interpreted as the limit from outside the lens in the upper half plane.
For $z$ in $B(1,\delta)$, outside the lens in the upper half plane,
we know from the definition \eqref{defhatR} of $R$ and by \eqref{def:P1}
that $S(z)= R(z) E(z)\Phi\left(\lambda_x(z);u_{N,x}\right)$.
By  \eqref{defR}, \eqref{def:E-1}, and \eqref{def:P-1}, we have 
\[
\Phi(\lambda;u)=\Xi(\lambda)M(\lambda)h_u(\lambda)^{\frac{\gamma}{{\sqrt{2}} i}\sigma_3}e^{\frac{2}{3}iu^{3/2}\sigma_3}\Phi_{\rm HG}\left(h_u(\lambda);\frac{\gamma}{{\sqrt{2}}i}\right)e^{\frac{2}{3}(\lambda u)^{3/2}\sigma_3}e^{-{\sqrt{2}\pi}\frac{\gamma}{4}\sigma_3},
\]
for $\lambda$ near $-1$ in the upper half plane.
Let  $H$ be given by
\begin{equation}\label{def:H}
H(z)=E(z)\Xi(\lambda_x(z))M(\lambda_x(z))h_{u_{N,x}}(\lambda_x(z))^{\frac{\gamma}{{\sqrt{2}} i}\sigma_3}e^{\frac{2i}{3}u_{N,x}^{3/2}\sigma_3}.
\end{equation}
Then for $z$ outside the lens with $\Im z>0$,
\begin{equation}\label{SRHPhi}S(z)= R(z) H(z) \Phi_{\rm HG}\left(h_u(\lambda_x(z))\right)e^{\frac{2}{3}(\lambda_x(z) u_{N,x})^{3/2}\sigma_3}e^{-{\sqrt{2}\pi}\frac{\gamma}{4}\sigma_3}.\end{equation} 
 Substituting \eqref{SRHPhi} into \eqref{diffid2} and using the formula $\xi_+(x)=\frac{2i}{3N}u_{N,x}^{ 3/2}$ (obtained from \eqref{def:u} and the definition of $\xi$) and the fact that $\lambda_{x,+}(x)^{3/2}=-i$,  we obtain
\begin{multline}\label{eq:quantityS2}
\frac{d}{dx}\log D_N(x,x;\gamma,0;0)=\frac{i}{2\pi}
\left(e^{-{\pi}\gamma/\sqrt{2}}-e^{{\pi}\gamma/\sqrt{2}}
\right)\\ \Bigg[\frac{d}{dz}\left. h_{u_{N,x}}(\lambda_x(z))\right|_{z=x} \left(\Phi_{{\rm HG},+}^{-1}(0)\Phi_{{\rm HG},+}'(0)\right)_{21}+
\left(\Phi_{{\rm HG},+}^{-1}(0)H^{-1}(x)H'(x)\Phi_{{\rm HG},+}(0)\right)_{21}\\ +
\left(\Phi_{{\rm HG},+}^{-1}(0)H^{-1}(x)R_+^{-1}(x) R_+'(x)H(x)\Phi_{{\rm HG},+}(0)\right)_{21}\Bigg].
\end{multline}
Here $\Phi_{{\rm HG},+}(0)$ is the boundary value of $\Phi_{\rm HG}$ as $0$ is approached from the sector $\pi/4<\arg\zeta<3\pi/4$.\footnote{Strictly speaking, the second column of $\Phi_{\mathrm{HG}}$ blows up at zero, but the $21$-entries we consider above can be expressed solely in terms of the first column of $\Phi_{\mathrm{HG}}$ -- this is due to the fact that $\det\Phi_{\mathrm{HG}}=1$. This is the way we interpret the above statement.}

 We will now evaluate all the quantities appearing at the right hand side of \eqref{eq:quantityS2}.
 
The first term at the right hand side of  \eqref{eq:quantityS2} can be evaluated directly,
by using \eqref{def:f1},
 \eqref{asym:h}, and Lemma \ref{Lemma1} in the {a}ppendix. 
For the  third term of \eqref{eq:quantityS2}, we need to use \eqref{order:R}, which implies that
 \begin{equation}
\label{eq:Restimate} (R^{-1}R')(x)={\mathcal O}(N^{-1/3}),
\end{equation}
uniformly as $N\to\infty$, for $N^{2/3}|x-1|\geq {m}$ for large enough fixed $m$.

This results in
\begin{multline}
\frac{d}{dx}\log D_N(x,x;\gamma,0;0)={\sqrt{2}\pi}\gamma N{\psi_V}(x)(1-x^2)^{1/2}
\\ +
\frac{i}{2\pi}
\left(e^{-{\pi}\gamma/\sqrt{2}}-e^{{\pi}\gamma/\sqrt{2}}
\right)\left(\Phi_{{\rm HG},+}^{-1}(0)H^{-1}(x)H'(x)\Phi_{{\rm HG},+}(0)\right)_{21}
\\+\left(\Phi_{{\rm HG},+}^{-1}(0)H^{-1}(x)\mathcal O(N^{-1/3})H(x)\Phi_{{\rm HG},+}(0)\right)_{21}.
\end{multline}

The relevant values for $\Phi_{{\rm HG_+}}(0)$ may be found in the appendix,
and so all that remains is to obtain an expression for $H^{-1}H'$, and to bound $H$ and $H^{-1}$, which we do in Lemma \ref{LemmaH} below. 
By Lemma \ref{LemmaH} and Lemma \ref{Lemma1}, we obtain
\[
\frac{d}{dx}\log D_N(x,x;\gamma,0;0)={\sqrt{2}\pi}\gamma N{\psi_V}(x)(1-x^2)^{1/2}
- { \frac{3\gamma^2}{{4}(1-x)}}+ B_{N,x}+ \mathcal O(1),
\]
with $B_{N,x}$ such that $\int_{1-\epsilon}^xB_{N,\eta}d\eta=\mathcal O(1)$ for sufficiently small $\epsilon>0$.

Integrating this expression in $x$ from $1-\epsilon$ to $x$, we finally obtain
\begin{multline*}
\log D_N(x,x;\gamma,0;0)=\log D_N(1-\epsilon, 1-\epsilon;\gamma,0;0)+{\sqrt{2}\pi}\gamma N\int_{1- \epsilon}^x{\psi_V}(\eta)(1-\eta^2)^{1/2}
d\eta\\ +{ \frac{3\gamma^2}{{4}}\log(1-x)}+ \mathcal O(1),
\end{multline*}
as $N\to\infty$, uniformly for $x$ with $|x-1|<\epsilon$ for some small fixed $\epsilon$ and $N^{2/3}|x-1|\geq m$ for some large fixed $m$.

Using the asymptotics from \cite[Theorem 1.1]{Charlier} for $\log D_N(1-\epsilon, 1-\epsilon;\gamma,0;0)$, we get (with the required uniformity, including in $\gamma$) 
\[
\log \frac{D_N(x,x;\gamma,0;0)}{D_N(x,x;0,0;0)}={ {\sqrt{2}\pi}\gamma N\left(1-\int_{x}^1{\psi_V}(\xi)(1-\xi^2)^{1/2}
d\xi \right)}+ \frac{\gamma^2}{{2}}\log N+{ \frac{3\gamma^2}{{4}}}\log(1-x)+ \mathcal O(1).
\]

To complete the proof of Theorem \ref{th:Hankel3}, we now state and prove Lemma \ref{LemmaH}, which we already used above.
\begin{lemma}\label{LemmaH}
As $N\to\infty$, we have {for large enough fixed $m$ and small enough fixed $\epsilon>0$, for} $x$ with $N^{2/3}|x-1|\geq {m}$ and $|x-1|<\epsilon$
\begin{equation}\label{estimateH1}
H(x)=\mathcal O((1-x)^{-1/4}),\qquad H^{-1}(x)=\mathcal O((1-x)^{-1/4})
\end{equation}
\begin{equation}\label{estimateH2}
H^{-1}(x)H'(x)=
-\frac{3\gamma}{{4\sqrt{2}} i(1-x)}\sigma_3
+A_{N,x}+\mathcal O(1),
\end{equation}
for some matrix $A_{N,x}$ which is such that, for sufficiently small $\epsilon>0$,
\[\int_{1-\epsilon}^{x}A_{N,\eta}d\eta=\mathcal O(1)\]
where the estimates are uniform in $x$ and  in $\gamma\in [-\Gamma, \Gamma]$.
\end{lemma}

\begin{proof}
Our goal is to estimate $H^{\pm 1}(x)$ and $H^{-1}(x)H'(x)$ by estimating the various terms appearing in the definition \eqref{def:H}. We begin with $E$.	
	
{ \underline{Estimating $E^{\pm 1}(x)$.}}  { Recall that $E$} was defined in \eqref{def:E1}. By \eqref{def:E1}, \eqref{def:M}, and \eqref{def:Pinfty}, we have
\begin{align}\label{FormE}
E(z)=D_\infty^{\sigma_3}Q(z)F(z)A^{-1}(\lambda_x(z))^{\frac{\sigma_3}{4}},\end{align}
where 
\begin{equation}\label{defF}
F(z)=D(z)^{-\sigma_3}\left(\frac{1+e^{-\frac{\pi i}{2}}\sqrt{\lambda_x(z)}}{\sqrt{\lambda_x(z)+1}}\right)^{{\sqrt{2}i}\gamma\sigma_3},
\end{equation}
\[
A=\frac{1}{\sqrt{2}}\begin{pmatrix}
	1 & i\\i & 1
	\end{pmatrix},
\]
and $\lambda_x$ is as in \eqref{def:f1}. 

By \eqref{expressionD} (note that compared to \eqref{def:f1}, the change in the sign of the term inside the root compensates for the missing prefactor $e^{\pi\gamma/\sqrt{2}}$), we have
 \[F(z)=
\left(\frac{-zx+1+e^{\frac{\pi i}{2}}\sqrt{(z^2-1)(1-x^2)}}{z-x}\right)^{\frac{-\gamma}{{\sqrt{2}} i}\sigma_3}\left(\frac{1+e^{-\frac{\pi i}{2}}\sqrt{\lambda_x(z)}}{\sqrt{\lambda_x(z)+1}}\right)^{{\sqrt{2}i}\gamma\sigma_3},
\] 
and since $\lambda_x(z)\in(-1,0)$ for $z\in(x,1)$, it follows that that $|F_+(z)|=1$ for $z\in (x,1)$.

As $z\to x_+$, we find that
\begin{multline}\label{expansionF}
\left(\frac{-zx+1+e^{\frac{\pi i}{2}}\sqrt{(z^2-1)(1-x^2)}}{z-x}\right)\left(\frac{1+e^{-\frac{\pi i}{2}}\sqrt{\lambda_x(z)}}{\sqrt{\lambda_x(z)+1}}\right)^2\\=
\frac{1}{c_1(1-x^2)}\left(2+(z-x)\left[\frac{2}{1-x^2}-c_1-2\frac{c_2}{c_1}\right]\right)+\mathcal O((z-x)^2),\end{multline}
where $c_1$ and $c_2$ are as in \eqref{eq:lambdx0}. By  \eqref{derivativezeta1}, \eqref{expansionF} it follows that
\begin{equation}\label{asymF} F_{+}(x)=I+\mathcal O(1-x),\end{equation}
as $x\to 1$.  Furthermore, we note for future reference that \eqref{derivativezeta1} and \eqref{expansionF} imply that $F_+'(x)$ is bounded as $x\to 1$.

Using also the fact that $Q_+(x)^{\pm 1}=\mathcal O((1-x)^{-1/4})$ as $x\to 1$, we immediately get that $E(x), E^{-1}(x)=\mathcal O((1-x)^{-1/4})$ as $x\to 1$, where we also note that $E(x)$ is independent of $N$. { This concludes our estimate for $E$.}

\medskip

{ \underline{Estimating $H^{\pm 1}(x)$:}} The asymptotics for $E{ (x)}$ and $\Xi{ (z)}$ are in fact sufficient to deduce \eqref{estimateH1}: using \eqref{def:E1}, \eqref{def:H}, \eqref{eq:smallnormPhi} we get for $z$ close\footnote{{ We are using this representation to estimate $H$ at $z$, so close now means that $z\to x$ e.g. faster than $x\to 1$ and $\lambda_x(z)^{\pm 1}$ is close to $-1$ so we can ignore the $\lambda_x(z)^{\pm \sigma_3/4}$-terms.}} to $x$  and in the upper half plane outside the lens,
\begin{multline*}H(z)=P^\infty(z)h_u(\lambda_x(z))^{\frac{\gamma}{{\sqrt{2}} i}\sigma_3}e^{\frac{2i}{3}u_{N,x}^{3/2}\sigma_3}\\
+
E(z)\mathcal O((1-x)^{-3/2}N^{-1})M(\lambda_x(z)){h_{u_{N,x}}}(\lambda_x(z))^{\frac{\gamma}{{\sqrt{2}} i}\sigma_3}e^{\frac{2i}{3}u_{N,x}^{3/2}\sigma_3},
\end{multline*}
as $N\to \infty$, uniformly for $x\in(1-\epsilon,1-m/N^{2/3})$, for some $\epsilon,m>0$.

If we now take the limit $z\to x$, since $P^\infty(x)=\mathcal O((1-x)^{-1/4})$ as $x\to 1$, the first term at the right hand side is $\mathcal O((1-x)^{-1/4})${ -- note that while $h_{u_{N,x}}(\lambda_x(x))=0$, this does not cause any further blowup since the power of $h_{u_{N,x}}$ is imaginary and we take the limit $z\to x$ in a way where there is no winding}.
The same applies to the second term, uniformly for $x\in(1-\epsilon,1-m/N^{2/3})$, since $E$ is independent of $N$ and we have shown that $E( x)=\mathcal O((1-x)^{-1/4})$ as $x\to 1$ and  $M(\lambda_x(z))h_{ u_{N,x}}(\lambda_x(z))^{\frac{\gamma}{{\sqrt{2}} i}\sigma_3}=\mathcal O(1)$ as $x\to 1$ (as follows from \eqref{def:M} and \eqref{asym:h}). The same argument is easily applied to $H^{-1}$.  We conclude that $H^{\pm 1}(x)=\mathcal O((1-x)^{-1/4})$, which was precisely the claim of \eqref{estimateH1}.

\medskip

{ \underline{Estimating $H^{-1}(x)H'(x)$.}} For \eqref{estimateH2}, 
we need more detailed expansions. Using \eqref{asym:h}  and \eqref{def:M}, it is straightforward to check that
\[M(\lambda)h_u(\lambda)^{\frac{\gamma}{{\sqrt{2}} i}\sigma_3}=\lambda^{-\frac{1}{4}\sigma_3}A(4u)^{\frac{3\gamma}{{2\sqrt{2}} i}\sigma_3}\left(I-\frac{3\gamma}{{4\sqrt{2}} i}\sigma_3(\lambda+1) +\mathcal O((\lambda+1)^2)\right),\]
as $\lambda\to -1$.  Let us now introduce the notation
\begin{equation}\label{defG} G(z)=E(z)\Xi(\lambda_x(z))\lambda{ _x(z)}^{-\frac{1}{4}\sigma_3}A.\end{equation}
Combining this and the above expansion with the definition of $H$ from \eqref{def:H}, we obtain after a short calculation that
\[
H(z)=G(z) (4u_{N,x})^{\frac{3\gamma}{2\sqrt{2}i}\gamma \sigma_3}\left(I-\frac{3\gamma}{{4\sqrt{2}} i}\sigma_3(\lambda_x(z)+1) +\mathcal O((\lambda_x(z)+1)^2)\right)e^{\frac{2i}{3}u_{N,x}^{3/2}\sigma_3}
\]
and it follows by  \eqref{eq:lambdx0}--\eqref{derivativezeta1} that
\begin{multline}\label{eq:Heq}
H^{-1}(x)H'(x) =  e^{\frac{-2i}{3}u_{N,x}^{3/2}\sigma_3}(4u_{N,x})^{-\frac{3\gamma}{{2\sqrt{2}} i}\sigma_3}
G^{-1}(x)G'(x)
(4u_{N,x})^{\frac{3\gamma}{{2\sqrt{2}}i}\sigma_3}
e^{\frac{2i}{3}u_{N,x}^{3/2}\sigma_3}\\
-\frac{3\gamma}{{4\sqrt{2}} i(1-x)}\sigma_3  +\mathcal O(1) 
,
\end{multline}
as $N\to\infty$ uniformly for $x\in (1-\epsilon,1-m/N^{2/3})$, for some $\epsilon,m>0$.

 Let
\begin{equation*}
\widehat \Xi(z)=A^{-1}\lambda{ _x(z)}^{\frac{1}{4}\sigma_3}\Xi(\lambda_x(z))\lambda{ _x(z)}^{-\frac{1}{4}\sigma_3}A.\end{equation*}
Then by \eqref{FormE} and \eqref{defG},
\begin{equation} \label{FormG}
G(z)=D_\infty^{\sigma_3}Q(z)F(z)\widehat \Xi(z).\end{equation}

Comparing with \eqref{estimateH2}, we see that our task is to estimate the $G^{-1}G'$-term.  
To do this, we first observe that, by \eqref{eq:smallnormPhi} and \eqref{eq:lambdx0}--\eqref{derivativezeta1},
\begin{equation}
\label{eq:Xiestimate}
\widehat \Xi(x)=I+\mathcal O(u_{N,x}^{-3/2}),\qquad \left(\widehat \Xi^{-1}\widehat \Xi'\right)(x)=\mathcal O\left(u_{N,x}^{-3/2}(1-x)^{-1}\right).
\end{equation}

Substituting \eqref{eq:Xiestimate} and \eqref{asymF} into \eqref{FormG}, and using the boundedness of $F_+'(x)$ as $x\to 1$, it follows that
\begin{equation*}
G^{-1}_+(x)G'_+(x)=Q^{-1}_+(x)Q'_+(x)\left(I+\mathcal O\left(u_{N,x}^{-3/2}\right)+\mathcal O(1-x)\right)+\mathcal O\left(u_{N,x}^{-3/2}(1-x)^{-1}\right)+\mathcal O(1),
\end{equation*}
as $N\to \infty$, uniformly for $x\in(1-\epsilon,1-m/N^{2/3})$, for some $\epsilon,m>0$.

By the fact that (see \eqref{def:Pinfty})
 \[\left(Q^{-1}Q'\right)(x)=\frac{i}{2 (1-x^2)}\begin{pmatrix}0&1\\-1&0\end{pmatrix},\]
and the fact that $u_{N,x}=\mathcal O\left(N^{2/3}(1-x)\right)$ as $N\to \infty$ uniformly for $x\in\left(1-\epsilon, 1-m/N^{2/3}\right)$ for some $\epsilon,m>0$, it follows that
\[G^{-1}_{ +}(x)G'_{ +}(x)=\frac{i}{2(1-x^2)}\begin{pmatrix}0&1\\-1&0\end{pmatrix}+\mathcal O(1)+\mathcal O\left(\frac{1}{N(1-x)^{5/2}}\right),\]
as $N\to \infty$, uniformly for $x\in(1-\epsilon,1-m/N^{2/3})$ for some $\epsilon,m>0$. Substituting this into \eqref{eq:Heq}, we conclude that
\begin{align*}
H^{-1}(x)H'(x)&=-\frac{3\gamma}{{4\sqrt{2}} i(1-x)}\sigma_3+\frac{i}{2(1-x^2)}
\begin{pmatrix}
0& e^{\frac{-4i}{3}u_{N,x}^{3/2}}(4u_{N,x})^{-\frac{3\gamma}{{\sqrt{2}} i}}\\
-e^{\frac{4i}{3}u_{N,x}^{3/2}}(4u_{N,x})^{\frac{3\gamma}{{\sqrt{2}} i}}&0
\end{pmatrix}
\\
&\quad +\mathcal O(1) +\mathcal O\left( \frac{1}{N(1-x)^{5/2}}\right)
.
\end{align*}
The last term integrates to $\mathcal O(N^{-1}(1-x)^{-3/2})$ which is of order one by our assumptions. All that remains to prove \eqref{estimateH2} is to show that the integral of the terms involving $u_{N,x}$ is bounded. Let 
\begin{equation}\nonumber
g_N(\eta)=-\frac{2}{3\pi}u_{N,\eta}^{3/2}+\frac{3\gamma}{2\sqrt{2}\pi}\log(4u_{N,\eta}).
\end{equation}
We will prove that
\begin{equation}\label{intgN0}
\int_{1-\epsilon}^{x}e^{2\pi i g_N(\eta)}\frac{d\eta}{1-\eta^2}=\mathcal O(1),
\end{equation}
as $N\to \infty,$ uniformly for $x\in\left(1-\epsilon, 1-m/N^{2/3}\right)$, which is precisely the integral of the $1,2$ entry in the matrix containing $u_{N,x}$. The $2,1$ entry is similar.
By \eqref{def:u}, we have
\begin{equation}\label{derivativegN}
\begin{aligned}
\frac{d}{d\eta}g_N(\eta)&=N\psi(\eta)\sqrt{1-\eta^2}-\frac{\gamma\psi(\eta)\sqrt{1-\eta^2}}{\sqrt{2}\pi\int_\eta^1 \psi(s)\sqrt{1-s^2}ds}\\
&=N\psi(\eta)\sqrt{1-\eta^2}\left(1+\mathcal O \left( \frac{1}{N(1-\eta)^{3/2}}\right)\right),
\end{aligned}\end{equation}
 uniformly for $\eta\in(1-\epsilon,1-m N^{-2/3})$.

To prove \eqref{intgN0}, and thus conclude the proof of the lemma, we first note that plugging \eqref{derivativegN} (in the form $1=\frac{g_N'(\eta)}{N\psi(\eta)\sqrt{1-\eta^2}}+\mathcal O(\frac{1}{N(1-\eta)^{3/2}})$) into the integral of \eqref{intgN0} and using the fact that $g_N$ is real-valued yields
\begin{align*}
\int_{1-\epsilon}^x e^{2\pi i g_N(\eta)}\frac{d\eta}{1-\eta^2}=\int_{1-\epsilon}^x e^{2\pi i g_N(\eta)}\frac{1}{1-\eta^2}\frac{g_N'(\eta)}{N\psi(\eta)\sqrt{1-\eta^2}}d\eta+\mathcal O\left(\frac{1}{N}\int_{1-\epsilon}^x \frac{1}{(1-\eta)^{5/2}}d\eta\right).
\end{align*}
For the last term, we note that the integral is $\mathcal O(N^{-1}(1-x)^{-3/2})$ which is $\mathcal O(1)$ under our assumptions, so it is sufficient to show that the first term on the right hand side is bounded. For this, we note that integrating by parts, 
\begin{align*}
\int_{1-\epsilon}^x e^{2\pi i g_N(\eta)}\frac{1}{1-\eta^2}\frac{g_N'(\eta)}{N\psi(\eta)\sqrt{1-\eta^2}}d\eta&=\mathcal O(N^{-1}(1-x)^{-3/2})\\
&\qquad -\frac{1}{2\pi i N}\int_{1-\epsilon}^x e^{2\pi i g_N(\eta)}\frac{d}{d\eta}\left(\frac{1}{\psi(\eta)(1-\eta^2)^{3/2}}\right)d\eta\\
&=\mathcal O(N^{-1}(1-x)^{-3/2})+\mathcal O\left(N^{-1}\int_{1-\epsilon}^x (1-\eta)^{-5/2}d\eta\right)\\
&=\mathcal O(N^{-1}(1-x)^{-3/2})
\end{align*}
which once again is bounded under our assumptions. This concludes the proof of both \eqref{intgN0} and the lemma.
\end{proof}

\appendix

\section{More on the log--correlated Gaussian process $X$}\label{app:logcor}

 In this appendix, we review some of the basic properties of the log-correlated field describing the limiting behavior of the eigenvalue counting function.
 
\subsection{Series representation and covariance structure} \label{sec:cov}

Recall that the Gaussian multiplicative chaos measure of Theorem~\ref{th:gmc} is associated with the (centered) log--correlated field $X$ on $(-1,1)$ with covariance kernel:
\[
\Sigma(x,y) = \log\bigg|\frac{1-xy+\sqrt{1-x^2}\sqrt{1-y^2}}{x-y}\bigg| . 
\]

As we pointed out in Section~\ref{sect:background}, in order to see that the kernel $\Sigma$ is positive definite on  $(-1,1)^2$,  one can represent $X$ as a random Chebyshev series:
\begin{equation} \label{X2}
X(x) :=  \sum_{k=1}^\infty \frac{\xi_k}{\sqrt{k/2}}U_{k-1}(x)\sqrt{1-x^2}, \qquad x\in(-1,1) , 
\end{equation}  
where $U_0, U_1, \dots$ are the usual Chebyshev polynomials of the second kind and $\xi_1, \xi_2, \dots$ are independent standard Gaussian random variables, see \eqref{eq:series}. 
Recall from Section \ref{sect:gmc} that such objects can be understood as random generalized functions in say $H^{-\epsilon}(\R^d)$. For the purposes of this appendix, we find it convenient to interpret the field $X$ as an element of the dual of the Sobolev space $\mathscr{H}^\epsilon= \big\{ g\in L^2(\mu_{\rm sc}) : \sum_{k=1}^{+\infty} k^\epsilon g_k^2 <+\infty \big\}$ for some $\epsilon>0$ where as in Section \ref{sec:intro}, $\frac{d\mu_{\rm sc}}{dx} = \frac 2\pi \sqrt{1-x^2} \1_{|x| \le 1}$ denotes the semicircle law, and $g_k=\int g U_k d\mu_{\rm sc}$. One can check that the above series converges almost surely in such a space. 
Indeed, since the Chebyshev polynomials of the second kind constitute an orthonormal basis with respect to the semicircle law, the meaning of the series \eqref{X2} is that for any $g\in \mathscr{H}^\epsilon$, 
\begin{equation} \label{X1}
\int_{-1}^1  g(x) X(x)  dx = \frac{\pi}{\sqrt 2} \sum_{k=1}^{+\infty} \frac{g_{k-1} \xi_k}{\sqrt{k}}. 
\end{equation}
Moreover, for all $k\in\N$ and $\theta\in [0,\pi]$, 
\[
U_{k-1}( \cos \theta) \sqrt{1-(\cos \theta)^2} =  \sin(k\theta) ,
\]
so that we have
\[
X(\cos \theta) =  \Im\bigg\{ \sum_{k=1}^\infty \frac{\xi_k}{\sqrt{k/2}} e^{\i k\theta} \bigg\} , \qquad \theta \in (0,\pi) . 
\]
 This reveals a connection between $X$ and the Gaussian Free Field (GFF).  
Namely, if we let $Z(z)$ be the restriction of the 2d GFF to the unit circle $|z| =1$,  normalized to have zero average over the unit circle, it is well-known that 
\begin{equation} \label{GFF}
Z(z) = \Re\bigg\{  \sum_{k=1}^\infty \frac{\zeta_k}{\sqrt{k/2}} z^k\bigg\}
\end{equation}
where $\zeta_1, \zeta_2$ are independent standard complex Gaussian random variables. Indeed, a straightforward computation shows that the covariance kernel of $Z$ is given by 
\begin{equation} \label{GFFcov}
\E Z(z) Z(z') = - \log|z-z'| , \qquad |z| = |z'| =1 . 
\end{equation}
Then, by choosing $\xi_k = - \sqrt{2} \Im(\zeta_k)$ for  $k\in\N$, we can couple our log-correlated fields $X$ and $Z$  in the following way:
\begin{equation} \label{X3}
X(\cos \theta)  = \frac{Z(e^{\i\theta}) - Z(e^{-\i\theta})  }{\sqrt{2}}  , \qquad \theta \in  (0,\pi) . 
\end{equation}
Now, using the representation \eqref{X3} and formula \eqref{GFFcov},  Thus, we verify that for any $\theta, \theta' \in  (0,\pi)$, 
\[ \begin{aligned}
\E{X(\cos \theta) X(\cos \theta')} & =  - \log|e^{\i \theta}- e^{\i \theta'}| + \log|e^{-\i \theta}- e^{\i \theta'}| \\
& = \frac 12 \log\bigg( \frac{1- \cos(\theta+\theta')}{1- \cos(\theta-\theta')} \bigg) . 
\end{aligned}\]
Using basic trigonometric expansion, this shows that  for all $x, y\in (-1,1)$,
\[ \begin{aligned}
\E{X(x) X(y)} 
& = \frac 12 \log\bigg( \frac{1- xy + \sqrt{1-x^2}\sqrt{1-y^2}}{1- xy- \sqrt{1-x^2}\sqrt{1-y^2}}\bigg)  \\
& = \Sigma(x,y) . 
\end{aligned}\]

\medskip

Now, let us explain the connection between the variance which appears in Johansson's CLT \eqref{eq:CLT} and our log--correlated field $X$. We claim that for any function $f \in C^1(\R)$ such that $f' \in \mathscr{H}^\epsilon$, by formula \eqref{X1}, 
\[ \begin{aligned}
\sigma^2(f) & =  \frac{1}{2\pi^2} \iint_{[-1,1]^2} f'(x) f'(y) \Sigma(x,y) dxdy \\
& = \frac{1}{2\pi^2} \E\left( \int_{-1}^1 f'(x) X(x) dx  \right)^2 \\
& = \frac{1}{4} \sum_{k=1}^{+\infty} \frac 1k (f')_{k-1}^2  .
\end{aligned}\]

If $T_1, T_2, \dots$ denote the Chebyshev polynomials of the first kind, we let $\displaystyle \hat{f}_{k} = \frac{2}{\pi} \int_{-1}^1 \frac{f(x) T_k(x)}{\sqrt{1-x^2}}dx$ for $k\in\N$ be the  Fourier--Chebyshev coefficients of the function $f$. 
Under our assumptions, we have $(f')_{k-1} = k\hat{f}_{k}$, see e.g. \cite[Formula (4.11)]{LLW}. 
This implies that for any function $f \in C^1(\R)$ such that $f' \in \mathscr{H}^\epsilon$, 
\begin{equation*}
\sigma^2(f) = \frac{1}{4} \sum_{k=1}^{+\infty} k \hat{f}_{k}^2 , 
\end{equation*}
which is the formula given by Johansson in \cite{Johansson}.  Using this representation for the variance $\sigma^2(f)$, it is easy to verify that under the same assumptions on $f$, we also have
\begin{equation*} 
\sigma^2(f) = -\frac{1}{2\pi}  \int_{-1}^1 f'(t) \mathcal U f (t)\sqrt{1-t^2} dt ,
\end{equation*}
where $\mathcal U$ denotes the finite Hilbert transform \eqref{U_transform}, 
see e.g. \cite[Formula (4.15)]{LLW}. By polarization, we define for any smooth functions $f,g :\R \to \R$,
\begin{equation} \label{cov3} 
\begin{aligned} 
\sigma^2(f;g) 
& =  \frac{1}{2\pi^2} \iint_{[-1,1]^2} f'(x) g'(y) \Sigma(x,y) dxdy \\
& = -\frac{1}{2\pi}  \int_{-1}^1 f'(t) \mathcal U g (t)\sqrt{1-t^2} dt . 
\end{aligned}\end{equation}

\subsection{Harmonic extension}

We consider the following approximation of the log--correlated field $X$: for any $\eta>0$ and $x\in \R$, let  
\begin{equation} \label{X4}
X_\eta(x) =  \int_{-1}^{1}  \varphi_\eta(x-u) X(u) du 
\end{equation}
where $\varphi(t) = \frac{1}{\pi(1+t^2)}$ and $\varphi_\eta(t) = \eta^{-1}\varphi(t \eta^{-1})$ for any $\eta>0$ and $t\in\R$. By our previous discussion, the Gaussian random variables \eqref{X4} are well-defined, centered and have covariance:
\begin{equation} \label{cov1}
\E  X_\eta(x) X_\epsilon(y)  = \iint_{[-1,1]^2}   \varphi_\eta(u-x)   \varphi_\epsilon(v-y) \Sigma(u,v) dudv  , \qquad  x, y\in \R , \ \eta, \epsilon>0 . 
\end{equation}
Recall that for any $\eta>0$ and $x\in\R$, we denote by $w_{\eta,x} = \sqrt{2} \pi  \varphi_\eta*\1_{(-\infty, x]}$. 
In particular, it is of relevance in Section~\ref{sec:special} to observe that
\begin{equation} \label{cov2}
\E  X_\eta(x) X_\epsilon(y)  =\sigma^2( w_{\eta,x}  ; w_{\epsilon,y} )  , \qquad  x, y\in \R , \ \eta, \epsilon>0 . 
\end{equation}
This follows directly from \eqref{cov3} and the fact that for any $\eta>0$ and $u,x\in\R$,
\begin{equation} \label{w'}
w_{\eta,x}'(u) =  \sqrt{2}\pi \int_{-\infty}^x \eta^{-2} \varphi'(\tfrac{u-t}{\eta}) dt = -\sqrt{2}\pi\varphi_\eta(u-x) . 
\end{equation}

Let us point out that the (upper half-plane) Poisson kernel $\varphi$ is not in the usual class of mollifiers used in GMC theory, \cite{Berestycki15}, but it is a very natural choice in our setting because $X_\eta$ corresponds to the harmonic extension of $X$ in the upper-half plane $\mathds{H}$. Then, according to formula \eqref{eq:hcomplex}, for any $x\in\R$ and $\eta>0$, as $N\to+\infty$
\begin{equation} \label{extension_cvg}
h_N(x+\i \eta)\ \Rightarrow\ X_\eta(x). 
\end{equation}
Moreover, it is not difficult to turn \eqref{extension_cvg} to a functional convergence in the space of  continuous functions on $\mathds{H}$ equipped with the topology of uniform convergence in compact subsets of $\mathds{H}$.  Let us also observe that almost surely, the series \eqref{GFF} converges uniformly for  all $|z| < r<1$ and that $Z$  corresponds to the harmonic extension of the Gaussian Free Field defined in the domain $\C\setminus\mathds{D}$   inside of the unit disk $\mathds{D}$ and we verify that 
\[
\E Z(z) Z(z') = - \log| 1- z'\overline{z}| , \qquad z, z' \in \mathds{D}.
\]
So, if $\varpi :  \C\setminus [-1,1]\to \mathbb{D} $ is the conformal bijection given by
\[
\varpi (z) = z- \sqrt{z^2-1} ,
\]
(with the branch of the root being fixed by the condition of being a conformal bijection)  by \eqref{X3}, we verify that for any $\epsilon>0$ and $x\in\R$
\[
X_\eta(x)  = \frac{Z\big(\varpi(x+\i \eta)\big) - Z\big(\overline{\varpi(x+\i \eta)}\big)  }{\sqrt{2}}  .
\]
This gives us an \emph{explicit formula} for the covariance \eqref{cov2}:
\begin{equation*}
\E  X_\eta(x) X_\epsilon(y)  
=  -\log\bigg|\frac{1-\varpi(x+\i \eta)\overline{\varpi(y+\i \epsilon)}}{1-  \varpi(x+\i \eta)\varpi(y+\i \epsilon)} \bigg|
, \qquad  x, y\in \R , \ \eta, \epsilon>0 . 
\end{equation*}

\medskip

Now, let us check that the covariance formulae used in Section~\ref{sec:special} hold.  In particular, we define for any $\eta>0$ and $x, y \in (-1,1)$, 
\begin{equation} \label{cov4}
 \E  X_\eta(x) X(y)  := \int_{-1}^1 \Sigma(y,t)  \varphi_\eta(x-t) dt . 
\end{equation}
Note that this function is continuous on $(-1,1)^2$
and we easily verify using formula \eqref{cov1} and the definition of $\Sigma$ that for all $x,y\in(-1,1)$,
\begin{equation} \label{cov5}
\E  X_\eta(x) X(y)  = \lim_{\epsilon\to 0 } \E  X_\eta(x) X_\epsilon(y) 
\end{equation}
and $\displaystyle \Sigma(x,y) = \lim_{\eta\to0 }\E  X_\eta(x) X(y) $ for all $x\neq y$. 
In particular, using the representation \eqref{cov3} and \eqref{cov2}, \eqref{w'}, \eqref{cov5},   this shows that
\[
\E  X_\eta(x) X(y)  = \lim_{\epsilon\to 0 } \frac{1}{\sqrt{2}}   \int_{-1}^1 \varphi_\epsilon(t-y)  \mathcal U w_{\eta,x}(t) \sqrt{1-t^2} dt ,
\]
Since the function $w_{\eta,x}$ is smooth on $\R$, so is its finite Hilbert transform and the function $t \mapsto \mathcal U w_{\eta,x}(t) \sqrt{1-t^2} \1_{|t| \le 1}$ is continuous and bounded on $\R$.  
Since $\varphi_\epsilon(t-y) dt \to \delta_y$ weakly, this implies that  
\begin{equation} \label{cov6}
\E  X_\eta(x) X(y) = \frac{1}{\sqrt{2}} \mathcal U w_{\eta,x}(y) \sqrt{1-y^2} .
\end{equation}

\section{The Airy model RH problem}\label{section:appendixAiry}

In this appendix we review the basic facts about the Airy model RH problem. For details, see e.g. \cite[Section 7]{DKMVZ}.

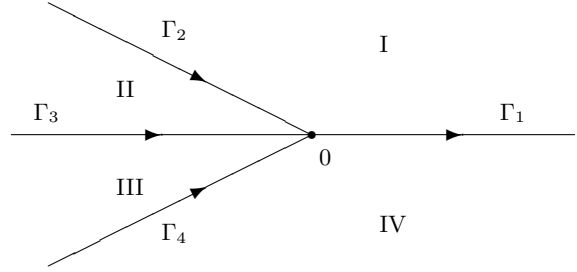
\begin{figure}[h]
    \begin{center}
    \setlength{\unitlength}{1mm}
    \begin{picture}(95,47)(0,2)
        \put(30,38){\small $\Gamma_2$}
        \put(13,27){\small $\Gamma_3$}
        \put(30,11){\small $\Gamma_4$}
        \put(75,27){\small $\Gamma_1$}

        \put(59,36){\small ${\rm I}$}
        \put(24,30){\small ${\rm II}$}
        \put(24,17){\small ${\rm III}$}
        \put(59,12){\small ${\rm IV}$}

        \put(50,25){\thicklines\circle*{.9}}
        \put(51,21){\small 0}

        \put(50,25){\line(-2,1){35}} \put(36,32){\thicklines\vector(2,-1){.0001}}
        \put(50,25){\line(-2,-1){35}} \put(36,18){\thicklines\vector(2,1){.0001}}
        \put(50,25){\line(-1,0){40}} \put(30,25){\thicklines\vector(1,0){.0001}}
        \put(50,25){\line(1,0){35}} \put(70,25){\thicklines\vector(1,0){.0001}}
    \end{picture}
    \caption{The oriented contour $\Gamma$. }
        \label{figure: Gamma2}
    \end{center}
\end{figure}

\subsubsection*{RH problem for $\Phi_{\rm Ai}$:}
\begin{itemize}
    \item[(a)] $\Phi=\Phi_{\rm Ai}$ is analytic in $\mathbb{C}\setminus \left(\mathbb R\cup e^{\pm \frac{2i\pi}{3}}\mathbb R^+\right)$.
    \item[(b)] $\Phi$ satisfies the following jump relations on $\Gamma$,
        \begin{align}
            \label{RHP Psi: b1}
            \Phi_+(z) &= \Phi_-(z)
                \begin{pmatrix}
                    1 & 1\\
                    0 & 1
                \end{pmatrix},
                &\mbox{for $z\in (0,+\infty)$,} \\[1ex]
            \label{RHP Psi: b2}
            \Phi_+(z) &= \Phi_-(z)
                \begin{pmatrix}
                    1 & 0\\
                    1 & 1
                \end{pmatrix},
                &\mbox{for $z\in e^{\pm \frac{2i\pi}{3}}\mathbb R^+$,} \\[1ex]
            \label{RHP Psi: b3}
            \Phi_+(z) &= \Phi_-(z)
                \begin{pmatrix}
                    0 & 1\\
                    -1 & 0
                \end{pmatrix},
                &\mbox{for $z\in (-\infty,0)$.}
        \end{align}
    \item[(c)] $\Phi$ has the following behavior at infinity,
        \begin{equation}\label{RHP Psi: c}
            \Phi(z)=
                z^{-\frac{1}{4}\sigma_3}A\left(I+\bigO(z^{-3/2})\right)e^{-\frac{2}{3}z^{3/2}\sigma_3}, \qquad
                \mbox{as $z\to\infty$,} 
        \end{equation}
        uniformly for $z\to\infty$ and not on the jump contour,
                where
                \begin{equation}
                A=\frac{1}{\sqrt{2}}\begin{pmatrix}1&i\\i&1\end{pmatrix}.
                \end{equation}
     \item[(d)] $\Phi$ is bounded at zero.
\end{itemize}

Define $y_j$ in terms of the Airy function:
\[
    y_j=y_j(z;r)=\omega^j{\rm Ai}(\omega^j z),\qquad j=0,1,2,
\]
with $\omega=e^{\frac{2\pi i}{3}}$.
The unique solution to the RH problem is given by
\begin{equation}
    \Phi(z)=\sqrt{2\pi}\times 
    \begin{cases}
        \begin{pmatrix}
            y_0 & -y_2\\
            -iy_0' & iy_2'
        \end{pmatrix}, & \quad\mbox{for $z\in {\rm I}$,}
        \\[3ex]
        \begin{pmatrix}
            -y_1 & -y_2\\
            iy_1' & iy_2'
        \end{pmatrix}, & \quad\mbox{for $z\in {\rm II}$,}
        \\[3ex]
        \begin{pmatrix}
            -y_2 & y_1\\
            iy_2' & -iy_1'
        \end{pmatrix}, & \quad\mbox{for $z\in {\rm III}$,}
        \\[3ex]
        \begin{pmatrix}
            y_0 & y_1\\
            -iy_0' & -iy_1'
        \end{pmatrix}, & \quad\mbox{for $z\in {\rm IV}$.}
    \end{cases}
\end{equation}

\section{The confluent hypergeometric model RH problem}\label{section:appendixCFH}
In this appendix we review some basic facts about the hypergeometric model RH problem. For further details, we refer the reader e.g. to \cite{DIK11}, though note that the construction there is more general and the conventions are slightly different. The model RH problem of relevance to us is the following.

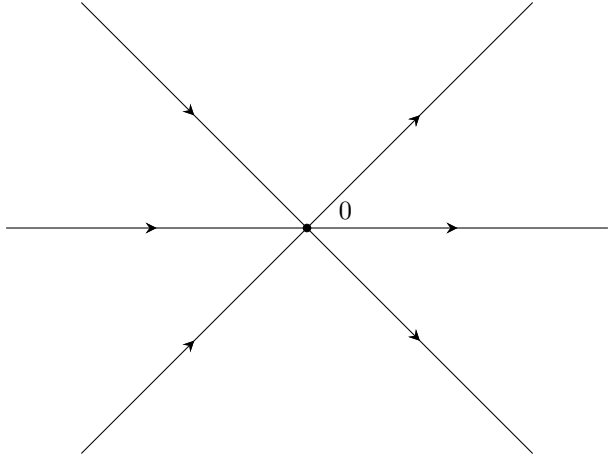
\begin{figure}
\begin{center}
\begin{tikzpicture}
\draw [decoration={markings, mark=at position 0.25 with {\arrow[thick]{>}}},
        postaction={decorate}, decoration={markings, mark=at position 0.75 with {\arrow[thick]{>}}},
        postaction={decorate}] (-1,0) -- (7,0);
\draw [decoration={markings, mark=at position 0.25 with {\arrow[thick]{>}}},
        postaction={decorate},decoration={markings, mark=at position 0.75 with {\arrow[thick]{>}}},
        postaction={decorate}] (0,3) -- (6,-3);        
\draw [decoration={markings, mark=at position 0.25 with {\arrow[thick]{>}}},
        postaction={decorate},decoration={markings, mark=at position 0.75 with {\arrow[thick]{>}}},
        postaction={decorate}] (0,-3) -- (6,3);
\node [above right] at (3.3,0) {$0$}; 
\draw[fill] (3,0) circle [radius=0.05];  
\end{tikzpicture}
\end{center}
\caption{The jump contour for $\Phi_{\rm HG}$}\label{fig:ContourCHG}
\end{figure}
 \subsubsection*{RH problem for $\Phi_{\rm HG}$}
\begin{itemize}
\item[(a)]
$\Phi=\Phi_{\rm HG}$ is analytic for $z\in \mathbb C\setminus \left(\mathbb R\cup  e^{\frac{i\pi}{4}}\mathbb R\cup e^{\frac{-i\pi}{4}}\mathbb R\right)$,  with the orientation chosen as in Figure \ref{fig:ContourCHG}. 
\item[(b)]On the jump contour, $\Phi$ has the jumps
\begin{equation}\label{eq:HGjump}
\begin{aligned}
 \Phi_+(z)&=\Phi_-(z) \begin{pmatrix}
1&0\\e^{\pi i \beta}&1
\end{pmatrix}	&& \textrm{for } z\in e^{\pm \frac{i\pi}{4}}\mathbb R^+,
\\
\Phi_+(z)&=\Phi_-(z) \begin{pmatrix}
1&0\\e^{-\pi i \beta}&1
\end{pmatrix}	&& \textrm{for } z\in e^{\pm \frac{3i\pi}{4}}\mathbb R^+,
\\
\Phi_+(z)&=\Phi_-(z) \begin{pmatrix}
0&e^{\pi i \beta}\\-e^{-\pi i \beta}&0
\end{pmatrix}	&& \textrm{for } z\in (-\infty,0),
\\
\Phi_+(z)&=\Phi_-(z) \begin{pmatrix}
0&e^{-\pi i \beta}\\-e^{\pi i \beta}&0
\end{pmatrix}	&& \textrm{for } z\in (0,\infty).
\end{aligned}
\end{equation}
\item[(c)]
As $z\to \infty$,
\begin{equation}\label{eq:HGasy}
\Phi(z)=e^{-\frac{\pi i}{2}\beta \sigma_3}(I+\mathcal O(z^{-1}))(z)^{-\beta \sigma_3}e^{-i\frac{z}{2}\sigma_3}\chi(z),
\end{equation}
where $\chi(z)$ is constant on each quarter of the plane:
\begin{equation*}\chi(z)=\begin{cases}
e^{\pi i\beta\sigma_3}, &\arg z\in(0,\pi),\\
\begin{pmatrix}
0&-1\\1&0
\end{pmatrix},&\arg z\in(\pi,2\pi).
\end{cases}
\end{equation*}
\item[(d)] As $z\to 0$, 
\begin{equation}\label{eq:PhiHG-log}
\Phi(z)=\mathcal O(|\log z|)
\end{equation}
\end{itemize}

The solution to this problem can be constructed in terms of confluent hypergeometric functions, see for instance \cite{ItsKrasovsky, DIK11, FoulquieMartinezSousa} or the appendix of \cite{Charlier} for various statements of this flavor --  for this exact formulation, see \cite[formula (4.9)]{CK}. The only thing we need to know here, is that
we have for $\arg z\in (\pi/4, \pi/2)$,
\begin{equation}
\Phi_{\rm HG}(z)=
\Bigg(\begin{matrix}
\Gamma(1-\beta)M(\beta,1,e^{\pi i/2}z)
\\ \Gamma(1+\beta)M(1+\beta,1,e^{\pi i/2}z)
\end{matrix}
\\
\begin{matrix}
-\frac{\Gamma(1-\beta)}{\Gamma(\beta)}U(1-\beta,1,e^{-\pi i/2}z)
\\U(-\beta,1,e^{-\pi i/2}z)
\end{matrix}\Bigg){
e^{-i\frac{z}{2}\sigma_3}},
\end{equation}
where $M(a,b,z)=1+\sum_{j=1}^\infty \frac{(a)_j}{(b)_jj!}z^j$ is Kummer's confluent hypergeometric function, and $U$ is the confluent hypergeometric function of the second kind -- we do not need the precise definition here, but for further details, see e.g. \cite[Section 13.2]{NIST}. 

\begin{lemma}\label{Lemma1}
For $\Re \beta=0$, if we take the limit where $z$ approaches $0$ with $\arg z\in (\pi/4,\pi/2)$, we have
\begin{equation}
\lim_{z\to 0}\Phi_{{\rm HG}, 11}(z)=
\Gamma(1-\beta),\qquad \lim_{z\to 0}\Phi_{{\rm HG}, 21}(z)=
\Gamma(1+\beta),
\end{equation}
and
\begin{equation}
\lim_{z\to 0}\left(\Phi_{\rm HG}^{-1}(z)\frac{d}{dz}\Phi_{\rm HG}(z)\right)_{21}={ -\frac{2\pi  \beta}{e^{\pi i \beta}-e^{-\pi i \beta}}}.
\end{equation}
\end{lemma}
\begin{proof}
By the definition of $M$,
\begin{equation}
\Phi_{\rm HG}(z)=
\begin{pmatrix}
\Gamma(1-\beta)&*\\
\Gamma(1+\beta)
&*
\end{pmatrix}
+iz{
\begin{pmatrix}
 (\beta-1/2)\Gamma(1-\beta)&*\\
(\beta+1/2)\Gamma(1+\beta)
&*
\end{pmatrix}}+\mathcal O(z^2)
\end{equation}
as $z\to 0$
 such that  $\arg z\in (\pi/4,\pi/2)$, and the result follows from the identity $\Re \beta=0$, $\Gamma(\beta)\Gamma(-\beta)={ -\frac{2\pi i}{\beta(e^{\pi i \beta}-e^{-\pi i \beta})}}$.
\end{proof}

We use this model RH problem in Section \ref{sec:model} in the case $\beta=\frac{\gamma}{{\sqrt{2}} i}\in i\mathbb R$. We should emphasize also that we use in Section \ref{sec:model} that the error term in \eqref{eq:HGasy} is uniform for $\gamma\in[-\Gamma,\Gamma]$, or in other words for $\beta$ in a compact subset of the imaginary line. This can be verified via the $z\to\infty$ asymptotics of Kummer's confluent hypergeometric function (see e.g. \cite[formulas 13.2.4 and 13.7.4-5]{NIST}), which are uniform in $a$.

\end{document}